\documentclass[12pt,a4paper]{amsart}
\usepackage{amssymb}
\usepackage{amsfonts}
\usepackage{amsmath}
\usepackage[mathscr]{eucal}

\usepackage{mathrsfs}

\usepackage{color}

\newtheorem{thm}{Theorem}[section]
\newtheorem{prop}[thm]{Proposition}
\newtheorem{lem}[thm]{Lemma}
\newtheorem{cor}[thm]{Corollary}
\newtheorem{df}[thm]{Definition}
\numberwithin{equation}{section}

\def\N{{\Bbb N}}
\def\Z{{\Bbb Z}}
\def\Q{{\Bbb Q}}
\def\R{{\Bbb R}}
\def\C{{\Bbb C}}
\def\A{{\Bbb A}}

\def\emp{\varnothing}

\def\fd{{\frak d}}
\def\fe{{\frak e}}

\def\fg{{\frak g}}
\def\fh{{\frak h}}

\def\fm{{\frak m}}

\def\fp{{\frak p}}
\def\fq{{\frak q}}

\def\ft{{\frak t}}

\def\fz{{\frak z}}

\def\fS{{\frak S}}

\def\fB{{\frak B}}

\def\fI{{\frak I}}
\def\fJ{{\frak J}}

\def\fL{{\frak L}}

\def\fX{{\frak X}}
\def\fN{{\frak N}}
\def\sG{\mathsf G}
\def\sH{\mathsf H}

\def\sN{\mathsf N}

\def\sa{{\mathsf a}}

\def\sn{\mathsf n}

\def\sf{\mathsf f}

\def\su{\mathsf u}

\def\sT{{\mathsf T}}

\def\cB{{\mathscr B}}

\def\cD{{\mathcal D}}

\def\cP{{\mathcal P}}
\def\cH{{\mathscr H}}

\def\cN{{\mathcal N}}

\def\cT{{\mathcal T}}

\def\cU{{\mathcal U}}

\def\GL{{\operatorname {GL}}}

\def\SL{{\operatorname{SL}}}
\def\SO{{\operatorname{SO}}}

\def\PGL{{\operatorname{{\mathbf {PGL}}}}}

\def\Re{{\operatorname {Re}}}
\def\Im{{\operatorname {Im}}}
\def\tr{{\operatorname{tr}}}

\def\Mat{{\operatorname{{\bf M}}}}

\def\Hom{{\rm{Hom}}}
\def\diag{{\operatorname {diag}}}

\def\Ad{{\operatorname{Ad}}} 
\def\vol{{\operatorname{vol}}}

\def\leq{\leqslant}
\def\geq{\geqslant}
\def\bsl{\backslash}

\def\d {{{\rm d}}}

\def\JJ{{\Bbb J}}

\def\bK{{\mathsf K}}

\def\1{{\bold 1}}

\def\fB{{\mathcal B}}

\renewcommand{\a}{\alpha}
\renewcommand{\b}{\beta}

\newcommand{\e}{\epsilon}

\renewcommand{\l}{\lambda}

\newcommand{\s}{\sigma}

\newcommand{\Bcal}{{\mathscr B}}

\newcommand{\Ecal}{{\mathcal E}}

\newcommand{\Ical}{{\mathcal I}}

\newcommand{\Ncal}{{\mathcal N}}

\newcommand{\Pcal}{{\mathcal P}}

\newcommand{\Scal}{{\mathscr S}}

\newcommand{\Ucal}{{\mathcal U}}
\newcommand{\Vcal}{{\mathcal V}}
\newcommand{\Wcal}{{\mathscr W}}
\newcommand{\Xcal}{{\mathcal X}}

\newcommand{\PP}{\frak{P}}

\newcommand{\TT}{\mathbb{T}}

\newcommand{\bfs}{{\mathbf s}}

\def\GL{{\mathbf {GL}}}

\newcommand{\fin}{{\rm fin}}
\renewcommand{\Re}{\operatorname{Re}}

\def\fD{{\frak D}}

\def\ii{{\mathsf i}}
\def\XX{\Bbb X}

\def\sS{{\mathsf S}}

\def\sF{{\mathsf X}}
\def\WW{{\Bbb W}}
\def\taut{{}^\lozenge}
\def\sB{{\mathbf B}}
\def\sT{{\mathbf T}}
\def\sU{{\mathbf U}}
\def\sG{{\mathbf G}}
\def\sH{{\mathbf H}}
\def\sZ{{\mathbf Z}}

\def\fN{{\mathscr J}}
\newcommand{\cchi}{\operatorname{\mbox{{\rm 1}}\hspace{-0.25em}\mbox{{\rm l}}}}
\def\fJ{\fI}
\def\nnu{\lambda}
\def\rP{{\rm P}}
\def\rN{{\rm N}}
\def\fQ{{\frak q}}

\def\DD{{\mathbb D}}
\def\HH{{\mathbb H}}
\def\iiota{{\rm h}}
\def\vf{\xi}
\def\GL{{\mathrm {GL}}}
\def\PGL{{\mathrm {PGL}}}
\def\SL{{\mathrm {SL}}}
\def\iita{{\rm j}}
\def\TT{{\mathsf g}}
\def\tint{{\textstyle \int}}

\title{Hecke-Maass cusp forms on $\PGL_n$ of large levels with non-vanishing $L$-values}

\author{Masao Tsuzuki}
\address{Faculty of Science and Technology, Sophia University, Kioi-cho 7-1 Chiyoda-ku Tokyo, 102-8554, Japan}
\email{m-tsuduk@sophia.ac.jp}

\begin{document}
\maketitle
\begin{abstract}
We introduce a new trace formula of Kuznetsov type involving the central standard $L$-values and the Whittaker periods of cuspidal automorphic representations of $\PGL_n(\Q)$ $(n\geq 3)$ which are spherical at the archimedean place. As an application, we show a simultaneous non-vanishing of standard $L$-values at $n-1$ points on the critical strip for infinitely many Hecke-Maass cuspidal newforms of growing prime levels. 
\end{abstract}

\section{Introduction} \label{Intro}
Let $\A$ be the adele ring of $\Q$ and $\A_\fin$ the subring of finite adeles. Let $n\geq 3$ be an integer, which is fixed throughout this article. For a positive integer $N$, define an open compact subgroup $\bK_1(N)$ of $\GL_n(\A_\fin)$ as $\prod_{p<\infty}\bK_{1}(N\Z_p)$, where for any prime $p<\infty$ 
\begin{align}
\bK_1(N\Z_p)=\{(k_{ij})_{ij}\in \GL_n(\Z_p)|\,k_{ni}\in N\Z_p\,(1\leq i\leq n-1),\,k_{nn}-1\in N\Z_p\}.
 \label{LocalK1fa}
\end{align}
For a positive integer $M$, let $S(M)$ denote the set of all the prime divisors of $M$. In this paper, we shall prove the following.  
\begin{thm} \label{MAINTHM2.5}
Fix a point $\nu=(\nu_j)_{j=1}^{n-1}\in \C^{n-1}$ such that 
$$|\Re \nu_j| < \tfrac{1}{2}\quad(1\leq j\leq n-1).$$
 Let $M>1$ be an integer, and $\tau_p$ for $p\in S(M)$ an irreducible smooth supercuspidal representation of $\GL_n(\Q_p)$ of conductor $M\Z_p$ with trivial central character. Let $S_0$ be a finite set of prime numbers not dividing $M$. Then there exists a constant $N_0\,(>M\prod_{p\in S_0}p)$ such that for any prime number $N>N_0$, there exists an irreducible cuspidal automorphic representation $\pi\cong \otimes_v\pi_v$ of $\GL_n(\A)$ with trivial central character possessing the following properties:
\begin{itemize}
\item[(i)] $\pi_{\infty}$ has non-zero vectors fixed by the orthogonal group ${\rm O}_n(\R)$ of degree $n$.
\item[(ii)] $\pi_p \cong \tau_p$ for all $p\in S(M)$ and the conductor of $\pi$ is $MN$. 
\item[(iii)] $\pi_v$ is not self-dual for all $v\in S_0\cup\{\infty\}$. 
\item[(iv)] $L\left(\tfrac{1}{2}+\nu_j,\pi\right)\not=0$ for all $1\leq j\leq n-1$, where $L(s,\pi)$ is the completed standard $L$-function of $\pi$. 
\end{itemize}
\end{thm}
%For the description of the $\bK_v$-spherical tempered unitary dual of $\PGL_n(\Q_v)$, we refer to \S~\ref{sec:URps}. 
The theorem may be viewed as a form of result which ensures the existence of a globalization of an arbitrary irreducible supercuspidal representation of a reductive group over a local field (\cite{Vigneras}, \cite{Henniart}, \cite{Shahidi}, \cite{Muic}). Among such, one of the novelties in Theorem~\ref{MAINTHM2.5} is that the cuspidal automorphic representation $\pi$ can be found in such a way that its standard $L$-values are non-zero simultaneously at the $n-1$ prescribed points inside the critical strip; actually, the theorem yields infinitely many cuspidal representations with that property. We remark that the infinitum of non-self dual Maass cusp forms on $\PGL_n$ $(n\geq 3)$ is known by \cite{Kala}. 
\subsection{Asymptotic formula} \label{GWP}
The proof of Theorem~\ref{MAINTHM2.5} relies on an asymptotic formula to be stated in Theorem~\ref{MAINTHM1}; for its description, we need notation. Let $\psi=\prod_{v}\psi_v$ be a character of $\Q\bsl \A$ such that $\psi_\infty(x)=e^{2\pi \ii x}$ for $x\in \R$, where $\ii=\sqrt{-1}\in \C$ is the imaginary unit. Let us fix a Haar measure $\d x_v$ on $\Q_v$ such that $\vol(\Z_p)=1$ if $v=p<\infty$ and $\vol([0,1])=1$ if $v=\infty$, and endow $\A$ with the Haar measure $\d x=\prod_{v}\d x_v$. 
%We use the Tamagawa measures on the group $\GL_{m}(\Q_v)\,(m\geq 1)$ of $\Q_v$-points and on the adelization $\GL_{m}(\A)\,(m\geq 1)$ defined as follows. 
Set 
\begin{align}
\Delta_{\GL_m}(z)=\prod_{j=1}^{m}\zeta_\Q(z+j-1)
 \label{DelsG0}
\end{align}
and $\Delta_{\GL_m}(1)^{*}:={\rm Res}_{z=1}\Delta_{\GL_m}(z)$ with $\zeta_\Q(z)=\Gamma_{\R}(z)\prod_{p<\infty}(1-p^{-z})^{-1}\,(\Re z>1)$ the completed Riemann zeta-function, where $\Gamma_{\R}(z):=\pi^{-z/2}\Gamma(z/2)$. The global Tamagawa measure on $\GL_m(\A)$ is defined to be 
\begin{align*}
\d g=(\Delta_{\GL_m}(1)^{*})^{-1} \,\prod_{v} \d g_v,\quad g=(g_v)_{v}\in \GL_m(\A), 
\end{align*}
where $\d g_v$ for $v\leq \infty$ is the normalized Tamagawa measure on $\GL_m(\Q_v)$ given as 
\begin{align}
\d g_v=\Delta_{\GL_m,v}(1)|\det g_v|_v^{-m}\prod_{ij}\d g_{ij}, \quad g_v=(g_{ij})_{ij} \in \GL_{m}(\Q_v)
 \label{MeasGv}
\end{align}
with $\Delta_{\GL_m,v}(z)$ the Euler $v$-factor of $\Delta_{\GL_m}(z)$. Note that $\int_{\GL_m(\Z_p)}\d g_p=1$ for $p<\infty$. Fix $n\geq 3$. Set $\sG=\GL_n$ and $\sZ$ denote the center of $\sG$. Note that, by the natural isomorphism $\sZ\cong \GL_1$, $\sZ(\Q_v)$ and $\sZ(\A)$ are endowed with the Haar measures. Let $\bK_v$ be the standard maximal compact subgroup of $\sG(\Q_v)$, i.e., $\bK_v=\sG(\Z_p)$ if $v=p<\infty$ and $\bK_\infty$ is the orthogonal group ${\rm O}_n(\R)$ if $v=\infty$. Set $\bK=\prod_{v}\bK_v$ viewed as a subgroup of $\sG(\A)$. For $f_v\in C_{\rm c}^{\infty}(\bK_v\bsl \sG(\Q_v)/\bK_v)$ and an admissible unitary representation $\pi_v$ of $\sG(\Q_v)$, let $\widehat{f_v}(\pi_v)$ be the trace of the finite rank operator $\pi_v(f_v)=\int_{\sG(\Q_v)}f_v(g_v)\pi_v(g_v)\d g_v$ on the space of $\pi_v$. For $s\in \C^{n}$, let $I_v^{\sG}(s)$ be the unramified principal series representation of $\sG(\Q_v)$ (see \S\ref{sec:URps}). For an irreducible unitary generic $\bK_v$-spherical representation $\pi_v$ of $\sG(\Q_v)$, there exists a unique point $s \in \C^{n}/\sS_n$ such that $\pi_v\cong I^{\sG}_v(s)$ (see \S\ref{sec:Lfunction}); we write $\widehat {f_v}(s)$ in place of $\widehat{f_v}(I_v^{\sG}(s))$. Let $\Pi_{\rm cusp}(\sG)_{\sZ}$ denote the set of all the irreducible cuspidal automorphic representations of $\sG(\A)$ with trivial central character. It should be recalled that for $\pi\cong \otimes_v\pi_v\in \Pi_{\rm cusp}(\sG)_{\sZ}$, $\pi_v$ is unitarizable and generic for all $v$ (\cite{Shalika}). Let $\bar \pi$ denote the complex conjugate representation of $\pi\in \Pi_{\rm cups}(\sG)_{\sZ}$, which is isomorphic to the contragredient $\pi^\vee$ of $\pi$. The adjoint $L$-function of $\pi\in \Pi_{\rm cusp}(\sG)_{\sZ}$ is defined as $L(z,\pi;\Ad)=\zeta(z)^{-1}L(z,\pi \times \bar \pi)$. Here and elsewhere in this article, all the $L$-functions are completed by appropriate gamma factors. Let $\nu=(\nu_i)_{i=1}^{n-1}\in \C^{n-1}$. By the theory of Eisenstein series, the isobaric sum $\boxplus_{i=1}^{n-1}|\,|_\A^{\nu_i}$ of quasi-characters $|\,|_\A^{\nu_j}\,(1\leq j \leq n-1)$ of the idele class group $\Q^{\times}\bsl \A^{\times}$ has a realization in the space of automorphic forms on $\GL_{n-1}(\A)$ (\cite{Langlands1979}). The Rankin-Selberg $L$-function $L(z, \boxplus_{i=1}^{n-1}|\,|_\A^{\nu_i} \times \bar \pi)$ is defined to be the product $\prod_{j=1}^{n-1} L(z+\nu_j,\bar \pi)$ of the shifted standard $L$-function of $\bar \pi$, whose value at $z=1/2$ is our main concern in this article. Let $K_\fin=\prod_{p<\infty}K_p$ be an open compact subgroup of $\sG(\A_\fin)$ and set $K=K_\fin \bK_\infty$. Let $\Pi_{\rm cusp}(\sG)_{\sZ}^{K}$ be the set of all those $\pi \in \Pi_{\rm cusp}(\sG)_{\sZ}$ having non-zero $K$-fixed vectors. Let $f=\otimes_v f_v$ be a decomposable function in $C_{\rm{c}}^{\infty}(K\bsl \sG(\A)/K)$ with $f_p$ being the characteristic function of $\bK_p$ for almost all primes $p$. Then we shall study the values $L(\frac{1}{2}, \boxplus_{i=1}^{n-1}|\,|_\A^{\nu_i} \times \bar \pi)$ ($\pi \in \Pi_{\rm cusp}(\sG)^{K}_{\sZ}$) collectively by considering the sum {\allowdisplaybreaks\begin{align}
{\bf I}_{\psi}(\nu;K,f):=\sum_{\pi\cong \otimes_{v} \pi_v \in \Pi_{\rm cusp}(\sG)_\sZ^{K}} \widehat{f^{S}}(\bar \pi^{S})\,A_{\psi}^{(\nu)}(\pi;K_S,f_{S}), \quad \nu \in \C^{n-1}
\label{GlobalWhittPer}
\end{align}}over the infinite set $\Pi_{\rm cusp}(\sG)_{\sZ}^{K}$ of the quantity {\allowdisplaybreaks\begin{align}
A_{\psi}^{(\nu)}(\pi;K_S,f_S):=\frac{\prod_{j=1}^{n-1} L\left(\tfrac{1}{2}+\nu_j,\bar\pi\right)}{L(1,\pi;\Ad)}\,\biggl\{\prod_{p\in S}\PP_{\psi_p}^{(\nu)}(\pi_p,K_p;f_{p})\biggr\}\,W_{\sG(\R)}^{0}(\nu({\pi_\infty});1_n)
\label{GlobalWhittPer-f1}
\end{align}}multiplied by $\widehat{f^{S}}(\bar \pi^{S}):=\prod_{v\not \in S}{\widehat {f_v}}(\bar \pi_v)$, where $S$ is any finite set of prime numbers containing $\{p<\infty|\,K_p\not=\bK_p\}$, $K_S:=\prod_{p\in S}K_p$ and $f_S:=\otimes_{p\in S}f_p$. The archimedean factor $W^{0}_{\sG(\R)}(\nu(\pi_\infty);1_n)$ of \eqref{GlobalWhittPer-f1} is the value at the identity $1_n\in \sG(\R)$ of the normalized Jacquet integral \eqref{normJInt} for the spectral parameter $\nu(\pi_\infty)$ of $\pi_{\infty}$ (see \S\ref{sec:URps}), and the $p$-factor $\PP_{\psi_p}^{(\nu)}(\pi_p,K_p;f_p)$ is the local quantity defined by the formula \eqref{LocalWHittPer}. From Lemma~\ref{unramifiedPer}, the right-hand side of \eqref{GlobalWhittPer} is independent of the choice of $S$. Moreover, it is shown that the infinite sum \eqref{GlobalWhittPer} is absolutely convergent, and that \eqref{GlobalWhittPer} as a function in $\nu$ is holomorphic (Proposition~\ref{SpectExpPerL-3}). 

Let $N$ be a positive integer. Let $M>1$ be an integer relatively prime to $N$ and $\{\tau_p\}_{p\in S(M)}$ a family of irreducible smooth supercuspidal representations of $\sG(\Q_p)$ with trivial central character such that $\tau_p$ has the conductor $M\Z_p$ (see \S\ref{sec:Lfunction}). Let $\WW^{\psi_p}(\tau_p)$ be the $\psi_p$-Whittaker model of $\tau_p$ (see \S\ref{LWFunc}); the representation of $\sG(\Q_p)$ on $\WW^{\psi_p}(\tau_p)$ by the right-translation $R$ is equivalent to $\tau_p$ and is known to be unitary with respect to the Hermitian inner-product $\langle\,|\,\rangle_{\cP(\Q_p)}$ on $\WW^{\psi_p}(\tau_p)$ defined by \eqref{WhitProd}. Let $S_0$ be a finite set of prime numbers such that $S_0\cap S(MN)=\emp$. Towards Thereom~\ref{MAINTHM2.5}, depending on $(N,\{\tau_p\}_{p\in S(M)},S_0)$, we take a more specific function $f=\otimes_{v}f_v$ with $v$-components satisfying the conditions: 
\begin{itemize}
\item[(a)] For each $p\in S(M)$, let $W_p\in \WW^{\psi_p}(\tau_p)^{\bK_1(M\Z_p)}$ be the essential vector (\cite[Theorem 1]{Jacquet3}, \cite{Matringe}), and choose a function $f_p\in C_{\rm{c}}^{\infty}(\bK_1(M\Z_p)\bsl \sG(\Q_p)/\bK_1(M\Z_p))$ such that $\tilde f_p(g):=\int_{\sZ(\Q_p)}f_p(zg)\,\d z$ coincides with the matrix coefficient $\Delta_{\sG,p}(1)^{-1}\langle R(g)W_p|W_p\rangle_{\cP(\Q_p)}$ of $\tau_p$. Set $f_{M}=\otimes_{p\in S(M)}f_p$.
 \item[(b)] For each $v\in S_0\cup \{\infty\}$, choose $f_v\in C_{\rm{c}}^{\infty}(\bK_v\bsl \sG(\Q_v)/\bK_v)$ and set $f_{S_0}=\otimes_{v\in S_0}f_v$ and $f_{S_0,\infty}=f_{S_0}\otimes f_{\infty}$.
\item[(c)] For each prime number $p \not\in S_0 \cup S(M)$, let $f_{p}$ be the characteristic function of $\bK_1(N\Z_p)$. Set $f_{N}^{0}=\otimes_{p\not\in S_0\cup S(M)}f_p$. 
\end{itemize} 
For a place $v \in S_0\cup\{\infty\}$, set $\fX_v^{0}(1)=\{s \in (\R/\ell_v\Z)^n|\sum_{j=1}^{n}s_j=0\}$, where $\ell_v=2\pi (\log p)^{-1}$ if $v=p<\infty$ and $\ell_\infty=0$. Let $\d\mu^{\rm Pl}_v$ be the Plancherel measure on $\fX_{v}^{0}(1)$ for $L^2(\sZ(\Q_v)\bK_v\bsl \sG(\Q_v)/\bK_v)$ pertaining to the Haar measure on $\sG(\Q_v)$; it is a positive Borel measure on $\fX_v^0(1)$ such that 
$\phi_v(1_n)=\int_{\fX_v^0(1)}\widehat \phi_v(s)\,\d\mu_v^{\rm Pl}(s)$ for all $\phi_v\in C_{\rm c}^{\infty}(\sZ(\Q_v)\bK_v\bsl \sG(\Q_v)/\bK_v)$ explicitly given in Propositions~\ref{LNVpadicL1} and \ref{LNVAr-L10}. For $v\in S_0\cup \{\infty\}$ and $s\in \C^n$, let $W^0_{\sG(\Q_v)}(s)$ be the normalized Jacquet integral for $I^{\sG}_v(s)$ defined in \S\ref{Jacquetint}; note that the value $W_{\sG(\Q_v)}^{0}(s;1_n)$ at the identity $1_n\in \sG(\Q_v)$ is $1$ if $v=p<\infty$ and is a vertically bounded holomorphic function in $s$ if $v=\infty$ (Lemma~\ref{HolWhitt}). Here is the main result of this paper, which gives an exact formula of the quantity \eqref{GlobalWhittPer} for $f=f_{N}^0\otimes f_{M} \otimes f_{S_0,\infty}$ with sufficiently large $N$.    

\begin{thm} \label{MAINTHM1}
There exists a positive integer $N_0=N_0(f_{M},f_{S_0, \infty})$ depending only on the support of the function $f_{M} \otimes f_{S_0,\infty}$ such that for any integer $N>N_0$ relatively prime to $M\prod_{p\in S_0}p$ and for any $\nu=(\nu_j)_{j=1}^{n-1}\in \C^{n-1}$ such that $\Re\nu_j>-1/2\,(1\leq j\leq n-1)$, 
\begin{align*}
&{\bf I}_{\psi}(\nu;\bK_1(MN),f_{N}^0 \otimes f_{M}\otimes f_{S_0,\infty})=\varphi(N)^{-1}\,{\bf J}_{\psi}^{0}(\nu; f_{S_0,\infty}),  
\end{align*}
where $\varphi(N)$ denotes the Euler totient function, 
\begin{align*}
&{\bf J}_\psi^{0}(\nu; f_{S_0,\infty}) 
=\prod_{v\in S_0\cup \{\infty\}}\int_{\fX_v^0(1)}\widehat {f_v}(-\ii s)\,W_{\sG(\Q_v)}^{0}(\ii s;1_n)\,\frac{\prod_{j=1}^{n-1}L\left(\tfrac{1}{2}+\nu_j, I^{\sG}_v(-\ii s)\right)}{\zeta_v(1)L\left(1,I_v^{\sG}(\ii s);\Ad\right)}\,\Delta_{\sG,v}^{(\infty)}(1)\,\d\mu_v^{\rm Pl}(s)
\end{align*}
with $\Delta_{\sG,v}^{(\infty)}(1)$ being $\Delta_{\sG,v}(1)$ or $1$ according as $v<\infty$ or $v=\infty$. 
\end{thm}
There are many works which study the weighted vertical Sato-Tate law for various families of automorphic forms on $\GL_2$ (\cite{KnightlyLi}, \cite{Ramakrishnan-Rogawski}, \cite{FeigonWhitehouse}, \cite{Tsud}, \cite{Sugiyama2}). For Hecke-Maass cusp forms on $\GL_n$ ($n\geq 3$) of a fixed level with varying infinitesimal characters, Matz-Templier \cite{MatzTemplier} established the vertical Sato-Tate law by means of the Arthur-Selberg trace formula. A weighted case for Hecke-Maass cusp forms on ${\SL}_n(\Z)$ ($n\geq 3$) is investigated by \cite{BlomerButtcaneRaulf} and  \cite{Zhou}. The formula in Theorem~\ref{MAINTHM1} is deduced from a Kuznetsov type (pre)trace formula pertaining to the pair $(\sZ\sB_0,\sU)$, where $\sB_0$ is the upper-triangular Borel subgroup of $\GL_{n-1}$ embedded to the upper-left $(n-1)\times (n-1)$-block of $\sG$ and $\sU$ is the upper-triangular maximal unipotent subgroup of $\sG$. In the usual deduction of the Kuznetsov trace formula (\cite{Kuz}, \cite{Brugg}, \cite{Blomer1}), one prepares two Poincar\'{e} series and computes their inner product in two ways. Unlike this, we start with one particular Poincar\'{e} series on $\sG(\A)$ and calculate its Whittaker-Fourier coefficient in two ways: by the spectral expansion in terms of automorphic functions, and by the geometric expansion in terms of orbital integrals associated with $(\sZ(\Q)\sB_0(\Q),\sU(\Q))$-double cosets in $\sG(\Q)$. In this methodological aspect, our argument has similarity to \cite{Qi}. Another remarkable feature of our formula is its simplicity, which is brought by the compact support condition of the test function $f=\otimes_{v}f_v$. Indeed, in the geometric side, after the main term is retrieved from the double coset containing the large Bruhat cell of $\GL_{n-1}(\Q)$, the remaining terms are shown to be zero if $N$ is sufficiently large by the support condition. We remark that a similar simplification of the geometric side (``stability'') is observed in \cite[Corollary 1]{MichelRamakrishinan}. In the spectral side, the supercuspidality of $\tilde f_p\,(p\in S(M))$ works to kill the potential terms from the Eisenstein series. Although the average \eqref{GlobalWhittPer} is made over an infinite family $\Pi_{\rm cusp}(\sG)_{\sZ}^{\bK_1(MN)\bK_\infty}$, the rapidly decaying weight factor $\widehat f_\infty(\pi_\infty)$ allows us to consider at least heuristically that the main contribution comes from those $\pi\in \Pi_{\rm cusp}(\sG)_{\sZ}^{\bK_1(MN)\bK_\infty}$ with bounded infinitesimal characters; total number of such $\pi$ is approximately $N^{n-1+o(1)}$ on one hand. On the other hand, the size of the conductor of the $L$-function $L(s,\boxtimes_{j=1}^{n-1}|\,|_\A^{\nu_j} \times \pi)$ for $\pi \in \Pi_{\rm cusp}(\sG)_{\sZ}^{ \bK_1(MN)\bK_\infty}$ is about $N^{n-1}$, so that the logarithmic ratio of conductor and family is $(n-1)/(n-1)=1$. Potentially, the Kuznetsov trace formula can handle much more complicated situations. For instance, Blomer-Buttcane-Maga \cite{BBM} obtained a sixth moment formula for the $L$-values of automorphic forms on $\GL_3$ averaged over a family of size $N^{2+o(1)}$, so that the logarithmic ratio of conductor versus family is $6/2=3>1$. An advantage of our method is that one can deal with arbitrary $\GL_n$ uniformly by a relatively soft method except that the exact evaluation of local zeta-integrals is required to compute the main term. We should mention recent researches \cite{Blomer3}, \cite{Jana} and \cite{Nelson} done in a similar spirit to our work dealing with analytic problems related to automorphic forms on $\GL(n)$ by a soft method. 
\subsection{Structure of paper}
 In \S2, we prepare notation and basic materials for later use. In \S3, we define a Poincar\'{e} series $\tilde {\bf\Phi}_{f,\b}$ by the formula \eqref{Pser-f1} for a general decomposable function $f=\otimes_v f_v$ on $\sG(\A)$ and an auxiliary entire function $\beta(\nu)$ on $\C^{n-1}$ with rapid decay on vertical strips using the notation \eqref{tildefnu-def} and \eqref{tildefB}. By means of a gauge estimate (Lemma~\ref{L1}), we prove the absolute convergence and the square integrability of $\tilde {\bf \Phi}_{f,\b}(g)$ in Proposition~\ref{SQint}. After recalling the normalize minimal parabolic Eisenstein series $\hat E(\nu)$ on $\GL_{n-1}$ in \S\ref{sec:minEis}, we prove a uniform estimate of the global ${\GL}_n \times {\GL}_{n-1}$ Rankin-Selberg zeta-integrals in the special case where the $\GL_{n-1}$-factor is $\hat E(\nu)$ (Lemma~\ref{SpectExpPerLLL}). In \S\ref{InnProdFor}, we prove the inner product formula \eqref{RSconv}, which plays a pivotal role in describing the spectral expansion of $\tilde {\bf \Phi}_{f,\b}$ by square-integrable automorphic forms. Actually, after \S\ref{SPEXP}, we assume that the local factor $f_p$ of $f$ at a prime $p$ is a matrix coefficient of a supercuspidal representation which implies that $\tilde{\bf\Phi}_{f,\b}$ satisfies the cuspidality condition (see Proposition~\ref{CuspidalPer}) so that the spectral expansion has no contribution from the continuous spectrum. The spectral expansion is written down only in this specific setting; see Proposition~\ref{SpectExpPer}. However we remark that the construction of $\tilde{\bf\Phi}_{f,\b}$ itself as well as its $L^2$-estimate does not require $M>1$, i.e., the presence of a supercuspidal matrix coefficient in a factor of $f$. In a forthcoming work \cite{Tsuzuki2021}, we elaborate the spectral resolution of $\tilde{\bf \Phi}_{f,\beta}$ incuding the case $M=1$. In \S4, we compute the Whittaker coefficient of $\tilde {\bf \Phi}_{f,\b}$ by using the spectral expansion; it is described by a contour integral against $\b(\nu)$ of a certain entire function ${\bf I}_f(\nu)$ defined as the sum $\sum_{\pi}\widehat {f^S}(\bar \pi^{S})\sum_{\varphi}P_\nu(\varphi* {f_S})$ over cuspidal automorphic representations $\pi\cong \otimes_{v}\pi_v$ of $\sG(\A)$ with trivial central characters, where the inner sum $A_\nu(\pi):=\sum_{\varphi} P_\nu(\varphi*{f_S})$ is taken over cusp forms $\varphi$ belonging to an orthonormal basis of $\pi$ with $P_\nu(\varphi* {f_S})$ being a product of the global Whittaker coefficient $\Wcal^{\psi}(\varphi)$ (defined by \eqref{WFCoeffDef}) and the global Rankin-Selberg zeta-integral $Z(\frac{1}{2},\hat E(\nu),\varphi*{f_S})$ (\S\ref{InnProdFor}). By the theory of Rankin-Selberg $L$-function (\cite{JPSS1983}, \cite{JPSS}), $A_\nu(\pi)$ is eventually expressed by the local Whittaker period of $\pi_S=\otimes_{v\in S}\pi_v$ and the Rankin-Selberg $L$-function $L(\frac{1}{2}, \boxtimes_{j=1}^{n-1}|\,|_\A^{\nu_j}\times \bar \pi)$ in the form \eqref{GlobalWhittPer}; see Lemma~\ref{AverageWHittPerL1}. In \S\ref{sec:WFcoefficient}, we examine the expression \eqref{FWGeo-f1} of the Whittaker coefficient of $\tilde{\bf \Phi}_{f,\b}$, which is obtained immediately by unfolding the integrals. Since our construction involves the restriction to $\sU(\Q)\bsl \sU(\A)$ of the Poincar\'{e} series on $\sZ(\A)\sG(\Q)\bsl \sG(\A)$ induced from a $\sZ(\A)\sB_0(\Q)$-invariant function, we are naturally lead to a detailed study of the $(\sZ(\Q)\sB_0(\Q),\sU(\Q))$-double cosets and the associated orbital integrals in the spirit of the relative trace formula of Jacquet (\cite{Jacquet2}). In \S\ref{sec:doublecosets}, we construct a complete set of representatives $\sn(y)w$ for the double coset space, labeled by vectors $y\in \Q^{n-1}$ and Weyl group elements $w$ of $\GL_n$ (for notation $\sn(y)$, see \S\ref{sec:DistSubgrp}). After that, we determine the condition for the first integral $a(\sn(y)w)$ in \eqref{FWGeo-aJ} to be non-zero (Lemma~\ref{agammanot0}). In \S\ref{sec: DefOrbInt}, we write the second integral $\JJ_{f,\b}(\sn(y)w)$ in \eqref{FWGeo-aJ}, referred to as the smoothed orbital integral, as an iterated integral \eqref{GlobalJbetaQy}. The upshot of \S\ref{sec:WFcoefficient} is the formula \eqref{GeoSideP1}, which expresses the Whittaker coefficient of $\tilde{\bf \Phi}_{f,\b}$ by an infinite sum of the smoothed orbital integrals $\JJ_{f,\b}(\sn(y) w)$. In \S\ref{sec:LocalOrbInt}, which is the most technical part of this article, we analyze the local orbital integral defined by \eqref{LocalJnuQy} and its adelic counterpart \eqref{GlobalJJnuQy} in detail to establish their absolute convergence in a certain range of parameters $\nu$ and to obtain the contour integral expressions \eqref{ErrorL1-f1} and \eqref{MaintermP1-f} of the smoothed orbital integrals. As a consequence, we obtain a ``pre-trace formula'' identity \eqref{RTFbeta} stated in Theorem~\ref{preRTF}, where one can find exact conditions on the test function $f$ to which \eqref{RTFbeta} can be applied. It is desirable to remove the function $\beta(\nu)$ for smoothing from \eqref{RTFbeta}, which is attained only for functions with large levels in this article. Indeed, by carefully examining the support condition of $f$, it is shown that there exists a constant $N_0$ depending on the component $f_{S_0\cup S(M)\cup\{\infty\}}$ of $f$ such that for $N$ larger than $N_0$, all the smoothed orbital integrals $\JJ_{f,\b}(\sn(y)w)$ are zero unless $y=0$ and $w$ is the longest element $w_{\ell}^{0}$ in the Weyl group of $\GL_{n-1}$ (Proposition~\ref{JJL-7}). The smoothed orbital integral $\JJ_{f,\b}(w_\ell^0)$ is evaluated exactly in \S\ref{GSST}. Then, the identity ${\bf I}_f(\nu)=\JJ_{f}^{(\nu)}(0,w_\ell^0)$ (without the $\beta$-smoothing) is obtained for all $\nu$ for large levels $N>N_0$ (Theorem~\ref{RTF}). In \S8, we further compute the orbital integral $\JJ_{f}^{(\nu)}(0,w_\ell^0)$ in terms of local $L$-functions invoking the exact evaluation of local zeta-integrals for spherical Whittaker functions due to Jacquet, Piatetski-Shapiro and Shalika (\cite{JPSS}) for non-archimedean places and to Ishii-Stade (\cite{IshiiStade2013}) for the archimedean place. The main technical tools here are the stable integral studied by Lapid-Mao (\cite{LapidMao}) at non-archimedean places, and the uniform estimate of the Jacquet integral at the archimedean place that is established by Blomer-Harcos-Maga (\cite{BlomerHarcosMaga}) based on an inductive formula for the Jacquet integrals by Stade (\cite{Stade}). For the proof of Theorem \ref{MAINTHM2.5}, we have to settle the issue of old forms. In \S\ref{sec:OldF}, by using this description of old forms due to \cite{Reeder}, we compute the local period \eqref{LocalWHittPer} for level one old forms in a spherical generic representation and obtain an upper estimation (Proposition \ref{OldF-L7}) invoking the Luo-Rudnick-Sarnak bound for Satake parameters of cuspidal automorphic representations. The proofs of Theorems~\ref{MAINTHM1} and \ref{MAINTHM2.5} are given in \S\ref{sec:PfMTHM}. In Theorem~\ref{MAINTHM2.5}, to circumvent further technical complications in the proof, we restrict ourselves to the case of prime levels, whereas the proof will be easily modified to cover the square-free level case. 
\section{Preliminaries}
Set $\N=\{j\in \Z \mid j>0\}$ and $\N_0:=\N\cup\{0\}$. For $a,b\in \Z$, let $[a,b]_{\Z}$ denote the set $\{j\in \Z\mid a\leq j \leq b\}$. Set $\R_+:=\{x\in \R\mid x>0\}$. For $m,l\in \N$ and a ring $R$ with the unit $1$, $\Mat_{m,l}(R)$ denotes the space of $m\times l$-matrices with entries in $R$; the identity matrix in $\Mat_{m}(R):=\Mat_{m,m}(R)$ will be denoted by $1_m$ and the set $\Mat_{m,1}(R)$ will be also denoted by $R^{m}$. For a condition ${\rm P}$, we set $\delta({\rm P})$ to be $1$ if ${\rm P}$ is true and to be $0$ if ${\rm P}$ is false. For a set $X$ and its subset $Y$, the characteristic function of $Y$ on $X$ will be denoted by $\cchi_{Y}$. For a function $f(g)$ on a group, we set $\check f(g)=f(g^{-1})$. For a topological group $H$, its connected component of the identity is denoted by $H^\circ$. For an algebraic $\Q$-group $\sH$, $X^*(\sH)$ denotes the group of the rational characters defined over $\Q$. The $\Z$-scheme ${\rm Spec}(\Z[t_1,\dots,t_m])$ is denoted by ${\rm Aff}^m$. For an $\R$-vector space $V$, $V_{\C}:=V\otimes_\R \C$. The idele norm of $a=(a_v)_{v}\in \A^{\times}$ is denoted by $|a|_\A:=\prod_{v\leq \infty}|a_v|_v$. The $v$-factor of the Euler product $\zeta_{\Q}(z)$ is denoted by $\zeta_v(z)$. 
\subsection{Measures on groups} \label{sec:Measure}
We fix notation for general linear groups $\GL_m\,(m\geq 1)$. The center of $\GL_m$ is dgenoted by $Z_{m}$. Let $T_{m}$ be the maximal $\Q$-split torus of diagonal matrices in $\GL_m$ and $U_{m}$ the maximal unipotent subgroup of upper-triangular matrices in $\GL_m$. Then $B_{m}:=T_mU_m$ is a Borel subgroup of $\GL_m$ . We already fixed Haar measures on $\GL_m(\Q_v)$ and $\GL_m(\A)$ in \S\ref{GWP}. The Haar measure on $\Q_v^\times=\GL_1(\Q_v)$ (resp. $\A^\times=\GL_1(\A)$) is transported to $Z_m(\Q_v)$ by the isomorphism $z\mapsto z\,1_{m}$. For a matrix $(g_{ij})_{ij}$ with coefficients $g_{ij}$ in $\Q_v$ viewed as independent variables, we denote by $\d g_{ij}$ the Haar measure on the additive group $\Q_v$ fixed in \S \ref{GWP}. With this notation, the Haar measures on $T_m(\Q_v)$ and $U_m(\Q_v)$ are fixed explicitly by $\d t_v=\prod_{i=1}^{m}\zeta_v(1)|t_{ii}|_v^{-1}\,\d t_{ii}$ for $t_v=(t_{ij})_{ij} \in T_m(\Q_v)$ and $\d u_v=\prod_{1\leq i<j \leq n}\d u_{ij}$ for $u_v=(u_{ij})_{ij}\in U_m(\Q_v)$. The left Haar measure on $B_m(\Q_v)=T_m(\Q_v)\,U_m(\Q_v)$ is fixed by $\d_lb_v=\d t_v\,\d u_v$ for $b_v=t_v u_v$ with $t_v\in T_m(\Q_v)$ and $u_v\in U_m(\Q_v)$. Note that the compact groups $Z_m(\Z_p)$, $T_m(\Z_p)$, $U_m(\Z_p)$, $B_m(\Z_p)$ and $\GL_m(\Z_p)$ have measure $1$ for a prime number $p$. Let $\bK_v^{(m)}$ denote the standard maximal compact subgroup of $\GL_m(\Q_v)$ (see \S~\ref{Intro}) and $\d k_v$ the Haar measure on $\bK_v^{(m)}$ of total mass $1$. By $\GL_m(\Q_v)=B_m(\Q_v)\bK_v^{(m)}$, we have the integration formula $\int_{\GL_m(\Q_v)}f(g_v)\,\d g_v=\int_{B_m(\Q_v)}\int_{\bK_v^{(m)}}f(b_vk_v)\,\d _lb_v\,\d k_v$ for all $f\in C_{\rm c}^{\infty}(\GL_m(\Q_v))$, which is symbolically written as $\d g_v=\d_{l}b_v\,\d k_v$; for $v=\infty$, this is shown by the easily confirmed identities $\int_{\GL_m(\R)}\phi(g)|\det g|_\infty^{-m}\prod_{ij}\d g_{ij}=1$ and $\int_{B_m(\R)}\phi(b)\,\d_{l}b=\Delta_{\GL_m,\infty}(1)$ for $\phi(g)=|\det g|^m_\infty \exp(-\pi\tr(g{}^tg))\,(g\in \GL_m(\R))$. The adelizations $T_m(\A)$, $U_m(\A)$, $B_m(\A)$ are endowed with the Tamagawa measures, which are the product of the Haar measures on the corresponding groups of $\Q_v$-points defined above. The compact group $\bK^{(m)}:=\prod_{v}\bK_v^{(m)}$ is endowed with the Haar measure of total mass $1$. The Tamagawa measure $d g$ on $\GL_m(\A)=B_m(\A)\,\bK^{(m)}$ is decomposed as $\d g=(\Delta_{\GL_m}(1)^{*})^{-1}\,\d_l b\,\d k$ for $g=bk$ with $b\in B_m(\A)$ and $k\in \bK^{(m)}$. Let $\delta_{B_m}:B_m(\A)\rightarrow \R_+$ be the modulus character of $B_m(\A)$, i.e., $\delta_{B_m}(b)\,\d_l b$ is a right invariant Haar measure on $B_m(\A)$. Set $\GL_m(\A)^1=Z_m(\R)^\circ \bsl \GL_m(\A)$, which is often identified with the subgroup $\{g\in \GL_m(\A) \mid |\det g|_\A=1\,\}$ so that $\GL_m(\Q)\subset \GL_m(\A)^1$. We consider the quotient measure on $\GL_m(\A)^1$. Then $\vol(\GL_m(\Q)\bsl \GL_m(\A)^1)$ is finite. Define $B_m(\A)^{1}=\{b=(b_{ij})_{ij} \in B_m(\A) \mid |b_{ii}|_{\A}=1\,(i\in [1,m]_\Z)\,\}$. Then we have the direct product $B_m(\A)=T_m(\R)^\circ\,B_m(\A)^{1}$, $B_m(\Q)\subset B_m(\A)^1$ and that the quotient $B_m(\Q)\bsl B_m(\A)^1$ is compact.

\subsection{Distinguished subgroups} \label{sec:DistSubgrp}
Throughout this paper, we keep the notation introduced in \S\ref{GWP}. In particular, $n\geq 3$ is an integer and $\sG=\GL_n$. Set $\sG_0=\GL_{n-1}$. 
The algebraic subgroups $Z_{n}$, $T_n$, $U_n$ and $B_n$ (resp. $Z_{n-1}$, $T_{n-1}$, $U_{n-1}$ and $B_{n-1}$) of $\sG=\GL_n$ (resp. $\sG_0=\GL_{n-1}$) introduced in \S\ref{sec:Measure} will be denoted by $\sZ$, $\sT$, $\sU$ and $\sB$ (resp. $\sZ_0$, $\sT_0$, $\sU_0$ and $\sB_0$), respectively. The compact groups $\bK_v^{(n-1)}\,(\subset \sG_0(\Q_v))$ for a place $v$ of $\Q$ and $\bK^{(n-1)}\,(\subset \sG_0(\A))$ are denoted by $\bK_{\sG_0,v}$ and $\bK_{\sG_0}$, respectively. For $z\in \GL_1$, $[z]$ denotes the scalar matrix $z\,1_n\in \sG$. Thus $\sZ=\{[z]\mid z\in \GL_1\}$. The group $\sG_0$ is embedded into $\sG$ by the mapping $
 \iota: h \mapsto \left[\begin{smallmatrix} h & 0 \\ 0 & 1 \end{smallmatrix} \right]$. The subgroup $\sZ\,\iota(\sG_0)$ coincides with the fixed point set of the involution $g\mapsto \taut g$ of $\sG$ defined as 
$$\taut g=\left[\begin{smallmatrix} 1_{n-1} & 0 \\ 0 & -1 \end{smallmatrix} \right]\,g\,\left[\begin{smallmatrix} 1_{n-1} & 0 \\ 0 & -1 \end{smallmatrix} \right], \quad g\in \sG.
$$
Note that the operator $g\mapsto \taut g$ commutes with inverses so that $\taut g^{-1}$ is well-defined. We have the relations $\iota(\sB_0)=\iota(\sG_0)\cap \sB$, $\iota(\sU_0)=\iota(\sG_0)\cap \sU$ and $\iota(\sT_0)=\iota(\sG_0)\cap \sT$. For $x\in \A^{n-1}$, set
$$
\sn(x)=\left[\begin{smallmatrix} 1_{n-1} & x \\ 0 & 1 \end{smallmatrix} \right]\quad \in \sU(\A).
$$
Then $\sU(\A)=\iota(\sU_0(\A))\,\{\sn(x)\mid x\in \A^{n-1}\}$. 

\subsection{Vector groups and roots for general linear groups}
\label{sec:Vector}
Set $T_{\infty}:=\sT(\R)^{\circ}$ and $T_{0,\infty}:=\sT_0(\R)^\circ$. It is worth noting that $\iota(\sB_0(\A)^1)=\iota(\sG_0(\A))\cap \sB(\A)^1$ and $\iota(T_{0,\infty})=\iota(\sG_0(\A))\cap T_{\infty}$. 

For any algebraic $\Q$-group $\sH$, the $\R$-vector spaces $X^{*}(\sH)\otimes \R$ and $\Hom(X^{*}(\sH),\R)$ will be denoted by $\fh^{*}$ and $\fh$, respectively, by using the corresponding German lowercase letter. Thus, we have paris of $\R$-vector spaces $(\ft^{*},\ft)$ for $\sT$, $(\ft_{0}^*, \ft_{0})$ for $\sT_0$, and $(\fz^{*}, \fz)$ for $\sZ$. For any $m\geq 1$, we dentify $X^{*}((\GL_1)^{m})=\Z^{m}$ in such a way that a tuple $(x_1,\dots,x_m)\in \Z^{m}$ corresponds to the character $(t_1,\dots,t_m)\mapsto \prod_{j=1}^{m}t_j^{x_j}$ of $(\GL_1)^{m}$; in particular, $X^{*}(\sT)=\Z^{n}$, $X^{*}(\sT_0)=\Z^{n-1}$, and $X^{*}(\sZ)=\Z$, by which we made identifications $\ft^{*}=\ft=\R^{n}$, $\ft_{0}^*=\ft^{0}=\R^{n-1}$ and $\fz^{*}=\fz=\R$ hereinafter. The isomorphism $(z,t)\mapsto z\,\iota(t)$ from $\sZ\times \sT_0$ onto $\sT$ induces an $\R$-linear isomorphism $\ft^{*}\cong \fz^*\oplus \ft_{0}^{*}$, which embeds $\ft^{*}=\R^{n-1}$ to the first $n-1$-coordinates of $\ft^{*}=\R^{n}$ and $\fz^{*}=\R$ diagonally to $\ft^{*}=\R^{n}$. Let $(\,,\,)$ denote the standard inner product on $\ft^*=\R^{n}$ as well as its $\C$-bi-linear scalar extension to $\ft_\C^*$. The norm of $\nu \in \ft_{\C}^*$ is defined as $\|\nu\|=(\nu,\bar\nu)^{1/2}$. Let $H:\sG(\A)\rightarrow \ft=\R^{n}$ be the Harish-Chandra projection along $\sG(\A)=T_\infty\sB(\A)^1\bK$ defined as 
$$
H(b\,t\,k)=(\log t_j)_{j=1}^{n}, \quad b\in \sB(\A)^1,\,t=\diag({t_j}\mid 1\leq j\leq n)\in T_{\infty},\,k\in \bK.
$$
Note that the function $H$ is left $\sB(\Q)$-invariant. We define $H^{\sG_0}(g)\in \ft_{0}$ to be the projection of $H(g)\in \ft$ along the decomposition $\ft\cong \ft_{0}\oplus \fz$. The pull-back $H_{\sG_0}=H^{\sG_0}\circ \iota$ coincides with the Harish-Chandra projection of $\sG_0(\A)$ along $\sG_0(\A)=T_{0,\infty}\sB_0(\A)^1\bK_{\sG_0}$. Set 
\begin{align}
\rho_{\sB}=(\tfrac{n-2j+1}{2})_{j=1}^{n}\in \ft^{*}, \quad 
\rho_{\sB_0}=(\tfrac{n-2j}{2})_{j=1}^{n-1}\in \ft_{0}^{*} .
 \label{rhosB0}
\end{align}
Then we have 
\begin{align}
\delta_{\sB_0}(b)&=\exp(\langle H_{\sG_0}(b),2\rho_{\sB_0}\rangle)\quad \text{for $b\in \sB_0(\A)$}
 \label{delta-rho}
\end{align}
and a similar formula for $\delta_{\sB}$. The symmetric group $\sS_n$ of degree $n$ is realized as a subgroup of $\sG(\Q)$ by identifying $w\in \sS_n$ with the permutation matrix $(\delta_{i\,w(j)})_{ij}$. Thus $\sS_n$ is identified with the Weyl group ${\rm N}_{\sG}(\sT)/\sT$ of $(\sG,\sT)$, where ${\rm N}_{\sG}(\sT)$ is the normalizer of $\sT$ in $\sG$. Similarly, $\sS_{n-1}$ viewed as a subgroup of permutation matrices in $\sG_0(\Q)$ is isomorphic to the Weyl group of $(\sG_0,\sT_0)$. Note that $\iota(\sS_{n-1})=\iota(\sG_0(\Q))\cap {\sS_{n}}=\{w\in \sS_n\mid w(n)=n\,\}$. The longest element of $\sS_n$ (resp. $\sS_{n-1}$) is denoted by $w_{\ell}$ (resp. $w_{\ell}^{0}$). The natural actions of the Weyl group $\sS_n$ on $\sT$ and on $X^*(\sT)$ are given by ${}^w t=w t w^{-1}$ for $t\in \sT$ and by ${}^w\chi(t)=\chi( w^{-1}t  w)$ for $\chi\in X^{*}(\sT)$, $t\in \sT$, respectively. Then we have the induced action of $\sS_n$ on $\ft^*$ such that $\langle H({}^wt),{}^w\nu\rangle=\langle H(t),\nu\rangle$ for all $t\in \sT(\A)$ and $\nu \in \ft^*$, which is described on the standard basis $\{\varepsilon_j\}_{j=1}^{n}$ of $X^{*}(\sT)=\Z^r$ as ${}^w\varepsilon_{j}=\varepsilon_{w(j)}$ for $w\in \sS_{n}$ and $1\leq j\leq n$. We also remark that the permutation matrix $w\in \sS_n$ when viewed as an element of $\sG(\A)$ belongs to the maximal compact subgroup $\bK$.  

Let $\Delta(\sG,\sT)\,(\subset X^{*}(\sT))$ (resp. $\Delta(\sG_0,\sT_0)\,(\subset X^{*}(\sT_0))$) be the set of simple $\sT$-roots (resp. $\sT_0$-roots) on $\sG$ (resp. $\sG_0)$ corresponding to the Borel subgroup $\sB$ (resp. $\sB_0)$. Note that $\langle \rho_{\sB_0}, \a\rangle=1$ for $\alpha\in \Delta(\sG_0,\sT_0)$. We set 
{\allowdisplaybreaks\begin{align*}
(\ft_{0}^*)^{+}&=\{\nu \in \ft_{0}^*\mid \langle \nu,\a\rangle>0\,(\alpha  \in \Delta(\sG_0,\sT_0)\,\},  \\
(\ft_{0}^{*})^{++}&=\rho_{\sB_0}+(\ft_{0}^*)^{+}=\{\nu \in \ft_{0}^*\mid \langle \nu, \a \rangle>1\,(\alpha \in \Delta(\sG_0,\sT_0))\}.
\end{align*}} For any subset $\Ncal\subset \ft_{0}^{*}$, set 
$$
{\fJ}(\Ncal)=\Ncal+\ii\,\ft_{0}^*\,\subset \ft_{0,\C}^*. 
$$
When $\Ncal$ consists of a single point $\sigma\in \ft_{0}^*$, ${\fJ}(\Ncal)$ is abbreviated to $\fJ(\sigma)$.  

\subsection{Siegel domain and an integration formula}\label{sec:SiegelDom}
We collect basic facts on the Siegel domain of $\sG_0$ and $\sG$, making all the statements only for $\sG_0$. For $c>0$, set
$$
T_{0,\infty}(c)=\{\diag({{t_1}},\dots,{{t_{n-1}}})\in T_{0,\infty}\mid t_j/t_{j+1} \geq c \,(1\leq j\leq n-2)\,\}.  
$$
A subset $\fS\subset \sG_0(\A)$ is called a Siegel domain if there exist 
  $c>0$ and a relatively compact subset $\omega\subset \sB_0(\A)^1$ such that $\fS=\omega T_{0,\infty}(c)\bK_{\sG_0}$. 
\begin{lem} \label{SiegelDomL0}
 Let $\fS \subset \sG_0(\A)$ be a Siegel domain. 
\begin{itemize}
\item[(i)] There exist a relatively compact subset $\cU'\subset \sG_0(\A)$ and $c>0$ such that $\fS\subset T_{0,\infty}(c)\,\cU'$. 
\item[(ii)] 
Let $\cU\subset \sG_0(\A)$ be a relatively compact subset. Then there exists a Siegel domain $\fS'\subset \sG_0(\A)$ such that $\fS\,\cU\subset \sB_0(\Q)\fS'$. 
\end{itemize}
\end{lem}
\begin{proof} Let $\fS=\omega T_{0,\infty}(c)\bK_{\sG_0}$. For (i), it suffices to note that $\cU'=\{t^{-1}bt \mid t\in T_{0,\infty}(c),\,b\in \omega\,\}\bK_{\sG_0}$ is relatively compact. To show (ii), by $\sG_0(\A)=\sB_0(\A)^{1}T_{0,\infty}\bK_{\sG_0}$, we choose compact sets $\omega_1\subset \sB_0(\A)^1$ and $\cU_{T} \subset T_{0,\infty}$ such that $\bK_{\sG_0}\cU\subset \omega_1\,\cU_T\,\bK_{\sG_0}$. Then $\fS\,\cU\subset \omega T_{0,\infty}(c)\,\omega_1\,\cU_T \,\bK_{\sG_0} \subset \sB_0(\A)^1\,T_{0,\infty}(c)\,\cU_T\,\bK_{\sG_0}$. Since $\cU_T$ is compact, $T_{0,\infty}(c)\,\cU_{T}\subset T_{0,\infty}(c')$ for some $c'>0$. Let $\omega'\subset \sB(\A)^1$ be a compact set such that $\sB_0(\Q)\omega'=\sB_0(\A)^1$ and set $\fS'=\omega'\,T_{0,\infty}(c')\bK_{\sG_0}$. Then $\fS\,\cU\subset \sB_0(\Q)\fS'$ as required.  
\end{proof}

We introduce a coordinate system $y=(y_1,\dots,y_{n-1})$ on $T_{0,\infty}$ by setting $y_{j}=t_{j}/t_{j+1}$ $(1\leq j\leq n-2)$ and $y_{n-1}=t_{n-1}$ for $t=\diag(t_j|1\leq j\leq n-1) \in T_{0,\infty}$. Then the map $t\mapsto y$ is a bijection from $T_{0,\infty}$ onto $(\R_+)^{n-1}$, whose inverse map is given by $y\mapsto a_{\sG_0}(y)$ with
$$
a_{\sG_0}(y)=\diag\biggl(\prod_{j=i}^{n-1}y_j\,\biggm|\,1\leq i \leq n-1\biggr), \quad y=(y_i)_{i=1}^{n-1}\in (\R_+)^{n-1}.
$$
We have $T_{0,\infty}(c)=\{a_{\sG_0}(y)\mid y\in [c,+\infty)^{n-2} \times \R_+\,\}$ for any $c>0$. Moreover, $a_{\sG_0}(y)\in \sZ_0(\R)$ if and only if $y_j=1\,(1\leq j\leq n-2)$. The Haar measure of $T_{0,\infty}$ becomes $\d t=\prod_{j=1}^{n-1}\d^*y_j$ for $t=a_{\sG_0}(y)$, where $\d^*y_j=\d y_j/y_j$. 
From \eqref{delta-rho}, we have  
\begin{align}
\delta_{\sB_0}(a_{\sG_0}(y))=\prod_{j=1}^{n-1}t_j^{n-2j}=\prod_{j=1}^{n-2}y_j^{j(n-1-j)}.
 \label{y-coordinate}
\end{align}
For $y=(y_j)_{j=1}^{n-2}\in (\R_+)^{n-2}$, set $\tilde y=(y_1,\dots,y_{n-2},\prod_{j=1}^{n-2}y_j^{-j/(n-1)})\in (\R_{+})^{n-1}$. Then $T_{0,\infty}(c)\cap \sG_0(\A)^1=\{a_{\sG_0}(\tilde y)\mid y\in [c,+\infty)^{n-2}\,\}$. 

\begin{lem}\label{SieDom} Let $\fS=\omega_{\sB_0}T_{0,\infty}(c)\bK_{\sG_0}$ with small enough $c>$ be a Siegel domain of $\sG_0(\A)$ such that $\sG_0(\A)=\sG_0(\Q)\fS$. Then, for any non-negative right $\bK_{\sG_0}$-invariant function $f$ on $\sG_0(\Q)\bsl \sG_0(\A)^1$, we have the inequality
\begin{align*}
\int_{\sG_0(\Q)\bsl \sG_0(\A)^1}f(h)\d h \leq (\Delta_{\sG_0}(1)^*)^{-1}\,\int_{\omega_{\sB_0}} \d b \int_{[c,\infty)^{n-2}}f(b\,a_{\sG_0}(\tilde y))\,\prod_{j=1}^{n-2}y_j^{-j(n-1-j)}\,\prod_{j=1}^{n-2}\d^* y_j,
\end{align*}
where $\d b$ is the quotient Haar measure on $\sB_0(\A)^1\cong \sB_0(\A)/T_{0,\infty}$. 
\end{lem}
\begin{proof} Since the quotient map $\fS \cap \sG_0(\A)^1 \rightarrow \sG_0(\Q)\bsl \sG_0(\A)^1$ is surjective, the integral of $f(h)$ over $\sG_0(\Q)\bsl \sG_0(\A)^1$ is no grater than the integral of $f(h)$ on $\fS \cap \sG_0(\A)^1$. From \S\ref{sec:Measure}, the Haar measure on $\sG_0(\A)$ is given as $\d h=(\Delta_{\sG_0}(1))^{*})^{-1}\delta_{\sB_0}(t)^{-1}\d b\,\d t\,\d k$ for $h=btk$ with $b\in \sB_0(\A)^1$, $t\in T_{0,\infty}$ and $k\in \bK_{\sG_0}$. By \eqref{y-coordinate}, we are done. 
\end{proof}

\subsection{Right-regular representations on $L^2$-space} \label{sec:convolutionP}
The convolution product of two $\C$-valued functions $\varphi$ and $f$ on $\sG(\A)$ is defined as 
\begin{align*}
\varphi*f\,(g)=\int_{\sG(\A)} \varphi(gg_1)\,f(g_1^{-1})\,\d g_1,\quad g\in \sG(\A)
\end{align*}  
with $\d g_1$ the Haar measure on $\sG(\A)$. For a finite set $S$ of places, we set $\Q_S=\prod_{v\in S}\Q_v$. For $f\in C_{\rm c}(\sG(\Q_S))$ and a function $\varphi$ on $\sG(\A)$, the convolution product $\varphi*f_S(g)\,(g\in \sG(\A))$ is defined as the integral of $\varphi(gg_S)f_S(g_S^{-1})$ in $g_S\in \sG(\Q_S)$ with respect to the product measure $\d g_S=\otimes_{v\in S}\d g_v$ of the Haar measures $\d g_v$ on $\sG(\Q_v)$ $(v\in S)$.  

Let $\langle \varphi|\varphi_1\rangle_{\sG}=\int_{Z_\infty\sG(\Q)\bsl \sG(\A)} \varphi(g)\bar\varphi_1(g)\,\d g$ be the inner product of $L^2(Z_\infty\sG(\Q)\bsl \sG(\A))$ and $\|\varphi\|_{L^2}:=\langle \varphi|\varphi\rangle_{\sG}^{1/2}$ the associated $L^2$-norm. The right-regular representation of $\sG(\A)$ on $L^2(Z_\infty \sG(\Q)\bsl \sG(\A))$ is denoted by $R$. Then we have the derived action of the convolution algebra $C_{\rm c}^{\infty}(\sG(\A))$ on $L^2(Z_\infty \sG(\Q)\bsl \sG(\A))$ given as 
\begin{align*}
[R(f)\varphi](g)=\int_{\sG(\A)}\varphi(gg_1)f(g_1)\d g_1=\varphi*\check{f}\,(g), \quad g\in \sG(\A)
\end{align*}
for $f \in C_{\rm c}^{\infty}(\sG(\A))$ and $\varphi \in L^2(Z_\infty \sG(\Q)\bsl \sG(\A))$. Let $\fg_\infty={\frak{gl}}_n(\C)$ be the complexified Lie algebra of $\sG(\R)$ and $U(\fg_\infty)$ the universal enveloping algebra of $\fg_\infty$. For $D\in U(\fg_\infty)$ and a smooth function $f$ on $\sG(\A)$, $f*D$ denote the derivatives of $f$ by $D$ from the right.

\subsection{Phragm\'{e}n-Lindel\"{o}f principle}
Given a domain $\fD\subset \ft_{0}^*$, we say that a holomorphic function $J(\nu)$ on $\fD$ is vertically of moderate growth (resp. vertically of exponential growth) over $\fD$ if for any compact set $\Ncal\subset \fD$ there exists $m\in \R$ depending on $\Ncal$ such that $\sup_{\nu\in {\fI}(\Ncal)}(1+\|\Im\,\nu\|)^{-m}|J(\nu)|<\infty$ (resp. $\sup_{\nu \in {\fI}(\cN)}e^{-m\|\Im\,\nu\|}|J(\nu)|<+\infty$). The following lemma will be used in the proof of Lemma~\ref{SpectExpPerL-2} (but has no role in other parts of this article). 

\begin{lem} \label{PRbound}
Let $\Lambda$ be a set and $f:\Lambda\rightarrow \R_{+}$ a function. Let $a,b\in \R$, $\kappa\in \N$, and $J_{\lambda}(\nu)$ $(\lambda \in \Lambda)$ a family of holomorphic functions on $\ft_{0,\C}^*$ with the following properties:
\begin{itemize}
\item[(i)] For a fixed $\lambda\in \Lambda$, the function $J_{\lambda}$ is vertically of exponential growth over $\ft_{0}^*$. 
\item[(ii)] For any compact set $\Ncal_0\subset (\ft_{0}^*)^{++}$ we have a constant $C_0>0$ such that   
\begin{align*}
|J_{\lambda}(\nu)|&\leq C_0\,(1+\|\nu\|)^{\kappa}\,f(\lambda)^{a}, \quad |J_{\lambda}(w_{\ell}^{0}\nu)|\leq C_0\, (1+\|\nu\|)^{\kappa}\,f(\lambda)^{b}
\end{align*}
for all $\nu\in {\fI}(\Ncal_0)$ and $\lambda \in \Lambda$.
\end{itemize}
Then for any compact set $\Ncal \subset \ft_{0}^*$, there exists a constant $C_1>0$ such that 
$$
|J_\lambda(\nu)|\leq C_1\,(1+\|\nu\|)^{\kappa}\,f(\lambda)^{(a+b)/2}, \quad\nu\in {\fI}(\Ncal),\,\lambda \in \Lambda.
$$
\end{lem}
\begin{proof} Since $\cN\subset \ft_{0}^*$ is compact, we can choose $\s>0$ large enough such that $(\cN\cup w_\ell^{0}\cN)+\s\rho_{\sB_0}\subset (\ft_{0}^*)^{++}$. For $z\in \C$, set $\mu(z)=(2\s)^{-1}(a-b)z+2^{-1}(a+b)$. Set 
$$
 g(z):=\exp\left(-\kappa(\log(z+2\s)-\tfrac{\pi \ii}{2})\right), $$
where $\log z$ is the principal branch of the logarithm defined as $\log z=\log |z|+\ii \theta$ for $z=|z|e^{\ii \theta}$ $(|z|\not=0,\theta \in [-\pi,\pi))$. Since $\log (t\ii)=\log t+\frac{\pi \ii}{2}$ for $t\in \R$, $|t|>1$, 
\begin{align*}
\log(x+\ii t+2\s)-\log t-\tfrac{\pi \ii}{2}&=\log(x+\ii t+2\s)-\log(t\ii)={\textstyle \int}^{x+2\s+\ii t}_{\ii t}\tfrac{\d u}{u}=O_{\s}\left(\tfrac{1}{t}\right)
\end{align*}
uniformly in $|t|\geq 1$ and $x\in[-\s,+\s]$. 
%If $t>1$, then $\ii t\log t\in \ii \R$; if $t<-1$, then $\ii t\log t=\ii t(\log|t|+\ii(-\pi))=\ii t\log|t|+\pi t$. Hence $|\exp(-\frac{r-q}{2\s}\ii t\log t)|=1$ if $t>1$ and $|\exp(-\frac{r-q}{2\s}\ii t\log t)|=\exp(\frac{q-r}{2\s}\pi t)\leq 1$ ($t<-1$) by $q\geq r$. Noting this, 
We have $g(x+\ii t)\leq \exp\left(-\kappa\log t+O_{\s,\kappa}(1)\right)$ for $|t|\geq 1$ and $x\in [-\s,+\s]$, which shows the bound 
\begin{align}
|g(z)| \ll_{\s,\kappa}(1+|\Im z|)^{-\kappa}
\label{PRbound1} 
\end{align}
on the vertical strip $|\Re\,z|\leq \s$. The bound \eqref{PRbound1} yields a constant $C_2=C_2(\s,r,q)>0$ such that  
$$
\sup_{\Re z=\s}|(1+|z|\|\rho_{\sB_0}\|)^{\kappa}g(z)|\leq C_2, 
\quad \sup_{\Re z=-\s}|(1+|z|\|\rho_{\sB_0}\|)^{\kappa}g(z)|\leq C_2.
$$
Let $C_0=C_0(\cN_0)$ be the constant in (ii) for $\cN_0=(\cN+\s\rho_{\sB_0})\cup w_{\ell}^{0}(\Ncal-\s\rho_{\sB_0})$. Fix a point $\nu \in {\fI}(\Ncal)$ and consider the one variable function
$$
\varphi_{\lambda}(z)=f(\lambda)^{-\mu(z)}\,(1+\|\nu\|)^{-\kappa}g(z)\,J_{\l}(\nu+z\rho_{\sB_0}), \quad z\in \C-(-\infty,-2\s]. 
$$
Then on the line $\Re z=\s$, 
\begin{align*}
|\varphi_\l(z)|&\leq  f(\l)^{-a}|J_\l(\nu+z\rho_{\sB_0})|\times(1+\|\nu\|)^{-\kappa}|g(z)|
\\
&\leq C_0 |(1+\|\nu+z\rho_{\sB_0}\|)^{\kappa}(1+\|\nu\|)^{-r}g(z)|
\leq C_0\,(1+|z|\|\rho_{\sB_0}\|)^{\kappa}|g(z)| \leq C_0C_2.
\end{align*}
and similarly on $\Re z=-\s$. Thus we obtain the inequality $|\varphi_{\lambda}(z)|\leq C_0\,C_2$ on $\Re z=\pm \s$ for all $\lambda \in \Lambda$. From the bound \eqref{PRbound1} combined with (i), we see that the function $\varphi_{\l}(z)$ has an exponential bound $O_{\lambda}(e^{N|\Im z|})$ on the strip $|\Re z|\leq \s$. From the Phragm\'{e}n-Lindel\"{o}f principle, we conclude that the same inequality $|\varphi_{\l}(z)|\leq C_0C_2$ holds true on the strip $|\Re z|\leq \s$ for all $\l\in \Lambda$. Putting $z=0$, we get the desired inequality. \end{proof}

\section{Poincar\'{e} series}

\subsection{Hecke functions} \label{NonArchTestFtn}
A function $f:\sG(\A)\rightarrow \C$ is said to be decomposable if there exist functions $f_v\in C_{\rm{c}}^\infty(\sG(\Q_v))$ for all $v$ such that $f_v=\cchi_{\bK_v}$ for almost all $v=p<\infty$ and $f(g)=\prod_{v}f_v(g_v)$ for all $g=(g_v)_{v}\in \sG(\A)$; in this case we write $f=\otimes_{v}f_v$. For such $f$, we set 
\begin{align}
\tilde f(g)=\int_{\A^\times}f([z] g)\,\d^* z, \quad g\in \sG(\A),  
 \label{CentralProj}
\end{align}
where $\d^* z$ is the Haar measure on $\A^\times$. Then,  
\begin{align}
\tilde f(g)=\prod_{v} \tilde f_v(g_v) \quad \text{with $\tilde f_{v}(g_v)=\int_{\Q^\times_v}f_v([z_v]g_v)\d^*z_v$}
 \label{CentralProjLoc}
\end{align}
for all $g=(g_v)_v\in \sG(\A)$, where $\d^{*}z_v$ is the Haar measure on $\Q_v^\times$. Set 
\begin{align}
\tilde f^{(\nu)}(g)=(\Delta_{\sG_0}(1)^*)^{-1}\int_{\sB_0(\A)}\tilde f(\iota(b^{-1})\,g)\,e^{\langle H_{\sG_0}(b),\nu+\rho_{\sB_0}\rangle} 
\,\d_l b, \quad g\in \sG(\A), \,\nu\in \ft_{0,\C}^{*}, 
 \label{tildefnu-def}
\end{align}
where $\d_lb$ is the left Haar measure on $\sB_0(\A)$. Then, 
\begin{align}
\tilde f^{(\nu)}([z]\iota(b)\,g)=
e^{\langle H_{\sG_0}(b),\nu+\rho_{\sB_0}\rangle} 
\,\tilde f^{(\nu)}(g), \quad z\in \A^\times,\,b\in \sB_0(\A), \,g\in \sG(\A). \label{Equiv-f}
\end{align}
Moreover $\tilde f^{(\nu)}(g)$ is a smooth function of compact support modulo $\sZ(\A)\iota(\sB_0(\A))$ on $\sG(\A)$. For a compact subset $\Ucal\subset \sG(\A)$, we set $X(\Ucal)=\{x\in \A^{n-1}\mid (\taut\sn(x))^{-1}\sn(x)\in \Ucal\}$, which is easily seen to be a compact subset of $\A^{n-1}$ from the relation $(\taut\sn(x))^{-1}\sn(x)=\sn(2x)$. For $f$ as above, we fix a compact set $\omega_{\sB}(f) \subset \sB(\A)$ such that ${\rm supp}(f) \subset \omega_{\sB}(f)\,\bK$ by $\sG(\A)=\sB(\A)\bK$.

\begin{lem} \label{fBoHyouka}
Let $\Ucal\subset \sG(\A)$ be a compact set which contains $\{\taut h^{-1}h\mid h\in \omega_{\sB}(f)\,\bK\,\}$, and $\cN\subset \ft_{0}^*$ a compact set. Then for any $\sigma\in \ft_{0}^*$, we have a bound\begin{align}
|\tilde f^{(\nu)}(g)|\ll_{\cN,f}
e^{\langle H^{\sG_0}(g), \s+\rho_{\sB_0}\rangle}\,\delta(\taut g^{-1}g\in \Ucal), \quad g\in \sG(\A), \, \nu \in \fJ(\cN).
 \label{fBoHyouka-1}
\end{align}
\end{lem}
\begin{proof}
Let $g=[z]\,\iota(b)\,\sn(x)\,k$ with $z\in \A^\times$, $b\in \sB_0(\A)$, $x\in \A^{n-1}$ and $k\in \bK$. Note that $H_{\sG_0}(b)=H^{\sG_0}(g)$. From \eqref{Equiv-f} and \eqref{tildefnu-def}, 
\begin{align*}
\tilde f^{(\nu)}(g)&=e^{\langle H^{\sG_0}(g), \nu+\rho_{\sB_0}\rangle}\, \tilde f^{(\nu)}(\sn(x)k)
\\
&=e^{\langle H^{\sG_0}(g), \nu+\rho_{\sB_0}\rangle} \int_{\sB_0(\A)\times \A^\times} f([z_1]\iota(b_1^{-1})\,\sn(x)k)\,e^{\langle H^{\sG_0}(b_1), \nu+\rho_{\sB_0}\rangle}\d_{l}b_1\d^{*}z_1.
\end{align*}
Since $\omega_{\sB}(f)$ is compact, we may restrict the integration domain of the $(b_1,z_1)$-integral to a compact subset $Y\subset \sB_0(\A)\times \A^\times$. Moreover, since $\sZ(\A)\iota(\sB_0(\A))$ is contained in the fixed point set of the involution $h\mapsto \taut h$, the relation $h=[z_1]\iota(b_1^{-1})\,\sn(x)k\in \omega_{\sB}(f)\bK$ implies $\taut g^{-1} g=\taut k^{-1}\,\sn(2x)k=\taut h^{-1}h \in \cU$. Hence, 
\begin{align*}
|\tilde f^{(\nu)}(\sn(x)k)|\leq \delta(\taut g^{-1}g\in \Ucal)\, C\,
\int_{(b_1,z_1)\in Y} e^{\langle H_{\sG_0}(b_1),\s+\rho_{\sB_0}\rangle}\,\d_lb_1\, \d^{*} z_1
\end{align*}
with $\s=\Re\nu$ and $C=(\Delta_{G_0}(1)^*)^{-1}\, \max_{h\in \sG(\A)}|f(h)|$. The last integral is bounded as $\s$ varies in a compact set $\cN$. 
\end{proof}

For a place $v$, set  
\begin{align}
\tilde f_v^{(\nu)}(g_v)=\int_{\sB_0(\Q_v)}\tilde f_v(\iota(b_v^{-1})\,g_v)\,e^{\langle H_{\sG_0}(b_v),\nu+\rho_{\sB_0}\rangle} \,\d_l b_v, \quad g_v\in \sG(\Q_v), \,\nu\in \ft_{0,\C}^{*}. 
 \label{Loc-tildefv}
\end{align}
Note that $\tilde f_v^{(\nu)}$ is identically $1$ on $\bK_v$ if $v=p<\infty$ and $f_v=\cchi_{\bK_p}$. We have the product formula 
\begin{align}
\tilde f^{(\nu)}(g)=(\Delta_{\sG_0}(1)^{*})^{-1}\prod_{v} \tilde f_{v}^{(\nu)}(g_v), \quad g=(g_v)_v \in \sG(\A).
 \label{prodformulaGsect}
\end{align}

\subsection{Construction of test functions} \label{AdeleTestFtn}
A holomorphic function $\beta$ on $\ft_{0,\C}^{*}$ is said to be vertically rapidly decreasing if for any $r>0$ and any compact subset $\Ncal \subset \ft_{0}^*$, there exists a constant $C>0$ such that
\begin{align}
|\beta(\nu)|\leq C(1+\|\Im \nu\|)^{-r}, \quad \nu \in {\fI}(\Ncal).
\label{Beta-Hyouka}
\end{align}
Let $\fB$ be the space of all such $\beta$. The functions of the form $Q(\nu)e^{T(\nu-\sigma, \nu-\sigma)}$ ($\sigma \in \ft_{0}^*$, $T>0$, $Q\in \C[\ft_{0,\C}]$) belong to $\fB$, where $(\nu,\nu)=\sum_{j=1}^{n-1}\nu_j^2$ for $\nu=(\nu_j)_{j=1}^{n-1}$ and $\C[\ft_{0,\C}]$ denotes the space of polynomial functions on $\ft_{0,\C}^{*}$. For $\s\in \ft_{0}^*$ and $\b\in \fB$, we define $\int_{\fJ(\s)}\beta(\nu)\,\d\nu:=\int_{\ft_{0}^*}\beta(\s+\ii t)\,\d t$, where $\d t$ is the Euclidean Lebesgue measure on $\ft_{0}^{*}$. By \eqref{Beta-Hyouka}, the integral converges absolutely. Moreover, by a repeated application of Cauchy's theorem, it is shown that the integral is independent of the choice of $\s$.  

Let $\fB_0$ be the set of $\beta\in \fB$ divisible by the polynomial 
\begin{align}
r(\nu)&:=\prod_{1\leq i<j\leq n-1}(\nu_i-\nu_j)(1-(\nu_i-\nu_j)^2), \quad \nu=(\nu_j)_{j=1}^{n-1} \in \ft_{0,\C}^*,\label{singpolynom}
\end{align}
i.e., $\b\in \fB_0$ if and only if $\beta(\nu)=r(\nu)\,\beta_0(\nu)$ with some $\b_0\in \fB$. Depending on $f\in C_{\rm{c}}^{\infty}(\sG(\A))$ and on $\beta\in \fB_0$, we define a $\C$-valued function $\tilde \Phi_{f,\b}(g)$ on $\sG(\A)$ as 
\begin{align}
\tilde \Phi_{f,\b}(g)=\int_{\fJ(\s)}
\b(\nu)\,{M}_{\sG_0}(\nu)\,\tilde f^{(\nu)}(g)\, \d\nu,\quad g\in \sG(\A), 
 \label{tildefB}
\end{align}
where 
\begin{align}
{M}_{\sG_0}(\nu):=\prod_{1\leq i<j\leq n-1}\zeta_{\Q}(\nu_i-\nu_j+1)
 \label{NormalFact}
\end{align}
with $\zeta_{\Q}(z)$ being the {completed} Riemann zeta-function. 
%Note that $M_{\sG_0}(\nu)$ is the Euler product of \eqref{f-MsG0v} over all $v$ if $\nu \in \fJ((\ft_{0}^*)^{++})$. 
By the presence of the factor $r(\nu)$ of $\b$, the singularity of ${M}_{\sG_0}(\nu)$ in the integrand is removed. Since $z(1-z)\zeta_\Q(z)$ is a vertically bounded holomorphic function on $\C$, due to the estimate \eqref{fBoHyouka-1}, the integrand  of \eqref{tildefB} belongs to the space $\fB$, hence \eqref{tildefB} is absolutely convergent and is independent of the choice of $\s$.

\subsection{Gauge estimate}
We shall construct a majorant of $\tilde\Phi_{f,\b}$ on $\sG(\A)$. Let $\|\cdot\|_\A:\sG(\A)\rightarrow \R_{+}$ be the height function on $\sG(\A)$ (\cite[\S 1.2.2]{MW}) defined as $\|g\|_{\A}=\prod_{v}\|g_v\|_v$ with $\|g_v\|_v=\max_{1\leq i,j\leq n}\{|(g_v)_{ij}|_{v},\,|(g_v^{-1})_{ij}|_v\}$ for $g=(g_v)_v\in \sG(\A)$. Let $\fS_{\sG}$ be a Siegel domain of $\sG(\A)$ (see \S\ref{sec:SiegelDom}). 

\begin{lem}\label{L4}
Let $\Ucal\subset \sG(\A)$ be a compact set. There exist constants $C'>0$ and $m>0$ such that 
$$
e^{-m\langle H(g),\rho_{\sB}\rangle}\geq C' \,\|h\|_{\A}^{-1}
$$
for all $h\in \sG(\A)$ and $g\in \fS_{\sG}\cap \sG(\A)^1$ with $h\in \sG(\Q)\,g\,\Ucal. $ 
\end{lem}
\begin{proof}
Let $(\chi,V)$ be an irreducible finite dimensional rational representation of $\sG$ of highest weight $2\rho_\sB$ and $\xi_0\in V(\Q)-\{0\}$ a highest weight vector. We define a height function $\|\xi\|_\A$ on the primitive elements of $V(\A)$ and a height function $\|\cdot\|_{\GL(V)(\A)}$ on $\GL(V)(\A)$ and argue in the same way as \cite[Lemma 3.3 (2)]{Tsud}. The constant $m>0$ comes in due to the estimate $\|\chi(g)\|_{\GL(V)(\A)}\ll \|g\|_\A^{m}$. 
\end{proof}

For $r>0$, $q>0$ and an open compact subset $\Vcal\subset \sG(\A_\fin)$, let us define a gauge function $\Xi_{r,q,\Vcal}:\sG(\A)\rightarrow \R_{+}$ as 
\begin{align*}
\Xi_{r,q,\Vcal}(g)&=\inf\{e^{-q\langle {{H}}^{\sG_0}(g),{}^w\rho_{\sB}\rangle}\mid w\in \sS_{n}\}\,\|\taut g^{-1}g\|_{\A}^{-r} \,\delta(\taut g_\fin^{-1}g_\fin\in \Vcal),
\end{align*} 
where $g_\fin=(g_p)_{p<\infty}\in \sG(\A_\fin)$ denotes the finite component of $g=(g_v)_{v}\in \sG(\A)$. 

\begin{lem} \label{L2} 
Let $r,q>0$ and $\Vcal\subset \sG(\A_\fin)$ an open compact subset. There exist a positive constant $C_\Vcal>0$ and a relatively compact neighborhood $\Ucal=\Ucal(\Vcal)$ of the identity in $\sG(\A)$ such that 
\begin{align}
\Xi_{r,q,\Vcal}(g)\leq C_\Vcal \,\Xi_{r,q,\Vcal}(gh), \quad g\in \sG(\A),\,h\in \Ucal. 
\label{L2-1}
\end{align}
\end{lem}
\begin{proof}
Let $\Ucal_\infty$ be a compact neighborhood of the identity in $\sG(\R)$, and $\Ucal_\fin \subset \bK_\fin$ be an open compact subgroup such that $(\taut h)^{-1}\Vcal h \subset \Vcal$ for all $h \in \Ucal_\fin$. Set $\Ucal=\Ucal_\infty \Ucal_{\fin}$. Since $\{H^{\sG_0}(kh)|\,k\in \bK,\,h\in \Ucal\} \subset \ft_{0}$ is compact, there exists a constant $C_1>0$ such that 
\begin{align}
\sup_{w\in \sS_{n} }
|\langle H^{\sG_0}(gh)-H^{\sG_0}(g),{}^w\rho_{\sB}\rangle|\leq C_1 
 \label{L2-2}
\end{align}
for all $(g,h)\in \sG(\A)\times \Ucal$. We have $\delta(\taut g^{-1}g\in \Vcal)\leq \delta((\taut (g h))^{-1}g h\in \Vcal)$ for all $(g,h)\in \sG(\A) \times \Ucal_\fin$. We have
$$
\|(\taut (gh))^{-1}(gh)\|_\A\leq \|\taut h^{-1}\|_\A\|\taut g^{-1}g\|_\A\|h\|_\A \leq C_2 \|\taut g^{-1}g\|_\A
$$
 for all $(g,h)\in \sG(\A)\times \Ucal$ with $C_2=\sup_{h\in \Ucal}\|\taut h^{-1}\|_\A\|h\|_\A$. Thus, we have the desired inequality \eqref{L2-1} by setting $C_\Vcal=\exp(qC_1)\,C_2^{r}$. 
\end{proof}

Noting that $\Xi_{r,q,\Vcal}$ is a left $\sZ(\A)\,\iota(\sB_0(\Q))$-invariant non-negative function, we define  
$$
{\bf \Xi}_{r,q,\Vcal}(g):=\sum_{\gamma\in \sZ(\Q)\,\iota(\sB_0(\Q))\bsl \sG(\Q)} \Xi_{r,q,\Vcal}(\gamma g), \quad g\in \sG(\A) .
$$

\begin{lem} \label{L3}
There exist constants $q_0>0$, $N_1>0$, $m>0$ and $N_2\in \R$ such that for any $q>q_0$ and $r>q N_1+N_2+n-1$ we have the bound: 
\begin{align*}
{\bf \Xi}_{r,q,\Vcal}(g)\ll_{r,q,\Vcal} e^{-m(N_1q+N_2)\langle H(g),\rho_{\sB}\rangle}, \quad g\in \fS_\sG\cap \sG(\A)^1.
\end{align*}
\end{lem}
\begin{proof} %In this proof, we regard $\sB_0$ as a subgroup of $\sG$ by the embedding $\iota$. 
Let $\Ucal\subset \sG(\A)$ be a compact neighborhood of unity as in Lemma~\ref{L2}, which depends on $\Vcal$, and $\bar\Ucal$ its image in the quotient group $\bar \sG(\A)=\sZ(\A)\bsl \sG(\A)$. Let $C_{\Vcal}>0$ be the constant in Lemma~\ref{L2} and set $C=C_{\Vcal}\,\vol(\bar \Ucal)^{-1}$. From \eqref{L2-1}, we have 
{\allowdisplaybreaks
\begin{align*}
C^{-1}\,
{\bf \Xi}_{r,q,\Vcal}(g)
&\leq \int_{\bar \Ucal} \sum_{\gamma\in \sZ(\Q)\iota(\sB_0(\Q))\bsl \sG(\Q)}\Xi_{r,q,\Vcal}(\gamma g h)\,\d h
\\
&=\sum_{\gamma \in \sZ(\Q)\iota(\sB_0(\Q))\bsl \sG(\Q)} \int_{\sZ(\A)\bsl  \sG(\A)}\Xi_{r,q,\Vcal}(h)\cchi_{\bar \Ucal}(g^{-1}\gamma^{-1}h)\,\d h \\
&=\sum_{\gamma\in \sZ(\Q)\iota(\sB_0(\Q))\bsl \sG(\Q)} \int_{\sZ(\A)\iota(\sB_0(\Q))\bsl \sG(\A)} \Xi_{r,q,\Vcal}(h)\sum_{\delta\in \iota(\sB_0(\Q))}\cchi_{\bar \Ucal}(g^{-1}\gamma^{-1}\delta h)\,\d h
\\
&=\int_{\sZ(\A)\iota(\sB_0(\Q))\bsl \sG(\A)} \Xi_{r,q,\Vcal}(h)\,\biggl(\sum_{\gamma\in \sZ(\Q)\bsl \sG(\Q)} \cchi_{\bar \Ucal}(g^{-1}\gamma^{-1}h)\biggr)\,\d h.
\end{align*}}Applying the estimate 
\begin{align}
\sum_{\gamma\in \sZ(\Q)\bsl \sG(\Q)} \cchi_{\bar \Ucal}(g^{-1}\gamma^{-1}h)\ll_{\bar\Ucal} e^{\langle H(g),2\rho_{\sB}\rangle}\,\cchi_{\sG(\Q)\,g\,\Ucal}(h), \quad h\in \sG(\A),\,g\in \fS_{\sG}\cap \sG(\A)^{1}, 
\label{L3-1}
\end{align}
which is found in \cite[p.59-60]{Langlands76} in a setting of Lie groups and whose proof for $\GL_2$ in a setting of adeles is recalled in \cite[Lemma 3.3 (1)]{Tsud}, we obtain
\begin{align}
{\bf \Xi}_{r,q,\Vcal}(g)\ll_{\Vcal} e^{\langle H(g),2\rho_{\sB}\rangle}\,\int_{\sZ(\A)\iota(\sB_0(\Q))\bsl (\sZ(\A)\sG(\Q)\,g\,\Ucal)} \Xi_{r,q,\Vcal}(h)\,\d h, \quad g\in \fS_\sG\cap \sG(\A)^{1}. 
\label{L3-10}
\end{align}
Suppose $h\in \sZ(\A)\iota(\sB_0(\Q))\bsl \sG(\A)$ satisfies $\Xi_{r,q,\Vcal}(h)\not=0$. From the decompositions $\sG(\A)=\sB(\A)\bK$, $\sB(\A)=\sZ(\A)\iota(\sB_0(\A))\,\sn(\A^{n-1})$, and $\sB_0(\A)=\sB_0(\A)^1T_{0,\infty}$, we can set
\begin{align}
h=\iota(bt)\, \sn(x_\infty+x_\fin )\, k, \quad (b\in \omega_{\sB_0},\, t\in T_{0,\infty},\, x_\fin \in X(\Vcal),\, x_\infty\in \R^{n-1},\,k\in \bK),
 \label{L3-2}
\end{align}
where $\omega_{\sB_0}$ is a fundamental domain for $\sB_0(\Q)\bsl \sB_0(\A)^{1}$ and $X(\Vcal)$ is a compact set of $\A_\fin^{n-1}$. If we further suppose $h\in\sZ(\A)\sG(\Q)\,g\,\Ucal$, we have $e^{-m\langle H(g),\rho_{\sB}\rangle}\gg \|h\|_\A^{-1}$ from Lemma~\ref{L4}. Since $\omega_{\sB_0}$, $\bK$ and $X(\Vcal)$ are compact, we easily have $\|h\|_{\A}^{-1}\gg \|t\|_\A^{-1}(1+\|x_\infty\|^2)^{-1/2}$, where $\|x_\infty\|$ is an Euclidean norm on $\R^{n-1}$. Thus, in the integral \eqref{L3-10} only those $h$ of the form \eqref{L3-2} satisfying $e^{-m\langle H(g),\rho_{\sB}\rangle}\gg \|t\|_\A^{-1}(1+\|x_\infty\|^2)^{-1/2}$ matter. For such $h$, from the definition we have 
$$\Xi_{r,q,\Vcal}(h)\ll\inf\{e^{-q\langle H_{\sG_0}(t),{}^w\rho_{\sB}\rangle}\mid w\in \sS_{n}\}\,(1+\|x_\infty\|^2)^{-r/2}.  
$$
By the Iwasawa decomposition of $\sG(\A)$, the integral in \eqref{L3-10} can be replaced by the integral
{\small\begin{align}
&\int_{\R^{n-1}} (1+\|x_\infty\|^2)^{-r/2} \d x_\infty
 \notag
\\
& \times \int_{\substack{t\in T_{0,\infty} \\ \|t\|_\A \gg e^{m\langle H(g),\rho_{\sB}\rangle} (1+\|x_\infty\|^2)^{-1/2}}}\,\inf\{e^{-q\langle H_{\sG_0}(t),{}^w\rho_{\sB}\rangle}\mid w\in \sS_{n} \}\,\delta_{\sB_0}^{-1}(t)\,\d t.
 \label{L3-3}
\end{align}}To discuss the convergence of \eqref{L3-3}, we examine the $t$-integral using the decomposition $T_{0,\infty}$ by subsets ${}^w (T_{0,\infty}(1))=\{t\in T_{0,\infty}|\,\langle H_{\sG_0}(t),{}^w\alpha \rangle\geq 0\,(\alpha \in \Delta(\sG_0,\sT_0))\,\}$ with $w\in \sS_{n-1}$. It suffices to consider the integral
$$
\int_{\substack{t\in  {}^w(T_{0,\infty}(1)) \\ \|t\|_\A \gg {\rm T}(g,x_\infty)}}e^{-q\langle H_{\sG_0}(t),{}^w\rho_{\sB}\rangle}\,\delta_{\sB_0}^{-1}(t)\,\d t
$$
for all $w\in \sS_{n-1}$, where ${\rm T}(g,x_\infty):=\exp(m\langle H(g),\rho_{\sB}\rangle)(1+\|x_\infty\|^2)^{-1/2}$. We make the variable change $t\mapsto {}^w t$ to have the integral over the set $t\in T_{0,\infty}(1)$ such that $\|t\|_\A\gg {\rm T}(g,x_\infty)$. Note that $\|{}^wt\|_\A=\|t\|_\A$. As in \S\ref{sec:SiegelDom}, we introduce the coordinates $y=(y_1,\dots,y_{n-1})\in (\R_+)^{n-1}$ on $T_{0,\infty}$ by the relation $t_j=\prod_{i=j}^{n-1}y_i$ for $1\leq j\leq n-1$. Then $t\in T_{0,\infty}(1)$ if and only if $y\in [1,\infty)^{n-2}\times \R_{+}$, and the constraint $\|t\|_\A \gg {\rm T}(g,x_\infty)$, which reads $\max(t_{j},t_j^{-1}|1\leq j\leq n-1) \gg {\rm T}(g,x_\infty)$, yields $y_{n-1}=t_{n-1}\gg {\rm T}(g,x_\infty)$. Obviously there exists $\{k_j\}_{j=1}^{n-1}\in \Z^{n-1}$ such that $\delta_{\sB_0}(wtw^{-1})=\prod_{j=1}^{n-1}y_j^{k_j}$. Thus from Lemma~\ref{SieDom}, we are reduced to estimate the integral 
\begin{align}
I({\rm T})=\int_{(y_1,\dots,y_{n-1})\in [1,+\infty)^{n-2} \times[{\rm T},\infty)}\prod_{j=1}^{n-1}y_j^{-qm_j-k_j}\,\d^* y_j, \quad {\rm T}>0
 \label{L3-5}
\end{align}
with $m_j=j(n-j)/2\,(1\leq j\leq n-1)$. We have
\begin{align*}
I({\rm T})=
\prod_{j=1}^{n-2} \int_{1}^{\infty} y_j^{-qm_j-k_j}\,\d^* y_j
\,\times \int_{{\rm T}}^{\infty} y_{n-1}^{-qm_{n-1}-k_{n-1}}\d^* y_{n-1},
\end{align*}
which is convergent if $qm_i+k_i>0$ for all $1\leq i\leq n-1$. Under this condition, from the $y_{n-1}$-integral, we have the majorant ${\rm T}^{-qm_{n-1}-k_{n-1}}$ for $I({\rm T})$. In this way, we get
$$
I({\rm T})\ll {\rm T}^{-(qm_{m-1}+k_{n-1})}, \quad {\rm T}>0
$$
as long as $q>q(w):=\max\{\frac{-k_j}{m_j}\mid 1\leq j\leq n-1\}$. Note that $m_{n-1}=\frac{n-1}{2}$ is independent of $w$ but $k_{n-1}$ and $q(w)$ depend on $w$. Since $w$'s are finite in number, we have a constant $q_0>0$, $N_1\,(:=m_{n-1})>0$ and $N_2\in \R$ such that, when $q>q_0$, the integral \eqref{L3-3} is majorized by 
$$
\int_{\R^{n-1}}(1+\|x_\infty\|^2)^{-r/2+(qN_1+N_2)/2} e^{-m(N_1 q +N_2)\langle H(g),\rho_{\sB}\rangle} \,\d x_\infty .
$$
The $x_\infty$-integral is convergent if $r-(qN_1+N_2)>n-1$.  
\end{proof}

\begin{lem} \label{L1}
Let $\Vcal\subset \sG(\A_\fin)$ be an open compact subset such that $\{\taut h^{-1}h\mid h\in (\omega_{\sB}(f))_\fin \bK_\fin \,\}$ is contained in $\Vcal$. Let $\b\in \fB_0$, $D\in U(\fg_\infty)$ and $r,q>0$. Then   
\begin{align}
|[\tilde \Phi_{f,\b}*D](g)|\ll_{f,r,q,D,\beta} \Xi_{r,q,\Vcal}(g), \quad g\in \sG(\A). \label{L1-0}
\end{align}
\end{lem}
\begin{proof} Since $\tilde\Phi_{f,\beta}*D=\tilde \Phi_{f*D,\beta}$, we may suppose $D=1$ without loss of generality. For $w\in \sS_{n}$, set $$\sG(\A)_{w}=\{g\in \sG(\A)\mid \,\langle {}^wH^{\sG_0}(g),\alpha \rangle \geq 0\,(\alpha \in \Delta(\sG,\sT))\}.$$
 Since $\sG(\A)$ is a union of subsets $\sG(\A)_{w}$ $(w\in \sS_{n})$, it is enough to confirm the estimate on each of $\sG(\A)_{w}$. Recall the decomposition $\ft^*=\ft_{0}^*\oplus \fz^*$ from \S\ref{sec:Vector}. Depending on $w$, we choose $\s_{w}\in \ft_{0}^*$ and $c_{w}\in \R$ such that $\s_{w}+c_{w}\zeta=-w^{-1}(q\rho_{\sB})-\rho_{\sB_0}$, where $\zeta$ is a basis vector of $\fz^*$. From Lemma~\ref{fBoHyouka} we have  
\begin{align}
|\tilde f^{(\nu)}(g)|\ll_{\Vcal,\sigma_w} e^{\langle H^{\sG_0}(g), \s_{w}+\rho_{\sB_0}\rangle} \|\taut g^{-1}g\|_\A^{-r}\,\delta(\taut g_\fin^{-1}g_\fin\in \Vcal), \quad \nu \in \fJ(\s_{w}),\,g\in \sG(\A)_{w}.
\label{L1-11}
\end{align}
We have $\langle H^{\sG_0}(g),\s_{w}+\rho_{\sB_0}\rangle=\langle{}^wH^{\sG_0}(g),w(\s_{w}+\rho_{\sB_0}+c_{w}\zeta)\rangle=-q \langle {}^wH^{\sG_0}(g),\rho_{\sB}\rangle$. Hence from \eqref{tildefB} (with the contour being shifted to $\fJ(\sigma_w)$), the absolute value $|{\tilde{\Phi}}_{f,\b}(g)|$ is majorized by 
\begin{align}
&\Bigl\{\int_{\fJ(\s_{w})}|{{M}_{\sG_0}(\nu)\b(\nu)}| \,|\d\nu|\Bigr\}\,e^{-q \langle {}^w H^{\sG_0}(g), \rho_{\sB}\rangle}
\,\|\taut g^{-1}g\|_\A^{-r}\delta(\taut g_\fin^{-1}g_\fin\in \Vcal)
 \label{L1-1}
\end{align}
for $g\in G(\A)_{w}$. Since $\rho_{\sB}-{}^{w}\rho_{\sB}$ for $w\in \sS_{n}$ belongs to the closed cone spanned by $\Delta(\sG,\sT)$, we have $\langle{}^{w'
} H^{\sG_0}(g),\rho_{\sB}\rangle\leq \langle {}^{w}H^{\sG_0}(g),\rho_{\sB}\rangle$ for all $g\in \sG(\A)_{w}$ and $w'\in \sS_{n}-\{w\}$, which shows that $\Xi_{r,q,\Vcal}(g)$ is a majorant of \eqref{L1-1} on $\sG(\A)_{w}$. This completes the proof. \end{proof}

\subsection{The Poincar\'{e} series} \label{Pser}
Since the function $\tilde\Phi_{f,\b}$ constructed in \S\ref{AdeleTestFtn} is left $\sZ(\Q)\iota(\sB_0(\Q))$-invariant, the following sum is well-defined if convergent. 
\begin{align}
&\tilde{\bf \Phi}_{f,\b}(g)=\sum_{\gamma \in \sZ(\Q)\iota(\sB_0(\Q))\bsl \sG(\Q)}\tilde\Phi_{f,\b}(\gamma g),\quad g\in \sG(\A).
 \label{Pser-f1}
\end{align}

\begin{prop} \label{SQint}
Let $f\in C_{\rm{c}}^\infty(\sG(\A))$ and $\b\in\fB_0$. The series \eqref{Pser-f1} converges absolutely and locally uniformly in $g\in \sG(\A)$ defining a function $\tilde {\bf \Phi}_{f,\b}$ on $\sZ(\A)\sG(\Q) \bsl \sG(\A)$ which is smooth and is of rapid decay on $\fS\cap \sG(\A)^1$ for any Siegel domain $\fS \subset \sG(\A)$. In particular, $\tilde{\bf\Phi}_{f,\b}\in L^{2}(Z_\infty\sG(\Q)\bsl \sG(\A))$. \end{prop}
\begin{proof} 
This follows from Lemmas~\ref{L1} and \ref{L3}. Indeed, the statement on the rapid decay is confirmed as follows. For any $l>0$, choose $q_l>0$ such that $m(q_lN_1+N_2)>l$ and $q_l>q_0$, where $N_1>0$, $m>0$, $N_2\in \R$ and $q_0>0$ are constants in Lemma~\ref{L3}. Then fix $r_l>0$ such that $r_l>q_lN_1+N_2+n-1$. Let $\Vcal\subset \sG(\A_\fin)$ be an open compact set containing $\{{}\taut h^{-1}\,h\mid h\in (\omega_{\sB}(f)\bK)_{\fin}\}$. Then from Lemmas~\ref{L1} and \ref{L3}, 
$$
|\tilde\Phi_{f,\beta}(g)|\ll_{f,\beta,l} {\Xi}_{r_l,q_l,\Vcal}(g)\ll_{l}e^{-m(N_1q_l+N_2)\langle H(g),\rho_{\sB}\rangle}\ll e^{-l\langle H(g),\rho_{\sB}\rangle}$$
for $g\in \fS_{\sG}\cap \sG(\A)^1$. This implies that $\tilde \Phi_{f,\beta}$ is rapidly decreasing on $\fS\cap \sG(\A)^1$.  
\end{proof}

\begin{prop} \label{CuspidalPer} Let $q$ be a prime number and $(\tau_q,V(\tau_q))$ an irreducible smooth supercuspidal representation of $\sG(\Q_q)$ with trivial central character. Let $f=\otimes_{v}f_v \in C_{\rm c}^{\infty}(\sG(\A))$ be a decomposable function such that $\tilde f_q$ is a matrix coefficient of $\tau_q$. Then for any $\b\in \fB_0$, $\tilde {\bf \Phi}_{f,\b} \in L_{\rm{cusp}}^2(Z_\infty\sG(\Q)\bsl \sG(\A))$. \end{prop}
\begin{proof} By Proposition~\ref{SQint}, it remains to check its cuspidality of $\tilde {\bf \Phi}_{f,\beta}$, i.e., . the vanishing of the integral \begin{align}
\int_{\sN(\Q)\bsl \sN(\A)}\tilde {\bf \Phi}_{f,\beta}(ug)\d u
 \label{CuspidalPer-1}
\end{align}
for any $g\in \sG(\A)$ and for the unipotent radical ${\bf N}$ of a proper $\Q$-parabolic subgroup of $\sG$, where $\d u$ is a Haar measure on ${\bf N}(\A)$. Note that ${\bf N}(\Q)\bsl {\bf N}(\A)$ is compact. It is seen that \eqref{CuspidalPer-1} depends linearly on the $q$-component $\tilde f_q$ of $\tilde f$, which in turn depends bi-linearly on a pair of vectors $(\xi,\xi')\in V(\tau_q)\times V(\tau_q^\vee)$ as $\tilde f_{q}(h)=\langle \tau_q(h)\,\xi,\xi'\rangle$ ($h\in \sG(\Q_q))$. We fix $\xi'$ and view \eqref{CuspidalPer-1} as a linear functional in $\xi\in V(\tau_q)$, say $\ell$. Fix $g=g_{q} g^{(q)}\in \sG(\A)$ with $g_{q} \in \sG(\Q_q)$. Set $\ell'(\eta)=\ell(\tau_q(g_q)^{-1}\eta)$ for $\eta\in V(\tau_q)$. Then, by a variable change, $\ell'(\tau_q(u)\eta)=\ell'(\eta)$ for any $(\eta,u)\in V(\tau_q)\times {\bf N}(\Q_q)$, which implies that $\ell'$ factors through the Jacquet module of $\tau_q$ along ${\bf N}(\Q_q)$. Since $\tau_q$ is supercuspidal, the Jacquet module is zero (\cite[Theorem 5.3.1]{Casselman}, \cite[Theorem 3.21]{BernsteinZelevinskii}). This completes the proof. 
\end{proof}

\noindent
{\bf Remark}: Here is a heuristic explanation why we prefer \eqref{Pser-f1} to the series 
\begin{align} 
&\sum_{\gamma\in \sZ(\Q)\iota(\sB_0(\Q))\bsl \sG(\Q)}\tilde f^{(\nu)}(\gamma g), \quad g\in \sG(\A),
 \label{Pser-f0}
\end{align} which seems to be a more natural object to consider at a first glance. For our purpose, the square-integrability of the Poincar\'{e} series is critical. Due to the obvious resemblance to the Eisenstein series, \eqref{Pser-f0} might be called the relative Eisenstein series. At this point, it should be recalled that the usual Eisenstein series is not square-integrable ({\it cf}. \cite[\S II.1.7]{MW}). Thus there seems to be no good reason to expect that \eqref{Pser-f0} yields an $L^2$-function on $\sG(\Q)\bsl \sG(\A)^1$, even if its absolute convergence might be well settled. (Note that the point-wise bound \eqref{L1-11} with decaying majorant is only valid on the subset $\sG(\A)_{w}$.) Still, it might be the case that if a supercuspidal matrix coefficient is present as a factor of $f$, which will be assumed after \S\ref{SPEXP}, the series \eqref{Pser-f0} with large $\Re\nu$ can be absolutely convergent to a cusp form on $\sG(\A)^1$. However, in this article we do not pursue \eqref{Pser-f0} any further and will work solely with \eqref{Pser-f1}; indeed, our method by $\tilde{\bf\Phi}_{f,\beta}$ works on a general case involving no supercuspidal components (\cite{Tsuzuki2021}) and the function $\beta(z)$ for smoothing is harmless, makes the analysis easier in some aspect and is expected to be removed eventually (see \S\ref{sec: sumformula}).   

\subsection{Unramified principal series} \label{sec:URps}
Let $v$ be a place of $\Q$. The set of all those irreducible unitary generic representations of $\sG(\Q_v)$ upt to equivalence is denoted by $\widehat {\sG(\Q_v)}_{\rm gen}$. A represnetation $\pi \in \widehat {\sG(\Q_v)}_{\rm gen}$ is called spherical if $\pi^{\bK_v}\not=0$; $\widehat {\sG(\Q_v)}_{\rm gen}^{\rm ur}$ denote the set of all spherical $\pi \in \widehat {\sG(\Q_v)}_{\rm gen}$. For $\nu\in \ft_{\C}^{*}$, the unramified principal series representation $I_v^{\sG}(\nu)$ of $\sG(\Q_v)$ is defined on the $\C$-vector space of all the smooth functions $\vf:\sG(\Q_v)\rightarrow \C$ such that $\vf(bg)=\prod_{i=1}^{n}|b_{ii}|_v^{\nu_j+\frac{n-2j+1}{2}}\,\vf(g)$ for all $b=(b_{ij})_{ij} \in \sB(\Q_v)$ and $g\in \sG(\Q_v)$ by letting the group $\sG(\Q_v)$ act on the space by right-translation. 

It is known that $\pi\in \widehat{\sG(\Q_v)}_{\rm gen}^{\rm ur}$ is isomorphic to $I_{v}^{\sG}(\nu(\pi))$ with a point $\nu(\pi)\in (\C^{n}/\ell_{v}\Z)^{n}$ which is uniqe up to coordinate-permutations and will be referred to as the spectral parameter of $\pi$, where $\ell_p=2\pi (\log p)^{-1}$ if $v=p<\infty$ and $\ell_\infty=0$; let $\fX_v^{0+}$ be the locus of the points $\nu(\pi)/\ii$. Thus $\fX_v^{0+}/\sS_n \cong \widehat {\sG(\Q_v)}_{\rm gen}^{ur}$ by the map $\nu \mapsto I_v^{\sG}(\nu)$. From the explicit determination of $\fX_v^{0+}$ in \cite{Tadic1986}, \cite{Tadic2009}, it is seen that $\fX_v^{0}:=(\R/\ell_v\Z)^{n}$ is contained in $\fX_v^{0+}$ and that 
\begin{align}
|\Re (\nu_j)|<1/2 \qquad \nu=(\nu_{j})_{j=1}^{n}\in \ii \fX_{v}^{0+}. 
\label{ReBound}
\end{align} 
The representation $I_v^{\sG}(\nu)$ with $\nu\in \ii \fX_v^{0+}$ is tempered if and only if $\nu\in \ii \fX_{v}^{0}$. The central character of $I_{v}^{\sG}(\nu)$ is trivial if and only if $\nu$ belongs to the subset $\ii\fX_v^{0+}(1)$ defined in \S\ref{GWP}.

\subsection{Archimedean spectral parameters}
For $\pi \in \widehat{\sG(\R)}_{\rm gen}^{\rm ur}$, define
\begin{align}
 \fQ(\pi)=1+\|\nu(\pi)\|^2,
 \label{def:fq-infty}
\end{align}
where $\nu(\pi)\in \ii\fX_{\infty}^{0+}$ is the spectral parameter of $\pi$ and $\|\,\|$ is the norm on $\ft_{\C}^*$. Let $\Omega$ be the Casimir element of $\sG(\R)$ corresponding to the $\sG(\R)$-invariant bilinear form $\tr(XY)$ on the Lie algebra $\fg_\infty=\Mat_{n}(\R)$ (\cite[Ch VIII \S3]{Knapp}). Then $\pi(\Omega)$ acts on the $C^\infty$-vectors of $\pi$ by the scalar $\chi_{\pi}(\Omega):=(\nu(\pi),\nu(\pi))-\|\rho_{\sB}\|^2$ (\cite[Proposition 8.22]{Knapp}), where $(\,,\,)$ denote the standard bi-linear form on $\ft^{*}_\C=\C^{n}$ so that $(x,\bar x)=\|x\|^2$ for $x\in \ft_{\C}^{*}$. Set $\nu(\pi)=x+\ii y$ with $x,y\in \R^n$, Then from $\chi_{\pi}(\Omega)=-\|\rho_{\sB}\|^2+\|x\|^2-\|y\|^2+2\ii\,(x,y)$ being non-positive real number,
\begin{align}
&(x,y)=0, 
\quad -\chi_{\pi}(\Omega)=\|y\|^2-\|x\|^2+\|\rho_{\sB}\|^2 \geq 0. 
 \label{Mar06-0}
\end{align}
By this and by the uniform estimate $\|x\|^2\leq n$ for $\pi$ obtained by \eqref{ReBound}, we easily have the uniform bound   
\begin{align}
2-\chi_\pi(\Omega) \asymp \fq(\pi), \quad \pi \in \widehat{\sG(\R)}_{\rm gen}^{\rm ur}.
 \label{GaugeInfChar}
\end{align}

\begin{lem} \label{SpectExpPerL-4-1}
For any $m>0$ and for any $f\in C_{\rm c}^{\infty}(\bK_\infty\bsl \sG(\R)/\bK_\infty)$, we have a constant $C_m(f)>0$ such that $|\widehat{f}(\pi)|\leq C_m(f)\,\fQ(\pi)^{-m}$ for all $\pi \in \widehat {\sG(\R)}_{\rm gen}^{\rm ur}$. 
\end{lem}
\begin{proof} 
Set $\nu=\nu(\pi)$ and write it as $\nu=x+\ii y$ with $x,y\in \R^{n}$.  Let $(\,|\,)$ be the $\sG(\R)$-invariant inner product on the space of $\pi$. Fix a unit vector $\xi_0\in \pi^{\bK_\infty}$. Note that $\xi_0$ is a $C^\infty$-vector of $\pi$ ({\it cf}. \cite[Theorem 8.1 and Proposition 8.5]{Knapp}). From the relation $\pi(D)\,\xi_0=(2-\chi_{\pi}(\Omega))\,\xi_0$ for $D:=2-\Omega$, combined with $\pi(f)\,\xi_0=\widehat{f}(\pi)\xi_0$, we have
\begin{align*}
\{2-\chi_{\pi}(\Omega) \}^{m}\,\widehat{f}(\pi)=(\pi(f)\xi_0|\pi(D^{m})\,\xi_0)
=(\pi(f*D^{m})\,\xi_0|\xi_0)
\end{align*}
noting that $D$ is in the center of $U(\fg_{\infty})$. Hence, by the Cauchy-Schwarz inequality and by \eqref{GaugeInfChar}, 
$$|\widehat{f}(\pi)|\leq \|\pi(f*D^m)\xi_0\|\,\{2-\chi_{\pi}(\Omega)\}^{-m} \ll \|\pi(f*D^m)\xi_0\|\fq(\pi)^{-m}.
$$
Since $f*D^m \in C_{\rm c}^{0}(\sG(\R))$, we have $C_m'(f):=1+\int_{\sG(\R)} |f*D^m(g)|\,\d g$ is a positive constant independent of $\pi$ and $\|\pi(f*D^m)\xi_0\|\leq C_m'(f)$.
\end{proof}

\subsection{The standard $L$-function} \label{sec:Lfunction} 
Basis references here are \cite{GodmentJacquet}, \cite{JPSS1981-1}, \cite{Jacquet5}, \cite{JPSS}, \cite{JPSS2}, \cite{JPSS1983}, \cite{Cogdell}. Let $\pi\cong \otimes_{v}\pi_{v} \in \Pi_{\rm cusp}(\sG)_{\sZ}$; then $\pi_v \in \widehat {\sG(\Q_v)}_{\rm gen}$ for all $v$ (\cite{Shalika}) and $\pi_{v}\in {\sG(\Q_v)}_{\rm gen}^{\rm ur}$ for $v$ outside a finite set $S_\pi$ of places of $\Q$. The standard $L$-function of $\pi$ is originally defined as the absolutely convergent Euler product $L(s,\pi)=\prod_{v\leq \infty} L(s,\pi_v)\, (\Re s>1)$ (\cite[Theorem (5.3)]{JPSS}) of local $L$-functions $L(s,\pi_v)$ defined either by \cite{GodmentJacquet} or as a special case of the Rankin-Selberg convolution $L$-function $L(s,\pi_v\times {\bf 1})$ for $\GL_n\times \GL_1$ (see \cite[Theorem (4.3)]{JPSS1981-1}, \cite[\S (5.1)]{JPSS1983}). Since $n\geq 3$, the function $L(s,\pi)\,(\Re s>1)$ admits a holomorphic continuation to $\C$ with a functional equation relating $L(s,\pi)$ and $L(1-s,\pi^\vee)$, and is bounded on any vertical strip (\cite[Theorem 4.1]{Cogdell}). Let $c(\pi) \in \N$ be the conductor of $\pi$, i.e., $c(\pi)=\prod_{p} c(\pi_p)$, where $c(\pi_p)$ is the conductor of $\pi_p$ defined as the least non-negative power of $p$ such that $\pi_p^{\bK_1(c(\pi_p)\Z_p)}\not=0$ (\cite[(5.1) Th\'{e}or\`{e}me]{JPSS1981}). Note that $\pi_p$ is $\bK_p$-spherical if and only if $c(\pi_p)=1$, which implies $S_\pi-\{\infty\}=\{v=p<\infty\mid c(\pi_p)>1\,\}$. 
%Therefore, $\pi_v$ for $v\not\in S_\pi$ is a unitary generic $\bK_v$-spherical representation of $\sG(\Q_v)$; as such, we have the spectral parameter $\nu_v=(\nu_{v,j})_{j=1}^{n} \in \ii\fX_v^{0+}/\sS_n$ of $\pi_{v}$ (see \S \ref{sec:URps}). For $v\not \in S_\pi$, we have $L(s,\pi_v)=\prod_{j=1}^{n}\Gamma_{\R}(s+\nu_{\infty,j})$ if $v=\infty$ and $L(s,\pi_v)=\prod_{j=1}^{n}(1-p^{-\nu_{v,j}-s})^{-1}$ if $v=p<\infty$. For $p\in S_\pi$, it is known that $L(s,\pi_p)$ has of the form $(1+P_p(p^{-s}))^{-1}$ with a certain polynomial $P_p(X)\in X\C[X]$ of degree no greater than $n$ (\cite[Proposition (2.5)]{JPSS1981}). 

\subsection{The minimal parabolic Eisenstein series} \label{sec:minEis}
We need a very specific Eisenstein series on $\sG_0(\A)$ described as follows. For $\nu \in \ft_{0,\C}^*$, define a function $\sf_{\sG_0}^{(\nu)}:\sG_0(\A) \rightarrow \C$ by  
 \begin{align}
\sf_{\sG_0}^{(\nu)}(h)=\exp(\langle H_{\sG_0}(h),\nu+\rho_{\sB_0}\rangle), \quad h \in \sG_0(\A).
 \label{minEis-section}
\end{align}
Then $\sf_{\sG_0}^{(\nu)}(bh)=e^{\langle H_{\sG_0}(b),\nu+\rho_{\sB_0}\rangle}\,\sf_{\sG_0}^{(\nu)}(h)$ for all $(h,b)\in \sG_0(\A)\times \sB_0(\A)$, $\sf_{\sG_0}^{(\nu)}$ is right $\bK_{\sG_0}$-invariant and $\sf_{\sG_0}^{(\nu)}(1_{n-1})=1$. By $\sG_0(\A)=\sB_0(\A)\bK_{\sG_0}$, these properties determine $\sf_{\sG_0}^{(\nu)}$ uniquely. Set
\begin{align}
E(\nu;h):=\sum_{\delta \in \sB_0(\Q)\bsl \sG_0(\Q)} \sf_{\sG_0}^{(\nu)}(\delta h), \quad (\nu,h)\in \fJ((\ft_{0}^{*})^{++})\times \sG_0(\A),
\label{EisserieDef}
\end{align}
which is absolutely and normally convergent (\cite[Proposition II.1.5]{MW}) yielding a right $\bK_{\sG_0}$-invariant automorphic form on $\sG_{0}(\A)$. From \cite[Theorem 1 in Appendix I]{Langlands76}, we see that $\nu\mapsto E(\nu;h)$ is continued to a meromorphic function on $\ft_{0,\C}^{*}$ satisfying the functional equation 
\begin{align}
\hat E(w \nu;h)=\hat E(\nu;h), \quad w\in \sS_{n-1},\,h\in \sG_0(\A),
 \label{hatE*-feq}
\end{align}
where $\hat E(\nu;h):={M}_{\sG_0}(\nu)\,E(\nu;h)$ is the normalized Eisenstein series with $M_{\sG_0}(\nu)$ being defined by \eqref{NormalFact}. The function $\nu \mapsto \hat E(\nu;h)$ is holomorphic away from the possible poles on the union of the hyperplanes $\nu_i-\nu_j=\e$ for $\e\in \{1,-1\}$, $1\leq i<j\leq n-1$; to remove these singularities, we set
\begin{align}
\hat E^{*}(\nu;h)&=r(\nu)\,\hat E(\nu;h), \quad h\in \sG_0(\A),
 \label{regcompEis}
\end{align}
where $r(\nu)$ is the polynomial \eqref{singpolynom}. Note that $\hat E^{*}(\nu;zh)=|\det\,z|_\A^{c}\hat E^{*}(\nu;h)$ for $z\in \sZ_0(\A)$ with $c:=\frac{1}{n-1}\sum_{j=1}^{n-1}\nu_j$. We need a uniform bound of the normalized Eisenstein series $\hat E^{*}(\nu;h)$ as follows by Phragm\'{e}n-Lindel\"{o}f principle. For $E(\nu;h)$, a uniform polynomial estimate is much more delicate (\cite{Yukie}, \cite{Lapid2006}).
\begin{lem} \label{SpectExpPerL-2}
For a fixed $h\in \sG_0(\A)$, the function $\nu \mapsto \hat E^{*}(\nu;h)$ on $\ft_{0,\C}^*$ is entire. Let $\fS_{\sG_0}\subset \sG_0(\A)$ be a Siegel domain; then for any compact set $\Ncal \subset \ft_{0}^*$ there exist constants $C>0$ and $r_0>0$ such that 
$$
|\hat E^{*}(\nu;h)|\leq C \,e^{r_0\langle H_{\sG_0}(h),\rho_{\sB_0}\rangle}\max(|\det\,h|_\A,|\det\,h|_\A^{-1})^{r_0}, \quad \nu \in {\fJ}(\Ncal), \quad h\in \fS_{\sG_0}.
$$
\end{lem}
\begin{proof} Set $f(h)=e^{\langle H_{\sG_0}(h),\rho_{\sB_0}\rangle}\max(|\det\,h|_\A,|\det\,h|_\A^{-1})$ for $h\in \fS_{\sG_0}$. From \cite[Theorem 1 (iii) in Appendix 1]{Langlands76}, for a fixed $h\in \sG_0(\A)$ and for a compact set $\cN\subset \ft_{0}^*$ there exists $d\in \N$ such that $|\hat E^{*}(\nu;h)|\ll_{h,\cN}(1+\|\Im(\nu)\|)^{d}$ for $\nu\in {\fJ}(\cN)$. For a compact set $\cN\subset (\ft_{0}^*)^{++}$ we have a uniform estimate of $\hat E^{*}(\nu;h)$ on $\fI(\cN)$: there exists $r>0$ such that $|\hat E^{*}(\nu;h)|\ll_{\Ncal} f(h)^{r}$ for $(\nu, h)\in {\fJ}(\cN)\times \fS_{\sG_0}$ (\cite[\S~II.1.5]{MW}). The same uniform estimation is valid for $\hat E^{*}(w_{\ell}^{0}\nu;h)$ by \eqref{hatE*-feq}. Note that $|r(\nu)|$ is $\sS_{n-1}$-invariant. By applying Lemma~\ref{PRbound} (with $\Lambda=\fS_{\sG_0}$, $J_{h}(\nu)=\hat E^{*}(\nu;h)$), we are done.    
\end{proof}

\subsection{A uniform estimate of the Rankin-Selberg integral} \label{sec: UERSI}
 For a cusp form $\varphi$ on $\sG(\A)^1$ and for a continuous function $E$ on $\sG_0(\Q)\bsl \sG_0(\A)$ of moderate growth, the global Rankin-Selberg integral is defined by 
\begin{align}
{Z}(s,E,\varphi)=\int_{\sG_0(\Q)\bsl \sG_0(\A)}E(h)\,\varphi(\iota(h))|\det h|_\A^{s-1/2}\,\d h, \quad s\in \C.
 \label{G0G-zetaInt}
\end{align}
The goal of this subsection is Lemma~\ref{SpectExpPerLLL}, which yields a certain uniformity of the absolute convergence of \eqref{G0G-zetaInt} for all $s\in \C$ when $(\varphi, E)$ varies in a family. Let $K_\fin=\prod_{p<\infty }K_p$ be an open compact subgroup of $\sG(\A_\fin)$ and set $K=\bK_\infty K_\fin$. Recall the set $\Pi_{\rm{cusp}}(\sG)^{K}$ from \S\ref{GWP}. For each $\pi\cong \otimes_{v}\pi_{v} \in \Pi_{\rm{cusp}}(\sG)_{\sZ}^{K}$, we fix an orthonormal basis $\Bcal(\pi^{K})$ of the $K$-fixed part of $\pi$. We refer to \S\ref{sec:convolutionP} for the definition of the convolution product on $\sG(\A)$. Recall the quantity $\fq(\pi_{\infty})$ defined by \eqref{def:fq-infty} and that the convolution product $*$ on $\sG(\A)$ (see \S\ref{sec:convolutionP}).

\begin{lem} \label{SpectExpPerL-1}
Let $f\in C_{\rm c}^{\infty}(K\bsl \sG(\A)/K)$ and $\fS\subset \sG(\A)$ any Siegel domain. For any $m>0$ there exists $r>0$ such that
\begin{align*}
|\varphi(g)|\ll_{\fS,m,K_\fin, f} e^{-m\langle H(g),\rho_{\sB}\rangle}\,\fQ(\pi_\infty)^{r}, \quad g\in \sB(\Q)\fS,\,\varphi\in \Bcal(\pi^{K})\cup \Bcal(\pi^K)*f, \,\pi \in \Pi_{\rm{cusp}}(\sG)_\sZ^{K}. 
\end{align*}
\end{lem}
\begin{proof} Since both $\varphi$ and $H$ are left $\sB(\Q)$-invariant, we may suppose $g\in \fS$. The estimation for $\varphi \in \Bcal(\pi^K)$ is proved by the same argument as \cite[\S 15.2]{Tsud} and \cite[\S 5.3]{GelbartLapid} by using \eqref{GaugeInfChar}. To obtain the estimate for $\varphi\in \Bcal(\pi^K)*f$, we argue as follows. Since $\Ucal={\rm supp}(\check f)$ is compact, by Lemma~\ref{SiegelDomL0} (ii) applied to $\sG(\A)$, there exists a Siegel domain $\fS'\subset \sG(\A)$ and a constant $C \in \R$ such that $\fS\,\Ucal\subset \sB(\Q)\fS'$ and $\langle H(gh)-H(g), \rho_{\sB}\rangle\geq C $ for $(g,h)\in \fS\times \Ucal$. If $C_1>0$ is a constant such that $|\varphi(g)|\leq C_1 e^{-m\langle H(g),\rho_{\sB}\rangle}\,\fQ(\pi_{\infty})^{r}$ for $g\in \sB(\Q)\fS'$, $\pi \in \Pi_{\rm cusp}(G)_\sZ^{K}$ and $\varphi\in \Bcal(\pi^{K})$, then   \begin{align*}
|[\varphi*f](g)|\leq \int_{\Ucal}|\varphi(gh)|\,|f(h^{-1})|\,\d h
\leq\,C_1 e^{-mC}\,\vol(\Ucal)\,\|f\|_{\infty}\times \fQ(\pi_\infty)^{r}\, e^{-m\langle H(g),\rho_{\sB}\rangle}
\end{align*}
for $g\in \fS$. This completes the proof. 
\end{proof}

\begin{cor} \label{SpectExpPerCor-1}
 Let $f\in C_{\rm c}^{\infty}(K\bsl \sG(\A)/K)$ and $\cU\subset\sG(\A)$ a compact subset. For any $m>0$ there exists $r>0$ such that
\begin{align*}
|\varphi(\iota&(t)g)|\ll_{\Ucal,m,K_\fin,f} \max(t_{n-1},t_{n-1}^{-1})^{-m}\,\delta_{\sB_0}(t)^{-m}\,\fQ(\pi_{\infty})^{r}
\end{align*}
for all $(t,g)\in T_{0,\infty}\times \Ucal$ and for all $\varphi\in \Bcal(\pi^{K})\cup \Bcal(\pi^K)*f$ with $\pi \in \Pi_{\rm{cusp}}(\sG)_\sZ^{K}$. 
\end{cor}
\begin{proof}
Set $\lambda_{\pm}=\sum_{j=1}^{n-2}j(n-1-j)\alpha_j\pm \alpha_{n-1}$ with $\alpha_j=\varepsilon_j-\varepsilon_{j+1}\,(1\leq j\leq n-1)$. Since ${}^{w}\lambda_{\pm}$ with $w\in \sS_n$ is a linear combination of simple roots $\alpha_j\,(1\leq j\leq n-1)$, it is possible to choose a constant $\kappa>0$ such that
\begin{align}
\langle H(a), {}^w\lambda_{\pm} \rangle \leq \langle H(a), \kappa \rho_{\sB}\rangle \quad w\in \sS_{n},\,a\in T_\infty(1),
 \label{SpectExpPerCor-1-f1}
\end{align}
where $T_{\infty}(1)$ is the positive chamber of $T_{\infty}$ (see \S\ref{sec:SiegelDom}). Since $\sS_{n}\Ucal$ is compact, by Lemma~\ref{SiegelDomL0} (ii) applied to $\sG(\A)$, we can take a Siegel domain $\fS$ of $\sG_(\A)$ such that 
\begin{align}
T_{\infty}(1)\,\sS_n\,\Ucal \subset \sB(\Q)\fS. 
 \label{SpectExpPerCor-1-f2}
\end{align}
From Lemma~\ref{SpectExpPerL-1} applied to this Siegel domain, we have a constant $r>0$ such that 
\begin{align}
|\varphi(g)|\ll_{\Ucal,m,K_\fin,f}e^{-m\kappa\langle H(g),\rho_{\sB}\rangle}\fQ(\pi_{\infty})^{r} 
\label{SpectExpPerCor-1-f3}
\end{align}  
for $g\in \sB(\Q)\fS$, $\pi \in \Pi_{\rm cusp}(\sG)_\sZ^{K}$, and $\varphi\in \Bcal(\pi^{K})\cup \Bcal(\pi^{K})*f$. Let $t\in T_{0,\infty}$ and $g\in \Ucal$; then there exists a $w\in \sS_n$ such that $\iota(t)=w aw^{-1}$ with $a\in T_{\infty}(1)$. By \eqref{SpectExpPerCor-1-f1}, \eqref{SpectExpPerCor-1-f2} and \eqref{SpectExpPerCor-1-f3}, we have
\begin{align*}
|\varphi(\iota(t)g)|=|\varphi(waw^{-1}g)|=|\varphi(aw^{-1}g)|
&\ll e^{-m\kappa\langle H(aw^{-1}g),\rho_{\sB}\rangle}\fQ(\pi_\infty)^{r}
\\
&\ll e^{-m\kappa\langle H(a),\rho_{\sB}\rangle}\,\fQ(\pi_\infty)^{r}
\\
&\ll
\inf(e^{-m\langle H(a),{}^{w^{-1}}\lambda_{+} \rangle},e^{-m\langle H(a),{}^{w^{-1}}\lambda_{-}
\rangle})\,\fQ(\pi_\infty)^{r}.
\end{align*}
Since $\langle H(a),{}^{w^{-1}}\lambda_{\pm}\rangle=\langle H(waw^{-1}),\lambda_{\pm}\rangle=\langle H_{\sG_0}(t),\lambda_{\pm}\rangle$ and $\exp(\langle H_{\sG_0}(t), \lambda_{\pm}\rangle)=\delta_{\sB_0}(t)t_{n-1}^{\pm 1}$ by \eqref{y-coordinate}, we are done. \end{proof}

\begin{lem}\label{SpectExpPerLLL}
Let $\Ncal\subset \ft_{0}^*$ be a compact set such that $\hat E(\nu;h)$ is regular on $\fI(\Ncal)$. Let $f\in C_{\rm c}^{\infty}(K\bsl \sG(\A)/K)$. There exists a constant $r>0$ such that 
$$
\int_{\sG_0(\Q)\bsl \sG_0(\A)} |\hat E(\nu;h)|\,|\varphi(\iota(h))|\,\d h
 \ll_{K_\fin, f,\Ncal} \fQ(\pi_\infty)^{r} 
$$
for $\varphi \in \cB(\pi^{K})\cup \cB(\pi^{K})*f$ with $\pi \in \Pi_{\rm{cusp}}(\sG)_\sZ^{K}$ and for $\nu \in \fI(\Ncal)$.
\end{lem}
\begin{proof} Let $\fS_{\sG_0}$ be a Siegel domain of $\sG_0(\A)$. Then, by Lemma~\ref{SiegelDomL0} (i), $\fS_{\sG_0}\subset T_{0,\infty}(c)\,\Ucal_0$ with some $c>0$ and a compact set $\Ucal_0\subset \sG_0(\A)$. Recall the element $a_{\sG_0}(y)$ from \S\ref{sec:SiegelDom}. Then $\det a_{\sG_0}(y)=\prod_{j=1}^{n-1}y_j^{j}$ for $y=(y_j)_{j=1}^{n-1}\in (\R_{+})^{n-1}$, which shows the estimate 
$$
\max(\det a_{\sG_0}(y),\det a_{\sG_0}(y)^{-1})\ll \biggl(\prod_{j=1}^{n-2}y_j^j \biggr) \times \max(y_{n-1}, y_{n-1}^{-1})^{n-1} \quad \text{for $y\in [c,\infty)^{n-2} \times \R_{+}$}.
$$
Thus by Lemma~\ref{SieDom} and \eqref{y-coordinate}, using Lemma~\ref{SpectExpPerL-2} and Corollary \ref{SpectExpPerCor-1}, we have the following majorizations. {\allowdisplaybreaks
\begin{align*}
\int_{\fS_{\sG_0}}|\hat E(\nu;h)|\,|\varphi(\iota(h))|\,\d h
&\ll
\int_{[c,\infty)^{n-2}\times \R_+} 
\{\prod_{j=1}^{n-2}y_j^{j(n-1-j)}\}^{r_0/2}\,
\{\prod_{j=1}^{n-2}y_j^j\}^{r_0}\max(y_{n-1},y_{n-1}^{-1})^{(n-1)r_0}
 \\
&\quad \times 
\{\prod_{j=1}^{n-2}y_j^{j(n-1-j)}\}^{-m-1}\,\max(y_{n-1}, y_{n-1}^{-1})^{-m}\,\prod_{j=1}^{n-1}\d^* y_j
\times \fQ_\infty(\pi)^{r}
\end{align*}}
for $\varphi \in \Bcal(\pi^K)\cup \Bcal(\pi^K)*f$, where the implied constants are independent of $(\varphi,\pi)$ and $\nu$ but depend on $K_\fin$, $f$ and $\Ncal$. By choosing $m$ large enough, the last integral is convergent. 
\end{proof}

\subsection{The inner product formula} \label{InnProdFor}
For $\b\in \fB_0$, set
\begin{align*}
\Ecal_{\b}(h)=\int_{\fJ(\s)} 
\b(\nu)\,\hat E(\nu;h)
\,\d\nu, \quad h\in \sG_0(\A) 
\end{align*}
with $\s\in \ft_{0}^{*}$. Note that the integrand is holomorphic since $\b$ is divisible by $r(\nu)$. From Lemma~\ref{SpectExpPerL-2}, the integral $\Ecal_{\b}(h)$ converges absolutely and is independent of the choice of $\s$, and the function $h\mapsto \Ecal_{\b}(h)$ is of moderate growth on $\sG_0(\A)$. 
 
\begin{lem} \label{RSconv}
 Let $\nu \in {\fJ}((\ft_{0}^{*})^{++})$ and suppose $f\in C_{\rm{c}}^{\infty}(\sG(\A))$ is left $\iota(\bK_{\sG_0})$-invariant. Then for any cusp form $\varphi$ on $\sG(\A)^1$, both of the integrals 
\begin{align}
\int_{\sZ(\A)\sB_0(\Q)\bsl \sG(\A)} \bar\varphi(g)\,\tilde f^{(\nu)}(g)\,\d g
\label{RSconv-1}
\end{align} 
and
\begin{align}
{Z}\left(\tfrac{1}{2},E(\nu),{{\bar\varphi*\check f}}\right)=
\int_{\sG_0(\Q)\bsl \sG_0(\A)} E(\nu;h)\,[{\bar \varphi* \check f}](\iota(h))\,\d h
 \label{RSconv-2}
\end{align}
are absolutely convergent and are equal to each other. 
\end{lem}
\begin{proof} For simplicity, we write $C_{\sG_0}$ for the constant $(\Delta_{\sG_0}(1)^*)^{-1}$. Set $\s=\Re\nu\in (\ft_{0}^*)^{++}$. 
%Since $(|\varphi|*|\check f|)(g)$ and $(|\varphi|*|f|)({}^tg^{-1})$ are rapidly decreasing on any Siegel domain of $\sG(\A)$, 
The proof shows that the bounds in Lemma \ref{SpectExpPerL-1} and Corollary \ref{SpectExpPerCor-1} actually holds true for the function $|\varphi|*|\check f|$; hence by the same argument as in the proof of Lemma~\ref{SpectExpPerLLL}, 
\begin{align} \int_{\sG_0(\Q)\bsl\sG_0(\A)} E(\s;h)\,(|\varphi|*|\check f|)(\iota(h))\,\d h<+\infty.
 \label{RSconv-3}
\end{align}
Since $\nu$ is on the convergent range of \eqref{EisserieDef}, $|E(\nu;h)|\leq E(\s;h)$. Thus we have the absolute convergence of the integral \eqref{RSconv-2}. By substituting the series expression \eqref{EisserieDef} for $E(\sigma;h)$ and by \eqref{minEis-section}, the integral \eqref{RSconv-3} equals
\begin{align*}
&\int_{\sB_0(\Q)\bsl \sG_0(\A)} e^{\langle H_{\sG_0}(h),\s+\rho_{\sB_0}\rangle}\,(|\varphi|*|\check f|)(\iota(h))\,\d h. 
\end{align*}
This becomes
\begin{align}
&C_{\sG_0}\int_{\sB_0(\Q)\bsl \sB_0(\A)} \biggr(\int_{\sG(\A)} |\varphi(g)|\,|f(\iota(b)^{-1}g)|\,\d g \biggl)\,e^{\langle H_{\sG_0}(b),\s+\rho_{\sB_0}\rangle}\,\d_l b
\label{RSconv-4}
\end{align}
by the $\sG_0(\A)=\sB_0(\A)\bK_{\sG_0}$ and by the left $\iota(\bK_{\sG_0})$-invariance of the function $f$. 
Consider the function $\widetilde{|f|}^{(\sigma)}$ defined by \eqref{tildefnu-def} with $f$ being replaced by $|f|$. We compute  
{\allowdisplaybreaks \begin{align}
&
C_{\sG_0}^{-1}\int_{\sZ(\A)\iota(\sB_0(\Q))\bsl \sG(\A)} |\varphi(g)|\, \widetilde{|f|}^{(\s)}(g)\, \d g \label{RSconv-5}
\\
&=\int_{\sZ(\A)\iota(\sB_0(\Q))\bsl \sG(\A)} |\varphi(g)|\,\biggl(\int_{\A^\times}\int_{\sB_0(\A)}|f|([z]\iota(b)^{-1}g)\,e^{\langle H_{\sG_0}(b),\s+\rho_{\sB_0}\rangle}\,\d_l b\,\d^{*}z \biggr)\, \d {g}
 \notag
\\
&=\int_{\iota(\sB_0(\Q))\bsl \sG(\A)}|\varphi|(g)\biggl( \int_{\sB_0(\A)}
|f(\iota(b)^{-1}g)|\,e^{\langle H_{\sG_0}(b),\s+\rho_{\sB_0}\rangle}
\,\d_l b \biggr)\,\d g. 
 \notag
\end{align}}By writing the $b$-integral as an integral over the quotient $\sB_0(\Q)\bsl \sB_0(\A)$ after a summation over $\sB_0(\Q)$, we extend the computation as
{\allowdisplaybreaks
\begin{align}
&\int_{\iota(\sB_0(\Q))\bsl \sG(\A)}|\varphi|(g)\biggl( \int_{\sB_0(\Q)\bsl \sB_0(\A)} \sum_{\delta \in \sB_0(\Q)} 
|f(\iota(b)^{-1}\iota(\delta)^{-1} g)|\,e^{\langle H_{\sG_0}(b),\s+\rho_{\sB_0}\rangle}
\,\d_l b \biggr)\,\d {g} 
 \notag
\\
&=\int_{\sB_0(\Q)\bsl \sB_0(\A)}\biggl( \int_{\iota(\sB_0(\Q))\bsl \sG(\A)} \sum_{\delta \in \sB_0(\Q)} |\varphi(\iota(\delta)^{-1} g)|\,
|f(\iota(b)^{-1}\iota(\delta)^{-1} g)|\,\d\dot{g} \biggr)
\,e^{\langle H_{\sG_0}(b),\s+\rho_{\sB_0}\rangle}
\,\d_l b
 \notag
\\
&=\int_{\sB_0(\Q)\bsl \sB_0(\A)}\biggl( \int_{\sG(\A)}
|\varphi(g)|\, 
|f(\iota(b)^{-1}g)|\,\d\dot{g} \biggr)
\,e^{\langle H_{\sG_0}(b),\s+\rho_{\sB_0}\rangle}
\,\d_l b.
 \notag
\end{align}}Hence we have the convergence of \eqref{RSconv-5} from that of \eqref{RSconv-4}. By $|\bar\varphi(g) \tilde f^{(\nu)}(g)|\leq |\varphi|(g)\widetilde{|f|}^{(\s)}(g)$, we obtain the absolute convergence of \eqref{RSconv-1}. By the same type of computation as above which shows the equality between \eqref{RSconv-5} multiplied by $C_{\sG_0}$ and \eqref{RSconv-3}, we get the equality between \eqref{RSconv-1} and \eqref{RSconv-2}. 
\end{proof}

Recall the notation $\langle\,|\,\rangle_{\sG}$ for the inner product of $L^2(Z_\infty\sG(\Q)\bsl \sG(\A))$ (\S\ref{sec:convolutionP}), and the global Rankin-Selberg integral \eqref{G0G-zetaInt} for the $\GL_{n-1}\times \GL_n$-convolution $L$-function. The following proposition plays a pivotal role in this article. 
\begin{prop} \label{InnerprodPer}
Let $f=\otimes_{v} f_v \in C_{\rm{c}}^\infty(\sG(\A))$ be a decomposable function which is left $\iota(\bK_{\sG_0})$-invariant, and define 
$$\tilde f^{(\nu)}(g)=(\Delta_{\sG_0}(1)^{*})^{-1}\int_{\sB_0(\A)} \int_{\A^{\times}}f([z]\iota(b)^{-1}\,g)e^{\langle H_{\sG_0}(b),\nu+\rho_{\sB_0}\rangle}\,\d_lb\,\d^*z \quad \text{for $g\in \sG(\A)$}. $$ Let $\b\in \fB_0$ (see \S\ref{AdeleTestFtn}). Then the absolutely convergent Poincar\'{e} series 
$$
\tilde{\bf \Phi}_{f,\beta}(g)=\sum_{\gamma \in \sZ(\Q)\iota(\sB_0(\Q))\bsl \sG(\Q)} \tilde \Phi_{f,\beta}(\gamma g) \quad \text{with $\tilde \Phi_{f,\beta}(g)=\int_{\fJ(\sigma)}\beta(\nu)\,M_{\sG_0}(\nu)\,\tilde f^{(\nu)}(g)\,\d \nu$
}
$$ for $g\in \sG(\A)$, and the $\bK_{\sG_0}$-spherical smoothed Eisenstein series on $\sG_0(\A)$ 
$$
\Ecal_{\beta}(h)=\int_{\fJ(\sigma)}\beta(\nu)\,\hat E(\nu;h)\,\d \nu, \quad (\sigma\in (\ft_{0}^*)^{++})
$$
satisfy the identity 
\begin{align*}
\langle \tilde {\bf \Phi}_{f,\b}|\varphi\rangle_{\sG}=Z\left(\tfrac{1}{2},\Ecal_\b,{\bar\varphi*\check f}\right) \quad \text{
for any cusp form $\varphi$ on $\sG(\A)^1$,}
\end{align*}
where the right-hand side is defined by the integral \eqref{G0G-zetaInt} with $\check f(g)=f(g^{-1})$. 
\end{prop}
\begin{proof} Since $\vol(Z_\infty\sZ(\Q)\bsl \sZ(\A))=1$ and $\tilde{\bf\Phi}_{f,\b}$ is $\sZ(\A)$-invariant, by \eqref{Pser-f1},   
{\allowdisplaybreaks\begin{align*}
\langle \tilde{\bf\Phi}_{f,\b}|\varphi\rangle_{\sG}&=\int_{\sZ(\A) \iota(\sB_0(\Q))\bsl \sG(\A)}\tilde\Phi_{f,\b}(g)\,\bar\varphi(g)\,\d g\\
&=\int_{\sZ(\A)\iota(\sB_0(\Q))\bsl \sG(\A)}\bar \varphi(g)\left(\int_{\fJ(\s)}{\b(\nu)}\,{ M}_{\sG_0}(\nu)\,\tilde f^{(\nu)}(g)
\,\d\nu\right)\,\d g.
\end{align*}}If we choose $\s$ in $(\ft_{0}^*)^{++}$, then by Lemma~\ref{RSconv}, we exchange the order of integrals to see that this equals
\begin{align*}
\int_{\fJ(\s)}{\b(\nu)}\,{M}_{\sG_0}(\nu)\,\left(\int_{\sG_0(\Q)\bsl \sG_0(\A)} E(\nu;h)\,[{{\bar \varphi*\check f}}](\iota(h))\,\d h\right)\,\d\nu.
\end{align*} 
By exchanging the order of integrals again, we have the desired equality. 
\end{proof}

\subsection{Spectral expansion} \label{SPEXP}
From now on, we keep holding the following conditions on our function $f=\otimes_{v}f_v\in C_{\rm{c}}^{\infty}(\sG(\A))$: 
\begin{itemize}
\item[(i)] $f$ is left $\iota(\bK_{\sG_0})$-invariant.
\item[(ii)] $f_\infty\in C_{\rm c}^{\infty}(\bK_\infty\bsl \sG(\R)/\bK_\infty)$. \item[(iii)] there exists $M \in \N$ and a family of irreducible smooth supercuspidal representations $\tau_p\,(p\in S(M))$ of $\sG(\Q_p)$ with trivial central characters such that $\tilde f_p\in C_{\rm c}^{\infty}(\sZ(\Q_p)\bK_1(M\Z_p)\bsl \sG(\Q_p)/\bK_1(M\Z_p))$ is a matrix coefficent of $\tau_p$. 
%(g_p)=\langle \tau_p(g_p)\xi_p|\xi_p'\rangle$ (with $\tilde f_p$ defined by \eqref{CentralProjLoc}) with $\langle\,|\,\rangle$ being a $\sG(\Q_p)$-invariant Hermitian inner product on $V(\tau_p)$ and $\xi_p,\,\xi_p'\in V(\tau_p)-\{0\}$ being invariant by $\bK_{1}(M\Z_p)$ for all $p \in S(M)$.  
\end{itemize}
%Note that if $M_\tau$ denotes the product of conductors (see \S \ref{sec:Lfunction}) of $\tau_p$ for $p\in S(M)$ then $M_\tau|M$ and $S(M_\tau)=S(M)$. 
The property (iii) implies the cuspidality of our Poincar\'e series $\tilde {\bf \Phi}_{f,\b}$ (Proposition~\ref{CuspidalPer}), which considerably simplifies its spectral resolution. For its description, we need additional notation. Let $R$ denote a finite set of places and $\pi\cong \otimes_{v}\pi_v\in \Pi_{\rm cusp}(\sG)_\sZ$. Set $\A^{R}=\prod_{v\not\in R}'\Q_v$, so that $\A=\A^{R}\times \prod_{v\in R}\Q_v$. For any $\Q$-group $\sH$, let $\sH(\A^{R})$ be the restricted direct product of $\sH(\Q_v)\,(v\not\in R)$ and set $\sH(\Q_R)=\prod_{v\in R}\sH(\Q_v)$. We set $f^{R}=\otimes_{v\not\in R}f_v$ viewed as a function on $\sG(\A^{R})$ and $f_{R}=\otimes_{v\in R}f_v$ viewed as a function on $\sG(\Q_R)$. Set $\pi^{R}=\otimes_{v\not\in R}\pi_{v}$ and $\pi_{R}=\otimes_{v\in R}\pi_{v}$. Then the numbers $\widehat {f^{R}}(\pi^{R})$ and $\widehat {f_{R}}(\pi_{R})$ are defined to be the trace of the operator $\pi^{R}(f^{R})$ on the unitary representation $\pi^{R}$ of $\sG(\A^R)$ and the trace of $\pi_{R}(f_R)$ on the unitary representation $\pi_{R}$ of $\sG(\Q_R)$, respectively. Recall $\dim_{\C}(\pi_p^{\bK_p})\leq 1$ for all $p<\infty$; thus, when $f_p$ is bi-$\bK_p$-invariant, then $\widehat {f_p}(\pi_p)=0$ unless $\pi_p^{\bK_p}\not=\{0\}$ in which case it coincides with the scalar by which $\pi_p(f_p)$ acts on $\pi_p^{\bK_p}\cong \C$. For almost all $p<\infty$, from the normalization of the measure on $\sG(\Q_p)$, $\widehat {f_p}(\pi_p)=\widehat {\cchi_{\bK_p}}(\pi_p)$ equals $1$ or $0$ according to $\pi_p$ has non-zero $\bK_p$-fixed vectors or not. We have the product formulas $\widehat {f^{R}}(\pi^{R})=\prod_{v\not\in R}\widehat {f_v}(\pi_{v})$ and $\widehat {f_{R}}(\pi_{R})=\prod_{v\in R}\widehat{f_v}(\pi_v)$. Let $S$ be a finite set of places $p<\infty$ such that $p\not\in S$ implies $f_p$ is bi-$\bK_p$-invariant. For $p<\infty$, let $K_p$ be an open compact subgroup of $\bK_p$ such that $f_p$ is bi-$K_p$-invariant and is maximal among all those such subgroups. Thus $K_p=\bK_p$ for $p \not\in S$. Set $K_\fin=\prod_{p<\infty}K_p$ and $K=\bK_\infty K_\fin$. 
%Recall that $\Bcal(\pi^{K})$ is an orthonormal basis of the space of $K$-fixed vectors $V_\pi^K$ in the space of $\pi \in \Pi_{\rm cusp}(\sG)^{K}$. 

\begin{prop} \label{SpectExpPer}
Let $\b\in \fB_0$. We have the point-wise equality
\begin{align}
\tilde{\bf \Phi}_{f,\b}(g)=(\Delta_{\sG}(1)^{*})^{-1}\sum_{\pi \in \Pi_{\rm{cusp}}(\sG)_\sZ^K}\widehat{f^S}(\bar \pi^{S})\,\sum_{\varphi \in \Bcal(\pi^K)} Z\left(\tfrac{1}{2},\Ecal_\b, \bar \varphi* \check f_{S}\right)\,\varphi(g), \quad g\in \sG(\A),
\label{SpectExpPer-0}
\end{align}
where the series on the right-hand side is absolutely convergent normally on $\sG(\A)$. 
\end{prop}
\begin{proof}
 From Propositions~\ref{SQint} and \ref{CuspidalPer}, $\tilde{\bf\Phi}_{f,\b}\in  L_{\rm{cusp}}^2(Z_\infty\sG(\Q)\bsl \sG(\A))$ is expanded as
\begin{align}
\tilde{\bf\Phi}_{f,\b}(g)=\sum_{\pi\in \Pi_{\rm{cusp}}(\sG)_\sZ}\sum_{\varphi\in \cB(\pi)}\langle \tilde{\bf\Phi}_{f,\b}|\varphi\rangle_{\sG}\,\varphi(g)
 \label{SpectExpPer-1}
\end{align}
in the $L^2$-space, where $\cB(\pi)$ is an orthonormal basis of $V_\pi$. 
Since $\tilde{\bf \Phi}_{f,\beta}$ is right $K$-invariant, $\langle \tilde{\bf\Phi}_{f,\b}|\varphi\rangle_{\sG}=\langle \tilde{\bf\Phi}_{f,\b}|\varphi^{K}\rangle_{\sG}$ with $\varphi^{K}=\int_{K}R(k)\varphi\,\d k$. Thus the summation range of \eqref{SpectExpPer-1} may be reduced to $\pi\in \Pi_{\rm{cusp}}(\sG)_\sZ^{K}$, $\varphi\in \cB(\pi^{K})$. For $\varphi\in \cB(\pi^{K})$, $\bar \varphi*\check f=\{(\Delta_\sG(1)^*)^{-1}\widehat{f^S}(\bar \pi^S)\} \times (\bar \varphi*\check f_{S})$. Thus the identity \eqref{SpectExpPer-0} in the $L^2$-space is obtained from \eqref{SpectExpPer-1} and Proposition \ref{InnerprodPer}. Note the relation $\widehat {f^S}(\bar \pi^{S})=\widehat {f^{S\cup\{\infty\}}}(\bar \pi^{S\cup\{\infty\}})\,\widehat f_{\infty}(\bar \pi_{\infty})$. Thus by combining Lemmas~\ref{SpectExpPerL-1} and ~\ref{SpectExpPerLLL} with Lemmas~\ref{SpectExpPerL-4-1}, \ref{SpectExpPerL-4} and \ref{SpectExpPerL-5} shown below, we see that the right-hand side of \eqref{SpectExpPer-0} is normally convergent on $\sG(\A)$, which combined with Proposition~\ref{SQint} implies the continuity of the both sides of \eqref{SpectExpPer-0}. Hence the equality \eqref{SpectExpPer-0} holds true point-wisely. \end{proof}

\begin{lem} \label{SpectExpPerL-4}
Let $K_\fin=\prod_{p<\infty}K_p$ be an open compact subgroup of $\bK$ and $f=\otimes_{v}f_v\in C_{\rm c}^{\infty}(K \bsl \sG(\A)/K)$ with $K=K_\fin \bK_\infty$. Let $S=\{p<\infty\mid K_p\not=\bK_p\}$. Then $\widehat{f^{S\cup \{\infty\}}}(\bar \pi^{S\cup \{\infty\}} )$ is bounded as $\pi \in \Pi_{\rm{cusp}}(\sG)_\sZ^{K}$.\end{lem}
\begin{proof}
Let $S_1=\{p<\infty\mid f_p\not=\cchi_{\bK_p}\,\}$. Then $S\subset S_1$ and $\widehat{f^{S\cup\{\infty\}}}(\pi^{S\cup \{\infty\}})={\widehat{f_{S_1-S}}}(\pi_{S_1-S})$ for all $\pi\in \Pi_{\rm{cusp}}(\sG)_\sZ^{K}$ because $\widehat {f_p}(\pi_p)=1$ for all $p \not\in S_1$. Since $\pi_{S_1-S}$ with $\pi\in \Pi_{\rm{cusp}}(\sG)_\sZ^{K}$ is unitarizable, by \eqref{ReBound}, the spectral parameter of $\pi_{p}$ for $p\in S_1-S$ lies in a compact subset of $(\C^{n}/2\pi \ii (\log p)^{-1}\Z^n)/\sS_n$ independent of $\pi$. Thus $\widehat{f_{S_1-S}}(\pi_{S_1-S})$ is bounded when $\pi$ varies over $\Pi_{\rm{cusp}}(\sG)_\sZ^{K}$.
\end{proof}

\begin{lem} \label{SpectExpPerL-5} 
There exists $d\in \N$ such that for any open compact subgroup $K'_\fin \subset \sG(\A_\fin)$, 
\begin{align*}
\sum_{\pi \in \Pi_{\rm{cusp}}(\sG)^{K'_\fin\bK_{\infty}}_\sZ} \#(\Bcal(\pi^{K'_\fin\bK_\infty}))\,\fQ_\infty(\pi)^{-d}<\infty. 
\end{align*}
\end{lem}
\begin{proof}
It suffices to settle the case when $K'_\fin$ is a principal congruence subgroup with a sufficiently large level. By \eqref{GaugeInfChar} and Weyl's law for the cuspidal spectrum of $\GL_n$ (\cite{MullerLapid}, \cite{LindenstraussVenkatesh}), we can take $d=\dim ({\bf SL}_n(\R)/{\bf SL}_n(\R)\cap \bK_\infty)+1=\frac{n(n+1)}{2}$.   
\end{proof}

\section{Whittaker-Fourier coefficient of $\tilde{\bf\Phi}_{f,\b}$: the spectral side}\label{WFCSP}
For a continuous function $\varphi$ on $\sG(\Q)\bsl \sG(\A)$, its global Whittaker-Fourier coefficient is defined by 
\begin{align}
\Wcal^{\psi}(\varphi)=\int_{\sU(\Q)\bsl \sU(\A)}\varphi(u)\psi_{\sU}(u)^{-1}\,\d u,
 \label{WFCoeffDef}
\end{align}
where $\psi_\sU$ is a character of $\sU(\A)$ defined as $\psi_\sU(u)=\psi(\sum_{j=1}^{n-1}u_{j\, j+1})$ $(u=(u_{ij})_{ij} \in \sU(\A))$ with $\psi:\Q\bsl \A \rightarrow \C^\times$ as in \S\ref{GWP}. Note that $\psi_{\sU}$ is trivial on $\sU(\Q)$. In this section, we fix $f\in C_{\rm{c}}^\infty(\sG(\A))$ satisfying the conditions (i), (ii), and (iii) in \S\ref{SPEXP} to consider the Poincar\'{e} series $\tilde{\bf \Phi}_{f,\b}$ defined as \eqref{Pser-f1}, and then compute the integral $\Wcal^{\psi}(\tilde{\bf\Phi}_{f,\b})$ by Proposition~\ref{SpectExpPerL-3}. Before that, we recall basic facts on Rankin-Selberg $L$-functions to define our main local quantity \eqref{LocalWHittPer}, and then prove an average formula \eqref{AverageWHittPerL1-1} of period integrals.

\subsection{Jacquet integrals} \label{Jacquetint}
For $\nu\in \fJ((\ft_{0}^*)^{+})$ and $\vf\in I_{v}^{\sG_0}(\nu)$ $(\nu \in \ft_{0,\C}^{*}$), the absolutely convergent integral 
\begin{align}
J^{\psi_v}_{\sG_0(\Q_v)}(\vf;h)=\int_{\sU_0(\Q_v)}\vf(w_\ell^0 uh)\,\psi_{\sU_0,v}(u)^{-1}\,\d u, \quad h\in \sG_0(\Q_v)
 \notag
\end{align}
is called the Jacquet integral of $\vf$, where $\psi_{\sU_0,v}$ is a unitary character of $\sU_0(\Q_v)$ defined as 
\begin{align}
\psi_{\sU_0,v}(u)=\psi_{v}\Bigl(\sum_{j=1}^{n-2}u_{j\, j+1}\Bigr), \quad  u=(u_{ij})_{ij} \in \sU_0(\Q_v)
\label{psiUv}
\end{align}
with $\psi_v$ being the additive character of $\Q_v$ fixed in \S\ref{GWP}. 
%The mapping $f\mapsto J^{\psi_v}_{\sG_0(\Q_v)}(f;\bullet)$ yields a $\sG_0(\Q_v)$-intertwining operator from $I^{\sG_0}_v(\nu)$ to the space of $\psi_v$-Whittaker functions on $\sG_0(\Q_v)$ (see \S~\ref{LWFunc}). 
Given a holomorphic $\bK_v$-finite section $\vf(\nu)$ of $I^{\sG_0}_v(\nu)$, the function $\nu\mapsto J^{\psi_v}_{\sG_0(\Q_v)}(\vf{(\nu)};h)$ with any fixed $h\in \sG_0(\Q_v)$ has a holomorphic continuation to $\ft_{0,\C}^*$ (\cite[Corollary 3.5]{Jacquet}, \cite[Proposition 2.1]{CasselmanShalika}, \cite[Theorem 15.4.1]{Wallach}). Let $\sf_{\sG_0,v}^{(\nu)}\in I^{\sG_0}_v(\nu)$ be the unique $\bK_{\sG_0,v}$-invariant vector such that $\sf_{\sG_0,v}^{(\nu)}(1_{n-1})=1$. We  abbreviate $J_{\sG_0(\Q_v)}^{\psi_v}(\sf_{\sG_0,v}^{(\nu)};h)$ to $J_{\sG_0(\Q_v)}^{\psi_v}(\nu;h)$ and set
\begin{align}
\text{${M}_{\sG_0,v}(\nu)=\prod_{1\leq i<j\leq n-1}\zeta_v(1+\nu_i-\nu_j)$,}
 \label{f-MsG0v}
\end{align}
and   
\begin{align}
W_{\sG_0(\Q_v)}^{0}(\nu;h)={M}_{\sG_0,v}(\nu)\,J_{\sG_0(\Q_v)}^{\psi_v}(\nu;h), \quad h\in \sG_0(\Q_v).
 \label{normJInt}
\end{align}
Initially this makes sense only for regular points $\nu\in \ft_{0,\C}^{*}$ of the function $M_{\sG_0,v}(\nu)$, but eventually for all $\nu\in \ft_{0,\C}^{*}$ with the aid of explicit formulas (see Lemma~\ref{HolWhitt} below). The functional equation of Whittaker function takes the form (\cite{Jacquet}, \cite[Corollary 5.3]{CasselmanShalika}):
\begin{align}
W^{0}_{\sG_0(\Q_v)}(w\nu;h)=W^{0}_{\sG_0(\Q_v)}(\nu;h), \quad w\in \sS_{n-1},\,
h\in \sG_0(\Q_v).
 \label{FEqWhitt}
\end{align}
The constructions so far explained for the group $\sG_0(\Q_v)$ is applied to the group $\sG(\Q_v)$, yielding the objects $J_{\sG(\Q_v)}^{\psi_v}(s;g)$, $M_{\sG,v}(s)$ and $W^{0}_{\sG(\Q_v)}(s;g)$ for $s \in \ft_{\C}^{*}$ and $g\in \sG(\Q_v)$. 
Next, we recall the bounds of the Whittaker functions on $\sG_0(\Q_v)$ and $\sG(\Q_v)$ which are stated explicitly only for $\sG_0(\Q_v)$ but equally valid for $\sG(\Q_v)$. 

For $\nu=(\nu_j)_{j=1}^{n-1} \in \ft_{0,\C}^*$, set $\tilde \nu:=(\nu_j-c)_{j=1}^{n-1}$ with $c=\frac{1}{n-1}\sum_{j=1}^{n-1}\nu_j$; then $\sf_{\sG_0,v}^{(\nu)}(h)=\sf_{\sG_0,v}^{(\tilde \nu)}(h)|\det\,h|_v^{c}$, which in turn yields 
\begin{align}
J_{\sG_0(\Q_v)}^{\psi_v}(\nu;h)=J_{\sG_0(\Q_v)}^{\psi_v}(\tilde \nu ;h)|\det\,h|_v^{c}, \quad h\in \sG_0(\Q).
\label{W-centralTwist}
\end{align}

\begin{lem}\label{HolWhitt}
For any $h\in \sG_0(\Q_v)$, the function $\nu\mapsto W_{\sG_0(\Q_v)}^{0}(\nu;h)$ is continued to a holomorphic function on $\ft_{0,\C}^{*}$ vertically of moderate growth. When $v=p<\infty$, we have $W_{\sG_0(\Q_p)}^{0}(\nu;1_{n-1})=1$ for all $\nu\in \ft_{0,\C}^{*}$. 
\end{lem}
\begin{proof}
When $v=p<\infty$, the claim follows from the explicit formula of $W^{0}_{\sG_0(\Q_p)}(\nu)|_{\sT_0(\Q_p)}$ obtained from \cite[Theorem 5.4]{CasselmanShalika}, which is originally due to \cite{Shintani1976}. When $v=\infty$, the claim follows from \cite[Corollary (p.358)]{Stade}, which combined with \eqref{W-centralTwist} yields the inequality $|W_{\sG_0(\R)}^{0}(\nu;h)| \leq |W_{\sG_0(\R)}^0(\Re\,\nu;h)|$ for all $h\in \sG_0(\R)$ and $\nu\in \ft_{0,\C}^{*}$. This actually shows that $\nu \mapsto W_{\sG_0(\R)}^0(\nu;g)$ is vertically bounded.  
\end{proof}

\subsubsection{}\label{sec:URWFTN}
 For $\pi \in \widehat {\sG(\R)}_{\rm gen}^{\rm ur}$, we define $W_{\pi}^{0}$ to be $W_{\sG(\Q_v)}^{0}(\nu(\pi))$ with $\nu(\pi)\in \fX_{v}^{0+}$ being the spectral parameter of $\pi$ (see \S~\ref{sec:URps}); by \eqref{FEqWhitt}, this is well-defined. When $v=p<\infty$, then $W_{\pi}^{0}$ is the unique $\bK_v$-invariant element of $\WW^{\psi_p}(\pi)$ such that $W_{\pi}^0(1_n)=1$. 

\subsection{Mirabolic subgroup} \label{sec:MBS}
Let $\Pcal$ be the mirabolic subgroup of $\sG$, which consists of all the invertible $n\times n$-matrices whose last row is $(0,\dots,0,1)$. We have an obvious $\Q$-isomorphism given as $(g_0,x)\mapsto \iota(g_0)\sn(x)$ from $\sG_0\times \Q^{n-1}$ onto $\Pcal$. We endow a left Haar measure on $\Pcal(\Q_v)$ as $\d_{l}q=\d g_0\,\prod_{j=1}^{n-1}\d x_j$ for $q=\iota(g_0)\,\sn(x)$ with $g_0\in \sG(\Q_v)$ and $x=(x_j)_{j=1}^{n-1}\in \Q_v^{n-1}$; note that $\vol(\Pcal(\Z_p))=1$ if $v=p<\infty$. The modulus character of $\Pcal(\Q_v)$ is $|\det|_v$, i.e., $\d_r q=|\det q|_v\d _l q$ is a right Haar measure on $\Pcal(\Q_v)$.
\subsection{Local Whittaker functions} \label{LWFunc}
Given an additive non-trivial character $\theta_v:\Q_v\rightarrow \C^\times$, a smooth function $W:\sG(\Q_v)\rightarrow\C$ is called a $\theta_v$-Whittaker function if it satisfies 
$$
W(ug)=\theta_{v}(u_{12}+\cdots+u_{n-1\,n})\,W(g), \quad u=(u_{ij})_{ij}\in \sU(\Q_v), \,g\in \sG(\Q_v)
$$
and is of moderate growth on $\sU(\Q_v)\bsl \sG(\Q_v)$ if $v=\infty$. For $\pi_v \in \widehat{\sG(\Q_v)}_{\rm gen}$, let $\WW^{\theta_v}(\pi_v)$ denote its $\theta_v$-Whittaker model, i.e., the unique realization of $\pi_v$ in the space of $\theta_v$-Whittaker functions on $\sG(\Q_v)$ (\cite{Shalika}). The integral 
\begin{align}
\langle W|W' \rangle_{\Pcal(\Q_v)}:=\zeta_v(n)\,\int_{\sU(\Q_v)\bsl \Pcal(\Q_v)}W(q_v)\bar W'(q_v)
\,\d_r q_v, \quad W,\,W' \in \WW^{\theta_v}(\pi_v) \label{WhitProd}
\end{align}
converges absolutely and defines a unitary structure on $\WW^{\theta_v}(\pi_v)$ (\cite[\S6.2]{Bernstein}, \cite[\S10.2]{Baruch}). From now on, we take $\theta_v$ to be the character $\psi_v:\Q_v\rightarrow \C^\times$ fixed in \S\ref{GWP} or the inverse $\psi_v^{-1}$. Note that $\psi_{p}|\Z_p=1$ and $\psi_p|p^{-1}\Z_p\not=1$ if $v=p<\infty$. Let us recall basic facts from \cite[\S2]{JPSS1983}, \cite{JS1990} and \cite[\S3]{Cogdell} on the local zeta-integral for the Rankin-Selberg $L$-function. The representation $I_v^{\sG_0}(\nu)$ $(\nu \in \ft_{0,\C}^*)$ is known to be of Whittaker type (\cite[\S(2.1)]{JPSS1983}, \cite[Corollary 1.8]{CasselmanShalika}, \cite[\S2]{JS1990}), i.e, the space of (continuous) linear functionals $\lambda$ on the space $I_{v}^{\sG_0}(\nu)$ such that $\lambda(R(u)\vf)=\psi_{\sU_0,v}(u)\,\lambda(\vf)$ for all $u \in \sU_0(\Q_v)$ and $\vf\in I_v^{\sG_0}(\Q_v)$ is one dimensional; indeed, such $\lambda$ is a constant multiple of the analytic continuation of the Jacquet integral $\vf \mapsto J_{\sG_0(\Q_v)}^{\psi_v}(\vf;1_{n-1})$ (see \S\ref{Jacquetint}). Let $\WW^{\psi_v}(I_{v}^{\sG_0}(\nu))$ denote the image of the $\sG_0(\Q_v)$-intertwining map $J_{\sG_0(\Q_v)}^{\psi_v}: I_v^{\sG_0}(\nu) \rightarrow \WW^{\psi_v}(I_v^{\sG_0}(\nu))$. When $\nu \in \ii \ft_{0}^*$, the representation $I_{v}^{\sG_0}(\nu)$ is irreducible and the map $J_{\sG_0(\Q_v)}^{\psi_v}$ is bijective; this latter fact is true even for $\nu\in \ft_{0,\C}^{*}$ satisfying $\Re\, \nu_{j}\geq \Re\,\nu_{j+1}$ $(1\leq j \leq n-1)$ (\cite[Proposition (3.2)]{JS1983}, \cite[Proposition 2.4]{JS1990}). Recall the function $W_{\sG_0(\Q_v)}^{0}(\nu)$ ($\nu\in \ft_{0,\C}^*$) on $\sG_0(\Q_v)$ defined in \S\ref{Jacquetint}. Let $w\in \sS_{n-1}$ be an element such that $w\nu=(\nu_{w(j)})_{j=1}^{n-1}$ satisfies $\Re(\nu_{w(j)})\geq \Re(\nu_{w(j+1)})$ $(j\in [1,n-2]_{\Z})$, then $W_{\sG_0(\Q_v)}^0(\nu) \in \WW^{\psi_v}(I_v^{\sG_0}(w\nu))$; this follows easily from the injectivity of $J_{\sG_0(\Q_v)}^{\psi_v}$ on $I_v^{\sG_0}(w\nu)$ recalled above by the argument in the proof of \cite[Lemma 1]{Jacquet3} and by \eqref{FEqWhitt}. For $W_0\in\WW^{\psi_v}(I_v^{\sG_0}(\nu))$ and $W\in \WW^{\psi_v^{-1}}(\pi_v)$, set 
\begin{align}
Z(z,W_0\otimes W):=\int_{\sU_0(\Q_v)\bsl \sG_0(\Q_v)} W_0(h)W\left(\iota(h)\right)\,|\det h|_{v}^{z-1/2}\,\d h, \quad z\in \C,
\label{LocalRSint}
\end{align}
which is known to be absolutely convergent for $\Re z\gg 0$. The local Rankin-Selberg $L$-function $L(s,I_{v}^{\sG_0}(\nu)\times \pi_v)$ is defined in such a way that the ratio $Z(z,W_{0}\otimes W)/L(z,I_v^{\sG_0}(\nu) \times \pi_v)$ has a holomorphic continuation to the whole $z$-plane satisfying the local functional equations with the contragredient objects (\cite[Theorem (2.7)]{JPSS1983}, \cite[Proposition 5.1]{JS1990}). The first part of the following lemma is implicit in the work \cite{JPSS1981}. 
\begin{lem}\label{lem:20201116}
Let $v=p<\infty$. 
\begin{itemize}
\item[(i)] If $\nu=(\nu_j)_{j=1}^{n-1}\in \ft_{0,\C}^*$ satisfies $\Re\,\nu_j\geq \Re\,\nu_{j+1}$ $(j\in [1,n-1]_{\Z})$, then 
$$
L(s,I_{v}^{\sG_0}(\nu)\times \pi_v)=\prod_{j=1}^{n-1}L(s+\nu_j,\pi_v),
$$
where $L(s,\pi_v)$ on the right-hand side is the standard $L$-function (see \S\ref{sec:Lfunction}).
\item[(ii)] There exists a unique $\C$-linear map $W\mapsto \Xi(W;X_1,\dots,X_{n-1})$ from $\WW^{\psi_p^{-1}}(\pi_v)$ to $\C[X_1,X_1^{-1},\dots,X_{n-1},X_{n-1}^{-1}]^{\sS_{n-1}}$ such that for any $\nu=(\nu_j)_{j=1}^{n-1}\in \ft_{0,\C}^*$, 
$$Z(z,W^{0}_{\sG_0(\Q_v)}(\nu) \otimes W)/\prod_{j=1}^{n-1}L(z+\nu_j,\pi_v)=\Xi(W;p^{-z-\nu_1}, \dots,p^{-z-\nu_{n-1}}) \quad \Re\,z \gg 0.$$
\end{itemize}
\end{lem}
\begin{proof} 
%We use the notation and terminology from \cite{JPSS1983} freely. Since $\pi_v \in \widehat {\sG(\Q_v)}_{\rm gen}$, it is isomorphic to an induced representation of Langlands type ${\rm Ind}(\sG(\Q_v),{\bf Q}(\Q_v);\eta_1,\dots,\eta_r)$ with a parabolic subgroup ${\bf Q}\subset \sG$ of type $(m_1,\dots,m_r)$ and $\eta_j$ an irreducible essentially tempered representation of $\GL_{m_j}(\Q_v)$. Let ${\bf R}$ be the standard parabolic subgroup of $\sG_0$ of type $(a_1,\dots,a_t)$ determined by the condition that $i\mapsto \Re\,\nu_i$ is constant on $I_j:=[1+\sum_{j=1}^{i-1}a_j,\sum_{j=1}^{i}a_j]_{\Z}$ but $\Re\,\nu_{i}<\Re\,\nu_{i'}$ for $i\in I_{j},\,i'\in I_{j'}$, $j>j'$. The induced representation $\sigma_j:={\rm Ind}(\GL_{a_j}(\Q_v),{B}_{a_j}(\Q_v);|\,|_v^{\nu_i}\,(i\in I_j))$ from the Borel subgroup ${B}_{a_j} \subset \GL_{a_j}$ is irreducible and essentially tempered. By induction by stages, $I_{v}^{\sG_0}(\nu)$ is isomorphic to ${\rm Ind}(\sG_0(\Q_v),{\bf R}(\Q_v);\sigma_1,\dots,\sigma_t)$, which is of Langlands type. By \cite[Theorems (9.4) and (8.4)]{JPSS1983}, we see that $L(s,I^{\sG_0}_{v}(\nu)\times \pi_v)$ equals $\prod_{j=1}^{t}\prod_{h=1}^{r} L(s, \sigma_i \times \eta_h)=\prod_{j=1}^{t}\prod_{i\in I_j}\prod_{h=1}^{r}L(s,|\,|_v^{\nu_i} \times \eta_h)$; by the same reasoning, each subproduct $\prod_{h=1}^{r}L(s,|\,|_v^{\nu_i}\times \eta_h)$ equals $L(s,|\,|_v^{\nu_i}\times \pi_v)$. 
(i) is shown by \cite[Theorems (9.4) and (8.4)]{JPSS1983}. It should be noted that, by \cite[\S(5.1)]{JPSS1983} and \cite[Theorem (4.3)]{JPSS1981-1}, the $\GL_1\times \GL_{n}$ Rankin-Selberg $L$-function $L(s,|\,|_v^{\nu_i}\times \pi_v)=L(s,1\times (\pi_v\otimes|\,|_v^{\nu_i}))$ coincides with the standard $L$-function $L(s+\nu_i,\pi_v)$. If $\Re\,\nu_i\geq \Re\,\nu_{i+1}\,(i\in [1,n-1]_{\Z})$, then $W_{\sG_0(\Q_v)}^{0}(\nu)\in \WW^{\psi_v}(I_{v}^{\sG_0}(\nu))$ as noticed above. Hence (i) implies that the function $z\mapsto Z(z,W_{\sG_0(\Q_v)}^0(\nu)\otimes W)/\prod_{i=1}^{n-1}L(z+\nu_i,\pi_v)$ is entire , which is true for any $\nu$ by \eqref{FEqWhitt}. For (ii), we refer to \cite[(4.2)]{JPSS1981}. 
\end{proof}
Suppose $v=p<\infty$ for a while. Since $\pi_v \in \widehat{\sG(\Q_v)}_{\rm gen}$, the $L$-function $\prod_{j=1}^{n-1}L(z+\nu_j,\bar \pi_v)$ is holomorphic on $\Re z\geq \tfrac{1}{2}$ if $\Re \nu_j\geq 0\,(j\in[1,n-1]_{\Z})$ from the proof of \cite[Proposition 2.1]{BarthelRamakrishnan}. Hence $z \mapsto Z(z,W_{\sG_0(\Q_v)}^{0}(\nu) \otimes \bar W)$ with $W\in \WW^{\psi_v}(\pi_v)$ is holomorphic on $\Re z\geq \tfrac{1}{2}$ when $\Re\nu_j\geq 0\,(j \in [1,n-1]_{\Z})$. We consider the function 
\begin{align}
Z^{0}(\nu;\bar W):={Z\left(\tfrac{1}{2},W_{\sG_0(\Q_v)}^{0}(\nu) \otimes \bar W\right)}/\prod_{j=1}^{n-1}{L\left(\tfrac{1}{2}+\nu_j,\bar \pi_v\right)}
 \label{20201116}
\end{align}
on $\Re \nu_j\geq 0\,(j \in [1,n-1]_{\Z})$. Given an open compact subgroup $K_v\subset \sG(\Q_v)$ and a function $f\in C_{\rm{c}}^\infty(K_v\bsl \sG(\Q_v)/K_v)$, define
\begin{align}
\PP_{\psi_v}^{(\nu)}(\pi_v,K_v;f)=
L(1,\pi_v\times \bar \pi_v)\,\sum_{W\in \Bcal(\pi_v,K_v)}Z^{0}(\nu; \bar W*\check f)\, W(1_n)    
\label{LocalWHittPer}
\end{align}
for $\nu\in \ft_{0,\C}^{*}$ with $\Re\,\nu_j\geq 0\,(j\in [1, n-1]_{\Z})$, where $\Bcal(\pi_v,K_v)$ is an orthonormal basis of the finite dimensional space $\WW^{\psi_v}(\pi_v)^{K_v}$ with respect to the inner product \eqref{WhitProd}. It is easily seen that the sum is independent of the choice of such a basis. From Lemma~\ref{lem:20201116} (ii), the quantity \eqref{LocalWHittPer} as a function in $\nu$ has a holomorphic continuation to $\ft_{0,\C}^*$.

At the archimedean place $v=\infty$, we restrictively consider the representations $\pi_{v} \in \widehat {\sG(\R)}_{\rm gen}^{\rm ur}$. Then by \eqref{URZetaInt} in Lemma ~\ref{Ishii-Stade}, the quantity \eqref{20201116} for $W_v=W_{\pi_v}^{0}$ (see \S\ref{sec:URWFTN}) is explicitly computed to be $1$ so that the definition \eqref{LocalWHittPer} makes sense in this particular case at least for $K_\infty=\bK_\infty={\bf O}_n(\R)$. 

\smallskip
Let $v$ be arbitrary again. For a Schwartz-Bruhat function $\Phi_v$ on $(\Q_v)_{n}:=\Mat_{1,n}(\Q_v)$ and a pair of Whittaker functions $W_v, W_v' \in \WW^{\psi_v}(\pi_v)$, the local zeta-integral for $L(z,\pi_v\times \bar \pi_v)$ is defined as 
$$
\Psi(z,W_v,\bar W_v',\Phi)=\int_{\sU(\Q_v)\bsl \sG(\Q_v)}\Phi_v({\mathsf e}_ng)W_v(g)\bar W_v'(g)|\det g|_v^{z}\,\d g, 
$$
which, under the unitarity condition of $\pi_v$, is shown to be absolutely convergent for $\Re z\geq 1$ (\cite[Propositions (1.5) and (3.17)]{JPSS}), where ${\mathsf e}_n=(0,\cdots,0,1)\in (\Q_v)_{n}$; for our purpose, the convergence at $z=1$ is relevant. Here we quote known results on computations of unramified zeta integrals.  
\begin{lem} \label{Ishii-Stade}
Suppose $\nu \in \ii \ft_{0}^{*}$ and $\pi_{v}\in \widehat{\sG(\Q_v)}_{\rm gen}^{\rm ur}$ is tempered. Let $\Phi_p^{0}=\cchi_{\Z_p^{n}}$ if $v=p<\infty$ and $\Phi_\infty^{0}(x)=e^{-\pi\sum_{i=1}^{n}x_i^2}$ if $v=\infty$. Then the zeta-integrals $\Psi(z,W_{\pi_v}^{0},\overline{W_{\pi_v}^{0}})$ and $
Z(z,W^{0}_{\sG_0(\Q_v)}(\nu)\otimes  {\overline {W_{\pi_v}^{0}}})$ are absolutely and normally convergent for $\Re\,z>0$ and  
\begin{align}
\Psi(z,W^{0}_{\pi_v},\overline{W_{\pi_v}^{0}},\Phi_v^{0})&=L(z,\pi_v \times \bar \pi_v),  \label{URZetaInt2}
\\
Z(z,W^{0}_{\sG_0(\Q_v)}(\nu)\otimes  {\overline {W_{\pi_v}^{0}}})&=\prod_{j=1}^{n-1}L(z+\nu_j, \bar \pi_v).
 \label{URZetaInt}
\end{align}
\end{lem}
\begin{proof} The statement on the convergence is proved in \cite[Proposition 5.3]{JS1990}. The formulas are due to Ishii and Stade (\cite[Theorem 2.2]{IshiiStade2013}) for $v=\infty$ and to Jacquet, Piatetski-Shapiro and Shalika (\cite[Proposition (2.3)]{JPSS}, \cite[Theorem 3.3]{Cogdell}) when $v<\infty$.
\end{proof}

If we let $\sG(\Q_v)$ act linearly on $(\Q_v)_{n}$ from the right, then $\Pcal(\Q_v)$ is the stabilizer of the vector ${\mathsf e}_n$ whose $\sG(\Q_v)$-orbit is $(\Q_v)_{n}-\{0\}$. Let $|\det g|_v \d \dot{g}$ denote the quotient measure on $\Pcal(\Q_v)\bsl \sG(\Q_v)$ of the quasi-invariant measure $|\det g|_v\,\d g$ on $\sG(\Q_v)$ by the right Haar measure on $\Pcal(\Q_v)$. Then from our normalization of Haar measures (see \S\ref{sec:Measure} and \S\ref{sec:MBS}), the measure $\zeta_v(n)^{-1}|\det g|_v (\d \dot{g})$ on $\Pcal(\Q_v)\bsl \sG(\Q_v)$ corresponds to $\d x:=\prod_{j=1}^{n}\d x_j$ on $(\Q_v)_{n}-\{0\}$, where $\d x_j$ is the Haar measure on $\Q_v$. Thus by the $\sG(\Q_v)$-invariance of the pairing \eqref{WhitProd},
{\allowdisplaybreaks\begin{align}
\Psi(1,W_v,\bar W_v,\Phi_v)&=\int_{\cP(\Q_v)\bsl \sG(\Q_v)}\Phi_v({\mathsf e}_n g)
 \langle R(g)W_ v|R(g)W_v\rangle_{\cP(\Q_v)}\,
\zeta_v(n)^{-1}\,|\det g|_v\,\d \dot{g}
 \label{WInnZetaInt3}
\\
&=\int_{(\Q_v)_{n}}\Phi_v(x)\,\d x \times \langle W_v|W_v\rangle_{\cP(\Q_v)}=
\widehat{\Phi_v}(0)\, \langle W_{v}|W_{v}\rangle_{\Pcal(\Q_v)}
 \notag
\end{align}} for any $W_v\in \WW_v^{\psi_v}(\pi_v)$ and a Schwartz-Bruhat function $\Phi_v$ on $(\Q_v)_{n}$. Thus from the value at $z=1$ of \eqref{URZetaInt2}, for $\pi_v\in \widehat {\sG(\Q_v)}_{\rm gen}^{\rm ur}$, 
\begin{align} 
\langle W_{\pi_v}^0|W_{\pi_v}^0\rangle_{\Pcal(\Q_v)}
&=L\left(1,\pi_v \times \bar \pi_v\right).
 \label{normLocalUrWhitt}
\end{align}

\begin{lem} \label{unramifiedPer} 
Let $\pi_v \in \widehat {\sG(\Q_v)}_{\rm gen}^{\rm ur}$. Then, for $\nu=(\nu_j)_{j=1}^{n-1} \in \ft_{0,\C}^*$ with $\Re\,\nu_j\geq 0\,(j \in [1,n-1]_{\Z})$ and for any $f\in C_{\rm c}^{\infty}(\bK_v\bsl \sG(\Q_v)/\bK_v)$, 
$${\PP}_{\psi_v}^{(\nu)}(\pi_v,\bK_v;f)={\widehat f}(\bar \pi_v)\,\,W_{\pi_v}^{0}(1_{n}). 
$$ 
\end{lem}
\begin{proof} 
From \eqref{normLocalUrWhitt}, we can take $\Bcal(\pi_v,\bK_v)=\{W_{\pi_v}^{0}\,L(1,\pi_v\times \bar \pi_v)^{-1/2}\}$. Then the relation $\bar W_{\pi_v}^{0}*\check f={\widehat {f}}(\bar \pi_v )\bar W_{\pi_v}^0$ and \eqref{URZetaInt} complete the proof. Note $W_{\pi_v}^{0}(1_n)=1$ if $v=p<\infty$. \end{proof}

\subsection{Computation of global Whittaker period}
Let $\pi\cong \otimes_{v}\pi_v \in \Pi_{\rm cusp}(\sG)_\sZ$. Let $S$ be a finite set of places of $\Q$ containing $\{\infty\}\cup S_{\pi}$ (see \S\ref{sec:Lfunction}). For a finite place $p\not\in S$, recall $W_{\pi_p}^{0}\in \WW^{\psi_p}(\pi_p)^{\bK_p}$ satisfies $W_{\pi_p}^{0}(1_n)=1$ (see \S~\ref{sec:URWFTN}). The global Whittaker model $\WW^{\psi}(\pi)$ is defined to be the space of all those finite $\C$-linear combinations of functions 
\begin{align}
W(g)=\prod_{v}W_v(g_v), \quad g=(g_v)_{v} \in \sG(\A)
\label{GlobalWhittTensor}
\end{align}
which are identified with the pure tensors $\otimes_v W_v\in \bigotimes _{v}\WW^{\psi_v}(\pi_v)$ with $W_p=W_{\pi_p}^{0}$ for almost all $p<\infty$. By the Fourier expansion (\cite{Shalika}), the map which sends $W\in \WW^{\psi}(\pi)$ to the function $$\varphi_{W}(g)=\sum_{\gamma \in \sU(\Q)\bsl \Pcal(\Q)}W(\iota(\gamma)g), \quad g\in \sG(\A)
$$ yields a $\sG(\A)$-intertwining $\C$-linear map from $\WW^{\psi}(\pi)$ onto $V_\pi$ whose inverse map is $\varphi \mapsto\Wcal^{\psi}(\varphi:\bullet)$, where $\Wcal^{\psi}(\varphi:g):=\Wcal^{\psi}(R(g)\varphi)$ is the global $\psi$-Whittaker function for $\varphi$ (\cite[\S1.2]{Cogdell}). The $L^2$-norm of $\varphi\in V_\pi$ (see \S\ref{sec:convolutionP}) corresponding to the pure tensor \eqref{GlobalWhittTensor} is given as
\begin{align} 
\|\varphi\|_{L^2}^{2}=\lim_{s\rightarrow 1}({L^{S}(s,\pi \times \bar\pi)}/{\Delta_\sG(s)})\,\prod_{v\in S}\langle W_v|W_v\rangle_{\Pcal(\Q_v)},\label{normphiW}
\end{align}
whose right-hand side is seen, from \eqref{normLocalUrWhitt}, to be independent of the choice of $S$. Indeed, \eqref{normphiW} is obtained by taking the residue at $s=1$ of the quantities in the identity \cite[\S(4.5),(5)]{JPSS} with $\Phi=\otimes_{v}\Phi_v$ being a Schwartz-Bruhat function on $\A^{n}$ such that $\Phi_v=\Phi_v^{0}$ for $v\not\in S$; the residue of the left-hand side of the identity is computed as $\widehat{\Phi}(0)\,\|\varphi\|_{L^2}$ by the formula \cite[\S(4.3), (2)]{JPSS} and by our measure on $\sZ(\A)$, whereas the residue of the right-hand side of the identity is computed to be $(\Delta_{\sG}(1)^{*})^{-1}\,{\rm Res}_{s=1}L^{S}(s,\pi\times \bar \pi)\times \prod_{v\in S}\widehat \Phi_v(0)\,\langle W_v|W_v\rangle_{\Pcal(\Q_v)}$ by our measure on $\sG(\A)$ by using \cite[\S(4.7),(2)]{JPSS} and \cite[Proposition(2.3)]{JPSS} and by the computation \eqref{WInnZetaInt3}.

From now on, we suppose $\pi_\infty$ is $\bK_\infty$-spherical. Let $K_\fin=\prod_{p<\infty}K_p$ be an open compact subgroup of $\bK_\fin$ such that $V_\pi^{K}\not=\{0\}$, where we set $K=K_\fin \bK_\infty$. Recall the global Rankin-Selberg integrals whose convergence was discussed in \S\ref{sec: UERSI}.

\begin{lem} \label{AverageWHittPerL1}
 Let $S$ be a finite set of prime numbers containing $\{p<\infty\mid K_p\not=\bK_p\,\}$. Let $f_p\in C_{\rm c}^\infty(K_p\bsl \sG(\Q_p)/K_p)$ for $p\in S$ and set $f_S=\otimes_{p\in S} f_p$. 
\begin{itemize} 
\item[(1)] For any $\varphi\in V_\pi^{K}$, the function $\nu \mapsto Z\left(\tfrac{1}{2},\hat E(\nu),\bar \varphi*\check f_{S}\right)$, which is defined by an absolutely convergent integral outside the poles of $\hat E(\nu)$, has a holomorphic continuation to $\ft_{0,\C}^*$. 
\item[(2)] We have 
{\allowdisplaybreaks\begin{align}
&\sum_{\varphi \in \Bcal(\pi^K)} Z\left(\tfrac{1}{2},\hat E(\nu),\bar \varphi*\check f_{S}\right)\,\Wcal^\psi(\varphi) 
\label{AverageWHittPerL1-1}
\\ 
&
=\frac{(\Delta_{\sG_0}(1)^*)^{-1}\prod_{i=1}^{n-1}L\left(\tfrac{1}{2}+\nu_j,\bar\pi\right)}{(\Delta_{\sG}(1)^*)^{-1}L(1,\pi;\Ad)}\,\biggl\{\prod_{p\in S}\PP_{\psi_p}^{(\nu)}(\pi_p,K_p;f_{p})\biggr\}\,W_{\sG(\R)}^{0}(\nu(\pi_\infty);1_n),
\notag
\end{align}}where $\nu(\pi_\infty)\in \ft_\C^*/\sS_n$ is the spectral parameter of $\pi_\infty\in \widehat {\sG(\R)}_{\rm gen}^{\rm ur}$ (\S\ref{sec:URps}), and $\Bcal(\pi^{K})$ is an orthonormal basis of $V_\pi^{K}$.\end{itemize}
\end{lem}
\begin{proof} For $p\in S$, let $\Bcal(\pi_p,K_p)$ be an orthonormal basis of $\WW^{\psi_p}(\pi_p)^{K_p}$ with respect to the inner product \eqref{WhitProd}. 

(1) From \eqref{normphiW}, the elements $\varphi_{W}\,(W\in \Bcal(\WW^{\psi}(\pi)^{K}))$ form an orthonormal basis of $V_\pi^{K}$, where $\Bcal(\WW^{\psi}(\pi)^{K})$ is the set of functions of the form 
\begin{align}
W((g_v)_{v})=\{\lim_{s\rightarrow 1}({L^{S}(s,\pi \times \bar\pi)}/{\Delta_\sG(s)})\} ^{-1/2} \,\prod_{v}W_v(g_v)
\label{AverageWHittPerL1-f101}
\end{align}
with $W_v\in \Bcal(\pi_v,K_v)$ for $v\in S$ and $W_p=W_{\pi_p}^{0}$ for $p\not\in S$. Since $W_{\pi_p}^0(1_n)=1$ for all prime numbers $p\not\in S$, 
\begin{align}
\Wcal^{\psi}(\varphi_W)&=
\{\lim_{s\rightarrow 1}({L^{S}(s,\pi \times \bar\pi)}/{\Delta_\sG(s)})\} ^{-1/2}\,W_{\sG(\R)}^{0}(\nu(\pi_\infty);1_n)\,\prod_{p\in S}W_p(1_n).
\label{AverageWHittPerL1-f3}
\end{align}
Let $\Wcal^{\psi}(\hat E(\nu);h)$ be the global $\psi$-Whittaker function on $\sG_0(\A)$ for $\hat E(\nu)$. For $\nu \in \fJ((\ft_{0,\C}^{++}))$, by unfolding the integral as in \cite[\S7]{Shahidi2}, we see that $\Wcal^{\psi}(\hat E(\nu);h)$ equals the product of $M_{\sG_0}(\nu)$ and the integral $\int_{\sU_0(\A)}\sf_{\sG_0}^{(\nu)}(w_0^{\ell}uh)\psi(\sum_{j=1}^{n-2}u_{j\,j+1})^{-1}\,\d u$ that is a product of $J_{\sG_0(\Q_v)}^{\psi_v}(\nu;h_v)$ $(v\leq \infty)$ due to the product formula $\sf_{\sG_0}^{(\nu)}(h)=\prod_{v}\sf_{\sG_0,v}^{(\nu)}(h_v)$ for $h=(h_v)_v\in \sG_0(\A)$. Thus by analytic continuation, for any regular point $\nu \in \ft_{0,\C}^*$ of $\hat E(\nu)$,
\begin{align}
\Wcal^{\psi}(\hat E(\nu);h)= \prod_{v} W_{\sG_0(\Q_v)}^{0}(\nu; h_v), \quad h=(h_v)_{v} \in \sG_0(\A).
\label{AverageWHittPerL1-f100}
\end{align}
The computation of \cite[p.116]{Cogdell} in our setting shows the identity
$$
Z(z,\hat E(\nu), \bar\varphi_W*{\check f_S})=\int_{\sU_0(\A)\bsl \sG_0(\A)} \,\Wcal^{\psi}(\hat E(\nu);h)\,\overline{\Wcal^{\psi}(\varphi_{W}*{\overline{\check f_S}}; \iota(h))}\,|\det h|_{\A}^{z-1/2}\,\d h
$$
for $\Re z\gg 0$. By substituting the formula \eqref{AverageWHittPerL1-f100} and $\Wcal^{\psi}(\varphi_{W}*\overline{\check {f_{S}}})=W*\overline{\check {f_{S}}}$ with $W$ as \eqref{AverageWHittPerL1-f101} and by using the formula $\d h=(\Delta_{\sG_0}(1)^{*})^{-1}\prod_{v}\d h_v$, we decompose the right-hand side of the last equality as a product of $Z(z,W_{\sG_0(\Q_p)}^0(\nu) \otimes \bar W_p*\check f_p)$ over $p\in S$ and the unramified zeta-integrals (calculated by \eqref{URZetaInt}) over $v\not\in S$ to obtain the formula: 
\begin{align}
Z(z,\hat E(\nu), \bar\varphi_W*{\check f_S})
&=\frac{\prod_{i=1}^{n-1}L(z+\nu_i,\bar \pi)}{\Delta_{\sG_0}(1)^{*}}\,
\biggl\{\lim_{s\rightarrow 1}\frac{{L^{S}(s,\pi \times \bar\pi)}}{\Delta_\sG(s)}\biggr\} ^{-1/2}\,\prod_{p\in S} \frac{Z(z, W^{0}_{\sG_0(\Q_p)}(\nu)\otimes \bar W_p*{\check f_p})}{\prod_{i=1}^{n-1}L(z+\nu_i,\bar \pi_p)}
 \label{AverageWHittPerL1-f1}
\end{align} for $\Re z\gg 0$. By analytic continuation, the identity holds for $z=1/2$. The $L$-function $\prod_{j=1}^{n-1}L(\frac{1}{2}+\nu_j,\bar \pi)$ is holomorphic on $\ft_{0,\C}^*$. From Lemma~\ref{lem:20201116} (ii), $Z^{0}(\nu;\bar W_p)$ for $p\in S$ is also holomorphic on $\ft_{0,\C}^*$. Hence $Z(1/2,\hat E(\nu), \bar\varphi_W*{\check f_S})$ turns out to be regular on the whole $\ft_{0,\C}^*$. 

\noindent
(2) Since ${\rm Res}_{s=1}\zeta(s)=1$, we have $\lim_{s\rightarrow 1}(\Delta_{\sG}(s)/\zeta(s))=\Delta_{\sG}(1)^{*}$. Thus from \eqref{AverageWHittPerL1-f3} and \eqref{AverageWHittPerL1-f1} the left-hand side of \eqref{AverageWHittPerL1-1} becomes the product of 
\begin{align*}
\frac{(\Delta_{\sG_0}(1)^{*})^{-1} \prod_{i=1}^{n-1}L\left(\tfrac{1}{2}+\nu_i,
\bar \pi\right)}{\lim_{s\rightarrow 1}{L(s,\pi \times \bar\pi)}{\Delta_\sG(s)}^{-1}}=\frac{(\Delta_{\sG_0}(1)^*)^{-1} \prod_{i=1}^{n-1} L\left(\tfrac{1}{2}+\nu_i,\bar \pi\right)}{(\Delta_\sG(1)^*)^{-1}L(1,\pi;\Ad)},
\end{align*}
\begin{align} 
\sum_{W \in \Bcal(\WW^\psi(\pi)^K} \prod_{p\in S}
L(1,\pi_p\times \bar \pi_p)Z^0(\nu;\bar W_p*{\check f_p})\,W_p(1_n), 
 \label{AverageWHittPerL1-2}
\end{align}
and $W_{\sG(\R)}^{0}(\nu_\infty(\pi);1_n)$. The summation range of \eqref{AverageWHittPerL1-2} is replaced with all the systems $\{W_p\}\in \prod_{p\in S}\Bcal(\pi_p,K_p)$. Thus \eqref{AverageWHittPerL1-2} equals
$$
\prod_{p\in S} \biggl\{\sum_{W_p\in \Bcal(\pi_p,K_p)}L(1,\pi_p\times \bar\pi_p)Z^0(\nu;\bar W_p*{\check f_p})\, W_p(1_n)\biggr\}=\prod_{p\in S}\PP_{\psi_p}^{(\nu)}(\pi_p,K_p;f_p). 
$$
\end{proof}

\begin{lem}\label{ZetaIntUnifBound}
\begin{itemize}
\item[(1)]
Let $f_S$ be as in Lemma~\ref{AverageWHittPerL1} and $\varphi \in \Bcal(\pi^{K})$ with $\pi \in \Pi_{\rm cusp}(\sG)^{K}_\sZ$. Then the function $\nu\mapsto Z(1/2,\hat E(\nu),\bar\varphi*\check f_S)$ is holomorphic and vertically bounded on $\ft_{0,\C}^{*}$. 
\item[(2)]
Let $\Ncal \subset \ft_{0}^*$ be a compact set. There exists a constant $r>0$ such that \begin{align}
\left|Z\left(\tfrac{1}{2},\hat E(\nu), \bar \varphi* \check f_{S}\right)\right|\ll_{\Ncal} \fQ(\pi_\infty)^{r}, \quad \pi \in \Pi_{\rm cusp}(\sG)_\sZ^{K},\,\varphi\in \Bcal(\pi^{K}),\,\nu \in \fJ(\Ncal).  
\label{ZetaIntUnifBound-1}
\end{align}
\end{itemize}
\end{lem}
\begin{proof}
(1) Since $L(z,\bar \pi)$ is entire and vertically bounded on $\C$ (\cite[Theorem 4.1]{Cogdell}), from \eqref{AverageWHittPerL1-f1} together with Lemma~\ref{lem:20201116}, so is the function $Z(1/2,\hat E(\nu),\bar \varphi*\check f_S)$. 

(2) The pole divisor of $\hat E(\nu;h)$ is contained in the union $D$ of the hyperplanes $\{\nu\in \ft_{0,\C}^{*}|\,\nu_i-\nu_j=\e\}$ $(1\leq i<j\leq n-1,\,\e\in \{1,-1\})$. If $\Ncal$ is disjoint from $D$, the estimate follows from Lemma~\ref{SpectExpPerLLL}. Let $\nu^{0}=(\nu_j^{0})_{j=1}^{n-1}\in \ft_{0}^*\cap D$ and fix $\delta>0$ and $\mu \in \ft_0^{*}$ as in Lemma~\ref{lem:20201114}. By Cauchy's theorem and (1), $Z\left(\tfrac{1}{2},\hat E(\zeta), \bar \varphi* \check f_{S}\right)$ with $\|\nu^0-\Re\zeta\|<\delta/8$ is expressed as a finite sum of the contour integrals
$$\left(\tfrac{1}{2\pi \ii}\right)^{n-1} {\small \int_{\Re z_1=\pm \mu_1}\dots\int_{\Re z_{n-1}=\pm \mu_{n-1}}}Z\left(\tfrac{1}{2},\hat E(\zeta+z),\bar \varphi* \check f_{S}\right)\,\prod_{j=1}^{n-1}e^{z_j^2}\tfrac{\d z_j}{z_j}.
$$ Since the points $\zeta+z$ are away from $D$, $|Z\left(\tfrac{1}{2},\hat E(\zeta+z),\bar \varphi* \check f_{S}\right)|$ is bounded by $\fQ(\pi_\infty)^{r}$ uniformly in $z$ and in $\zeta$ with $\|\nu^0-\Re\zeta\|<\delta$. Hence the bound \eqref{ZetaIntUnifBound-1} holds true for $\nu$ with $\Re \nu$ being in the $8^{-1}\delta$-neighborhood of $\nu^0$.  
\end{proof}

\begin{lem} \label{lem:20201114} Let $D$ be the union of hyperplanes $\{\nu\in \ft_{0,\C}^{*}\mid \nu_i-\nu_j=\e\}$ $(1\leq i<j\leq n-1,\,\e\in \{1,-1\})$. Let $\nu^0=(\nu_j^0)_{j=1}^{n-1}\in \ft_{0}^{*}\cap D$. There exist $\delta>0$ and a point $\mu \in \ft_{0}^{*}$ such that $\delta/2<|\mu_i-\mu_j|<\delta\,(i\not=j)$ and $\nu+\mu\not\in D$ for all $\nu \in \ft_{0}^{*}$ with $\|\nu^0-\nu\|<\delta/8$.   \end{lem}
\begin{proof} Let $R$ (resp. $R'$) be the set of all $(i,j)$ such that $1\leq i<j\leq n-1$ and $\nu_{i}^{0}-\nu_{j}^{0}\in \{\pm 1\}$ (resp. $\nu_{i}^{0}-\nu_{j}^{0}\not\in \{\pm 1\}$). For each $(i,j)\in R$, set $\e_{ij}:=\nu_{i}^{0}-\nu_j^{0}$. Choose $C>0$ such that $|1-|\nu_{i}^{0}-\nu_j^{0}||>C$ for $(i,j)\in R'$ and $|-\e_{ij}-(\nu_i^{0}-\nu_j^{0})|>C$ for $(i,j)\in R$. Fix $\delta\in (0,C/2)$ small enough so that $\|\nu^0-\nu\|<\delta/8$ implies $|1-|\nu_i-\nu_j||>C/2$ for $(i,j)\in R'$ and $|-\e_{ij}-(\nu_i-\nu_j)|>C/2$ for $(i,j)\in R$. Let $\mu \in \ft_{0,\R}$ be any point such that $\delta/2<|\mu_i-\mu_j|<\delta\,(i\not=j)$ and let $\|\nu^0-\nu\|<\delta/8$. Then $|\e-(\nu_i+\mu_i-\nu_j-\mu_j)|\geq |1-|\nu_i-\nu_j||-|\mu_i-\mu_j|\geq C/2-\delta$ for $(i,j)\in R'$ and $\e\in \{\pm 1\}$, and $|\e_{ij}-(\nu_i+\mu_i-\nu_j-\mu_j)|\geq |\mu_i-\mu_j|-|(\nu_i^0-\nu_j^{0})-(\nu_i-\nu_j)|\geq \delta/2-\delta/4=\delta/4$ for $(i,j)\in R$; $|-\e_{ij}-(\nu_i+\mu_i-\nu_j-\mu_j)|\geq |-\e_{ij}-\nu_i+\nu_j|-|\mu_i-\mu_j|\geq C/2-\delta$ for $(i,j)\in R$. Hence $\nu+\mu\not\in D$. 
\end{proof}

\subsection{The spectral side} \label{sect:spectralside}
With the notation introduced in \S\ref{SPEXP}, we set
\begin{align*}
{\bf I}_f(\nu)=(\Delta_\sG(1)^*)^{-1}\sum_{\pi \in \Pi_{\rm{cusp}}(\sG)_\sZ^K}\widehat{f^S}(\bar \pi^{S})\,\sum_{\varphi \in \Bcal(\pi^K)} Z\left(\tfrac{1}{2},\hat E(\nu),\bar \varphi*\check f_{S}\right)\,\Wcal^\psi(\varphi), \quad \nu\in \ft_{0,\C}^*. 
\end{align*}
Note that from Lemma~\ref{ZetaIntUnifBound} (1) each summand is meaningful for all $\nu \in \ft_{0,\C}^{*}$. 
\begin{prop} \label{SpectExpPerL-3}
The series ${\bf I}_f(\nu)$ converges absolutely and locally uniformly in $\nu$ defining a holomorphic function of $\nu\in \ft_{0,\C}^*$ which is vertically of moderate growth over $\ft_{0}^{*}$. Moreover, for any $\sigma\in \ft_{0}^{*}$, 
\begin{align*}
\Wcal^{\psi}(\tilde{\bf \Phi}_{f,\b})=\int_{\fJ(\s)}\beta(\nu)\,{\bf I}_f(\nu)\,\d\nu.
\end{align*}
\end{prop}
\begin{proof} Let $\Ncal$ be a compact set of $\ft_{0}^*$. Lemmas~\ref{ZetaIntUnifBound}, \ref{SpectExpPerL-4},\ref{SpectExpPerL-4-1} and \ref{SpectExpPerL-1} show that 
$$
\sum_{\pi \in \Pi_{\rm{cusp}}(\sG)_\sZ^K}|\widehat{f^S}(\bar \pi^{S})|
\,\sum_{\varphi \in \Bcal(\pi^K)} |Z\left(\tfrac{1}{2},\hat E(\nu),\bar \varphi*\check f_{S}\right)| \times |\varphi(g)|
$$
has the majorant 
$$
\sum_{\pi\in \Pi_{\rm{cusp}}(\sG)_\sZ^{K}}\#(\cB(\pi^{K}))\,\fQ_\infty(\pi)^{-m_1} \times e^{-m\langle H(g),\rho_{\sB}\rangle}
$$
uniformly in $(\nu,g)\in {\fJ}(\Ncal)\times (\fS_{\sG}\cap \sG(\A)^1)$, where $m_1$ can be taken so that the sum is convergent by Lemma~\ref{SpectExpPerL-5}. By taking integral over the compact space $\sU(\Q)\bsl \sU(\A)$ of \eqref{SpectExpPer-0}, we have the assertions by Fubini's theorem and \eqref{hatE*-feq}.    
\end{proof}

\section{The Whittaker-Fourier coefficient of $\tilde{\bf \Phi}_{f,\b}$: the geometric side} \label{sec:WFcoefficient}
Let $f=\otimes_{v}f_v$ be as in \S\ref{SPEXP} and fix $\beta\in {\mathcal B}_0$. In this section, we continue to study the integral \eqref{WFCoeffDef} for $\varphi=\tilde{\bf \Phi}_{f,\beta}$ but from a different perspective than in the previous section. For $\gamma \in \sG(\Q)$, set 
$$
 \sU_{[\gamma]}=\gamma^{-1}\sZ\iota(\sB_0) \gamma\cap \sU.
$$
By the absolute convergence of the series \eqref{Pser-f1} shown in Proposition~\ref{SQint}, the summation over $\gamma \in \sZ(\Q)\iota(\sB_0(\Q))\bsl \sG(\Q)$ is written as an double sum over $\gamma \in \sZ(\Q)\iota(\sB_0(\Q))\bsl \sG(\Q)/\sU(\Q)$ and $\delta\in \sU_{[\gamma]}(\Q)\bsl \sU(\Q)$. Then by substituting this expression to \eqref{WFCoeffDef} and then by unfolding ({\it cf}. \cite{Jacquet2}), we have the identity; 
{\allowdisplaybreaks
\begin{align*}
\Wcal^{\psi}(\tilde{\bf \Phi}_{f,\b})&=\sum_{\gamma \in \sZ(\Q)\iota(\sB_0(\Q)) \bsl \sG(\Q)/\sU(\Q)} 
\int_{\sU_{[\gamma]}(\A)\bsl \sU(\A)} \biggl(
\int_{u_1\in \sU_{[\gamma]}(\Q)\bsl \sU_{[\gamma]}(\A)} \tilde\Phi_{f,\b}(\gamma u_1 u)\,\psi_{\sU}(u_1 u)^{-1}\,\d u_1\biggr)\,\d u,
\end{align*}}where $\d u_1$ is the Tamagawa measure on the unipotent group $\sU_{[\gamma]}(\A)$, which is a product of local Tamagawa measures on $\sU_{[\gamma]}(\Q_v)$ so that $\vol(\sU_{[\gamma]}(\Z_p))=1$ if $v=p<\infty$. As such, the volume of $\sU_{[\gamma]}(\Q)\bsl \sU_{[\gamma]}(\A)$ is $1$. Since $\gamma \sU_{[\gamma]} \gamma^{-1}$ consists of unipotent matrices in $\sZ\iota(\sB_0)$, we have $\gamma \sU_{[\gamma]} \gamma^{-1}\subset \iota(\sU_0)$, which shows the identity $\tilde\Phi_{f,\b}(\gamma u_1 u)=\tilde\Phi_{f,\b}(\gamma u)$ for all $u_1\in \sU_{[\gamma]}(\A)$ by the left $\iota(\sU_0(\A))$-invariance of $\tilde \Phi_{f,\b}$. Thus, 
\begin{align}
\Wcal^{\psi}(\tilde{\bf \Phi}_{f,\b})&=\sum_{\gamma \in \sZ(\Q)\iota(\sB_0(\Q))\bsl \sG(\Q)/\sU(\Q)}a(\gamma)\,\JJ_{f,\b}(\gamma)
 \label{FWGeo-f1}
\end{align}
with
\begin{align}
a(\gamma)&:=\int_{\sU_{[\gamma]}(\Q)\bsl \sU_{[\gamma]}(\A)}\psi_{\sU}(u_1)^{-1}\d u_1, \qquad\JJ_{f,\b}(\gamma):=\int_{\sU_{[\gamma]}(\A)\bsl \sU(\A)}\tilde\Phi_{f,\b}(\gamma u)\,\psi_{\sU}(u)^{-1}\d u
 \label{FWGeo-aJ}
\end{align}
and 
\begin{align}
\sum_{\gamma \in \sZ(\Q)\iota(\sB_0(\Q))\bsl \sG(\Q)/\sU(\Q)}|a(\gamma)\,\JJ_{f,\b}(\gamma)|<\infty. 
\label{ABSCONV-GEOM}
\end{align}
In \S\ref{sec:doublecosets}, we determine a complete set of representatives of the double coset space appearing in \eqref{FWGeo-f1} (see Lemma~\ref{Doublecoset}). In \S\ref{sec:IntStab}, we compute $a(\gamma)$ for each representative $\gamma$ (see Lemma~\ref{agammanot0}). In \S\ref{sec: DefOrbInt}, we study $\JJ_{f,\beta}(\gamma)$ for $\gamma$ with $a(\gamma)\not=0$ and obtain the expression \eqref{GeoSideP1}.

\subsection{The double cosets} \label{sec:doublecosets}
In this subsection, for any $\Q$-algebraic group $\sH$, we abbreviate $\sH(\Q)$ to $H$. Thus $\sZ(\Q)$, $\sB_0(\Q)$, $\sU(\Q)$ and $\sG(\Q)$ are denoted by $Z$, $B_0$, $U$ and $G$, respectively. To parametrize the $(Z\iota(B_0),U)$-double cosets in $G$, we need a series of definitions. First of all, set $I^0=[1,n-1]_{\Z}$ and $I=[1,n]_{\Z}$. For $i\in I^{0}$, let $I_{>i}^{0}$ (resp. $I_{<i}^{0}$) denote the interval $\{j\in I^0 \mid j>i \,\}$ (resp. $\{j\in I^0\mid j<i\,\}$). For $w\in \sS_{n}$, set $I_w^0:=\{i\in I^0\mid w^{-1}(i)>w^{-1}(n)\,\}$. For any $\Q$-algebra $F$ and $y\in F^{n-1}$, set  
$$
 {\rm{sp}}(y)=\{i\in I^0\mid y_i\not=0\,\}.
$$

\begin{df} \label{OrderRevLemm}
For $w\in \sS_{n}$, let $Y(w)$ be the set of $y\in \Q^{n-1}$ such that ${\rm sp}(y)\subset I_w^0$, and $w^{-1}$ is decreasing on ${\rm{sp}}(y)$, i.e., if $h,k\in {\rm{sp}}(y)$ and $h<k$, then $w^{-1}(h)>w^{-1}(k)$.
\end{df}
Note that we always have $0\in Y(w)$. 

%\begin{lem} \label{OrderRevLemm}
%Let $w\in \sS_{n}$ and $y \in \Q^{n-1}$. Suppose ${\rm sp}(y)\subset I_w^{0}$. %Then $y\in Y(w)$ if and only if . 
%\end{lem}
%\begin{proof} This follows from (II). 
%\end{proof}

The following relation will be used repeatedly: 
\begin{align}
(wAw^{-1})_{ij}=A_{w^{-1}(i),w^{-1}(j)}, \quad A\in \Mat_n(\A),\,w\in \sS_{n}.
\label{wAwinv}
\end{align}

For $w\in \sS_{n}$, we consider the subgroups $U^{w}=U\cap  w\, \bar U\, w^{-1}$ and $U_w=U\cap  w\, U\, w^{-1}$, where $\bar U$ denotes the group of all the lower triangular unipotent matrices in $G$. Then, 
$$
U_w=\{u\in U\mid u_{ij}=0\,(i<j,\,w^{-1}(i)>w^{-1}(j))\,\}, \quad 
U^{w}=\{u\in U\mid u_{ij}=0\,(i<j,\,w^{-1}(i)<w^{-1}(j))\,\}
$$
by \eqref{wAwinv}, and $U=U^{w}U_{w}=U_wU^w$ by the next lemma:
\begin{lem}\label{lemRR'prod} For any subset $R\subset \{(i,j)\in I \times I\mid i<j\}$, define $\Q$-algebraic subsets $\sU_{R}$ of $\sU$ by setting $\sU_{R}:=\{(u_{ij})_{ij}\in \sU\mid u_{ij}=0\,((i,j)\not\in R)\,\}$. Let $R'$ be the complement of $R$ in $\{(i,j)\in I \times I \mid i<j\}$. Then for any $\Q$-algebra $F$, the product map $\mu:(x,y)\mapsto xy$ is a bijection from $\sU_{R}(F)\times \sU_{R'}(F)$ onto $\sU(F)$. Let $\omega_{\sU}$, $\omega_{\sU_{R}}$, and $\omega_{\sU_{R'}}$ be everywhere non-zero algebraic differential forms of top-degree on $\sU$, $\sU_{R}$, and $\sU_{R'}$ given as $\omega_{\sU}=\wedge_{1\leq i<j\leq n}\d u_{ij}$, $\omega_{\sU_{R}}(x)=\wedge_{(i,j)\in R}\d x_{ij}$ and $\omega_{\sU_{R'}}(y)=\wedge_{(i,j)\in R'}\d y_{ij}$, respectively. Then $\mu^{*}\omega_{\sU}$ coincides with $\omega_{\sU_{R}}(x)\wedge \omega_{\sU_{R'}}(y)$ up to sign. 
\end{lem}
\begin{proof} It is enough to show that the relation $u=xy$ for $u\in \sU(F)$ is uniquely solved by a pair $(x,y)\in \sU_{R}(F) \times \sU_{R'}(F)$. Since $u,x,y\in \sU(F)$, the relation $u=xy$ is equivalent to $u_{ij}-\sum_{i<m<j}x_{im}y_{mj}$ being equal to $x_{ij}$ if $(i,j)\in R$ or to $y_{ij}$ if $(i,j)\in R'$; this recurrence relation uniquely determines the set of numbers $x_{ij}\,((i,j)\in R,\,|i-j|=d)$ and $y_{ij}\,((i,j)\in R',\,|i-j|=d)$ inductively in $d$. The last formula is also shown by the recurrence formula
\end{proof}

Recall the notation $\sn(y)$ from \S\ref{sec:DistSubgrp}. 
\begin{lem} \label{Doublecoset}
The set $G$ is a disjoint union of subsets 
\begin{align}
Z\iota(B_0)\,\sn(y) w\,U, \quad (w\in \sS_{n},\, y\in Y(w)).
\label{Doublecoset-0}
\end{align}
The double cosets in \eqref{FWGeo-f1} are parametrized by the points $\sn(y)\,w$ with $w\in \sS_{n}, y\in Y(w)$. 
\end{lem}
\begin{proof} For use in this proof, we set 
\begin{align*}
{\rP}(w)&=\{(i,j)\in I^2 \mid i<j,\,w^{-1}(i)>w^{-1}(j)\,\}, \quad 
\rN(w)=\{(i,j)\in I^2 \mid i<j,\,w^{-1}(i)<w^{-1}(j)\,\}
\end{align*}
for $w\in \sS_n$. By the Bruhat decomposition of $G$, we have that $G$ is a disjoint union of $TU^{w} w U$ $(w\in \sS_{n}$). Let $u\in U^{w}$ and write $u=\iota(v)\,\sn(y)$ with $v\in U_0$ and $y\in \Q^{n-1}$. Then we have 
\begin{align}
\text{$y_i=0$ for all $i\in I^{0}-I^{0}_w$.}
\label{Doublecoset-1}
\end{align}
Indeed, the condition $u\in U^{w}$ implies $v_{ij}=0$ for $(i,j)\in \rN(w)\cap (I^{0}\times I^{0})$, and $0=(vy)_i=y_i+\sum_{j \in I^{0}\,;(i,j)\in \rP(w)}v_{ij}y_j$ for all $i\in I^0$ such that $w^{-1}(i)<w^{-1}(n)$. From this, by downward-induction on $i$, we have \eqref{Doublecoset-1}. At this point, we have that $G$ is a union of subsets $Z\iota(B_0)\sn(y)w U$ with $(w,y)\in \sS_n \times \Q^{n-1}$ such that ${\rm sp}(y)\subset I_w^0$. 

For a while, we fix $(w,y)\in\sS_n\times\Q^{n-1}$ satisfying ${\rm sp}(y)\subset I_w^0$ and will find some $\tilde y\in Y(w)$ such that $Z\iota(B_0)\sn(y)w U=Z\iota(B_0)\sn(\tilde y)wU$. If $\iota(b)\in \iota(B_0)\cap w U w^{-1}=\iota(U_0)\cap U_w$, then
\begin{align*}
Z\iota(B_0)\,\sn(y) w\,U&=Z\iota(B_0)\,\iota(b)^{-1}\sn(by) w\,( w^{-1}\iota(b) w)U =Z\iota(B_0)\,\sn(by) w\, U.
\end{align*}
By this, it suffices to show that there exists $b\in U_0$ such that $\iota(b)\in U_w$ and $\tilde y:=by$ belongs to $Y(w)$. Note that $\iota(b) \in \iota(U_0)\cap U_w$ if and only if $b\in U_0$ and $b_{hk}=0$ for all $(h,k)\in \rP(w)\cap (I^0\times I^0)$. For such a matrix $b$, from ${\rm sp}(y)\subset I_w^0$, we have
$$
(by)_i=y_{i}+\sum_{\substack{j\in I^0_w \\ (i,j)\in \rN(w)}}b_{ij}y_{j}, \quad i \in I^{0}. 
$$
Let us define a matrix $b\in U_0$ by setting $b_{ij}\,(j\in I^0_{>i})$ for each $i\in I^0$ as follows separating cases: If $i\in I^0-I_w^0$, then we set $b_{ij}=0$ for all $j\in I^0_{>i}$. If $i\in I_w^0$ and if there exists $j\in{\rm{sp}}(y)$ such that $(i,j)\in \rN(w)$, then we pick such a $j$, say $j(i)$, and set $b_{ij}=-y_iy_{j(i)}^{-1}$ and $b_{ij'}=0$ for all $j'\in I^0_{>i}-\{j(i)\}$. If $i\in I_w^0$ and if all $j\in {\rm{sp}}(y)\cap I^{0}_{>i}$ satisfy $(i,j)\in \rP(w)$, then we set $b_{ij}=0$ for all $j\in I^0_{>i}$. For this $b$, we have $b_{ij}=0$ for $(i,j)\in \rP(w)\cap (I^0\times I^0)$; moreover, $(by)_i=y_i=0$ for $i\in I^0-I^0_w$ and $(by)_{i}=y_i+(-y_iy_{j(i)}^{-1})\,y_{j(i)}=0$ for $i\in I_w^0$ such that $(i,j)\in N(w)$ with some $j\in {\rm{sp}}(y)$. We have the equality \begin{align}
{\rm{sp}}(by)=\{i\in {\rm{sp}}(y)\mid (i,j)\in \rP(w) \,\text{for all $j\in {\rm{sp}}(y)\cap I_{>i}^0$}\,\}.
\label{Doublecoset-2}
\end{align}
This yields ${\rm sp}(by)\subset {\rm sp}(y)\subset I_w^0$. Then we have $by\in Y(w)$ by Definition~\ref{OrderRevLemm}. From the argument so far, we confirm that the set of double cosets \eqref{Doublecoset-0} cover the whole group $G$. Let us show that double cosets in \eqref{Doublecoset-0} are mutually disjoint. To argue, take $(w,y)$ and $(w',y')$ with $w,w'\in \sS_n$, $y\in Y(w)$, $y'\in Y(w')$ in such a way that the containment $\sn(y) w\in [z]\iota(b)^{-1}\sn(y') w'\,U$ holds with some $b\in B_0$ and $z\in \Q^\times$. From the Bruhat decomposition, this forces us to have $w=w'$ and $[z^{-1}] w^{-1}\sn(-y')\iota(b)\sn(y) w \in U$. Hence $z=1$, $b\in U_0$ and $\left[\begin{smallmatrix} b & by-y' \\ 0 & 1 \end{smallmatrix}\right] \in U\cap  w U w^{-1}.$ From this, $\iota(b)\in \iota(U_0)\cap U_{w}$, and also $(-by+y')_{i}=0$ for all $i\in I_w^{0}$, or equivalently 
$$
y_{i}'=y_i+\sum_{\substack{j\in I^{0} \\ (i,j)\in \rN(w)}}b_{ij} y_j, \quad i\in I_w^0. 
$$
If $i\in {\rm{sp}}(y)$, then $y_j=0$ for all $j\in I^0$ such that $(i,j)\in \rN(w)$ by Definition \ref{OrderRevLemm}. Thus $y_i'=y_i$. This shows the inclusion ${\rm{sp}}(y)\subset {\rm{sp}}(y')$. By the same argument reversing the roles of $y$ and $y'$, we also have the converse inclusion to get the equalities ${\rm{sp}}(y')={\rm{sp}}(y)$ and $y=y'$. 
\end{proof}

We examine the group $U_{[\gamma]}=\sU_{[\gamma]}(\Q)$ for particular $\gamma$'s. 

\begin{lem} \label{StabL}
Let $\gamma=\sn(y)\, w$ with $w\in \sS_{n}$ and $y=(y_j)_{j=1}^{n-1} \in Y(w)$.
Then $U_{[\gamma]}$ consists of all the matrices of the form 
$$
 w^{-1}\left[\begin{smallmatrix} u & uy-y \\ 0 & 1 \end{smallmatrix}\right] w
$$
with $u\in U_0$ satisfying the following conditions:
\begin{itemize}
\item[(i)] $u_{ij}=0$ for all $i,j\in I^0$ such that $i<j$, $w^{-1}(i)>w^{-1}(j)$. 
\item[(ii)] We have
\begin{align}
\sum_{\substack{j\in {\rm{sp}}(y)\cap I_{>i}^0 \\ w^{-1}(i)<w^{-1}(j)}} u_{ij}y_j=0 \quad {\text{ for all $i\in I_w^0-{\rm{sp}}(y)$.}}
\label{StabLii}
\end{align}
\end{itemize}
\end{lem}
\begin{proof}
Since the condition $\gamma^{-1}[z]\iota(b)\gamma \in U$ implies $z=1$ and $b\in U_0$, a general element of the set $U_{[\gamma]}=\gamma^{-1}Z\iota(B_0)\gamma\cap U$ can be written as 
$$
\gamma^{-1} \iota(u) \gamma= w^{-1} \sn(-y)\iota(u) \sn(y)  w=
 w^{-1} \left[\begin{smallmatrix} u & uy-y \\ 0 & 1 \end{smallmatrix}\right] w 
$$
with $u\in U_0$ satisfying $\left[\begin{smallmatrix} u & uy-y \\ 0 & 1 \end{smallmatrix}\right] \in U_w$, or equivalently $\iota(u)\in U_w$ and 
\begin{align}
\sum_{j\in I^{0}_{>i}} u_{ij}y_j=0 \quad \text{for all $i \in I_w^{0}$.}
 \label{StabL-1}
\end{align}
The condition $\iota(u)\in U_w$ is equivalent to (i); under this condition, the range of summation in \eqref{StabL-1} can be reduced to only those $j\in {\rm{sp}}(y)$ such that $i<j$, $w^{-1}(i)<w^{-1}(j)$; since $y\in Y(w)$, there is no such $j$ when $i\in {\rm sp}(y)$ (Definition~\ref{OrderRevLemm}). Thus for $\iota(u)\in U_w$, the condition \eqref{StabL-1} is equivalent to \eqref{StabLii}. 
\end{proof}

\subsection{The integral on stabilizers} \label{sec:IntStab}
We examine the condition of non-vanishing of the integral $a(\gamma)$ defined by \eqref{FWGeo-aJ}. For a subset $Q\subset I^0$, set
\begin{align}
m(Q)&=\# Q+1, \qquad Q'=I^{0}-Q, 
 \notag
\\
{\mathsf Y}_{Q}&=\{y\in \Q^{n-1}\mid {\rm sp}(y)\subset Q'\cap \bigcap_{q\in Q}I_{<q}^{0}\,\}.
 \label{sfYQ}
\end{align}
Note that $m(Q)\in I$ and that ${\mathsf Y}_Q=\emp$ if $Q=I^0$. For $w\in \sS_n$, set
\begin{align}
\fN(w):=\{i \in I^{0}\mid i>w^{-1}(n),\,w(i)<w(i+1)\,\}.
 \label{def:fNw}
\end{align}

\begin{df} \label{Def1}
For $Q\subset I^{0}$ and $y\in \Q^{n-1}$ with ${\rm sp}(y)\subset Q'$, let $\sS_n(Q,y)$ be the set of all $w\in \sS_n$ with the following properties.
\begin{itemize}
\item[(a)] $Q=w([1,m(Q)-1]_{\Z})$, and $w(m(Q))=n$. 
\item[(b)] $w(i)>w(j)$ if $1\leq i<j\leq m(Q)-1$.
\item[(c)] $w^{-1}(h)>w^{-1}(k)$ if $h<k$ and $h,k\in {\rm sp}(y)$.  
\item[(d)] If $j\in \fN(w)$, then $w(j)\not\in {\rm sp}(y)$, $w(j+1)\in {\rm sp}(y)$ and $[w(j),w(j+1))\cap {\rm sp}(y)=\emp$.
\end{itemize}
\end{df}
Note that $Q'=w([m(Q)+1,n]_{\Z})=I_w^{0}$ and $\fN(w)=\{i\in [m(Q)+1,n-1]_{\Z}\mid w(i)<w(i+1)\}$ from (a). Then from Definition \ref{OrderRevLemm} and (c), we have the containment $y \in Y(w)$ for all $w\in \sS_n(Q,y)$. 

\smallskip
\noindent
{\bf Remark} : When $y=0$, the set $\sS_n(Q,0)$ coincides with the set of $w\in \sS_n$ such that $w(i)>w(j)$ if $1\leq i<j<m(Q)$ or $m(Q)<i<j\leq n$. Indeed, $y=0$ (i.e., ${\rm sp}(y)=\emp$) implies that (d) is met only when $\fN(w)=\emp$, which in turn is equivalent to $w$'s being decreasing on the interval $[m(Q)+1,n]_\Z$.

\begin{lem} \label{agammanot0}
Let $\gamma=\sn(y)w$ with $w\in \sS_n$ and $y\in Y(w)$. Then $a(\gamma)\not=0$ if and only if there exists a unique subset $Q\subset I^0$ such that $y\in {\mathsf Y}_Q$ and $w\in \sS_n(Q,y)$. 
\end{lem}
\begin{proof}
Note $a(\gamma)\not=0$ if and only if $\psi_\sU|\sU_{[\gamma]}(\A)$ is identically $1$, in which case $a(\gamma)=\vol(\sU_{[\gamma]}(\Q)\bsl \sU_{[\gamma]}(\A))=1$ by our normalization of the Haar measure. A general element of $\sU_{[\gamma]}(\A)$ is written as $ w^{-1} v  w$ with $v=\left[\begin{smallmatrix} u & u y-y \\ 0 & 1 \end{smallmatrix}\right]$ and $u\in \sU_0(\A)$ satisfying the conditions (i) and (ii) in Lemma~\ref{StabL}. The value $\psi_{\sU}( w^{-1}v w)=\prod_{i=1}^{n-1} \psi(v_{w(i),w(i+1)})$ is equal to the product of 
{\allowdisplaybreaks\begin{align} 
&\prod_{\substack{i\in I^0 \\ i,i+1\in I^{0}_{>m}, w(i)<w(i+1)}} \psi(u_{w(i),w(i+1)})=\prod_{i\in \fN(w)} \psi(u_{w(i),w(i+1)}), \label{agammanot0-1} \\ 
&\prod_{\substack{i\in I^0 \\ i,i+1\in I^{0}_{<m}, w(i)<w(i+1)}} \psi(u_{w(i),w(i+1)}), \label{agammanot0-2}\\
&\prod_{\substack{j \in {\rm{sp}}(y) \\ w(m-1)<j, m-1<w^{-1}(j)}} \psi(u_{w(m-1),j}\,y_{j}), \quad (\text{if $m>1$}), \label{agammanot0-3}
\end{align}}where $m=w^{-1}(n)$. From Lemma~\ref{StabL}, the coordinate ring of the variety $\sU_{[\gamma]}$ is generated by functions $u_{w(i),w(j)}$ with $i,j\in I-\{m\}$ such that $i<j$ and $w(i)<w(j)$ subject to the the linear relations \eqref{StabLii} equivalently written as \begin{align}
\sum_{\substack{j\in I_{>i}^0 \\ w(i)<w(j),\,w(j)\in {\rm sp}(y)}}u_{w(i)\,w(j)}\,y_{w(j)}=0 \quad (i \in I^{0}_{>m}-w^{-1}({\rm sp}(y))). 
 \label{StabL-1prime}
\end{align}
These relations impose no constraints on $u_{w(i),w(i+1)}$ and on $u_{w(m-1),j}$ appearing in \eqref{agammanot0-2} and in \eqref{agammanot0-3}, respectively. Thus \eqref{agammanot0-2} and \eqref{agammanot0-3} are identically $1$ if and only if the product-ranges are empty. If $i\in \fN(w) \cap w^{-1}({\rm sp}(y))$, then the condition \eqref{StabL-1prime} imposes no constraints on $u_{w(i),w(i+1)}$ in \eqref{agammanot0-1}; thus the partial product over $i\in \fN(w)\cap w^{-1}({\rm sp}(y))$ is identically $1$ if and only if $\fN(w) \cap w^{-1}({\rm sp}(y))=\emp$. For $i \in \fN(w)-w^{-1}({\rm sp}(y))$, the relevant formula in \eqref{StabL-1prime} reduces to the one term relation $u_{w(i),w(i+1)}y_{w(i+1)}=0$ yielding $u_{w(i),w(i+1)}=0$, if and only if $i$ satisfies $w(i+1)\in {\rm sp}(y)$ and $[w(i),w(i+1))\cap {\rm sp}(y)=\emp$; to see this, note that $w^{-1}$ is decreasing on ${\rm sp}(y)$ (Definition~\ref{OrderRevLemm}) for $y\in Y(w)$. If $i\in \fN(w)-w^{-1}({\rm sp}(y))$ and $w(i+1)\not \in {\rm sp}(y)$, then $u_{w(i),w(i+1)}$ does not occur in \eqref{StabL-1prime}, which means that we have no constraint on $u_{w(i),w(i+1)}$. Suppose $i\in \fN(w)-w^{-1}({\rm sp}(y))$, $w(i+1)\in {\rm sp}(y)$, and that $[w(i),w(i+1))\cap {\rm sp}(y)$ contains an element $w(j)$. From Definition~\ref{OrderRevLemm}, we have $i+1<j$; then the sum \eqref{StabL-1prime} involves at least two terms $u_{w(i),w(i+1)}y_{w(i+1)}$ and $u_{w(i),w(j)}y_{w(j)}$ with $y_{w(i+1)},\,y_{w(j)}\not=0$, which means that $u_{w(i),w(i+1)}$ can take an arbitrary value so that $\psi(u_{w(i),w(i+1)})\not=1$. Thus $a(\gamma)\not=0$ if and only if 
\begin{itemize}
\item[(i)] If $i\in \fN(w)$, then $w(i)\not\in {\rm sp}(y)$, $w(i+1)\in {\rm sp}(y)$ and $[w(i),w(i+1))\cap {\rm sp}(y)=\emp$. 
\item[(ii)] $w(i)>w(i+1)$ for all $i\in I^{0}$ such that $i,i+1\in I^{0}_{<m}$, and 
\item[(iii)] $w(m-1)>j$ for all $j\in {\rm{sp}}(y)$ such that $m-1<w^{-1}(j)$, when $m>1$.
\end{itemize}
Suppose $a(\sn(y)w)\not=0$ with $w\in \sS_n$ and $y\in Y(w)$. Set $Q:=w([1,m-1]_{\Z})$ with $m=w^{-1}(n)$, so that $I_{w}^{0}=Q'$. Then the condition (a) is obvious. Since $y\in Y(w)$, we have ${\rm sp}(y)\subset I_{w}^{0}=Q'$ and the condition (c) from Definition~\ref{OrderRevLemm}. The condition (b) follows from (ii). The condition (d) is equivalent to (i). Thus we obtain $w\in \sS_n(Q,y)$. The condition (iii) is equivalent to the containment $y\in {\mathsf Y}_Q$. Conversely, if $y\in {\mathsf Y}_Q$ and $w\in \sS_n(Q,y)$, we easily have (i), (ii) and (iii) above. The uniqueness of $Q$ for $(y,w)$ is clear; indeed, from the condition (a), we have $Q=\{w(i)\mid 1\leq i\leq w^{-1}(n)-1\}$. \end{proof}

Let $y\in {\mathsf Y}_Q$ and $w\in \sS_n(Q,y)$. We have $m(Q)=w^{-1}(n)$ from (a) in Definition~\ref{Def1}. If $\fN(w) \not=\emp$, we enumerate elements of $\fN(w)$ in an increasing sequence as  
$$
\fN(w)=\{j_1<\cdots<j_{h}\}, \quad h=\# \fN(w)
$$
and set $j_{0}:=m(Q)$. Note that $j_0<j_1$. We have 
\begin{align}
\text{$w(j_{\lambda})<w(j_{\lambda}+1)$ $(1\leq \nnu\leq h)$ with $w(j_{\nnu}+1)\in {\rm sp}(y)$, $w(j_\nnu)\not\in {\rm sp}(y)$}
 \label{20200213}
\end{align}
 from (d) in Definition~\ref{Def1}. We also have 
$$
w(j_{\lambda-1})>w(j_{\lambda}+1) \quad (1\leq \lambda\leq h).$$
Indeed, this is evident if $\lambda=1$ for $w(j_{0})=n$ and $w(j_1+1)\in {\rm sp}(y)\subset I^{0}$. Otherwise, since $w(j_{\lambda}+1)\in {\rm sp}(y)$ and $[w(j_{\lambda-1}),w(j_{\lambda-1}+1)) \cap {\rm sp}(y)=\emp$, we have $w(j_{\lambda-1})>w(j_{\lambda}+1)$ or $w(j_{\lambda}+1)\geq w(j_{\lambda-1}+1)$; we see that the latter possibility does not happen by Definition~\ref{OrderRevLemm} noting the relations $w(j_{\lambda}+1),\,w(j_{\lambda-1}+1)\in {\rm sp}(y)$ and $j_{\lambda-1}<j_{\lambda}$. 

For $1\leq \nnu \leq h$, set
\begin{align}
J_{\nnu}(w)&:=[j_{\nnu-1}+1,j_\nnu]_\Z, \notag 
\\ 
J_{\nnu}^*(y,w)&:=\{i\in J_\nnu(w)-w^{-1}({\rm sp}(y))\mid w(i)<w(j_\nnu+1)\,\}. \label{Jnnu*yw}
\end{align} 
Note that the disjoint intervals $J_{\nnu}(w)$ $(1\leq \nnu\leq h)$ and $[j_{h}+1,n]_\Z$ form a covering of $[m+1,n]_\Z$, and that $j_{\lambda} \in J_{\lambda}^{*}(y,w)$ from \eqref{20200213}. By using these sets, the summation range of \eqref{StabL-1prime} is described as in the next lemma. 

\begin{lem}\label{StabLii-Lem}
Let $w\in \sS_n(Q,y)$. 
\begin{itemize}
\item[(1)] 
If $\fN(w)=\emp$, then $w$ is decreasing both on $[1,,m(Q)-1]_\Z$ and on $[m(Q)+1,n]_\Z$. In particular, the range of the $j$-summation in \eqref{StabL-1prime} is empty. 
\item[(2)] Suppose $\fN(w)\not=\emp$. Then $w$ is decreasing on each of the intervals $J_\nnu(w)$ $(1\leq \nnu\leq h)$ as well as on $[j_h+1,n]_\Z$. Suppose $i\in J_{\nnu}(w)-w^{-1}({\rm sp}(y))$; then the range of $j$-summation in \eqref{StabL-1prime}  is $\{j_{\nnu}+1\}$ or $\emp$ according as $i\in J_{\nnu}^{*}(y,w)$ or $i\not\in J_{\nnu}^{*}(y,w)$, respectively. Suppose $i\in I_{>j_h}^0-w^{-1}({\rm sp}(y))$; then the range of $j$-summation in \eqref{StabL-1prime} is empty. \end{itemize}
\end{lem}
\begin{proof} (1) is obvious from (b) in Definition~\ref{Def1} and \eqref{def:fNw}. The first assertion of (2) is also obvious from \eqref{def:fNw}. To show the second assertion of (2), let $j\in I_{>i}^{0}$ with $i\in J_\nnu(w)-w^{-1}({\rm sp}(y))$, $w(i)<w(j)$ and $w(j)\in {\rm sp}(y)$. Then $i<j$ and $w(i)<w(j)$ means $j_{\nnu}+1\leq j$ and $w(j_\nnu)<w(j)$, because $w$ is decreasing on the interval $[j_{\nnu-1}+1,j_{\nnu}]_\Z$ which contains $i$. Since $w(j_{\nnu}+1),w(j)$ are both in ${\rm sp}(y)$, we have $w(j_{\nnu}+1)\geq w(j)$ by (c) in Definition~\ref{Def1}. Since $[w(j_{\nnu}),w(j_{\nnu}+1))\cap {\rm sp}(y)=\emp$ and $w(j)\in {\rm sp}(y)$, we must have $j=j_{\nnu}+1$. The last assertion of (2) holds true obviously because $w$ is decreasing on $I_{>j_h}^0$. 
\end{proof}

\subsection{Smoothed orbital integrals} \label{sec: DefOrbInt}
Let $F$ denote any $\Q$-algebra. For $Q\subset I^{0}$, set 
\begin{align*}
\sF_{Q}(F)&:=\{x\in F^{n-1}\mid {\rm{sp}}(x)\subset Q\,\}. 
\end{align*}
Note that $\sF_{\emp}(F)=\{0\}$ and $\sF_{I^0}(F)=F^{n-1}$. For any $w\in \sS_n$, define an element $w_0\in \sS_{n-1}$ associated with $w$ by setting 
\begin{align}
w_0(j):=\begin{cases} w(j) \quad &(1\leq j\leq m-1), \\
 w(j+1) \quad &(m\leq j \leq n-1)
\end{cases}
\label{weyl0-weyl}
\end{align} 
with $m:=w^{-1}(n)$. Note that $w(I-\{m\})=I^0$. The following lemma is evident. \begin{lem} \label{lem:w-w0}
The mapping from $\{(j,i)\in I^0\times I^0\mid j<i\}$ to $\{(j',i')\in (I-\{m\})\times (I-\{m\})\mid j'<i'\}$ defined as  
$$
(j,i)\,\mapsto\, (j',i')=\begin{cases} 
(j,i) \quad &(1\leq j<i<m), \\
(j+1,i+1) \quad &(m\leq j<i\leq n-1), \\
(j,i+1) \quad &(1\leq j<m,\,m\leq i\leq n-1)
\end{cases}
$$
is a bijection such that $(w_0(j),w_0(i))=(w(j'),w(i'))$.  
\end{lem}
In the rest of this subsection, we suppose $w$ and $Q$ are related by $Q=w([1,m-1]_{\Z})$, so that $Q'=w([m+1,n-1]_{\Z})$. Recall $\sS_{n-1}\subset \sS_n \subset \sG(\Q)$ from our convention (see \S\ref{sec:Vector}). 

Then 
\begin{align}
\sF_{Q}(F)=\{w_0\left[\begin{smallmatrix} \xi \\ 0_{n-m} \end{smallmatrix}\right]\mid \xi\in F^{m-1}\}, \quad \sF_{Q'}(F)=\{w_0\left[\begin{smallmatrix} 0_{m-1} \\ \xi' \end{smallmatrix}\right]\mid  \xi'\in F^{n-m}\}. 
\label{Mar11-1}
\end{align}
From this, it is evident that $u\sF_{Q}(F)\subset \sF_{Q}(F)$, ${}^t\sF_{Q'}(F)u\subset {}^t\sF_{Q'}(F)$ for $u\in w_{0}\sU_0(F)w_0^{-1}$, and that ${}^tx'\,x=0$ for $x'\in \sF_{Q'}(F)$ and $x\in \sF_{Q}(F)$. We introduce the notation
$$[u;x,x']_w:=w^{-1} \left[\begin{smallmatrix} u & x \\ {}^t x' & 1 \end{smallmatrix} \right]w, \quad (u, x,x') \in w_0\, \sU_0(F)\,w_0^{-1} \times \sF_Q(F)\times \sF_{Q'}(F).   
$$
Then the product law for these elements is given as 
\begin{align}
[u_1;x_1,x_1']_w\,[u_2;x_2,x_2']_w=[u_1u_2+x_1{}^t x_2';x_1+
u_1 x_2, x_2'+{}^t u_2x_1']_w.
 \label{ProdLaw}
\end{align}

\begin{lem} \label{lem:U-elements}
 We have $\sU(F)=\left\{[w_0Zw_0^{-1} ;x,x']_w\mid  (Z,x,x') \in \sU_0(F) \times \sF_Q(F)\times \sF_{Q'}(F)\right\}$.
\end{lem}
\begin{proof} Write any $g=(g_{ij})_{ij}\in \Mat_{n}(F)$ in the form $g=w^{-1} \left[\begin{smallmatrix} u & x \\ {}^t x' & t \end{smallmatrix} \right]w$ with $u\in \Mat_{n-1}(F)$, $t\in F$ and $x,x'\in F^{n-1}$. Then $g\in \sU(F)$ if and only if $g_{ij}=0$ for all $1\leq j<i\leq n$ and $g_{ii}=1$ for all $i\in I$. By \eqref{wAwinv}, the $(i,j)$-entry $g_{ij}$ for $j\leq i$ equals $u_{w(i)\,w(j)}$, $x_{w(i)}$, $x_{w(j)}$, or $t$ according to $i,j\in I-\{m\}$, $j=m<i$, $j=i<m$, or $i=j=m$, respectively. The diagonal condition $g_{ii}=1\,(i\in I)$ is equivalent to $u_{jj}=1\,(j\in I^0)$, $t=1$. Set $Z=w_0u w_0^{-1}$; then under the bijection $(j,i)\mapsto (j',i')$ in Lemma~\ref{lem:w-w0}, $u_{w(i'),w(j')}=0$ $(j'<i',\,j',i'\in I-\{m\})$ if and only if $Z_{ij}=0\,(j<i,\,i,j\in I^0)$. Since $Q=w([1,m-1]_{\Z})$ and $Q'=w([m+1,n]_{\Z})$, the condition $x_{w(i)}=0\,(m<i)$, $x'_{w(j)}=0\,(j<m)$ is equivalent to $x\in \sF_{Q}(F)$, $x'\in \sF_{Q'}(F)$.
\end{proof}

For $i,j\in I^0$ with $i<j$, let $\sU_0(i,j)$ denote the root subgroup of $\sG_0$ generated by the matrix element ${\rm E}_{ij}\in {\rm Lie}(\sG_0)$. For $y\in {\mathsf Y}_Q$ and $w\in \sS_{n}(Q,y)$, define a subgroup ${\bf V}_{y,w}^{0}$ of $\sU_0$ to be trivial if $\fN(w)=\emp$, and
\begin{align}
{\bf V}_{y,w}^{0}=
\prod_{\nnu=1}^{h}\prod_{i\in J_\nnu^*(y,w)}\sU_{0}(i-1,j_{\nnu}):=\biggl\{\prod_{\nnu=1}^{h}\prod_{i\in J_\nnu^*(y,w)} x_{i,\nnu}\,\biggm|\,x_{i,\nnu}\in \sU_{0}(i-1,j_\nnu)\,\biggr\}
 \label{Vyw0}
\end{align}
with 
$\fN(w)=\{j_1<\cdots<j_h\}$, $h=\# \fN(w)$ and $J_\nnu^*(y,w)$ being defined by \eqref{Jnnu*yw} if $\fN(w)\not=\emp$. Note that the product \eqref{Vyw0} is well-defined because any two root subgroups occurring in the product commute with each other due to the relation $[{\rm E}_{i-1\,j_\nnu},{\rm E}_{i'-1\, j_{\mu}}]=0$ ($i\in J_{\nnu}^{*}(y,w)$, $i'\in J_{\mu}^{*}(y,w)$), which is confirmed by $w(i'),w(i)\not\in {\rm sp}(y)$ and $w(j_{\mu}+1),w(j_{\nnu}+1)\in {\rm sp}(y)$. This remark also shows that ${\bf V}_{y,w}^{0}$ is an abelian group.

\begin{lem} \label{DOI-L1}
Let $\gamma=\sn(y)\,w\,(y \in {\mathsf Y}_Q,\,w\in \sS_n(Q,y))$ with $Q\subset I^{0}$. Set 
\begin{align}
 {\bf V}_{y,w}={\bf V}^{0}_{y,w}\,(w_{0}^{-1}\,\bar \sU_0\,w_0 \cap \sU_0),
 \label{DOI-L1-f0}
\end{align}
where $w_0\in \sS_{n-1}$ is the element associated with $w$ by \eqref{weyl0-weyl}. Then for any $\Q$-algebra $F$, the matrices  
$$
[w_0Zw_0^{-1} ;x,x']_w \quad (Z\in {\bf V}_{y,w}(F),\,x\in \sF_Q(F),\,x'\in \sF_{Q'}(F))
$$
form a complete set of representatives of $\sU_{[\gamma]}(F)\bsl \sU(F)$. 
\end{lem}
\begin{proof}
Suppose $\gamma=\sn(y)w$ with $y\in {\mathsf Y}_Q$ and $w\in \sS_n(Q,y)$. Then from Lemma~\ref{StabL}, 
\begin{align}
\sU_{[\gamma]}(F)&=\{[w_0Vw_0^{-1}; w_0Vw_0^{-1}y- y,0]_{w} \mid V\in w_0^{-1}{\bf V}'(F) w_0\,\}, \label{Mar11-0}
\end{align}
where ${\bf V}'$ denotes the subgroup of all $u\in w_0\sU_0w_0^{-1}\cap \sU_0$ satisfying \eqref{StabLii}. Suppose $h=\# \fN(w)>0$. Then from Lemma~\ref{StabLii-Lem}, the condition \eqref{StabLii} (or equivalently \eqref{StabL-1prime}) reduces to $u_{w(i),w(j_{\nnu}+1)}=0$ for $i\in J^{*}_\nnu(y,w)$, $1\leq \nnu\leq h$. Hence by \eqref{weyl0-weyl}
$${\bf V}'=\{u \in w_0 \sU_0 w_0^{-1}\cap \sU_0\mid u_{w_0(i-1),w_0(j_{\nnu})}=0\,(1\leq \nnu \leq h, \,{i \in J^{*}_\nnu(y,w)})\}.
$$
By Lemma~\ref{lemRR'prod}, we have $\sU_0=(\sU_0\cap w_0^{-1}\sU_0w_0)(\sU_0\cap w_0^{-1}\bar \sU_0w_0)$ and $\sU_0\cap w_0^{-1}\sU_0w_0=w_0^{-1}{\bf V}'w_0\,{\bf V}^{0}_{y,w}.$ Then by Lemma~\ref{lem:U-elements}, \eqref{Mar11-0}, using \eqref{ProdLaw}, we are done. Suppose $\fN(w)=\emp$. Then from Lemma~\ref{StabLii-Lem}, the condition \eqref{StabLii} becomes empty; thus ${\bf V}'=w_{0} \sU_0w_0^{-1}\cap \sU_0$ and we are done. 
\end{proof}

Recall the integrals \eqref{FWGeo-aJ}; for $y\in {\mathsf Y}_Q$ and $w\in \sS_n(Q,y)$, the first integral $a(\sn(y)w)=1$ by Lemma~\ref{agammanot0}, and the second integral, referred to as the smoothed orbital integral, has the following expression by Lemma~\ref{DOI-L1} and by \eqref{Mar11-1}: 
\begin{align}
\JJ_{f,\b}(\sn(y)w)=&\int_{{\bf V}_{y,w}(\A)\times \A^{m-1}\times \A^{n-m}}
\tilde\Phi_{f,\b}\left(\sn(y)\,w\,\bigl[w_0 Z w_0^{-1} ;w_0\left[\begin{smallmatrix} \xi \\ 0_{n-m} \end{smallmatrix}\right], w_0\left[\begin{smallmatrix} 0_{m-1} \\ \xi' \end{smallmatrix}\right]\bigr]_w\right) 
\label{GlobalJbetaQy}
\\
&\quad \times \theta_{w}^{\psi^{-1}}(Z,\xi,\xi')\,\d Z\,\d \xi\,\d \xi', 
 \notag
\end{align}
where
\begin{align*}
\theta_w^{\psi}(Z,\xi,\xi'):=&\psi_{\sU}\left(\bigl[w_0 Z w_0^{-1};w_0\left[\begin{smallmatrix} \xi \\ 0_{n-m} \end{smallmatrix}\right], w_0\left[\begin{smallmatrix} 0_{m-1} \\ \xi' \end{smallmatrix}\right] \bigr]_w
\right)=
\psi\Bigl(\,\sum_{i=1}^{m-2}Z_{i,i+1}
+\sum_{j=m+1}^{n-1}Z_{j-1,j}+\xi'_{m+1}+\xi_{m-1}\,\Bigr)
\end{align*}
for $(Z,\xi,\xi')\in {\bf V}_{y,w}(\A)\times \A^{m-1}\times \A^{n-m}$, and $\d Z$, $\d \xi$ and $\d \xi'$ are the Tamagawa measures on ${\bf V}_{y,w}(\A)$, $\A^{m-1}$ and on $\A^{n-m}$, respectively. Indeed, from the second claim of Lemma~\ref{lemRR'prod} and the proof of Lemma~\ref{DOI-L1}, we see that the natural gauge-form $\d Z \wedge \d \xi \wedge \d xi'$ on the product ${\bf V}_{y,w}\times {\rm Aff}^{m-1} \times {\rm Aff}^{n-m}$ corresponds to the gauge-form on $\sU_{[\gamma]}\bsl \sU$. Then, from \eqref{FWGeo-f1}, Lemmas~\ref{Doublecoset} and \ref{agammanot0}, we have
\begin{align}
\Wcal^{\psi}(\tilde{\bf \Phi}_{f,\b})=\sum_{Q\subset I^0} \sum_{y\in {\mathsf Y}_Q}\sum_{w\in \sS_n(Q,y)} \,\JJ_{f,\b}(\sn(y)w).
 \label{GeoSideP1}
\end{align}
Note that the infinite summation in $(y,w)$ for each $Q$ is absolutely convergent from \eqref{ABSCONV-GEOM}. 
%\begin{align*}
%\sum_{y\in {\mathsf Y}_Q}\sum_{w\in \sS_n(Q,y)} \,|\JJ_{f,\b}(\sn(y)w)|<\infty.%
%\end{align*}

\section{Local orbital integrals} \label{sec:LocalOrbInt}
We launch a detailed study of the integral \eqref{GlobalJbetaQy}. To motivate the analysis in this section, we first proceed formally ignoring the issue of convergence. By substituting \eqref{tildefB} to \eqref{GlobalJbetaQy}, and then by exchanging the order of the $(Z,\xi,\xi')$-integral and the $\nu$-integral, we get an integral of $\tilde f^{(\nu)}$ over ${\bf V}_{y,w}(\A)\times \A^{m(Q)-1} \times \A^{n-m(Q)}$, which, from \eqref{prodformulaGsect}, is expected to become a product of the local integrals of $f_v^{(\nu)}$ over ${\bf V}_{y,w}(\Q_v) \times \Q_v^{m(Q)-1} \times \Q_v^{n-m(Q)}$. By this formal argument, we get \eqref{ErrorL1-f1} for regular orbital integrals (for this notion see below) and \eqref{MaintermP1-f} for the singular orbital integral. To be rigorous, the absolute convergence of the local integrals should be settled in the first place. Thus, we are naturally led to the following definition. Let $v$ be a place of $\Q$. Fix $Q\subset I^0$, $y\in {\mathsf Y}_{Q}$, and $w\in \sS_n(Q,y)$. Set $m=m(Q)$. Let $w_0\in \sS_{n-1}$ be the element associated with $w$ (see \eqref{weyl0-weyl}). Let $\theta_{w,v}^{\psi}(Z,\xi,\xi')$ $((Z,x,x')\in {\bf V}_{y,w}(\Q_v)\times \Q_v^{m-1} \times \Q_v^{n-m})$ denote the $v$-component of $\theta_{w}^{\psi}$. We endow the spaces ${\bf V}_{y,w}(\Q_v)$, $\Q_v^{m-1}$ and $\Q_v^{n-m}$ with the local Tamagawa measures $\d Z$, $\d \xi$ and $\d \xi'$.
%such that the product $\d Z_v\,\d x_v\,\d x_v'$ corresponds to the quotient measure on ${\bf U}_{[\sn(y)w]}(\Q_v)\bsl \sU(\Q_v)$ through the bijection ${\bf V}_{y,w}(\Q_v)\times \sF_{Q}(\Q_v)\times \sF_{Q'}(\Q_v)\ni (Z,x,x')\mapsto [w_0 Z w_0^{-1};x,x']_{w} \in \sU_{[\sn(y)w]}(\Q_v)\bsl \sU(\Q_v)$ from Lemma~\ref{DOI-L1}. 
Note that, if $v=p<\infty$, then ${\bf V}_{y,w}(\Z_p) \times \Z_p^{m-1} \times \Z_p^{n-m}$ has the measure $1$. Define
{\small\begin{align}
&\JJ_{f_v}^{(\nu)}(y,w)
=\int_{{\bf V}_{y,w}(\Q_v) \times \Q_v^{m-1} \times \Q_v^{n-m}}
 \tilde{f}^{(\nu)}_v\left(\sn(y)\,w\,\bigl[w_0 Z w_0^{-1};w_0\left[\begin{smallmatrix} \xi \\ 0_{n-m} \end{smallmatrix}\right], w_0\left[\begin{smallmatrix} 0_{m-1} \\ \xi' \end{smallmatrix}\right] \bigr]_w
\right)
\theta_{w,v}^{\psi^{-1}}(Z,\xi,\xi')\,\d Z\,\d \xi\,\d \xi', 
\label{LocalJnuQy}
\\
&{\mathbb {MJ}}(|\tilde f_v^{(\nu)}|;y,w)=
\int_{{\bf V}_{y,w}(\Q_v) \times \Q_v^{m-1} \times \Q_v^{n-m}}\biggl|\tilde f_v^{(\nu)}\left(\sn(y)w\,\bigl[w_0 Z w_0^{-1};w_0\left[\begin{smallmatrix} \xi \\ 0_{n-m} \end{smallmatrix}\right], w_0\left[\begin{smallmatrix} 0_{m-1} \\ \xi' \end{smallmatrix}\right] \bigr]_w\right)
\biggr|\,\d Z\,\d \xi\,\d \xi' 
\label{LocalMJnuQy}
\end{align}}for $\nu \in \ft_{0,\C}^*$ with $\tilde f_v^{(\nu)}$ being defined as in \S\ref{NonArchTestFtn} for a left $\iota(\bK_{\sG_0,v})$-invariant function $f_v\in C_{\rm{c}}^\infty(\sG(\Q_v))$. The integral \eqref{LocalJnuQy} with $Q\not=I^0$ (resp. $Q=I^0)$ is referred to as a regular local orbital integral (resp. a singular local orbital integral). The singular case will be treated in \S\ref{GSST}. The determination of an absolute convergence range of the regular local orbital integral \eqref{LocalJnuQy} is the main objective of this section. Indeed, we shall obtain the following.

\begin{prop} \label{JJ-L1} 
Let $\Q\subsetneq I^{0}$. If $\nu \in\fJ((\ft_{0}^*)^{++})$, the integral \eqref{LocalJnuQy} is absolutely convergent defining a holomorphic function on $\fJ((\ft_{0}^*)^{++})$. There exists a constant $c(f_v)>0$ depending only on the support of $f_v$ such that if $y=(y_{j})_{j=1}^{n-1}\in {\mathsf Y}_Q-\{0\}$, $|y_{q'}|_v >c(f_v)$ with $q'$ the minimal element of the set ${\rm sp}(y):=\{j\in [1,n-1]_{\Z}\mid y_j\not=0\}$, then $\JJ_{f_v}^{(\nu)}(y,w)=0$ for all $w\in \sS_n(Q,y)$ and for all $\nu\in \fJ((\ft_{0}^*)^{++})$.  
\end{prop}
Since \eqref{LocalMJnuQy} is an obvious majorizer of \eqref{LocalJnuQy}, this is a corollary to the following proposition, which yields an estimation of the local orbital integrals \eqref{LocalJnuQy}. 

\begin{prop} \label{LocalABSCONVJJ}
Let $\Ucal_v\subset \sG_0(\Q_v)$ be a compact subset. Then there exist two constants $C(\Ucal_v)>0$ and $c(\Ucal_v)>0$, and a compact set $\Vcal_{0,v}\subset \sG_0(\Q_v)$ depending only on $\Ucal_v$ such that for any left $\iota(\bK_{\sG_0,v})$-invariant function $f_v\in C_{\rm c}^{\infty}(\sG(\Q_v))$ with ${\rm supp}(f_v)\subset \Ucal_v$, and for any $\nu\in \fJ((\ft_{0}^*)^{++}))$, $Q\subsetneq I^{0}$, $y=(y_j)_{j=1}^{n-1}\in {\mathsf Y}_Q$, and $w\in \sS_n(Q,y)$,
{\allowdisplaybreaks\begin{align*}
&{\mathbb {MJ}}(|\tilde f_v^{(\nu)}|;y,w)
\leq C(\Ucal_v)\,\|\tilde f_v\|_{\infty}\,\tilde C_v(\nu,y,w)\times\prod_{\substack{1\leq i<j \leq n-1 \\ w_0^{-1}(i)>w_0^{-1}(j)}}
\tfrac{\zeta_v(\Re\,\nu_i-\Re\,\nu_j)}{\zeta_v(\Re\,\nu_i-\Re\,\nu_j+1)}
\times \int_{h \in \Vcal_{0,v}}\sf_{\sG_0,v}^{(w_0^{-1}\Re \nu)}(h)\,\d h, \end{align*}}where $\tilde C_v(\nu,y,w):=1$ if $y=0$, and $\tilde C_v(\nu,y,w)>0$ for $y\not=0$ is a constant  with the following properties:\begin{itemize}
\item $\tilde C_v(\nu,y,w)$ depends on $\nu$ and $\Ucal_v$ only through the numbers $\Re\,\nu_j$ and $c(\Ucal_v)$.
\item $\tilde C_{v}(\nu,y,w)=0$ if $|y_{q'}|_v> c(\Ucal_v)$, where $q'$ is the minimal element of ${\rm sp}(y)$, 
\item $\tilde C_v(\nu,y,w)=1$ for all $\nu \in \fI((\ft_{0}^*)^{++}))$ if $v=p<\infty$, $y_{j}\in \Z_p^\times\,(j\in {\rm sp}(y))$ and $c(\cU_p)=1$. 
\end{itemize}
If $v=p<\infty$ and $\Ucal_v=\bK_p$, then we can take $\Vcal_{0,v}=\bK_{\sG_0,p}$ and $c(\Ucal_v)=C(\Ucal_{v})=1$. 
\end{prop}The inequality in Proposition \ref{LocalABSCONVJJ} is designed only for the specific purpose to gain the convergence of the adelic orbital integrals in \S\ref{sec:global terms}. 
After a preliminary, the proof of this proposition will be given in \S\ref{sec:Proofof}. In \S\ref{sec:VanishLOInt}, we examine the regular local orbital integral $\JJ_{f_p}^{(\nu)}(y,w)$ $(\nu \in \fJ((\ft_{0}^*)^{++}))$ when $f_p$ is the characteristic function of the group $\bK_1(N\Z_p)$ (see \eqref{LocalK1fa}) with $|N|_p<1$ to show its vanishing for a certain class of $y$'s (see Proposition~\ref{JJL-5}).

\subsection{Preliminary analysis} \label{sec:ABSconv}
In this subsection, we fix $Q\subsetneq I^{0}$, $y\in {\mathsf Y}_Q$ and $w\in \sS_n(Q,y)$. Let $w_0\in \sS_{n-1}$ be the element associated with $w$ defined by \eqref{weyl0-weyl}. Then ${\rm sp}(y)$ is a subset of $Q':=I^0-Q$. Set $m=m(Q):=\#Q+1$ and 
$$
Q_{+}'={\rm sp}(y), \quad Q_{-}'=Q'-{\rm sp}(y). 
$$ From Definition~\ref{Def1}, we have $Q=w_0([1,m-1]_{\Z})$, $m=w(n)$ and $Q'=w_0([m,n-1]_{\Z})$. 
%As before, $x$ (resp. $x'$) denotes a general element of $\sF_{Q}(\Q_v)$ (resp. $\sF_{Q'}(\Q_v)$). From $Q'=Q_{-}' \cup Q_{+}'$, we have the direct sum decomposition $\sF_{Q'}(\Q_v)=\sF_{Q'_{-}}(\Q_v)\oplus \sF_{Q_{+}'}(\Q_v)$.
%; let $x'_{\pm}$ denote the projection of $x'\in \sF_{Q'}(\Q_v)$ to $\sF_{Q_{\pm}'}(\Q_v)$; thus $x'=x_{-}'+x_{+}'$, $x_{-}'\in \sF_{Q_-'}(\Q_v)$, and $x_{+}'\in \sF_{Q_{+}'}(\Q_v)$. 
Hereinafter we label the entries of vectors from $\Q_v^{n-m}$ by the set $[m,n-1]_{\Z}$ and let $\eta=(\eta_j)_{j=m}^{n-1} \in \Q^{n-m}$ be the element such that 
\begin{align}
w_0^{-1}y =\left[\begin{smallmatrix} 0_{m-1} \\ \eta \end{smallmatrix}\right].
 \label{w0Vectors}
\end{align} From definitions, for any $j\in [m,n-1]_{\Z}$ we have $\eta_{j}=y_{w(j+1)}$, thus 
\begin{align}
\text{$w(j+1)\,(=w_0(j))\in {\rm sp}(y)$ if and only if $\eta_{j}\not=0$.}
\label{eta-0cond}
\end{align}
%Let $\xi_{+}'\in \Q_v^{n-m}$ (resp. $\xi_{-}'\in \Q_v^{n-m})$ be the vector which corresponds to $x_{+}'\in \sF_{Q_+'}(\Q_v)$ (resp. $x_{-}'\in \sF_{Q_-'}(\Q_v)$) by the second relation of \eqref{w0Vectors}, and define $\Q_v$-subspaces $\Xi_{Q_{+}',v}^{w_0}$ (resp. $\Xi_{Q_{-}',v}^{w_0}$) of $\Q_v^{n-m}$ as the image of $x_{+}'\mapsto \xi_{+}'$ (resp. $x_{-}'\mapsto \xi_{-}'$), i.e., 
%We have $\sF_{Q}(\Q_v)=\{x=w_0 \left[\begin{smallmatrix} \xi \\ 0_{n-m}\end{smallmatrix}\right] \mid \xi \in \Q_{v}^{m-1}\}$. 

For $R\in \{Q_{+}, Q_{-}'\}$, set 
\begin{align*}
\Xi_{R,v}^{w_0}=\{\xi'=(\xi'_j)_{j=m}^{n-1} \in \Q_v^{n-m} \mid \xi'_j=0\,(j \in w_0^{-1}(R)\},
\end{align*}
so that $\sF_{R}(\Q_v)=\{x'=w_0\left[\begin{smallmatrix} 0_{m-1} \\ \xi'\end{smallmatrix}\right] \mid \xi' \in \Xi_{R, v}^{w_0}\}$ and $\Q_v^{n-m}=\Xi_{Q_{+}',v}^{w_0} \oplus \Xi_{Q_-',v}^{w_0}$ by \eqref{Mar11-1}. From \eqref{DOI-L1-f0}, we can write a general element $Z\in {\bf V}_{y,w}(\Q_v)$ as $Z=Z_0\,Z_1$ with $Z_0 \in {\bf V}_{y,w}^0(\Q_v)$ and $Z_1\in w_0^{-1}\bar \sU_0(\Q_v)w_0 \cap \sU_0(\Q_v)$. Since $w\in \sS_n(Q,y)$, from Definition \ref{Def1}, a general form of $Z_1$ is   
\begin{align}
 Z_1=\left[\begin{smallmatrix}A& B \\ 0 &  D\end{smallmatrix}\right] \, \in w_0^{-1}\bar \sU_0(\Q_v)w_0 \cap \sU_0(\Q_v) 
 \label{fZ_1-ABD}
\end{align}
 with $A\in \sU^{(m-1)}(\Q_v)$, $B\in {\bf N}_{w_0}(\Q_v)$, and $D\in \sU_{w_0}^{(n-m)}(\Q_v)$, where
$\sU^{(l)}(\Q_v)$ denotes the group of upper-triangular unipotent matrices of degree $l$ and 
\begin{align*}
{\bf N}_{w_0}(\Q_v)&=\{B=(B_{ij})_{\substack{1\leq i\leq m-1 \\ m\leq j \leq n-1}}\in {\bf M}_{m-1,n-m}(\Q_v)\mid B_{ij}=0\,(w_0(i)<w_0(j))\,\}, \\
\sU_{w_0}^{(n-m)}(\Q_v)&=\{D=(D_{ij})_{\substack {m\leq i\leq n-1 \\ m\leq j\leq n-1}}\in \sU^{(n-m)}(\Q_v)\mid D_{ij}=0\,(i<j,\,w_{0}(i)<w_0(j))\}. 
\end{align*}
If $y=0$, then ${\bf V}_{y,w}^{0}=\{1_{n-1}\}$. If $y\not=0$, then from \eqref{Vyw0}, 
\begin{align}
Z_0=\left[\begin{smallmatrix}1_{m-1} & 0 \\ 0 &  D_0\end{smallmatrix}\right],
\quad D_0 \in {\bf V}^{(n-m)}_{y,w}(\Q_v),
\label{fZ0-D0}
\end{align}
where ${\bf V}_{y,w}^{(n-m)}(\Q_v)\subset \sU^{(n-m)}(\Q_v)$ is the set all those matrices $(u_{i,j})_{m\leq ij\leq n-1}\in \sU^{(n-m)}(\Q_v)$ such that only the possibly non-zero off-diagonal entries are $u_{i,j_{\nnu}}$ with $m\leq i \leq n-1,\,1\leq \nnu \leq h $ such that $i+1\in J_{\nnu}^*(y,w)$. 

With the notation above, we have
\begin{align}
\sn(y)\,w\,
\bigl[w_0 Z w_0^{-1};
w_0\left[\begin{smallmatrix} \xi \\ 0_{n-m} \end{smallmatrix}\right], w_0\left[\begin{smallmatrix} 0_{m-1} \\ \xi' \end{smallmatrix}\right] \bigr]_w
&=\left[\begin{smallmatrix}w_0 & 0 \\ 0 & 1 \end{smallmatrix}\right]
\left[\begin{matrix} A & B & \xi \\ 
0 & D_0 D+\eta{}^t\xi' & \eta \\
0 & {}^t \xi' & 1 \end{matrix}\right]\left[\begin{smallmatrix}w_0 & 0 \\ 0 & 1 \end{smallmatrix}\right]^{-1}w.
 \label{matrix0}
\end{align}
for $Z\in {\bf V}_{y,w}(\Q_v)$, $\xi \in \Q_v^{m-1}$, $\xi'\in \Q_{v}^{n-m}$. Write $\xi'=\xi_{+}'+\xi_{-}'\,(\xi_{\pm}'\in \Xi_{Q_{\pm}',v}^{w_0})$. Since $\eta \in \Xi_{Q'_{+}}^{w_0}$, we have $\eta{}^t\xi_{-}'=0$. Noting this, we see that the right-hand side of \eqref{matrix0} becomes
\begin{align}
\left[\begin{smallmatrix}w_0 & 0 \\ 0 & 1 \end{smallmatrix}\right]
\left[\begin{matrix} A & B-\xi{}^t\xi_{+}' & \xi \\ 
0 & D_0 D & \eta \\
0 & {}^t \xi_{-}' & 1 \end{matrix}\right]
\left[\begin{matrix} 1_{m-1} & 0  & 0 \\
0 & 1_{n-m} & 0 \\
0 & {}^t \xi_{+}' & 1 
\end{matrix} \right]
\left[\begin{smallmatrix}w_0 & 0 \\ 0 & 1 \end{smallmatrix}\right]^{-1}w.
 \label{matrix1}
\end{align}
At this point we need the Iwasawa decomposition, 
\begin{align}
\left[\begin{smallmatrix} 1_{n-m} & 0 \\ 
{}^t \xi_{+}' & 1 \end{smallmatrix} \right]
=\left[\begin{smallmatrix} \Upsilon & \ell \\ 
 0 & 1 \end{smallmatrix} \right]\,\left[\begin{smallmatrix} \alpha & 0 \\ 
 0 & t \end{smallmatrix} \right] \kappa, \quad
\kappa=\left[\begin{smallmatrix} \kappa_{11} & \kappa_{12} \\ 
 \kappa_{21} & \kappa_{22} \end{smallmatrix} \right]
 \label{IwasawaDecxi}
\end{align}
with $\Upsilon \in \sU^{(n-m)}(\Q_v)$, $\ell \in \Q_v^{n-m}$, $\alpha=\diag(\alpha_1,\dots,\alpha_{n-m})$, $t\in \Q_v^\times$ and $\kappa\in \bK^{(n-m+1)}_v$, which is block decomposed in such a way that $\kappa_{11}\in {\bf M}_{n-m}(\Q_v)$ and $\kappa_{22}\in \Q_v$. Putting this, we have that the matrix \eqref{matrix1} becomes
\begin{align}
\left[\begin{smallmatrix}w_0 & 0 \\ 0 & 1 \end{smallmatrix}\right]
\left[\begin{matrix} A & \tilde B \Upsilon & \tilde B \ell+\xi \\ 
0 & D_0 D\Upsilon & D_0D\ell+\eta \\
0 & {}^t \xi_{-}'\Upsilon  & 1+{}^t\xi'_{-}\,\ell 
 \end{matrix}\right]
\left[\begin{matrix} 1_{m-1} & 0  & 0 \\
0 & \alpha & 0 \\
0 & 0 & t 
\end{matrix} \right]
\left[\begin{matrix} 1_{m-1} & 0  & 0 \\
0 & \kappa_{11} & \kappa_{12} \\
0 & \kappa_{21} & \kappa_{22} 
\end{matrix} \right]
\left[\begin{smallmatrix}w_0 & 0 \\ 0 & 1 \end{smallmatrix}\right]^{-1}w,
 \label{matrix2}
\end{align}
where 
$$
\tilde B=B-\xi\,{}^t\xi_{+}'. 
$$
We claim that $\xi\,{}^t\xi_{+}'\in {\bf N}_{w_0}(\Q_v)$ and that $B\mapsto \tilde B$ is a measure preserving bijection from ${\bf N}_{w_0}(\Q_v)$ onto itself. When $Q=\emp$, this is obvious because $\xi=0$. Suppose $Q\not=\emp$ and let $i \in[1,m-1]_{\Z}$ and $j \in [m, n-1]_{\Z}$ with $w_0(i)<w_0(j)$. Then $w_0(i) \in Q$. Since $y\in {\mathsf Y}_{Q}$, any element of ${\rm sp}(y)$ is smaller than $\min(Q)$ (see \eqref{sfYQ}). Thus $w_0(i)\in Q$ and $w_0(i)<w_0(j)$ implies $w_0(j)\not\in {\rm sp}(y)=Q_{+}'$, eqivalently $j\in w_{0}^{-1}(Q_{-}')$. Hence $(\xi,{}^t\xi_{+}')_{ij}=\xi_{i}\,(\xi_{+}')_{j}=\xi_{i}\times 0=0$. Thus $\xi\,{}^t\xi_{+}'\in {\bf N}_{w_0}(\Q_v)$. Note that ${\bf N}_{w_0}(\Q_v)$ is a linear subspace of $\Mat_{m-1.n-m}(\Q_v)$. 

We need an explicit formula of the Iwasawa data $(\Upsilon,\ell,\alpha,t)$ in \eqref{IwasawaDecxi}, which is derived from the next two lemmas. 
%As a temporally convention for these lemmas, we use the set $[1,n-m]_{\Z}$ instead of $[m,n-1]_{\Z}$ to label the entries of a vector $\Q_v^{n-m}$, i.e., for $i\in [1,n-m]_\Z$, the $i$-the entry of $(\xi'_j)_{j=m}^{n-1}\in \Q_v^{n-m}$ is $z_i=\xi'_{i+m-1}$. 
To state the first lemma, which concerns the case of $p$-adic field, we need additional notations. For $a\in \N_0$ and $\iiota=(\iiota(\nnu))_{\nnu=0}^{a}$ a strictly decreasing sequence of elements of $w_0^{-1}(Q_+')$, define $\Xi_{Q_{+}',p}^{w_0}(\iiota)$ to be the set of all those points $\xi_{+}'=(\xi_i)_{i=m}^{n-1}\in \Xi_{Q_{+}',p}^{w_0}$ satisfying the inequalities 
\begin{align*}
&|\xi_{\iiota({\nnu})}|_{p}=\max(|\xi_m|_{p},\dots,|\xi_{\iiota({\nnu-1})-1}|_p)>1, \quad (\nnu\in [0,a]_\Z), \\
&\max(|\xi_{m}|_p,\cdots,|z_{\iiota(a)-1}|_p)\leq 1,
\end{align*} 
where $\iiota(-1):=n$. Then the set $\Xi_{Q_{+}',p}^{w_0}-\Z_p^{n-m}$ is a union of subsets $\Xi_{Q_+',p}^{w_0}(\iiota)$ for all $\iiota$'s as above. Note that this is not a disjoint decomposition.
\begin{lem} \label{ExplicitIwasawaDec}
Let $a\in \N_0$ and $\iiota=(\iiota(\nnu))_{\nnu=0}^{a}$ as above. Let $\xi_{+}'=(\xi_i)_{i=m}^{n-1}\in \Xi_{Q_{+}',p}^{w_0}(\iiota)$. 
\begin{itemize}
\item[(i)] We have \eqref{IwasawaDecxi} with $\kappa\in \bK^{(n-m+1)}_p$, $\alpha=\diag(\alpha_i\mid i\in [1,n-m]_\Z)$, $t\in \Q_p^\times$, $\Upsilon \in \sU^{(n-m)}(\Q_p)$ and $\ell \in \Q_p^{n-m}$, where 
\begin{align*}
\alpha_{i}&=\begin{cases}1 \quad (\text{$i\not=\iiota(\lambda)$ for all $\lambda\in [0,a]_\Z$}), \\
 -\xi_{\iiota(a)}^{-1}, \quad (i=\iiota(a)), \\
 -\xi_{\iiota({\nnu})}\xi_{\iiota({\nnu-1})}^{-1}, \quad (i=\iiota{(\nnu-1)},\nnu\in [1,a]_\Z)\end{cases}, \qquad t=\xi_{\iiota({0})}, 
\end{align*}
and for $i\in [1,n-m]_\Z$ the $i$-th row of $\left[\begin{smallmatrix} \Upsilon & \ell \\ 0 & 1 \end{smallmatrix}\right]$ equals 
\begin{align}
&(\underbrace{0,\dots,0}_{i-1}, 1, \underbrace{0,\dots,0}_{n-m-i+1}), 
\quad{\text{if $i$ does not occur in the sequence $\iiota$}}, 
 \label{EIWD-1}
\\
& \biggl(\underbrace{0,\dots,0}_{\iiota(\nnu)-1}, 1 , -\frac{\xi_{\iiota({\nnu})+1}}{\xi_{\iiota(\nnu)}},\dots, -\frac{\xi_{\iiota(\nnu-1)}}{\xi_{\iiota(\nnu)}}, \underbrace{0,\dots,0}_{n-m-\iiota(\nnu-1)+1} \biggr) \quad {\text{if $i=\iiota(\nnu)$ with $\nnu\in [0,a]_\Z$}}, 
\end{align}
where we set $\xi_{\iiota(-1)}=\xi_{n}:=-1$.
\item[(ii)] We have $\ell \in \xi_{\iiota(0)}^{-1}\,\Z_p^{n-m}$ and $\Upsilon \alpha \in \Mat_{n-m}(\Z_p)$. 
Moreover, 
\begin{align}
{}^t\xi_{-}'\ell=0, \quad {}^t\xi_{-}'\,\Upsilon \alpha={}^t\xi_{-}' \quad \text{for all $\xi_{-}'\in \Xi_{Q_{-}',p}^{w_0}$}. \label{ExplicitIwasawaDec-0}
\end{align}
\end{itemize}
\end{lem}
\begin{proof} In the proof, we set $z_j:=\xi_{j+m-1}\,(j\in [1,n-m]_{\Z})$ and $\iita(\lambda):=\iiota(\lambda)-m+1$ for $\lambda\in [0,a]_{\Z}$, so that $z_{\iita(\lambda)}=\xi_{\iiota(\lambda)}$.  

 (i) For $1\leq i<j\leq n-m+1$ and $x\in \Q_p^\times$, let $\sa(i,j;x)$ denote the diagonal matrix in $\GL_{n-m+1}(\Q_p)$ such that $\sa(i,j;x)_{ii}=-x$, $\sa(i,j;x)_{jj}=x^{-1}$ and $\sa(i,j;x)_{hh}=1$ for $h\not=i,j$, and set $\su(i,j;x)=1_{n-m+1}+x{\rm E}_{ij}$, where ${\rm E}_{ij}\in {\bf M}_{n-m+1}(\Q_p)$ is the matrix element with $1$ on the $(i,j)$-entry. Starting with the matrix $\bar n=\left[\begin{smallmatrix} 1_{n-m} & 0 \\ 
{}^t \xi_{+}'  & 1 \end{smallmatrix} \right]$, we apply the following operations for matrices in this order: 
\begin{itemize}
\item Multiply $\su(\iita(0),n-m+1;-z_{\iita(0)}^{-1})$ from the left, so that the $-z_{\iita(0)}^{-1}$-times the $(n-m+1)$-th row is added to the $\iita(0)$-th row. 
\item Multiply $\sa(\iita(0),n-m+1;z_{\iita(0)})$ from the left, so that the $\iita(0)$-row is multiplied with $-z_{\iita(0)}$, and the $(n-m+1)$-th row is multiplied with $z_{\iita(0)}^{-1}$.  
\item Multiply the matrices $\su(\iita(0),i;-z_{i})\,(\iita(0)<i\leq n-m)$ from the left, so that the $-z_{i}$ times the $i$-th row is added  to the $\iita(0)$-th row.  
\end{itemize}
Set $\bar n_0=\prod_{i=\iita(0)+1}^{n-m}\su(\iita(0),i;-z_{i})\, \sa(\iita(0),n-m+1;z_{\iita(0)})\,\su(\iita(0),n-m+1;-z_{\iita(0)}^{-1})\,\bar n$. Then the $j$-th row of $\bar n_0$ is $(\bar n_0)_{j}=(\delta_{ij})_{1\leq i\leq n-m+1}$ if $j\not\in \{\iita(0),n-m+1\}$, 
\begin{align*}
(\bar n_0)_{j}=&(z_1,\dots,z_{\iita(0)-1},\underbrace{0,\dots,0}_{n-m-\iita(0)+1},1) \quad (j=\iita(0)), \\
(\bar n_0)_{j}=&\left(\tfrac{z_{1}}{z_{\iita(0)}},\dots,\tfrac{z_{n-m}}{z_{\iita(0)}},\tfrac{1}{z_{\iita(0)}}\right) \quad (j=n-m+1). 
\end{align*}
Note that $(\bar n_0)_{n-m+1}$ is a primitive vector of $\Z_p^{n-m+1}$. We proceed inductively: for $1\leq \nnu\leq a$, once a matrix $\bar n_{\nnu-1}\in \GL_{n-m+1}(\Q_p)$ is defined, we apply the following operations to $\bar n_{\nnu-1}$ in this order to define $\bar n_{\nnu}$: 
\begin{itemize}
\item Multiply $\su(\iita(\nnu),\iita(\nnu-1);-z_{\iita(\nnu)}^{-1})$ from the left, so that the $-z_{\iita(\nnu)}^{-1}$-times the $\iita(\nnu-1)$-th row is added to the $\iita(\nnu)$-th row. 
\item Multiply $\sa(\iita(\nnu),\iita(\nnu-1);z_{\iita(\nnu)})$ from the left, so that the $\iita(\nnu)$-row is multiplied with $-z_{\iita(\nnu)}$, and the $\iita(\nnu-1)$-th row is multiplied with $z_{\iita(\nnu)}^{-1}$.  
\item Multiply the matrices $\su(\iita(\nnu),i;-z_{i})\,(\iita(\nnu)<i\leq \iita(\nnu-1)-1)$ from the left, so that the $-z_{i}$ times the $i$-th row is added to the $\iita(\nnu)$-th row.  
\end{itemize}
We have $\bar n_{\nnu}=\prod_{i=\iita(\nnu)+1}^{\iita(\nnu-1)-1}\su(\iita(\nnu),i;-z_{i})\, \sa(\iita(\nnu),\iita(\nnu-1);z_{\iita(\nnu)})\,\su(\iita(\nnu),\iita(\nnu-1);-z_{\iita(\nnu)}^{-1})\,\bar n_{\nnu-1}$. The $j$-th row $(\bar n_{\nnu})_{j}$ of $\bar n_{\nnu}$ is $(\delta_{ij})_{1\leq i\leq n-m}$ if $j\not\in \{\iita(\nnu),\dots,\iita(0),n-m+1\}$, and 
\begin{align*}
(\bar n)_{\iota(\nnu)}&=\left(z_{1},\dots,z_{\iita(\nnu)-1}, 0,\dots,0,1\right), \\
(\bar n)_{\iita(\mu)}&=\left(\tfrac{z_{1}}{z_{\iita(\mu+1)}}, \dots, \tfrac{z_{\iita(\mu)-1}}{z_{\iita(\mu+1)}}, {0,\dots,0},\tfrac{1}{z_{\iita(\mu+1)}}\right), \quad (-1\leq \mu<\nu).
\end{align*}
Note that $j$-th row of $\bar n_\nnu$ is a primitive vectors of $\Z_p^{n-m+1}$ for $j>\iita(\nnu)$. In the end, we obtain an element $\bar n_{\iita(a)} \in \GL_{n-m+1}(\Z_p)$. Then we have the relation \eqref{IwasawaDecxi} with $\kappa=\bar n_{\iita(a)}$,  
{\allowdisplaybreaks\begin{align*}
 \left[\begin{smallmatrix} \Upsilon & \ell \\ 
 0 & 1 \end{smallmatrix} \right]&=\biggl\{\su(\iita(0),n-m+1;z_{\iita(0)}^{-1})\, \prod_{i=\iita(0)+1}^{n-m}\su(\iita(0),i;-z_{i}z_{\iita(0)}^{-1})\biggr\}\times \cdots \\
&\quad \times  
\biggl\{\su(\iita(\nnu),\iita(\nnu-1);-z_{\iita(\nnu-1)}z_{\iita(\nnu)}^{-1})\,
\prod_{i=\iita(\nnu)+1}^{\iita(\nnu-1)-1}\su(\iita(\nnu),i;-z_{i}z_{\iita(\nnu)}^{-1})\biggr\} \times \cdots \\
&\quad \times \biggl\{\su(\iita(a),\iita(a-1);-z_{\iita(a-1)}z_{\iita(a)}^{-1})\, 
\prod_{i=\iita(a)+1}^{\iita(a-1)-1}\su(\iita(a),i;-z_{i}z_{\iita(a)}^{-1})\biggr\}, \\
\left[\begin{smallmatrix} \alpha & 0 \\ 
 0 & t \end{smallmatrix} \right]&=\prod_{\nnu=0}^{a}\sa(\iita(\nnu),\iita(\nnu-1);z_{\iita(\nnu)}^{-1}).
\end{align*}}
(ii) From (i), we have that $\ell={}^t(z_{\iita(0)}^{-1}\delta_{i,\iita(0)})_{i=1}^{n-m}$, which belongs to $z_{\iita(0)}^{-1}\Z_p^{n-m}$, and that the $i$-th row of $\Upsilon\alpha$ is equal to $(\delta_{ij})_{j=1}^{n-m}$ if $i$ does not occur in the sequence $\iita$ and to 
$$\biggl(\underbrace{0,\dots,0}_{\iita(\nnu)-1}, -\frac{z_{\iita(\nnu+1)}}{z_{\iita(\nnu)}}, \Bigl(-\frac{z_{j}}{z_{\iita(\nnu)}}\Bigr)_{j=\iita(\nnu)+1}^{\iita(\nnu-1)-1}, 1, \underbrace{0,\dots,0}_{n-m-\iita(\nnu-1)} \biggr) \quad {\text{if $i=\iita(\nnu)$ with $0\leq \nnu\leq a$}},
$$
which belongs to $\Z_p^{n-m}$ because $|z_{j}|_p\leq |z_{\iita(\nnu)}|_{p}$ $(1\leq j <\iita(\nnu-1)$) and $\iita$ is strictly decreasing. Let $\xi_{-}'=(z'_j)_{j=1}^{n-m} \in \Xi_{Q_{-}',v}^{w_0}$; then the $\iita(\nnu)$-th entry $z_{\iita(\nnu)}$ is zero for all $\nnu$ because $\iita(\nnu)+m-1\in w_0^{-1}Q_{+}'$. From this and from the explicit description of $\ell$ and the row vectors of $\Upsilon \alpha$ above, we obtain ${}^t\xi_{-}'\ell=0$ and ${}^t\xi_{-}'\Upsilon\alpha=\xi_{-}'$. 
\end{proof}

\begin{lem} \label{ExplicitIwasawaDecArch}
For $\xi_{+}'=(z_{i})_{i=1}^{n-m}\in \Xi_{Q_{+}',\infty}^{w_0}$, we have the Iwasawa decomposition \eqref{IwasawaDecxi} with a datum $(\Upsilon,\ell,\alpha,t)$ such that
\begin{align*}
t&=(1+\|\xi_{+}'\|^2)^{1/2}, \quad \ell=(1+\|\xi_{+}'\|^2)^{-1}\,\xi_{+}', \\
\alpha&=\diag(\alpha_i|1\leq i\leq n-m) \quad \text{with} \quad 
\alpha_{i}=\bigl(1+\sum_{j=1}^{i-1}z_j^2\bigr)^{1/2}\bigl(1+\sum_{j=1}^{i}z_j^2\bigr)^{-1/2},
\end{align*}
and the Hilbert-Schmidt norm $\|\Upsilon\alpha\|_{\rm HS}$ is no greater than $n-m$. Moreover, 
\begin{align}
{}^t\xi_{-}'\ell=0, \quad \|{}^t(\Upsilon\alpha)\xi_{-}'\|=\|\xi_{-}'\| \quad \text{for all $\xi_{-}' \in \Xi_{Q_{-}',\infty}^{w_0}$}.
\label{ExplicitIwasawaDecArch-f1}
\end{align}
\end{lem}
\begin{proof} From \eqref{IwasawaDecxi} and $\kappa {}^t\kappa=1_{n-m+1}$, we have
\begin{align*}
\left[\begin{smallmatrix} 1_{n-m} & 0 \\ 
{}^t \xi_{+}' & 1 \end{smallmatrix} \right]\,
\left[\begin{smallmatrix} 1_{n-m} & \xi_{+}' \\ 
0& 1 \end{smallmatrix} \right]
=\left[\begin{smallmatrix} \Upsilon & \ell \\ 
 0 & 1 \end{smallmatrix} \right]\,\left[\begin{smallmatrix} \alpha^2 & 0 \\ 
 0 & t^2 \end{smallmatrix} \right] \left[\begin{smallmatrix} {}^t\Upsilon & 0 \\  {}^t\ell & 1 \end{smallmatrix} \right]. 
\end{align*}
Hence $t^{2}=1+{}^t\xi_{+}'\xi_{+}'$, $\xi_{+}'=t^2\,\ell$ and 
\begin{align}
\Upsilon\, \alpha^2\,{}^t\Upsilon=1_{n-m}-t^{2}\ell\,{}^t\ell
=1_{n-m}-({1+\|\xi_{+}'\|^2})^{-1}\, \xi_{+}'\,{}^t\xi_{+}'
.\label{ExplicitIwasawaDecArch-f0}
\end{align}
Since $|z_i z_j/(1+\|\xi_{+}'\|^2)|\leq 1$, all the entries of the last matrix are no greater than $1$. Thus, $\|\Upsilon \alpha\|_{\rm HS}^2=\tr((\Upsilon\alpha)\,{}^t(\Upsilon \alpha))\leq (n-m)^2$. Let $(\rho_j,V_j)$ be the $j$-th wedge product of the natural representation of $\GL_{n-m+1}(\R)$ on $\R^{n-m+1}$ and ${\bf e}_i\,(1\leq i\leq n-m+1)$ the standard basis of $\R^{n-m+1}$; $V_j$ is endowed with the natural inner product with the associated norm $\|\cdot\|$, which is ${\bf O}_{n-m+1}(\R)$-invariant. The vector $v_j={\bf e}_1\wedge \dots \wedge {\bf e}_j$ is fixed by $\sU^{(n-m+1)}(\R)$. Hence from \eqref{IwasawaDecxi}, 
\begin{align*}
(\alpha_1\cdots\alpha_j)^{-1}&=\|\rho_j\left(\left[\begin{smallmatrix} 1_{n-m} & 0 \\ {}^t \xi_{+}' & 1 \end{smallmatrix} \right] \right)^{-1}v_j\| \\   
&=\|v_j + \sum_{i=1}^{j}z_{i}{\bf e}_1 \wedge \dots \wedge {\bf e}_{i-1}\wedge {\bf e}_{n-m+1}\wedge {\bf e}_{i+1}\wedge \dots \wedge {\bf e}_{n-m}\|
=\bigl(1+\sum_{i=1}^{j}z_{i}^{2}\bigr)^{1/2}
\end{align*}
for $1\leq j \leq n-m$. From these relations, we have the formula of $\alpha$. Let $\xi_{-}'\in \Xi_{Q_{-}',\infty}^{w_0}$; then ${}^t\xi_{-}'\xi_{+}'=0$. Thus from \eqref{ExplicitIwasawaDecArch-f0} and the formula of $\ell$, we have $\|{}^t(\Upsilon\alpha)\xi_{-}'\|^2={}^t\xi_{-}'\Upsilon\, \alpha^2\,{}^t\Upsilon\xi_{-}'={}^t\xi_-'\xi_{-}'=\|\xi_-'\|^2$ and ${}^t\xi_{-}'\ell=0$ as desired.
\end{proof}
From ${}^t\xi_{-}'\,\ell=0$ obtained in \eqref{ExplicitIwasawaDec-0} and \eqref{ExplicitIwasawaDecArch-f0}, the last row of the second matrix in \eqref{matrix2} is $(\underbrace{0,\dots,0}_{m-1},{}^t\xi_{-}'\Upsilon ,1)$. Noting this, from the expression \eqref{matrix2} we have the equality
{\allowdisplaybreaks\begin{align}
&\sn(y)\,w\,
\bigl[w_0 Z w_0^{-1};
w_0\left[\begin{smallmatrix} \xi \\ 0_{n-m} \end{smallmatrix}\right], w_0\left[\begin{smallmatrix} 0_{m-1} \\ \xi' \end{smallmatrix}\right] \bigr]_w
=t\,\iota({\mathsf h}(A,\tilde B,D,\xi_{+}'))\, \TT_y(Z,\xi,\xi'), 
\label{matrix4}
\end{align}}where 
{\allowdisplaybreaks
\begin{align}
&{\mathsf h}(A,\tilde B,D,\xi_{+}'):=w_0 
\left[\begin{smallmatrix} A & \tilde B \Upsilon \\
0 & D_0D\Upsilon \end{smallmatrix} \right]
\left[\begin{smallmatrix} t^{-1} 1_{m-1} & 0\\
0 & t^{-1}\alpha\end{smallmatrix} \right] \quad (\in\sG_0(\Q_v)), 
 \label{matrix45}
\\
&\TT_y(Z,\xi,\xi',):=\left[\begin{smallmatrix} 1_{m-1} & 0 & tu_1 \\ 
0 & 1_{n-m} & t\alpha^{-1} u_2 \\
0 & t^{-1} {}^t \xi_{-}'\Upsilon \alpha  & 1 
 \end{smallmatrix}\right]
\left[\begin{smallmatrix} 1_{m-1} & 0  & 0 \\
0 & \kappa_{11} & \kappa_{12} \\
0 & \kappa_{21} & \kappa_{22} 
\end{smallmatrix} \right]
\left[\begin{smallmatrix}w_0 & 0 \\ 0 & 1 \end{smallmatrix}\right]^{-1}w \quad(\in \sG(\Q_v))
 \label{matrix44}
\end{align}}with $\tilde B=B-\xi\,{}^t\xi_{+}'$, 
$$
u_1=A^{-1}(\xi-\tilde BD^{-1}D_0^{-1}\eta), \quad u_2=\Upsilon^{-1}(\ell+D^{-1}D_0^{-1}\eta).
$$
Since each of the factors of the product $\sn(y)\,w\,
[w_0 Z w_0^{-1};x, x']_w$ has the determinant $\pm 1$ (see Lemma~\ref{lem:U-elements}), the matrix \eqref{matrix4} also has the determinant $\pm 1$. Since $\kappa\in \bK^{(n-m+1)}_v$ in \eqref{IwasawaDecxi}, we have $|\det\kappa|_v=1$; hence the absolute value of the determinants of the last three factors of $\TT_y(Z,\xi,\xi')$ are $1$. The second factor of ${\mathsf h}(A,\tilde B,D,\xi_+')$ is an upper-triangular unipotent matrix because $A$, $D$, $D_0$ and $\Upsilon$ are such. By taking the determinant of \eqref{IwasawaDecxi}, we have $|t\,\det\alpha|_v=1$; hence $|\det(t\,\iota({\mathsf h}(A,\tilde B,D,\xi_{+}'))|_v=|t^{n} \times \det(\iota(w_0))\times t^{-(n-1)}\,\det \alpha|_v=|t\,\det \alpha|_v=1$. Therefore, we obtain  
\begin{align}
\left|\det \,
\left[\begin{smallmatrix} 1_{m-1} & 0  & tu_1 \\
0 & 1_{n-m} & t\alpha^{-1} u_2 \\
0 & t^{-1} {}^t\xi_{-}'\Upsilon \alpha& 1 
\end{smallmatrix} \right]\right|_v=1.
 \label{determinant-one}
\end{align}

The next lemma is of critical importance in the following arguments. 
\begin{lem} \label{tildefsuppL}
 Let $\Ucal\subset \sG(\Q_v)$ be a compact set. There exists a compact subset $\Vcal=\Vcal(\Ucal)$ of $\sG_0(\Q_v)\times \Q_v^{n-1} \times \Q^{n-1}_v\times \Q_p^\times $ with the following property. If $h\in \sG_0(\Q_v)$, $u,\,u' \in \Q_v^{n-1}, z\in \Q_v^\times $ satisfies
\begin{align}
&\iota(h)^{-1} \left[\begin{smallmatrix} 1_{n-1} & u \\ {}^t u' & 1 \end{smallmatrix} \right] \in z\,\Ucal,
\qquad 
\biggl|\det\,
\left[\begin{smallmatrix} 1_{n-1} & u \\
{}^t u' & 1 
\end{smallmatrix} \right]\biggr|_v=1,
\label{tildefsuppL-0}
\end{align}
then $(h,u,u',z)\in \Vcal$. When $v=p<\infty$ and $\Ucal=\bK_v$, then the relation \eqref{tildefsuppL-0} implies $h\in \bK_{\sG_0,v}$, $u\in \Z_p^{n-1}$, $u'\in\Z_p^{n-1}$ and $z\in \Z_p^\times$. 
\end{lem}
\begin{proof}
We have a matrix $X=\left[\begin{smallmatrix} x_{11} & x_{12} \\ x_{21} & x_{22}\end{smallmatrix}\right]\in \Ucal$ with $x_{11}\in \Mat_{n-1}(\Q_v)$, $x_{12}, {}^t x_{21}\in \Q_v^{n-1}$ and $x_{22} \in \Q_v$ such that 
\begin{align}
\left[\begin{smallmatrix} h^{-1} & h^{-1} u  \\ {}^t u' & 1 \end{smallmatrix} \right]=z\left[\begin{smallmatrix} x_{11} & x_{12} \\ x_{21} & x_{22}\end{smallmatrix}\right].
\label{tildefsuppL-1}
\end{align}
Since $\cU\subset \sG(\Q_v)$ is compact, we have that $|\det X|_v\asymp 1$ and that all the entries of $X$ are bounded, which implies $|x_{22}|_v \ll 1$ and $|\det x_{11}|_v\ll 1$. Thus for the first assertion, it suffices to show $|\det h|_v\asymp 1$ and $|z|_v\asymp 1$. We have $|z|_v^{-1}=|x_{22}|_{v}\ll 1$, and thus $|z|_v\gg 1$. From the second condition of \eqref{tildefsuppL-0}, the absolute value of the determinant of the left-hand side of \eqref{tildefsuppL-1} is $|\det h|_v^{-1}$. By taking the determinant of the both sides of \eqref{tildefsuppL-1}, $|\det h|_v^{-1}=|z|_v^{n}|\det X|_v \asymp |z|_v^{n}$. Hence $|\det h|_v\asymp |z|_v^{-n}$, which, combined with $|z|_v\gg 1$, yields $|\det h|_v \ll 1$. From \eqref{tildefsuppL-1}, $h^{-1}=zx_{11}$; taking the determinant of this, we get $|\det h|_v^{-1}=|z|^{n-1}_v|\det x_{11}|_v$. Combining this with the bound $|\det h|_v|z|_v^{n}=|\det X|_v^{-1} \ll 1$ and $|\det x_{11}|_v\ll 1$, we have $|z|_v = |z|^{n}_v|\det h|_v|\det x_{11}|_v\ll 1$. Hence $|\det h|_v=|z|_v^{-n}|\det X|_v\gg 1$. Thus $|\det h|_v\asymp 1$ and $|z|_v\asymp 1$ as desired. Let $v=p<\infty$ and $\Ucal=\bK_p$; then by arguing as above using $|\det X|_p=1$, $|\det x_{11}|_p\leq 1$ and $|x_{22}|_p\leq 1$, we have $|\det h|_p=|z|_p=1$. From this observation, the last assertion follows. 
\end{proof}
 
\begin{cor} \label{tildesuppL-cor} Let $Q\subsetneq I^0$ and $m=m(Q)=\# Q+1$. Let $y\in {\mathsf Y}_Q$ and $w\in \sS_n(Q,y)$. For any compact set  $\Ucal$ of $\sG(\Q_v)$, there exists a compact set $\Vcal_0\times \Vcal_1 \times \Vcal_2 \times \Vcal_3 \times \Vcal_4$ of $\sG_0(\Q_v)\times \Q_v^{m-1}\times \Q_v^{n-m}\times \Q_v^{n-m}\times \Q_v^\times$ with the following property: If 
\begin{align}
\iota(h^{-1})\,\sn(y)\,w\,
\bigl[w_0 Z w_0^{-1}; 
w_0\left[\begin{smallmatrix} \xi \\ 0_{n-m} \end{smallmatrix}\right], w_0\left[\begin{smallmatrix} 0_{m-1} \\ \xi' \end{smallmatrix}\right] 
\bigr]_w \in z\,\Ucal
 \label{tildesuppL-cor-f1}
\end{align}
with $h\in \sG_0(\Q_v)$, $Z=\left[\begin{smallmatrix} A & B \\ 0 & D_0D \end{smallmatrix}\right]\in {\bf V}_{y,w}(\Q_v)$, $\xi\in \Q_v^{m-1}$, $\xi' \in \Q_v^{n-m}$, and $z\in \Q_p^\times$, then $h$ and $Z$ and the vectors $\xi$, $\xi'=\xi_{-}'+\xi_{+}'\,(\xi_{\pm}' \in \Xi_{Q_{\pm}',v}^{w_0})$ and $\eta$ defined by \eqref{w0Vectors} should satisfy\begin{align}
\begin{cases} 
&h^{-1}\,{\mathsf h}(A,\tilde B,D,\xi_{+}')
\in \Vcal_0, \\
&tu_1\in \Vcal_1, \qquad t\alpha^{-1}u_2\in \Vcal_2, \qquad t^{-1}\xi_{-}'\in \Vcal_3,, \quad t^{-1} z\in \Vcal_{4},
\end{cases}
 \label{tildesuppL-cor-f2}
\end{align}
where $u_1=A^{-1}(\xi-\tilde BD^{-1}D_0^{-1}\eta)$ and $u_2=\Upsilon^{-1}(\ell+D^{-1}D_0^{-1}\eta)$ with $\tilde B=B-\xi\,{}^t\xi_{+}'$, and $(\Upsilon,\ell,\alpha,t)$ for $\xi_+'$ is the data of the Iwasawa decomposition \eqref{IwasawaDecxi}. When $v=p<\infty$ and $\Ucal=\bK_0(N\Z_p)$ with $N\in \N$, then the relation \eqref{tildesuppL-cor-f1} imples \eqref{tildesuppL-cor-f2} with $\Vcal_0=\bK_{\sG_0,p}$, $\Vcal_1=\Z_p^{m-1}$, $\Vcal_2=\Vcal_3=\Z_p^{n-m}$ and $\Vcal_4=\Z_p^\times$, and that the vector $\xi'=(\xi_j')_{j=m}^{n-1}$ satisfies  
\begin{align}
t^{-1}\,(\underbrace{0,\dots,0}_{m-1}, 1, \underbrace{\xi_{m}', \dots,\xi_{n-1}'}_{n-m})\in (N\Z_p^{n-1} \oplus \Z_p)_{\rm prim},
 \label{tildesuppL-cor-f0}
\end{align}
where $L_{\rm prim}$ denotes the set of primitive vectors in a $\Z_p$-lattice $L$. 
\end{cor}
\begin{proof} Lemma~\ref{tildefsuppL} together with the formulas \eqref{matrix4}, \eqref{matrix44} and \eqref{determinant-one} yields $\Vcal_{i}\,(i=0,1,2,4)$ with the bounds in \eqref{tildesuppL-cor-f2} and also a bound for the points $t^{-1}\,{}^t\xi_{-}' \Upsilon \alpha$. To obtain the bounding set $\Vcal_3$ for the points $t^{-1}\xi_{-}'$ we further use \eqref{ExplicitIwasawaDec-0} and \eqref{ExplicitIwasawaDecArch-f1}.

If $\Ucal=\bK_0(N\Z_p)$, then from the expression \eqref{matrix4} and the second statement of Lemma~\ref{tildefsuppL}, observing the containment 
$$
\bK_1(N\Z_p)w^{-1} 
\left[\begin{smallmatrix} w_0 & 0 \\ 0 & 1 \end{smallmatrix}\right]
 \left[\begin{smallmatrix} 1_{m-1} & 0  & 0 \\
0 & \kappa_{11} & \kappa_{12} \\
0 & \kappa_{21} & \kappa_{22} 
\end{smallmatrix} \right]^{-1}\subset \bK_p,
$$
we see that the relation \eqref{tildesuppL-cor-f1} yields $t^{-1} z\in \Z_p^\times$. Then the $n$-th row of the matrix in \eqref{tildesuppL-cor-f1}, which, from \eqref{matrix0}, equals $$((\underbrace{0,\dots,0}_{m-1},{}^t\xi')w_0^{-1},1)\,w
=(\underbrace{0,\dots,0}_{m-1}, 1, \xi_{m}', \dots,\xi_{n-1}'),
$$should belong to $z\,(N\Z_p^{n-1} \bigoplus \Z_p)_{\rm prim}=t\,(N\Z_p^{n-1} \bigoplus \Z_p)_{\rm prim}$. \end{proof}

\subsection{The proof of Proposition~\ref{LocalABSCONVJJ}} \label{sec:Proofof}
We estimate the values
\begin{align}
&|\tilde f_v^{(\nu)}
(\sn(y)w\,\bigl[w_0Zw_0^{-1};
w_0\left[\begin{smallmatrix} \xi \\ 0_{n-m} \end{smallmatrix}\right], w_0\left[\begin{smallmatrix} 0_{m-1} \\ \xi' \end{smallmatrix}\right] \bigr]_w)|, \quad (Z,\xi,\xi') \in {\bf V}_{y,w}(\Q_v)\times \Q_v^{m-1} \times \Q_v^{n-m}.
 \label{JJ-L1-f0}
\end{align}
The element $\sf_{\sG_0,v}^{(\nu)} \in I^{\sG_0}_v(\nu)$ should be recalled from \S\ref{Jacquetint}. 
\begin{lem}\label{testftnconv}
Suppose $f_v \in C_{\rm{c}}^\infty(\sG(\Q_v))$ is left $\iota(\bK_{\sG_0,v})$-invariant. Then, 
$$
\tilde f_v^{(\nu)}(g_v)=\int_{\sG_0(\Q_v)} \sf_{\sG_0,v}^{(\nu)}(h_v)\,\tilde f_v(\iota(h_v)^{-1}g_v)\,\d h_v, \quad g_v\in \sG(\Q_v),\,\nu \in \ft_{0,\C}^*.
$$
\end{lem}
\begin{proof} Recall that the measures on $\sG_0(\Q_v)$, $\sB_0(\Q_v)$, and $\bK_{\sG_0,v}$ are related as $\d h_v=\d_l b \,\d k$ (see \S\ref{sec:Measure}), and apply \eqref{Loc-tildefv} and the formula $\sf_{\sG_0,v}^{(\nu)}(bk)=e^{\langle H_{\sG_0}(b), \nu+\rho_{\sB_0}\rangle}$. 
\end{proof}
From Lemma~\ref{testftnconv} and by \eqref{matrix4}, after an obvious variable change for the $h$-integral, we see that the value \eqref{JJ-L1-f0} is no greater than
{\allowdisplaybreaks\begin{align*}
&\int_{\sG_0(\Q_v)}\sf_{\sG_0,v}^{(\Re \nu)}({\mathsf h}(A,\tilde B,D,\xi_{+}')\,
h)\times |\tilde f_v(\iota(h^{-1})\,\TT_y(Z,\xi,\xi'))
|\,\d h.
\end{align*}}From Corollary~\ref{tildesuppL-cor} applied to a compact set $\Ucal\subset \sG(\Q_v)$ containing ${\rm supp}(f_v)$, we may restrict the domain of the $h$-integration to a compact set $\Vcal_0\subset \sG_0(\Q_v)$ and also may suppose $u_1 \in t^{-1}\,\Vcal_1$, $t\alpha^{-1}u_2\in \Vcal_2$ and $\xi_{-}'\in t\Vcal_3$, where $\Vcal_i\,(i=0,1,2,3)$ are certain compact sets determined only by $\Ucal$. Now consider the integral in $(Z,\xi,\xi')$ over ${\bf V}_{y,w}(\Q_v)\times \Q_v^{m-1} \times \Q_v^{n-m}$ of the last expression. By $\Q_v^{n-m}=\Xi_{Q_+',v}^{w_0}\bigoplus \Xi_{Q_{-}',v}^{w_0}$, the Haar measure $\d \xi'$ on $\Q^{n-m}_v$ is decomposed as $\d\xi_{+}'\,\d\xi_{-}'$ with Haar measures $\d\xi_{\pm}'$ on $\Xi_{Q_{\pm}',v}^{w_0}$ such that $\vol(\Z_p^{n-m}\cap \Xi_{Q_{\pm}',p}^{w_0})=1$ if $v=p<\infty$. We note that the measure on ${\bf V}_{y,w}(\Q_v)$ is given as $\d Z=\d D_0\,\d A\,\d B\,\d D$, where all factors are the Tamagawa measures on the unipotent groups. As noted before, the map $B\rightarrow \tilde B:=B-\xi{}^t\xi_{+}'$ is a measure preserving bijection of ${\bf N}_{w_0}(\Q_v)$. Recall that $(\Upsilon,\ell,\alpha,t,\kappa)$ is a function of $\xi_{+}'$ determined by the relation \eqref{IwasawaDecxi} and $u_1=A^{-1}(\xi-\tilde B D^{-1}D_0^{-1}\eta)$, $u_2=\Upsilon^{-1}(\ell+D^{-1}D_0^{-1}\eta)$. By the variable change, we can replace $\xi$ with $u_1$ as a variable for integration. Moreover, it is evident from the proof of Lemma~\ref{DOI-L1} that $w_0 \diag(1_{m-1},D_0)w^{-1}_0$ belongs to $\sU_0(\Q_v)$, which implies $\sf_{\sG_0,v}^{(\nu)}(w_0\diag(1_{m-1},D_0)h')=\sf_{\sG_0,v}^{(\nu)}(w_0 h')$ for all $h'\in \sG_0(\Q_v)$. The absolute value $|\tilde f_v|$ is bounded by $\|\tilde f_v\|_{\infty}:=\sup_{h\in \sG(\Q_v)}|\tilde f_{v}(h)|$. In this way, we have{\allowdisplaybreaks\begin{align}
&{\mathbb {MJ}}(|\tilde f_v^{(\nu)}|;y,w)
 \label{JJ-L1-f2}
\\
&\leq \|\tilde f_v\|_\infty \,
\int_{h \in \Vcal_0}\int_{(A,\tilde B,D)}\int_{\xi_{+}'} \sf_{\sG_0,v}^{(\Re \nu)}({\mathsf h}(A,\tilde B,D,\xi_{+}')\,h)\biggl\{\int_{(u_1,\xi_{-}') \in t^{-1}\Vcal_1\times t\Vcal_3} \,\d u_1 \,\d \xi'_{-}
\biggr\}
\notag
\\
&\quad \times \biggl\{\int_{D_0\in {\bf V}_{y,w}^{(n-m)}(\Q_v)}\delta(t\alpha^{-1} \Upsilon^{-1}(\ell+D^{-1}D_0^{-1}\eta) \in \Vcal_2)\d D_0\biggr\}\,\d h\,\d A\,\d \tilde B\,\d D\,\d \xi_{+}' .
 \notag 
\end{align}}If $y=0$ (i.e., $\eta=0$), then the $D_0$-integral is $1$, because $\fN(w)=\emp$ for $w\in \sS_n(Q,0)$ as was remarked after Definition~\ref{Def1}, and hence ${\bf V}_{y,w}^{0}\cong {\bf V}_{y,w}^{(n-m)}$ is trivial. Suppose $y\not=0$ for a while, and let us estimate the $D_0$-integral; for that, we examine the condition: 
\begin{align}
D^{-1}D_{0}^{-1}\eta\in -\ell+t^{-1}\Upsilon\alpha\,\Vcal_2. 
\label{3-19}
\end{align} 
Set $\fm(\xi_{+}')=|t|_v$, or explicitly
\begin{align}
\fm(\xi_{+}')=\begin{cases} 
\max\{1, \max\{|\xi_j|_{p}|\,j \in w_0^{-1}(Q_+')\}\} \quad & (v=p<\infty), \\
\bigl(1+\sum_{j \in w_0^{-1}(Q_+')}|\xi_j|_\infty^2\bigr)^{1/2} \quad & (v=\infty)
\end{cases}
 \label{fmxi-plus}
\end{align}
for $\xi_{+}'=(\xi_j)_{j=m}^{n-1} \in \Xi_{Q_{+}',v}^{w_0}$ from Lemmas~\ref{ExplicitIwasawaDec} and \ref{ExplicitIwasawaDecArch}. 
We have a lemma. 
\begin{lem} \label{L2020319} There exists a constant $B_{\Vcal_2}>0$ independent of $y\,(\not=0)$ and $\xi_{+}'$ such that \eqref{3-19} implies 
$$
|y_{q'}|_v\leq B_{\Vcal_2}\,\fm(\xi_{+}')^{-1}, \quad \xi_{+}'\in \Xi_{Q_{+}',v}^{w_0},
$$
where $q'=\min({\rm sp}(y))$. If $v=p<\infty$ and $\Vcal_2=\Z_p^{n-m}$, then we can take $B_{\Vcal_2}=1$. 
\end{lem}
\begin{proof} Define the index $\j\in [m,n-1]_\Z$ by $w(\j+1)=q'$. Since $w^{-1}$ is decreasing on ${\rm sp}(y)$, we see that $\eta_{i}=0$ for all $i\in I^{0}_{>\j}$. Therefore, noting that $D_0D$ is an upper-triangular unipotent matrix, we have $(D^{-1}D_0^{-1}\eta)_{\j}=\eta_{\j}=y_{q'}$. Hence \eqref{3-19} implies $ty_{q'}$ is in the set $Y(\xi_+'):=-t\ell_{\j}+(\Upsilon\alpha \Vcal_2)_{\j}$. From Lemmas~\ref{ExplicitIwasawaDec} and \ref{ExplicitIwasawaDecArch}, all the entries of $t\ell$ and $\Upsilon \alpha$ are bounded. Hence $Y(\xi_+')$ is contained in a fixed bounded subset of $\Q_v$ as $\xi_+'$ varies. For the last assertion, we remark that Lemma~\ref{ExplicitIwasawaDec} (ii) shows $Y(\xi_{+}')\subset \Z_p$ if $v=p<\infty$ and $\Vcal_2=\Z_p^{n-m}$.  
\end{proof}
In what follows, we fix a constant $B_{\Vcal_2}$ as in Lemma~\ref{L2020319} and define $\XX(y_{q'})$ to be the set of $\xi_{+}'\in \Xi_{Q_{+}',v}^{w_0}$ such that $|y_{q'}|_v\leq B_{\Vcal_2}\,\fm(\xi_{+}')^{-1}$. Note that $\XX(y_{q'})$ is a bounded subset of $\Xi_{Q_+',v}^{w_0}$ which is empty unless $|y_{q'}|_v\leq B_{\Vcal_2}$. 

\begin{lem} There exists a positive constant $C_{\Vcal_2}$ such that
\begin{align}
\int_{D_0 \in {\bf V}_{w_0}^{(n-m)}(\Q_v)}\delta(D^{-1}D_0^{-1}\eta \in -\ell+t^{-1}\Upsilon\alpha\,\Vcal_2)\,\d D_0 \leq C_{\Vcal_2}\delta(\xi_{+}' \in \XX(y_{q'}))\,\prod_{\nnu=1}^{h}|y_{w(j_\nnu+1)}|_p^{-\#J_\nnu^*(y,w)} \label{JJ-L1-f101}
\end{align}
uniformly in $D$, $y\,(\not=0)$ and $\xi_{+}'$. If $v=p<\infty$ and $\Vcal_2=\Z_p^{n-m}$, then we can take $C_{\Vcal_2}=1$. 
\end{lem}
\begin{proof} The factor $\delta(\xi_{+}'\in \XX(y_{q'}))$ of the majorant arises from Lemma~\ref{L2020319}. From Lemmas~\ref{ExplicitIwasawaDec} and \ref{ExplicitIwasawaDecArch} all the entries of the matrices $\ell$ and $t^{-1}\Upsilon \alpha$ are in the ball $\beta_v=\{x\in \Q_v|\,|x|_v\leq r_v\}$ with $r_v=1$ if $v=p<\infty$ and $r_v=n-m$ if $v=\infty$. Thus if we set $\tilde \Vcal_2=\beta_v^{n-m}+{\bf M}_{n-m}(\beta_v)\,\Vcal_2$, where ${\bf M}_{n-m}(\beta_v)$ is the set of $(n-m)\times (n-m)$ matrices with entries in $\beta_v$, then $\tilde \Vcal_2$ is a compact subset of $\Q_v^{n-m}$ and \eqref{3-19} implies 
\begin{align}
D^{-1}D_0^{-1}\eta \in \tilde \Vcal_2.
 \label{JJ-L1-f1000}
\end{align} 
By enlarging $\tilde \Vcal_2$ if necessary, we may suppose $\tilde \Vcal_2=\prod_{i=m}^{n-1}(\tilde \Vcal_2)_{i}$ with compact sets $(\tilde \Vcal_2)_{i}\subset \Q_v$. 

 Suppose $\#\fN(w)=h>0$ and let $\fN(w)=\{j_1,\dots,j_h\}$ and $J_{\nnu}(w)=[j_{\nnu-1}+1,j_{\nnu}]_\Z$ be as before. From \eqref{20200213}, $\eta_{j_\nnu}=y_{w(j_\nnu+1)}\not=0$. Set $\Ical:=\{(i,\lambda) \in [m,n-1]_{\Z} \times[1,h]_{\Z} \mid  i+1\in J_\lambda^*(y,w)\,\}$. Then as noted before, we may set $D_0^{-1}=(u_{i,j})_{m\leq i,j \leq n-1} \in {\bf V}_{y,w}^{(n-m)}(\Q_v)$ (see \eqref{fZ0-D0}), where only the possibly non-zero off-diagonal entries are $u_{i,j_\lambda}\,((i,\lambda)\in \Ical)$, and $\d D_0=\prod_{(i,\lambda)\in \Ical}\d u_{i,j_\lambda}$. Thus the $D_0$-integral to be estimated is written as an iterated integral with respect to the variables $u_{i,j_{\lambda}}$ over the set defined by \eqref{JJ-L1-f1000}. 
%Noting \eqref{eta-0cond} and \eqref{Jnnu*yw}, we see that 
For $i\in [m,n-1]_\Z$ the $i$-th entry of $D_0^{-1}\eta$ is given by 
\begin{align*}
(D_0^{-1}\eta)_{i}=\begin{cases} 
\eta_{i} \quad &(\text{if $i+1\not\in J^{*}(y,w)$}), \\
u_{i,j_{\nnu}}\eta_{j_\nnu} \quad &(\text{if $i+1\in J^{*}_\nnu(y,w)$ $(\nnu \in [1,h]_{\Z})$},
\end{cases}
\end{align*}
where $J^{*}(y,w)$ denotes the union of $J_{\nnu}^{*}(y,w)\,(1\leq \nnu\leq h)$. Set $D^{-1}=(d_{ij})_{m\leq i,j \leq n-1}\in \sU_{w_0}^{(n-m)}(\Q_v)$. Then for $i\in [m,n-1]_\Z$ such that $i+1\in J_{\nnu}^{*}(y,w)$ with $\nnu\in [1,h]_{\Z}$, the entry $(D^{-1}D_0^{-1}\eta)_{i}$ is equal to 
{\allowdisplaybreaks\begin{align}
U_{i,j_\nnu}:=&\eta_{j_\nnu}u_{i,j_{\nnu}}
+\sum_{\substack{j\in I_{>i}^{0}\\ w(i+1)>w(j+1), j+1 \in J_{\mu}^{*}(y,w)}} d_{ij}\eta_{j_\mu}u_{j,j_\mu}
 +\sum_{\substack{j \in I_{>i}^{0} \\ w(i+1)>w(j+1), j+1\notin J^{*}(y,w)}} d_{ij}\eta_{j} .
 \label{Affine}
\end{align}}If $j\in I_{>i}^{0}$ and $j+1\in J_\mu(w)$, we have that $\mu$ is uniquely determined by $j$, and that $\mu\geq \nnu$. Consider the sets of variables $u_{i,j_\nnu}$ and $U_{i,j_\nnu}$ both indexed by the set $\Ical$, and related to each other by the affine equations \eqref{Affine}. If we endow the set $\Ical$ with the lexicographical ordering, i.e., $(i,\nnu)\preceq (i',\nnu')$ if and only if $\nnu'>\nnu$ or $\nnu'=\nnu$, $i'\geq i$, then \eqref{Affine} is written in the form
$$
U_{i,j_\nnu}=\eta_{j_\nnu}u_{i,j_\nnu}+\sum_{\substack{ (j,\mu)\in \Ical \\ (i,\nnu)\preceq (j,\mu), (i,\nnu)\not=(j,\mu)}} \fd_{(i,\nnu),(j,\mu)}\,u_{j,j_{\mu}}+\fe_{(i,\nnu)}, 
$$
where $\fd_{(i,\nnu),(j,\mu)},\,\fe_{(i,\nnu)}\in \Q_v$ depend only on $\eta$ and the entries $d_{ij}$ of $D^{-1}$. The condition \eqref{JJ-L1-f1000} is equivalent to $U_{i,j_{\nnu}}\in (\tilde \Vcal_2)_{i}$ for all $(i,\nnu)\in \Ical$. Thus, the integral over $D_0\in {\bf V}_{y,w}^{(n-m)}(\Q_v)$ is an iteration of $u_{i,j_\nnu}$-integrals over the set defined by $U_{i,j_{\nnu}}\in (\tilde \Vcal_2)_i$, performed successively according to the total ordering on the index set $\Ical$. Let $(i(r),\nnu(r))\,(r=0,1,2,\dots)$ be the enumeration of $\Ical$ such that $(i(r),\nnu(r))\preceq(i(r+1),\nnu(r+1))$ for any $r$. Then the $u_{i(0),j_{\nnu(0)}}$-integral over the set $U_{i(0),j_{\nnu(0)}}\in (\tilde {\Vcal_2})_{i(0)}$ equals the volume of the set
$$\eta_{j_{\nnu(0)}}^{-1}((\tilde \Vcal_2)_{i(0)}+t) \quad \text{with $t=-\sum_{r\geq 1} \fd_{(i(0),\nnu(0)),(i(r),\nnu(r))}\,u_{i(r),j_{\nnu(r)}}-\fe_{(i(0),\nnu(0))}$,}
$$ which in turn equals $|\eta_{j_{\nnu(0)}}|_p^{-1}\vol((\tilde \Vcal_2)_{i(0)})$ independently of variables $D$ and $u_{i(r),j_{\nnu(r)}}$ ($r\geq 1)$. In a similar way, the $u_{i(1),j_{\nnu(1)}}$-integral over the set $U_{i(1),j_{\nnu(1)}}\in (\tilde \Vcal_2)_i$ is computed to be $|\eta_{j_{\nnu(1)}}|_p^{-1}\,\vol((\tilde \Vcal_2)_{i(1)})$ independently of $D$ and other variables $u_{i(r),j_{\nnu(r)}}$ with $r\geq 2$. By repeating this process, we finally get a constant $C_{\Vcal_2}=\prod_{(i,\nnu)\in \Ical}\vol((\tilde \Vcal_2)_{i})$ such that \eqref{JJ-L1-f101} holds uniformly in $D$ and $y\,(\not=0)$. Note that if $v=p<\infty$ and $\Vcal_2=\Z_p^{n-m}$, then $\tilde \Vcal_2=\Z_p^{n-m}$ and $(\tilde \Vcal_2)_i=\Z_p$; from this, $C_{\Z_p^{n-m}}=1$. 
\end{proof}

Let us start the proof of Proposition \ref{LocalABSCONVJJ}. Bounding the $(u_1,\xi_{-}')$-integral in \eqref{JJ-L1-f2} by $|t|_v^{-(m-1)}\times|t|_v^{\# Q_{-}'}$ and the $D_0$-integral by \eqref{JJ-L1-f101}, we see that ${\mathbb{MJ}}(|\tilde f_v^{(\nu)}|;y,w)$ is no greater than
{\allowdisplaybreaks\begin{align*}
&\|\tilde f_v\|_\infty\,C_{\Vcal_2}\,C(y)\, \prod_{\nnu=1}^{h}
|y_{w(j_\nnu+1)}|_v^{-\#J_\nnu^*(y,w)} \\
&\quad \times \int_{h \in \Vcal_0}\int_{\xi_{+}' \in \XX(y_{q'})} \biggl\{\int_{(A,\tilde B,D)}\sf_{\sG_0,v}^{(\Re \nu)}({\mathsf h}(A,\tilde B,D,\xi_{+}')\,
h)\d A\,\d \tilde B\,\d D\biggr\} |t|_p^{-(m-1)+\# Q_{-}'} \d h\, \d \xi_{+}',
\end{align*}}
where $C(y)=1$ if $y=0$ and $C(y)=\delta(|y_{q'}|_v\leq B_{\Vcal_2})$ if $y\not=0$.  
Let  
$$M_{w_0}(\sf_{\sG_0,v}^{(\nu)};h)=\int_{w_0^{-1}\bar \sU_0(\Q_v)w_0\cap \sU_0(\Q_v)}\sf_{\sG_0,v}^{(\nu)}(w_0 Z_1 h)\,\d Z_1
$$
be the standard intertwining operator on $I_v^{\sG_0}(\nu)$, which is known to be absolutely convergent for $ \nu \in \fJ((\ft_{0}^{*})^{++})$ and 
$$
M_{w_0}(\sf_{\sG_0,v}^{(\nu)};h)=
\prod_{\substack{1\leq i<j \leq n-1 \\ w_0^{-1}(i)>w_0^{-1}(j)}}
\frac{\zeta_v(\nu_i-\nu_j)}{\zeta_v(\nu_i-\nu_j+1)}\times \sf_{\sG_0,v}^{(w_0^{-1}\nu)}(h), \quad h\in \sG_0(\Q_v)
$$
by the Gindikin-Karpelvic formula \cite[\S4.2]{Shahidi2}. Then from \eqref{fZ_1-ABD}, the $(A,\tilde B,D)$-integral is 
\begin{align*}
M_{w_0}\biggl(\sf_{\sG_0,v}^{(\Re \nu)}; \left[\begin{smallmatrix} t^{-1} 1_{m-1} & 0 \\
0 & \Upsilon \alpha t^{-1} \end{smallmatrix} \right]h\biggr)
&=|t|_{v}^{-\sum_{j=1}^{n-1}\Re \nu_j}\, \exp\left(\langle H_{\sG_0}\left(\left[\begin{smallmatrix} 1_{m-1} & 0 \\ 0 & \alpha \end{smallmatrix}\right]\right), w_0^{-1}\Re \nu+\rho_{\sB_0}\rangle \right) \\ 
&\times 
\prod_{\substack{1\leq i<j \leq n-1 \\ w_0^{-1}(i)>w_0^{-1}(j)}}
\frac{\zeta_v(\Re\,\nu_i-\Re\,\nu_j)}{\zeta_v(\Re\, \nu_i-\Re\,\nu_j+1)}\times \sf_{\sG_0,v}^{(w_0^{-1} \Re \nu)}(h).\end{align*}
Note that $\Upsilon$ is an upper-triangular unipotent matrix. By applying this formula, 
{\allowdisplaybreaks\begin{align*}
&{\mathbb{MJ}}(|\tilde f_v^{(\nu)}|;y,w)
\\
&\leq \|\tilde f_v\|_\infty\,C_{\Vcal_2}\,C(y)\, \prod_{\nnu=1}^{h}
|y_{w(j_\nnu+1)}|_v^{-\#J_\nnu^*(y,w)}\,
\prod_{\substack{1\leq i<j \leq n-1 \\ w_0^{-1}(i)>w_0^{-1}(j)}}
\frac{\zeta_v(\Re\,\nu_i-\Re\,\nu_j)}{\zeta_v(\Re\,\nu_i-\Re\,\nu_j+1)}\,
\int_{h \in \Vcal_0}\sf_{\sG_0,v}^{(w_0^{-1}\Re \nu)}(h)\,\d h
 \notag
\\
&\quad \times  \,
\int_{\xi_{+}'\in \XX(y_{q'})}\exp\left(\langle H_{\sG_0}\left(\left[\begin{smallmatrix} 1_{m-1} & 0 \\ 0 & \alpha \end{smallmatrix}\right]\right), w_0^{-1}\Re \nu+\rho_{\sB_0}\rangle \right)|t|_v^{-(m-1)+\# Q_{-}'-\sum_{j=1}^{n-1}\Re \nu_j} \d \xi_{+}'. 
\end{align*}}We remind the readers that $\alpha$ and $t$ are functions of $\xi_{+}'$ (see Lemmas~\ref{ExplicitIwasawaDec} and \ref{ExplicitIwasawaDecArch}). In the rest of the proof, assuming $y\not=0$ (i.e., $Q_{+}'\not=\emp$), we majorize the $\xi_{+}'$-integral over the compact set $\XX(y_{q'})$. Let us first consider the case when $v=p<\infty$. The space $\Xi_{Q_{+}',p}^{w_0}$ is a union of $\Z_p^{n-m}\cap \Xi_{Q_{+}',p}^{w_0}$ and the sets $\Xi_{Q_+',p}^{w_0}(\iiota)$ (defined before Lemma~\ref{ExplicitIwasawaDec}), where $\iiota=(\iiota(\mu))_{\mu=0}^{a}$ is a strictly decreasing sequence of elements of $w_0^{-1}Q_{+}'\,(\subset [m,n-1]_{\Z})$ with $0\leq a\leq \# Q_{+}'-1$. For $\xi_{+}'=(\xi_j')_{j=m}^{n-1} \in \XX(y_{q'})\cap \Xi_{Q_+',p}^{w_0}(\tilde \iiota)$, from Lemma~\ref{ExplicitIwasawaDec} and Lemma~\ref{L2020319}, we have $1<|t|_p=|\xi'_{\iiota(0)}|_{p}\leq B_{\Vcal_2}|y_{q'}|_p^{-1}$, $1<|\xi'_{\iiota(\mu)}|_p\leq |\xi'_{\iiota(\mu-1)}|_p$ $(1\leq \mu\leq a)$, $|\xi'_{j}|_p\leq 1$ $(j<\iiota(a))$ and 
{\allowdisplaybreaks\begin{align*}
H_{\sG_0}\left(\left[\begin{smallmatrix} 1_{m-1} & 0 \\ 0 & \alpha \end{smallmatrix}\right]\right)&=\sum_{\mu=1}^{a} \log |\xi'_{\iiota(\mu)}/\xi'_{\iiota(\mu-1)}|_p\,\varepsilon_{\iiota(\mu-1)}-\log|\xi'_{\iiota(a)}|_{p}\,\varepsilon_{\iiota(a)}
\\
&=\sum_{\mu=1}^{a}\log|\xi_{\iiota(\mu)}'|_p\,(\varepsilon_{\iiota(\mu-1)}-\varepsilon_{\iiota(\mu)})-\log|\xi_{\iiota(0)}'|_p\,\varepsilon_{\iiota(0)}.
\end{align*}}Since $w^{-1}$ is decreasing on $Q_+'={\rm sp}(y)$ (see Definition \ref{Def1} (c)) and $\iiota(\mu)<\iiota(\mu-1)$, we have $w_0(\iiota(\mu))>w_0(\iiota(\mu-1))$, and hence $\Re\nu_{w_0(\iiota(\mu))}<\Re\nu_{w_0(\iiota(\mu-1))}$ ($\mu\in [1,a]_\Z)$ by $\Re\nu \in (\ft_{0}^{*})^{++}$. We also note $\langle \rho_{\sB_0}, \varepsilon_{\iiota(\mu)}\rangle =n/2-\iiota(\mu)$ by \eqref{rhosB0}. From these and $1<|\xi_{\iiota(\mu)}'|_p\leq |\xi_{\iiota(0)}'|_p \,(\mu\in[0, a]_\Z)$, 
{\allowdisplaybreaks  
\begin{align*}
& \exp(\langle H_{\sG_0}\left(\left[\begin{smallmatrix} 1_{m-1} & 0 \\ 0 & \alpha \end{smallmatrix}\right]\right), w_0^{-1}\Re \nu+\rho_{\sB_0}\rangle) 
\\ 
&=\prod_{\mu=1}^{a} |\xi_{\iiota(\mu)}'|_p^{\Re\,\nu_{w_0(\iiota(\mu-1))}-
\Re\,\nu_{w_0(\iiota(\mu))}+\iiota(\mu)-\iiota(\mu-1)}\times |\xi_{\iiota(0)}'|_p^{-\Re\nu_{w_0(\iiota(0))}-(n-2\iiota(0))/2}
\\
&\leq \prod_{\mu=1}^{a}|\xi_{\iiota(0)}'|_p^{\Re\,\nu_{w_0(\iiota(\mu-1))}-
\Re\,\nu_{w_0(\iiota(\mu))}}\times |\xi_{\iiota(0)}'|_p^{-\Re\nu_{w_0(\iiota(0))}-(n-2\iiota(0))/2}
=|\xi_{\iiota(0)}'|_p^{-\Re\,\nu_{w_0(\iiota(a))}+\iiota(0)-n/2}.
\end{align*}}Thus, by using the relation $\# Q_{+}'+\# Q_{-}'=n-m$ on the way,  {\allowdisplaybreaks\begin{align}
&\tint_{\xi_{+}' \in \XX(y_{q'})\cap \Xi_{Q_+',p}^{w_0}(\tilde\iiota)}\exp\left(\langle H_{\sG_0}\left(\left[\begin{smallmatrix} 1_{m-1} & 0 \\ 0 & \alpha \end{smallmatrix}\right]\right), w_0^{-1}\Re \nu+\rho_{\sB_0}\rangle \right)|t|_p^{-(m-1)+\# Q_{-}'-\sum_{j=1}^{n-1}\Re \nu_j} \d \xi_{+}' 
 \label{JJ-L1-f103}
\\
&\leq 
\tint_{\xi_{+}' \in \XX(y_{q'})\cap\Xi_{Q_+',p}^{w_0}(\tilde\iiota)}
|\xi_{\iiota(0)}'|_p^{-\Re\,\nu_{w_0(\iiota(a))}+\iiota(0)-n/2}\times |\xi_{\iiota(0)}'|_p^{-(m-1)+\# Q_{-}'-\sum_{j=1}^{n-1}\Re \nu_j} \d \xi_{+}' 
 \notag
\\
&\leq \tint_{1<|\xi_{\iiota(0)}|_p\leq B_{\Vcal_2}|y_{q'}|_p^{-1}}|\xi'_{\iiota(0)}|_p^{\#Q_{+}'} \times |\xi_{\iiota(0)}'|_p^{-\Re\,\nu_{w_0(\iiota(a))}+\iiota(0)-n/2-(m-1)+\# Q_{-}'-\sum_{j=1}^{n-1}\Re \nu_j} \d \xi'_{\iiota(0)} 
 \notag
\\
&=\tint_{1<|x|_p\leq B_{\Vcal_2}|y_{q'}|_p^{-1}}|x|_p^{-z(\iiota)}\,\d x
\notag
\end{align}} with $z(\iiota):=\sum_{j=1}^{n-1}\Re\nu_j+\Re\,\nu_{w_0(\iiota(a))}-\left(\tfrac{n}{2}-2m+1\right)-\iiota(0)$. We have $\alpha=1$, $t=1$ for $\xi_{+}' \in \Z_{p}^{n-m
}\cap \Xi_{Q_+',p}^{w_0}$; thus 
$$
\tint_{\xi_{+}' \in \Z_{p}^{n-m
}\cap \Xi_{Q_+',p}^{w_0}}\exp\left(\langle H_{\sG_0}\left(\left[\begin{smallmatrix} 1_{m-1} & 0 \\ 0 & \alpha \end{smallmatrix}\right]\right), w_0\Re \nu+\rho_{\sB_0}\rangle \right)|t|_p^{-(m-1)+\# Q_{-}'-\sum_{j=1}^{n-1}\Re \nu_j} \d \xi_{+}'=1.
$$
If we set $C(\Ucal_p)=C_{\Vcal_2}$, $c(\Ucal_p):=B_{\Vcal_2}$ and 
\begin{align*}
\tilde C_p(\nu,y,w):=
&\delta\bigl(|y_{q'}|_v\leq c(\Ucal_v)\bigr)\,
\prod_{\nnu=1}^{h}|y_{w(j_\nnu+1)}|_v^{-\#J_\nnu^*(y,w)}
\biggl\{1+ \sum_{\iiota} \tint_{1<|x|_p\leq c(\Ucal_p)|y_{q'}|_p^{-1}}|x|_p^{-z(\iiota)}\,\d x\biggr\},
\end{align*}
the we have the desired inequality in the proposition for $v=p<\infty$. Note that the integrals are zero if $c(\Ucal_p)|y_{q'}|_p^{-1}<1$. Let us consider the case $v=\infty$. Set $\gamma_{j}=(1+\sum_{i=m}^{j}\xi_i^2)^{1/2}$ for $j\in [m,n-1]_{\Z}$ and $\xi_{+}'=(\xi_j)_{j=m}^{n-1}\in \Xi_{Q_{+}',\infty}^{w_0}$. Then $1<\gamma_{j-1}\leq \gamma_{j}$ for $j \in [m+1,n-1]_{\Z}$, and $\gamma_{j}=\gamma_{j-1}$ iff $\xi_{j}=0$. From Lemma~\ref{ExplicitIwasawaDecArch}, we have $\alpha=\diag(\gamma_{j-1}/\gamma_{j}|\,j\in[m,n-1]_\Z)$. Since $\xi_j=0$ if $j\not\in w_0^{-1}Q_{+}'$, we have that $\alpha_j=\gamma_{j-1}/\gamma_j \not=1$ implies $j \in w_0^{-1}Q_{+}'$. Hence, if we set $w_0^{-1}Q_+'=\{\iiota(a)<\cdots \iiota(0)\}$ (so that $a+1=\# Q_{+}'$), then by arguing in the same way as above,  
{\allowdisplaybreaks\begin{align*}
 \exp(\langle H_{\sG_0}\left(\left[\begin{smallmatrix} 1_{m-1} & 0 \\ 0 & \alpha \end{smallmatrix}\right]\right), w_0^{-1}\Re \nu+\rho_{\sB_0}\rangle) 
&\leq \prod_{j=m}^{n-1} \gamma_{\iiota(0)}^{-\Re\nu_{w_0(\iiota(a))}+\iiota(0)-n/2}.
\end{align*}}By this, in conjunction with $t=\gamma_{\iiota(0)}=(1+\|\xi_{+}'\|^2)^{1/2}$ from Lemma~\ref{ExplicitIwasawaDecArch}, we estimate  
{\allowdisplaybreaks\begin{align*}
&\tint_{\xi_{+}' \in \XX(y_{q'})}\exp\left(\langle H_{\sG_0}\left(\left[\begin{smallmatrix} 1_{m-1} & 0 \\ 0 & \alpha \end{smallmatrix}\right]\right), w_0^{-1}\Re \nu+\rho_{\sB_0}\rangle \right)|t|_\infty ^{-(m-1)+\# Q_{-}'-\sum_{j=1}^{n-1}\Re \nu_j} \d \xi_{+}' \\
&\leq \tint_{\xi_+'\in \XX(y_{q'})} (1+\|\xi_{+}'\|^2)^{-(z(\iiota)+a+1)/2}\,\d \xi_{+}' 
\end{align*}}where $z(\iiota)$ is the same number as in the $p$-adic case. If we set $C(\Ucal_\infty)=C_{\Vcal_2}$, $c(\Ucal_\infty):=B_{\Vcal_2}$ and
\begin{align*}
\tilde C_\infty(\nu,y,w):=
&\delta\bigl(|y_{q'}|_\infty\leq c(\Ucal_\infty)\bigr)\,
\prod_{\nnu=1}^{h}|y_{w(j_\nnu+1)}|_\infty^{-\#J_\nnu^*(y,w)}
\tint_{\xi_+'\in \XX(y_{q'})} (1+\|\xi_{+}'\|^2)^{-(z(\iiota)+a+1)/2}\,\d \xi_{+}' 
\end{align*}
the we have the desired inequality in the proposition. \qed

\subsection{Vanishing of local orbital integrals on primes dividing the level}\label{sec:VanishLOInt}

Let $p$ be a prime and $N\in \N$. 

\begin{lem} \label{Mar10-L1}
 Set $U_p(N)=1+N\Z_p$ if $|N|_p<1$ and $U_p(N)=\Z_p^\times$ if $|N|_p=1$. Let $f_{p}=\cchi_{\bK_1(N\Z_p)}$. Then 
$$
 \tilde{f_p}=\#(\Z_p^{\times}/U_p(N))^{-1}\,\cchi_{\sZ(\Q_p)\bK_{1}(N\Z_p)}. 
$$
\end{lem}
\begin{proof} From the definition (see \eqref{CentralProjLoc}), $\tilde f_p(g)=0$ unless $g\in \sZ(\Q_p)\,\bK_{1}(N\Z_p)$. Let $g=[u]\,k$ with $u\in \Q_p^\times$ and $k\in \bK_1(N\Z_p)$. Then $f_p([z]g)\not=0$ if and only if $[zu]\in \sZ(\Q_p)\cap \bK_{1}(N\Z_p)$. Since $\sZ(\Q_p)\cap \bK_{1}(N\Z_p)=\{[x]\mid x\in U_p(N)\,\}$ as is easily confirmed, $\tilde f_{p}(g)=\int_{\sZ(\Q_p)}f_p([z]g)\,\d^* z=\int_{\sZ(\Q_p)\cap \bK_1(N\Z_p)}\d^* z=\#(\Z_p^\times/U_p(N))^{-1}$. 
\end{proof}

%The integral $\JJ_{\cchi_{\bK_1(N\Z_p)}}^{(\nu)}(y,w)$ with $p$ being a prime number and $|N|_p<1$ has the following vanishing property.  

\begin{prop} \label{JJL-5} 
 Suppose $|N|_p<1$. Let $Q$ be a proper subset of $I^0$. Let $y\in {\mathsf Y}_Q$ and $w\in \sS_{n}(Q,y)$. Let $\nu \in \fJ((\ft_{0}^*)^{++})$.
\begin{itemize}
\item[(1)] Suppose $y\not=0$ and let $q'=q'(y)$ be the minimal element of ${\rm sp}(y)$. Then 
$$
 \JJ_{{\cchi}_{\bK_1(N\Z_p)}}^{(\nu)}(y,w)=0 \quad \text{if $y_{q'} \not \in N\Z_p$}.
$$
\item[(2)] If $y=0$, then $\JJ_{{\cchi}_{\bK_1(N\Z_p)}}^{(\nu)} (0,w)=0$. 
\end{itemize} 
\end{prop}
\begin{proof} By Lemma~\ref{Mar10-L1}, the absolute value of $\JJ_{f_p}^{(\nu)}(y,w)$ is bounded by a constant multiple of 
\begin{align}
\int_{(Z,\xi,\xi')}\int_{\sG_0(\Q_p)} \sf_{\sG_0,v}^{(\Re\,\nu)}(h)\,\cchi_{\sZ(\Q_p)\bK_{1}(N\Z_p)}(\iota(h)^{-1}\sn(y)w[w_0Zw_0^{-1}; ;w_0\left[\begin{smallmatrix} \xi \\ 0_{n-m} \end{smallmatrix}\right], w_0\left[\begin{smallmatrix} 0_{m-1} \\ \xi' \end{smallmatrix}\right]
]_{w})\,\d h\,\d Z\,\d \xi\,\d \xi',
 \label{JJ-L5-f0}
\end{align}
whose integrand is zero unless $\iota(h)^{-1}\sn(y)w[w_0Zw_0^{-1};w_0\left[\begin{smallmatrix} \xi \\ 0_{n-m} \end{smallmatrix}\right], w_0\left[\begin{smallmatrix} 0_{m-1} \\ \xi' \end{smallmatrix}\right]]_{w}\in z \bK_1(N\Z_p)$ with some $z\in \Q_p^\times$. Set $\xi'=\xi_{+}'+\xi_{-}'$ with $\xi_\pm' \in \Xi_{Q_\pm',v}^{w_0}$ and let $(\Upsilon,\ell,\alpha,t)$ be the Iwasawa data for $\xi_{+}'$ (see \eqref{IwasawaDecxi}). Since $p|N$, we have $(N\Z_p^{n-1}\bigoplus \Z_p)_{\rm prim}=N\Z_p^{n-1}\times \Z_p^{\times}$; thus, from the last statement of Corollary~\ref{tildesuppL-cor}, we may suppose that $\xi'=(\xi_j')_{j=m}^{n-1}$ satisfies 
\begin{align}
(\underbrace{0,\dots,0}_{m-1},1,\underbrace{\xi_m',\dots,\xi_{n-1}'}_{n-m}) \in t\,(N\Z_p^{n-1}\times \Z_p ^\times)
 \label{JJ-L5-f1}
\end{align}
with $m=\# Q+1$. Note that $1\leq m \leq n-1$, for $Q\not=I^{0}$; thus, by looking at the $m$-th coordinate of \eqref{JJ-L5-f1}, we always have $1\in tN\Z_p$. \\ 
(1) Let $y\not=0$, or equivalently $Q_{+}':={\rm sp}(y)\not=\emp$. Recall the function $\fm(\xi'_{+})$ defined by \eqref{fmxi-plus}. Let $X_{>1}$ (resp. $X_{\leq 1}$) denotes the set of $\xi'\in \Q_p^{n-m}$ satisfying \eqref{JJ-L5-f1} and $\fm(\xi_+')>1$ (resp. $\fm(\xi_{+}')\leq 1$). Then the integral \eqref{JJ-L5-f0} is written as a sum of two integrals corresponding to $\xi'\in X_{>1}$ or $\xi'\in X_{\leq 1}$. Suppose $\xi\in X_{>1}$ for a while, and let $j_0=j_0(\xi'_{+})\in w_0^{-1}(Q_+')$ be the largest index such that $\fm(\xi'_{+})=|\xi_{j_0}'|_{p}$. Then we can find a strictly decreasing sequence $\iiota=(\iiota(\nnu))_{\nnu=0}^{a}$ such that $\iiota(\nnu)\in w_0^{-1}(Q_+')$ $(\forall \nnu \in [0,a]_{\Z})$, $\iiota(0)=j_0$ and $\xi_{+}' \in \Xi_{Q_{+}',p}^{w_0}(\iiota)$. By Lemma~\ref{ExplicitIwasawaDec}, 
\begin{align}
\fm(\xi'_{+}):=|t|_{p}>1, \quad {}^t\ell=(\underbrace{0,\dots,0}_{j_0-m},1/t, \underbrace{0,\dots,0}_{(n-1)-j_0}), \quad \Upsilon\alpha \in {\bf M}_{n-m}(\Z_p).
\label{JJ-L5-f2}
\end{align} Note that $m\leq j_0\leq n-1$. Consider the $m$-th and the $(j_0+1)$-th coordinates of \eqref{JJ-L5-f1}. We have the two possibilities:\begin{itemize}
\item[(i)] $j_0<n-1$ and $1\in tN\Z_p$, $\xi_{j_0}'\in tN\Z_p$. 
\item[(ii)] $j_0=n-1$ and $1\in tN\Z_p$, $\xi_{j_0}'=\xi_{n-1}' \in t\Z_p^\times$.
\end{itemize}
By $\fm(\xi_{+}')=|t|_p=|\xi_{j_0}'|_{p}$, the relation $\xi_{j_0}'\in tN\Z_p$ yields $1\in N\Z_p$, which is impossible when $p|N$. Thus the case (i) does not happen and we are forced to be in the case (ii), which implies $1\in \xi_{n-1}'N\Z_p$. We re-examine the condition $t\alpha^{-1} \Upsilon^{-1}(\ell+D^{-1}D_0^{-1}\eta) \in \Vcal_2$ in the right-hand side of \eqref{JJ-L1-f2}, where we may take $\Vcal_2=\Z_p^{n-m}$ by the last part of Corollary~\ref{tildesuppL-cor}. Let $\j\in [m,n-1]_{\Z}$ be the index defined as $w(\j+1)=q'$, then by looking at the $\j$-th coordinate of the relation $D^{-1}D_0^{-1}\eta \in -\ell+t^{-1}\Upsilon \alpha\,\Z_{p}^{n-m}$ noting \eqref{JJ-L5-f2} and $\fm(\xi_{+}')=|\xi_{n-1}'|_{p}=|t|_p$ as well as $1\in \xi_{n-1}'N\Z_p$ obtained above, we have   
$$y_{q'}=y_{w(\j+1)}=\eta_{\j} \in (-\ell+t^{-1}\,\Z_{p}^{n-m})_{\j}\subset t^{-1}\Z_p=(\xi_{n-1}')^{-1}\Z_p\subset N\Z_p.$$
Here for the second equality, the remark after \eqref{w0Vectors} should be recalled. Thus the integral over $\xi'\in X_{>1}$ is zero unless $y_{q'} \in N\Z_p$. Let $\xi'\in X_{\leq 1}$, which implies the vector $\xi_{+}'$ is integral. Hence we may take $(\Upsilon,\ell,\alpha,t)=(0,0,1_{n-m},1)$. By looking at the $m$-th coordinate of \eqref{JJ-L5-f1}, we have $1\in N\Z_p$, which is impossible due to $p|N$. Hence $X_{\leq 1}=\emp$. This proves (1). \\
(2) Let us consider the case when $y=0$, i.e., $Q_{+}'={\rm sp}(y)=\emp$. Since $\xi'=\xi_{-}'$ and $\xi_{+}'=0$, we may take $(\Upsilon,\ell,\alpha,t)=(0,0,1_{n-m},1)$. Then taking the $m$-th coordinate of \eqref{JJ-L5-f1} leads us to the impossible relation $1\in N\Z_p$ as above. 
\end{proof}

\section{Global terms on the geometric side}\label{sec:global terms}
Let $N,M\in \N$, and $S_0$ a finite set of prime numbers such that the sets $S(N)$, $S_0$ and $S(M)$ are mutually disjoint. Depending on $(N,M,S_0)$, we specify our test function $f=\otimes_v f_v\in C_{\rm{c}}^\infty(\sG(\A))$ further by requiring the conditions: 
\begin{itemize}
\item[(iv)] for $p\in S_0$, $f_p \in C^\infty_{\rm c}(\bK_p\bsl \sG(\Q_p)/\bK_p)$,
\item[(v)] for $p\not\in S(M)\cup S_0$, $f_p=\cchi_{\bK_{1}(N\Z_p)}$,
\end{itemize}
as well as (ii) and (iii) in \S\ref{SPEXP}; then $f$ is bi-$\bK_0(NM)$-invariant. Hereinafter the integer $NM$ will be referred to as the level of $f$. Note that the condition (i) in \S\ref{SPEXP} is automatic from these requirements. Let $Q\subset I^0$, $y\in {\mathsf Y}_Q$, and $w\in \sS_n(Q,y)$. Let $R$ be a finite set of places and recall the notation in \S\ref{SPEXP}. For the test function $f=\otimes_v f_v$, we define functions on $\sG(\A^{R})$ by $f^{R}=\otimes_{v\not\in R}f_v$, $\widetilde{f^R}=\otimes _{v\not\in R} \tilde f_{v}$, and  
$$
\widetilde{f^{R}}^{(\nu)}(g)=(\Delta_{\sG_0}^{R}(1)^{*})^{-1} \int_{\sB_0(\A^{R})}\widetilde{f^{R}}(\iota(b)^{-1}g)\,e^{\langle H_{\sG_0}(b),\nu+\rho_{\sB_0}\rangle} \,\d_{l}b, \quad g\in \sG(\A^{R}),
$$
where $\Delta_{\sG_0}^{R}(1)^{*}$ is the residue at $s=1$ of the analytic continuation of the partial Euler product $\Delta_{\sG_0}^{R}(s)=\prod_{v\not\in R}\Delta_{\sG_0,v}(s)\,(\Re s>1)$ with $\Delta_{\sG_0,v}(s)$ being the $v$-factor of \eqref{DelsG0}, and $\d_lb=\otimes_{v\not\in R}\d_l b_v$ is the product measure on $\sB_0(\A^R)$. Set 
\begin{align}
\JJ_{f^R}^{(\nu)}(y,w)=&\int_{{\bf V}_{y,w}(\A^{R})\times (\A^{R})^{m-1} \times (\A^{R})^{n-m}}\widetilde {f^{R}}^{(\nu)}\left(\sn(y)w\bigl[w_0Zw_0^{-1}; w_0\left[\begin{smallmatrix} \xi \\ 0_{n-m} \end{smallmatrix}\right], w_0\left[\begin{smallmatrix} 0_{m-1} \\ \xi' \end{smallmatrix}\right]
\bigr]_w\right)
\label{GlobalJJnuQy}
\\
& \times \theta_{w}^{\psi^{-1}}(Z,\xi,\xi')\,\d Z\,\d \xi\,\d\xi'
 \notag
\end{align}
for $\nu \in \fI((\ft_{0}^{*})^{++})$ and $\d Z=\otimes_{v\not\in R}\d Z_v$, $\d x=\otimes_{v\not\in R}\d x_v$ and $\d x'=\otimes_{v\not\in R}\d x_v'$ the product of the local Tamagawa measures in \eqref{LocalJnuQy}. We define $\JJ_{f}^{(\nu)}(y,w)$ to be the integral \eqref{GlobalJJnuQy} with $R=\emp$. In this section, the absolute convergence of \eqref{GlobalJJnuQy} for $\nu \in \fI((\ft_{0}^{*})^{+})$ will be established; then, from \eqref{prodformulaGsect} and \eqref{LocalJnuQy}, for any finite set $R$ of places, we have the product formula
\begin{align}
(\Delta_{\sG_0}(1)^{*})\,\JJ_{f}^{(\nu)}(y,w)=(\Delta_{\sG_0}^{R}(1)^{*})\,\JJ_{f^{R}}^{(\nu)}(y,w)\times \prod_{v\in R}\JJ_{f_v}^{(\nu)}(y,w), \quad \nu \in \fI((\ft_0^*)^{++}).\label{JJproductformaula}
\end{align}

\subsection{The regular orbital integrals over adeles}\label{GSRT} 
We first examine the integrals $\JJ_{f}^{(\nu)}(y,w)$ for $y\in {\mathsf Y}_Q$ when $Q$ is a proper subset of $I^0$. 

\begin{lem} \label{MonotoneCOnv}
 Let $R$ be a finite set of places containing $\infty$, and $\phi$ a non-negative function on ${\bf V}_{y,w}(\A^{R})\times (\A^{R})^{m-1}\times (\A^{R})^{n-m}$ such that $\phi(Z,\xi,\xi')=\prod_{v\not\in R}\phi_v(Z_v,\xi_v,\xi_v')$ with a family of non-negative measurable functions $\phi_v$ on ${\bf V}_{y,w}(\Q_v)\times \Q_v^{m-1}\times \Q_v^{n-m}$ such that $\phi_p$ is identically $1$ on ${\bf V}_{y,w}(\Z_p)\times \Z_p^{m-1} \times \Z_p^{n-m}$ for almost all $p\not\in R$. Then $\phi$ is integrable on ${\bf V}_{y,w}(\A^R)\times (\A^R)^{m-1} \times (\A^{R})^{n-m}$ provided that 
\begin{itemize}
\item[(i)] $\phi_v$ is integrable for all $v\not\in R$, and 
\item[(ii)] there exists a family of positive numbers $\{\lambda_v\}_{v\not\in R}$ such that $\lambda_v\geq 1\,(v\not\in R)$,  
$$
\int_{{\bf V}_{y,w}(\Q_v)\times \Q_v^{m-1} \times \Q_v^{n-m}}\phi_v(Z_v,\xi_v,\xi_v')\,\d Z_v\,\d \xi_v\,\d \xi_v' \leq \lambda_{v}, \quad v\not\in R
$$
and the infinite product $\prod_{v\not\in R}\lambda_v$ is convergent.
\end{itemize}
\end{lem}
\begin{proof} Set ${\Xcal}={\bf V}_{y,w}\times {\rm Aff}^{m-1} \times {\rm Aff}^{n-m}$ viewed as a $\Z$-scheme. Let ${\mathbb P}$ be the finite set of finite places of $\Q$. For any $R'\in {\mathbb P}$ containing $R$, set ${\Xcal}_{R'}:=\prod_{v\in R'}\Xcal(\Q_v)\times \prod_{p\not\in R'}\Xcal(\Z_p)$; then $\{\Xcal_{R'}\}_{R\subset R'}$ is an open covering of $\Xcal(\A^{R})$. Let $R_n\,(\in {\mathbb P})$ be an increasing sequence such that $R\subset R_n$ and the union of $R_n-R$ is the complement of $R$. For any Borel set $U$ of $\Xcal(\A^R)$, let $I(U)$ denote the integral of $\phi$ on $U$. It suffices to show $I(\Xcal(\A^{R}))<\infty$. For a place $v\not\in R$, $I_v(U_v)$ denotes the integral of $\phi_v$ on a Borel set $U_v\subset \Xcal(\Q_v)$; by the non-negativity of $\phi_v$, we have $I_v(U_v)\leq I_v(\Xcal(\Q_v))\leq \lambda_v$. Since $I_p(\Xcal(\Z_p))=1$ for almost all $p\not\in R$ (see the paragraph before the formula \eqref{LocalJnuQy}), the definition of the measure on $\Xcal(\A^{R})$ shows the equality $I(\Xcal_{R_{n}})=\prod_{v\in R_n}I_v(\Xcal(\Q_v))$ for sufficiently large $n$. Since $I_p(\Xcal(\Q_p))\geq I_p(\Xcal(\Z_p))=1$ for all $p\in R_{n+1}-R_n$, we have that $I(\Xcal_{R_n})$'s form an increasing sequence for sufficiently large $n$. By assumption, the increasing sequence $\prod_{v\in R_n}\lambda_v$ has a limit, say $\lambda$, as $n\rightarrow \infty$. Since $I(\Xcal_{R_n})\leq \prod_{v\in R_n}\lambda_v\leq \lambda$ for all $n$, the increasing sequence $I(\Xcal_{R_n})$ is also convergent. By the monotone convergence theorem, we have $I(\Xcal(\A^{R}))=\lim_{n\rightarrow \infty} I(\Xcal_{R_n})$. Thus, it is concluded that $I(\Xcal(\A^R))<\infty$, i.e., $\phi$ is integrable. 
\end{proof}

\begin{prop} \label{JJL-7} 
 Suppose $Q\not=I^{0}$. 
\begin{itemize} 
\item[(1)] Let $y\in {\mathsf Y}_Q$ and $w\in \sS_n(Q,y)$. For any finite set $R$ of places, the integral $\JJ_{f^{R}}^{(\nu)}(y,w)$ converges absolutely defining a vertically bounded holomorphic function on $\fI((\ft_0^*)^{++})$; moreover, the product formula \eqref{JJproductformaula} holds. 
\item[(2)] 
There exists a constant $N_0=N_0(f_\infty, f_{M}, f_{S_0})>0$ depending only on the support of $f_\infty\otimes f_{M}\otimes f_{S_0}$ such that if the level $NM$ of $f$ is greater than $N_0M$ then $\JJ_{f}^{(\nu)}(y,w)=0$ for all $\nu\in \fI((\ft_0^*)^{++})$, $y\in {\mathsf Y}_Q$ and $w\in \sS_n(Q,y)$. 
\end{itemize}
\end{prop}
\begin{proof} Let $S(y)$ be a finite set of prime numbers such that $p\not\in S(y)$ implies $y_{i}\in \Z_p^\times \cup\{0\}$ $(i\in [1,n-1]_\Z)$. Set $S=\{\infty\}\cup S_0\cup S(MN)\cup S(y)$. Note that $|\tilde f_p^{(\nu)}(k)|=1$ ($k\in \bK_p$) and $\sn(y)w[w_0Zw_0^{-1};w_0\left[\begin{smallmatrix} \xi \\ 0_{n-m} \end{smallmatrix}\right], w_0\left[\begin{smallmatrix} 0_{m-1} \\ \xi' \end{smallmatrix}\right]]_{w} \in \bK_p$ for $p\not\in S$ and $(Z,\xi,\xi')\in {\bf V}_{y,w}(\Z_p)\times \Z_p^{m-1}  \times \Z_p^{n-m}$. Since we have the absolute convergence of $\JJ_{f_v}^{(\nu)}(y,w)$ for any $v$ in Corollary~\ref{LocalABSCONVJJ}, it suffices to show the absolute convergence of $\JJ_{f^{R}}^{(\nu)}(y,w)$ for $R$ such that $S\subset R$. Let us consider the case $y\not=0$ first. For $p\not\in R$, $f_p=\cchi_{\bK_p}$; hence, the right-hand side of the inequality in Proposition~\ref{LocalABSCONVJJ} (applied with $\Ucal_p=\bK_p$) reduces to $\lambda_p(\nu):=\prod_{\substack{1\leq i<j \leq n-1
\\ w_0^{-1}(i)>w_0^{-1}(j)}}\tfrac{\zeta_p(\Re\,\nu_i-\Re\,\nu_j)}{\zeta_p(\Re\,\nu_i-\Re\,\nu_j+1)}$, bacause $\|\tilde f_p\|_\infty=1$, $c(\Ucal_p)=C(\Ucal_p)=1$ and $\Vcal_{0,p}=\bK_{\sG_0,p}$ so that $\tilde C_{p}(\nu,y,w)=1$ and $\int_{\Vcal_{0,p}}\sf^{(w_0^{-1}\Re\nu)}_{\sG_0,p}(h)\, \d h=\vol(\bK_{\sG_0,p})=1$ for all $\nu \in \fI((\ft_{0}^*)^{++})$. Thus by Lemma~\ref{MonotoneCOnv}, the convergence of  
\begin{align*}
\int_{{\bf V}_{y,w}(\A^R)\times (\A^R)^{m-1}\times(\A^R)^{n-m}}\left|\tilde f^{(\nu)}\left(\sn(y)w\,[w_0Zw_0^{-1};
;w_0\left[\begin{smallmatrix} \xi \\ 0_{n-m} \end{smallmatrix}\right], w_0\left[\begin{smallmatrix} 0_{m-1} \\ \xi' \end{smallmatrix}\right]]_{w}\right) \right|\,\d Z\,\d\xi\,\d\xi', \end{align*}
follows from the convergence of the product $\prod_{p\not\in R}\lambda_{p}(\nu)$, which in turn results from $\Re\nu \in(\ft_{0}^*)^{++}$. Indeed $\Re \nu \in (\ft_{0}^{*})^{++}$ implies $\Re\nu_i-\Re\nu_j>1$ for all $i<j$. This settles the case $y\not=0$. The case $y=0$ is similar and easier, because $\tilde C_p(\nu, 0,w)=1$ in Proposition~\ref{LocalABSCONVJJ}. From Proposition~\ref{JJ-L1}, Proposition~\ref{JJL-5} (1) and the product formula \eqref{JJproductformaula} applied to $R=\{v\}$ with $v$ varied, we have a constant $D(f_M,f_{S_0})=\prod_{p\in S(M)\cup S_0}p^{d_p}\in \N$ depending only on the support of $f_Mf_{S_0}$ and a constant $c(f_\infty)>1$ depending only on the support of $f_\infty$ such that for any $N$ with $(D(f_M,f_{S_0}),N)=1$ we have $\JJ_{f}^{(\nu)}(y,w)=0$ for $y=(y_j)_{j}\in {\mathsf Y}_{Q}-\{0\}$ unless $y_{q'}$ belongs to the set 
\begin{align}
\{t\in \Q-\{0\}\mid \text{$t\in N\,D(f_M,f_{S_0})^{-1}\,\Z$ and $|t|_\infty\leq c(f_\infty)$}\,\}.
\label{Mar10-0}
\end{align} If $N>N_0(f_\infty,f_M,f_{S_0}):=D(f_M,f_{S_0})\times c(f_\infty)$, we have $\JJ_{f}^{(\nu)}(y,w)=0$ because, for such $N$, the set \eqref{Mar10-0} is empty. If $N>N_0(f_\infty, f_M, f_{S_0})$, then we have a prime $p$ such that $p|N$; thus $\JJ_{f}^{(\nu)}(0,w)=0$ follows from \eqref{JJproductformaula} and Proposition~\ref{JJL-5} (2).
\end{proof}

Recall the smoothed orbital integrals $\JJ_{f,\b}(\sn(y)w)$, which appear in \eqref{GeoSideP1}.   
\begin{prop} \label{ErrorL1} 
 Suppose $Q\not=I^{0}$. Let $y\in {\mathsf Y}_Q$ and $w\in \sS_{n}(Q,y)$. Then for any $\s\in (\ft_{0}^*)^{++}$, 
\begin{align}
\JJ_{f,\beta}(\sn(y)w)= \int_{\fJ(\s)}\JJ_{f}^{(\nu)}(y,w)\beta(\nu)\,M_{\sG_0}(\nu)\,\d \nu, \quad \b\in {\mathcal B}_0.
 \label{ErrorL1-f1}
\end{align}
There exists a constant $N_0$ depending only on the support of $f_{\infty}\otimes f_{M} \otimes f_{S_0}$ such that $\JJ_{f,\beta}(\sn(y)w)=0$ for all $y\in {\mathsf Y}_Q$ and $w\in \sS_n(Q,y)$ if $f$ is of level $NM>N_0M$. 
\end{prop}
\begin{proof} Recall the formula \eqref{tildefB}. Then the assertions follow from \eqref{GlobalJbetaQy}, \eqref{GlobalJJnuQy} and Proposition~\ref{JJL-7} by Fubini's theorem. 
\end{proof}

\subsection{The singular orbital integrals over adeles}\label{GSST}
In this subsection, we examine the singular terms appearing in \eqref{GeoSideP1}, i.e., the terms $\JJ_{f,\b}(\sn(y)w)$ for $y\in {\mathsf Y}_Q$ with $Q=I^{0}$. From Definition \ref{Def1} and Remark after that, the equality $Q=I^0$ yields $m(Q)=n$, $y=0$ and $\sS_n(Q,y)=\{w_\ell^{0}\}$, where $w_\ell^0$ is the longest element of $\sS_{n-1}$. Let $v$ be a place. From Lemma~\ref{testftnconv}, we have that the singular local orbital integral $\JJ_{f_v}^{(\nu)}(0,w_\ell^{0})$ (originally defined by \eqref{LocalJnuQy}) equals
\begin{align}
&\int_{h\in\sG_0(\Q_v)}\int_{Z\in \sU_0(\Q_v)}\int_{x\in \Q_v^{n-1}} \sf_{\sG_0,v}^{(\nu)}(w_\ell^{0} h)\,\tilde f_v\left(\iota(h)^{-1} \iota(Z)\,\sn(x)\right)\,\psi_{v}\Bigl(\sum_{j=1}^{n-2}Z_{j,j+1}+x_{n-1}\Bigr)^{-1}\, \d h\,\d Z\,\d x.
 \label{GSST-f0}
\end{align}
To compute this, we consider the integral \eqref{LocalRSint} for any smooth $\psi_{v}^{-1}$-Whittaker function $W$ which does not necessarily belong to an irreducible representation. Set
\begin{align}
 \Wcal_v^{\psi_v}({\tilde f_v};g_v)=\int_{\sU(\Q_v)}\tilde f_v(ug_v)\,\psi_{\sU,v}(u)^{-1}\,\d u, \quad g_v\in \sG(\Q_v), 
\label{Wcalpsi}
\end{align} 
where $\psi_{\sU,v}$ is the character of $\sU(\Q_v)$ defined by $\psi_{\sU,v}(u)=\psi_v(\sum_{j=1}^{n-1}u_{j\,j+1})$ for $u=(u_{ij})_{ij}\in \sU(\Q_v)$. Since $f_v \in C_{\rm{c}}^\infty(\sG(\Q_v))$, from the Iwasawa decomposition of $\sG(\Q_v)$, it follows that the integral \eqref{Wcalpsi} is absolutely convergent defining a $\psi_{v}$-Whittaker function of compact support modulo $\sU(\Q_v)$ on $\sG(\Q_v)$. Recall the formulas \eqref{f-MsG0v} for $M_{\sG_0,v}(\nu)$, \eqref{normJInt} for $W_{\sG_0(\Q_v)}^{0}(\nu)$ and \eqref{LocalRSint} for the local zeta-integal. 
\begin{lem} \label{GSST-L1}
For $\nu \in \fJ((\ft_{0}^*)^{++})$, we have  
\begin{align*}
\JJ_{f_v}^{(\nu)}(0,w_\ell^0)=M_{\sG_0,v}(\nu)^{-1}\,Z\left(\tfrac{1}{2};W_{\sG_0(\Q_v)}^{0}(\nu) \otimes \Wcal_v^{\psi_v^{-1}}(\widetilde{\check f_v})\right).
\end{align*}
\end{lem}
\begin{proof} Since $\Wcal_v^{\psi_v^{-1}}(\widetilde{\check f_v})$ is of compact support modulo $\sU(\Q_v)$ on $\sG(\Q_v)$, the integral \eqref{LocalRSint} is convergent absolutely for all $z\in \C$. We have 
{\allowdisplaybreaks
\begin{align*}
&Z\left(\tfrac{1}{2};J_{\sG_0(\Q_v)}^{\psi_v}(\nu) \otimes \Wcal_v^{\psi_v^{-1}}(\widetilde{\check f_v})\right) \\
&=\int_{\sU_0(\Q_v)\bsl \sG_0(\Q_v)}
J_{\sG_0(\Q_v)}^{\psi_v}(\nu;h) \,\Wcal_v^{\psi_v^{-1}}(\widetilde{\check f_v};\iota(h))\,\d h
\\
&=\int_{\sU_0(\Q_v)\bsl \sG_0(\Q_v)} \int_{\sU_0(\Q_v)} \int_{\sU(\Q_v)} \sf_{\sG_0,v}^{(\nu)}(w_\ell^0 u_0 h)\,\psi_{\sU_0,v}(u_0)^{-1}\,\tilde f_v(\iota(h)^{-1} u)\psi_{\sU,v}(u)^{-1}\,\d h\,\d u_0\,\d u
\\
&=\int_{\sU_0(\Q_v)\bsl \sG_0(\Q_v)} \int_{\sU_0(\Q_v)} \int_{\sU_0(\Q_v)}\int_{\Q_v^{n-1}} \sf_{\sG_0,v}^{(\nu)}(w_\ell^0 u_0 h)\psi_{\sU_0,v}(u_0)^{-1}\,\tilde f_v(\iota(h)^{-1} \iota(u_1)\sn(x))
\\
&\quad \times \psi_{\sU_0,v}(u_1)^{-1}\,\psi_v(x_{n-1})^{-1}\,\d h\,\d u_0\,\d u_1\,\d x.\end{align*}}By the variable change $u_0 \rightarrow u_0u_1^{-1}$, this becomes{\allowdisplaybreaks \begin{align*}
&\int_{\sU_0(\Q_v)\bsl \sG_0(\Q_v)}\int_{\sU_0(\Q_v)} \int_{\sU_0(\Q_v)}\int_{\Q_v^{n-1}} \sf_{\sG_0,v}^{(\nu)}(w_\ell^0 u_0 u_1^{-1} h)\psi_{\sU_0,v}(u_0)^{-1}\\
&\qquad \times \tilde f_v(\iota(u_1^{-1} h)^{-1}\sn(x))\psi_v(x_{n-1})^{-1}\,\d h\,\d u_0\,\d u_1\,\d x \\
&=\int_{\sG_0(\Q_v)}\int_{\sU_0(\Q_v)} \int_{\Q_v^{n-1}} \sf_{\sG_0,v}^{(\nu)}(w_\ell^0 u_0 h)\,\psi_{\sU_0,v}(u_0)^{-1}\,\tilde f_v(\iota(h)^{-1} \sn(x))\psi_v(x_{n-1})^{-1}\,\d h\, \d u_0\,\d x\\
&=\int_{\sG_0(\Q_v)}\int_{\sU_0(\Q_v)} \int_{\Q_v^{n-1}} \sf_{\sG_0,v}^{(\nu)}(w_\ell^0 
h)\,\tilde f_v(\iota(h)^{-1} \iota(u_0)\,\sn(x))\,\psi_{\sU_0,v}(u_0)^{-1}\,\psi_v(x_{n-1})^{-1}\,\d h\,\d u_0\,\d x
\\
&=\JJ_{f_v}^{(\nu)}(0,w_\ell^0).
\end{align*}}By $W_{\sG_0(\Q_v)}^{0}(\nu)=M_{\sG_0,v}(\nu)\,J_{\sG_0(\Q_v)}^{\psi_v}(\nu)$, we are done. 
\end{proof}

The singular orbital integrals over almost all finite places are computed as follows. Recall the subgroups $U_p(N)\subset \Z_p^\times$ defined in Lemma~\ref{Mar10-L1}. 
\begin{lem} \label{GSST-L2}
Let $p\not \in S_0\cup S(M)$ be a prime. Then on the region $\fJ((\ft_{0}^*)^{++})$, 
\begin{align*}
\JJ_{f_p}^{(\nu)}(0,w_\ell^{0})=\#(\Z_p^\times/U_p(N))^{-1} \times {M}_{\sG_0,p}(\nu)^{-1}. \end{align*}
\end{lem}
\begin{proof} We start with the formula obtained from \eqref{GSST-f0} by the variable change $h\mapsto Zh$. From Lemma~\ref{Mar10-L1}, we have $\tilde f_p\left(\left[\begin{smallmatrix}h^{-1} & h^{-1} x \\ 0 & 1  \end{smallmatrix}\right]  \right)=\#(\Z_p^\times/U_p(N))^{-1}\,\delta(h\in \bK_{\sG_0,p},\,x\in \Z_p^{n-1})$. Thus $\#(\Z_p^\times/U_p(N))^{-1}\times \JJ_{f_p}^{(\nu)}(0,w_\ell^0)$ equals{\allowdisplaybreaks\begin{align*}
&\int_{h\in \bK_{\sG_0,p}} \int_{Z\in \sU_0(\Q_p)} \int_{x\in \Z_p ^{n-1}} \sf_{\sG_0,p}^{(\nu)}(w_\ell^{0} Z h)
\,\psi_{p}\Bigl(\sum_{j=1}^{n-2}Z_{j j+1}+x_{n-1}\Bigr)^{-1}\, \d h\, \d Z\,\d x\\
&=\vol_{\sG_0(\Q_p)}(\bK_{\sG_0,p})\,\int_{\sU_0(\Q_p)} \sf_{\sG_0,p}^{(\nu)}(w_\ell^{0}Z)\,\psi_p \Bigl(\sum_{j=1}^{n-2}Z_{j j+1}\Bigr)^{-1} \d Z
\\
&=J_{\sG_0(\Q_p)}^{\psi_p}(\sf_{\sG_0,p}^{(\nu)};1_{n-1})
=M_{\sG_0,p}(\nu)^{-1}\,W_{\sG_0(\Q_p)}^{0}(\nu;1_{n-1}).
\end{align*}}
We are done by $W_{\sG_0(\Q_p)}^{0}(\nu;1_{n-1})=1$ (Lemma~\ref{HolWhitt}).
\end{proof}

Since the infinite product $\prod_{p<\infty}M_{\sG_0,p}(\nu)^{-1}$ is seen to be absolutely convergent when $\nu \in \fJ((\ft_{0}^{*})^{++})$, we have the absolute convergence of the integral $\JJ_{f}^{(\nu)}(0,w_\ell^0)$ (defined by \eqref{GlobalJJnuQy}) together with the infinite product expression
\begin{align}
\JJ_{f}^{(\nu)}(0,w_\ell^0)=(\Delta_{\sG_0}(1)^{*})^{-1}\,\prod_{v}\JJ_{f_v}^{(\nu)}(0,w_\ell^{0}), \quad \nu \in \fJ((\ft_{0}^*)^{++}).
 \label{JJproductformaula-sing}
\end{align}

Define
\begin{align}
{\bf J}_f^{\circ}(\nu)=
(\Delta_{\sG_0}(1)^{*})^{-1}\prod_{v \in \{\infty\} \cup S(M) \cup S_0}
Z\left(\tfrac{1}{2};W_{\sG_0(\Q_v)}^{0}(\nu)\otimes \Wcal_v^{\psi_v^{-1}}(\widetilde{\check f_v})\right). 
 \label{Jcircnu}
\end{align} 

Recall the smoothed orbital integral $\JJ_{f,\b}(w_\ell^0)$, the term with $(y,w)=(0,w_\ell^0)$ in the formula \eqref{GeoSideP1}. Let $\varphi(N)$ be the Euler totient function. 
\begin{prop} \label{MaintermP1}
We have $\JJ_{f}^{(\nu)}(0,w_\ell^0)={M}_{\sG_0}(\nu)^{-1}\,\varphi(N)^{-1}\,{\bf J}_f^{\circ}(\nu)$ for all $\nu \in\fJ((\ft_{0}^*)^{++})$, and 
\begin{align}
\JJ_{f,\b}(w_\ell^0)=\varphi(N)^{-1}\,\int_{\fJ(\s)}\b(\nu)\,{\bf J}_f^\circ(\nu)\,\d \nu, \quad \b\in {\mathcal B}_0
 \label{MaintermP1-f}
\end{align}
for any $\sigma\in (\ft_{0}^*)^{++}$. The function ${\bf J}_f^{\circ}(\nu)$ is holomorphic on $\ft_{0,\C}^*$ and is vertically of moderate growth over $\ft_{0}^*$. It satisfies the functional equations ${\bf J}_{f}^{\circ}(w\nu)={\bf J}_f^{\circ}(\nu)$ for $w\in \sS_{n-1}$. 
\end{prop}
\begin{proof}
The first two assertions follow from Lemmas~\ref{GSST-L1}, \ref{GSST-L2}, and the product formula \eqref{JJproductformaula-sing}. Since $\Wcal_v^{\psi_v^{-1}}(\widetilde{\check f_v})$ is of compact support modulo $\sU(\Q_v)$ on $\sG(\Q_v)$,\\ the integral $Z\left(\tfrac{1}{2};W_{\sG_0(\Q_v)}^{0}(\nu)\otimes \Wcal_v^{\psi_v^{-1}}(\widetilde{\check f_v})\right)$ is convergent for all $\nu\in \C$ and defines a holomorphic function on $\ft_{0,\C}^*$ which is vertically of moderate growth from Lemma~\ref{HolWhitt}. The functional equation is evident from \eqref{Jcircnu} and \eqref{FEqWhitt}. 
\end{proof}

\subsection{Relative trace formulas} \label{sec: sumformula}
 An upshot of Proposition~\ref{SpectExpPerL-3}, \eqref{GeoSideP1}, Propositions~\ref{ErrorL1} and \ref{MaintermP1} is the identity \eqref{RTFbeta} stated in the following theorem, where we record all the conditions on the test function $f$ required so far, i.e., (ii), (iii) from \S\ref{SPEXP} and (iv), (v) from \S\ref{sec:global terms}, for convenience.

\begin{thm} \label{preRTF}
 Let $N$ and $M$ be positive integers and $S_0$ a finite set of prime numbers such that $S(N)$, $S(M)$ and $S_0$ are disjoint from each other. Let $\{\tau_p\}_{p\in S(M)}$ be a family of irreducible smooth supercuspidal representations $\tau_p$ of $\sG(\Q_p)$ with trivial central characters. Depending on $N$, $\{\tau_p\}_{p\in S(M)}$ and $S_0$, we take a function $f=\otimes_{v}f_v\in C_{\rm{c}}^\infty(\sG(\A))$ whose factors $f_v\in C_{\rm c}^{\infty}(\sG(\Q_v))$ satisfies the following conditions: 
\begin{itemize}
\item For $v\in S_0\cup \{\infty\}$, $f_v \in C_{\rm c}^{\infty}(\bK_v\bsl \sG(\Q_v)/\bK_v)$.  
\item For $p\in S(M)$, $\tilde f_{p}(g)=\int_{\Q_p^{\times}}f([z]g)\d^* z$ coincides with a matrix coefficient of $\tau_p$. 
\item For $p\not\in S(M)\cup S_0$, $f_p=\cchi_{\bK_{1}(N\Z_p)}$. 
\end{itemize}
Let ${\bf I}_{f}(\nu)$ and ${\bf J}_f^{\circ}(\nu)$ be as in \S\ref{sect:spectralside} and \S\ref{GSST}, respectively. Then \begin{align}
\int_{\fJ(\s)}\b(\nu)\,{\bf I}_f(\nu)\,\d \nu
&=\varphi(N)^{-1}\int_{\fJ(\s)}\b(\nu)\,{\bf J}_f^\circ(\nu)\,\d \nu
+\sum_{Q\subsetneq I^0} \sum_{y\in {\mathsf Y}_Q}\sum_{w\in \sS_n(Q,y)}\,\int_{\fJ(\s)}\JJ_{f}^{(\nu)}(y,w)\beta(\nu)\,M_{\sG_0}(\nu)\,\d \nu
 \label{RTFbeta}
\end{align}
for any $\beta\in {\mathcal B}_0$ and $\sigma \in (\ft_{0}^*)^{++}$ with the summation over $(Q,y,w)$ being absolutely convergent, $M_{\sG_0}(\nu)$ defined by \eqref{NormalFact} is the normalizing factor of the spherical Eisenstein series on $\sG_0$, and $\JJ_{f}^{(\nu)}(y,w)$ is the orbital integral defined by \eqref{GlobalJJnuQy} and \eqref{JJproductformaula}.    
\end{thm}
In what follows, we only consider the case when $N$ is sufficiently large so that all the terms except the first one on the right-hand side of \eqref{RTFbeta} are zero (see Proposition~\ref{ErrorL1}). Then the function $\beta$ for smoothing is removed as follows. 

\begin{thm}\label{RTF}
Let $N$ be a positive integer relatively prime to $M\prod_{p\in S_0}p$ and 
$$
f=f_{N}^0 \otimes f_{M}\otimes f_{S_0}\otimes f_{\infty}, \quad f_{N}^0=\otimes_{p\not\in S_0\cup S(M)}\cchi_{\bK_1(N\Z_p)},
$$
where $f_{M}=\otimes_{p\in S(M)}f_p$ and $f_{S_0}=\otimes_{p\in S_0}f_p$ and $f_\infty$ are functions fixed as in Theorem~\ref{preRTF}. Then, there exists a constant $N_0=N_0(f_{M},f_{S_0},f_{\infty})$ depending only on the support of $f_{M} \otimes f_{S_0} \otimes f_\infty$ such that for any $\nu \in \ft_{0,\C}^*$, 
\begin{align}
{\bf I}_{f}(\nu)=\varphi(N)^{-1}\,{\bf J}_f^{\circ}(\nu), \qquad N>N_0.
\label{RTF-f1}
\end{align}
\end{thm}
\begin{proof}
 Let $N_0$ be the constant in Proposition~\ref{ErrorL1} and suppose $N>N_0$. Then \eqref{RTFbeta} becomes  
\begin{align*}
\int_{\fJ(\s)}\b(\nu)
\{ {\bf I}_{f}(\nu)-\varphi(N)^{-1}{\bf J}_f^{\circ}(\nu)\}\d \nu=0, \quad \b\in \fB_0, \quad \s \in \fI((\ft_{0}^*)^{++}),
\end{align*}
and the function ${\bf I}_{f}(\nu)-\varphi(N)^{-1}{\bf J}_f^{\circ}(\nu)$ is holomorphic on $\ft_{0,\C}^*$ vertically of moderate growth over $\ft_{0}^*$. We have the desired identity for $\nu\in \fI((\ft_0^*)^{++})$ by Lemma~\ref{Jequal0} below. Then the identity \eqref{RTF-f1} holds on the whole space $\ft_{0,\C}^*$ by analytic continuation. 
\end{proof}

\begin{lem}\label{Jequal0} 
Let $J(\nu)$ a holomorphic function on $\fJ((\ft_{0}^*)^{++})$ which is vertically of moderate growth. If $\int_{\fJ(\s)}\beta(\nu)J(\nu)\,\d \nu=0$ for all $\beta\in \fB_0$ and for all $\s\in (\ft_0^*)^{++}$, then $J(\nu)=0$ for all $\nu\in \fJ((\ft_{0}^{*})^{++})$.
\end{lem}
\begin{proof} Let $\s\in (\ft_0^*)^{++}$. Any function $\beta(\nu)$ of the form $P(\nu-\s)\,e^{T(\nu-\s, \nu-\s)}$ with $T>0$ and $P(\nu)\in \C[\ft_{0,\C}]$ belongs to the space $\fB$. From the assumption, we have $\int_{\ft_{0}^*}J(\s+\ii t)r(\s+\ii t)\,P(\ii t)e^{-T (t, t)}\,\d t=0,
$ for all $T>0$ and for all $P\in \C[\ft_{0,\C}]$. Since the functions $Q(t)e^{-T(t,t)}$ $(Q\in \C[\ft_{0,\C}],\,T>0)$ are everywhere dense in the space of Schwartz functions on $\ft_{0}^*$, this implies $J(\s+\ii t)r(\s+\ii t)=0$ for all $t\in \ft_{0}^*$. Note that $r(\nu)\not=0$ on $(\ft_{0}^*)^{++}$ by \eqref{singpolynom}.
\end{proof}

%\noindent
%{\bf Remark}: We expect that there exists a domain $\cD'\subset \ft_{0}^*\cap \fD$ such that the $y$-summation and the $\nu$-integral for $\sigma \in \cD'$ on the right-hand side of \eqref{RTFbeta} can be interchanged for general $N$ so that the function $\beta$ can be removed from \eqref{RTFbeta} by the same argument as in Lemma~\ref{RTF}. This is surely the case if we could establish a bound of the form $\sum_{y\in {\mathsf Y}_Q}\sum_{w\in \sS_n(Q,y)}|\JJ_{f}^{(\nu)}(y,w)|\ll_{Q}(1+\|\Im(\nu)\|)^{d}$ $(\nu \in \fJ(\cD'))$ for some $d\in \R$.   

\section{Local non-vanishing results} \label{NVR}
For $f_v\in C_{\rm c}^{\infty}(\sG(\Q_v))$, recall that $\Wcal_v^{\psi_v^{-1}}(\widetilde{\check f_v})$ is a $\psi_v^{-1}$-Whittaker function on $\sG(\Q_v)$ defined by \eqref{Wcalpsi}, where $\widetilde{\check f_v}$ in turn is defined by the integral \eqref{CentralProjLoc} for $\check f_v(g)=f_v(g^{-1})$. In this section, we further compute the local factors $Z\bigl(z;W_{\sG_0(\Q_v)}^0(\nu)\otimes \Wcal_v^{\psi_v^{-1}}(\widetilde{\check f_v})\bigr)$ of ${\bf J}_f^\circ(\nu)$ in terms of the Rankin-Selberg $L$-functions to show its non-vanishing for a properly chosen $f_v$ for $v\in S(M)\cup S_0\cup\{\infty\}$. Over the $p$-adic field, we use the notion of the stable integral on $\sU(\Q_p)$ due to \cite{LapidMao}. Let $\pi\in \widehat{\sG(\Q_p)}_{\rm gen}$. From \cite[Proposition 2.3]{LapidMao}, for any $W,\,W'\in \WW^{\psi_p}(\pi)$ and $g_p\in \sG(\Q_p)$, the function $u\mapsto \varphi(u):=\langle R(g_p^{-1}u)W|W'\rangle_{\cP(\Q_p)}\,\psi_{\sU,p}(u)^{-1}$ on $\sU(\Q_p)$ has a stable integral $\int_{\sU(\Q_p)}^{\rm st} \varphi(u)\,\d u$. It should be noted that our measure on $\Pcal(\Q_p)$ (fixed in \S\ref{sec:MBS}) is $\Delta_{\sG_0,p}(1)$ times the measure used in \cite{LapidMao}; hence the pairing \eqref{WhitProd} coincides with the pairing $[W,\bar W']_{p}$ defined on \cite[p. 477]{LapidMao} multiplied by the factor $\Delta_{\sG,p}(1)=\Delta_{\sG_0,p}(1)\times \zeta_p(n)$. After this adjustment, \cite[Lemma 4.4]{LapidMao} yields the formula 
\begin{align}
\int_{\sU(\Q_p)}^{\rm st} \langle R(u) W|W'\rangle_{\Pcal(\Q_p)}\,\psi_{\sU,p}(n)^{-1}\,\d n=\Delta_{\sG,p}(1)\,W(1_n)\,\bar W'(1_n), \quad W,\,W'\in \WW^{\psi_p}(\pi).
 \label{LapidMaoF}
\end{align}

\subsection{Local non-vanishing over $S(M)$} \label{LocalNVSM}
Recall that $\tau_p$ for $p\in S(M)$ is an irreducible smooth supercuspidal representation of $\sG(\Q_p)$ with the trivial central character such that $\tau_p$ has the conductor $M\Z_p$. As such, $\tau_p$ is generic (\cite[Theorem B]{GelfandKazhdan2}, \cite[\S5.18]{BernsteinZelevinskii}) and is unitarizable. Let $W_p$ be the essential vector in $\WW^{\psi_p}(\tau_p)$ (\cite[Th\'{e}or\`{e}me (4.1)]{JPSS1981}, \cite{Matringe}, \cite{Jacquet3}). 
%Note that the original proof of \cite[Th\'{e}or\`{e}me (4.1)]{JPSS1981}, which shows the existence of the essential vectors for a general smooth irreducible generic representations, contains a gap, which was later corrected by \cite{Jacquet3} and independently by \cite{Matringe}. 
The essential vector $W_p$ has the property that $W_p$ generates the $1$ dimensional space $\WW^{\psi_p}(\tau_p)^{\bK_{1}(M\Z_p)}$ (\cite[Th\'{e}or\`{e}me (5.1)]{JPSS1981}) and that $W_p(1_n)=1$ (\cite[Theorem 3.1]{Matringe}, \cite[Corollary 4.4]{Miyauchi}).

\begin{lem} \label{NV-SC}
Let $p\in S(M)$ and $W_p$ the essential vector in $\WW^{\psi_p}(\tau_p)$. Then $W_p(1_n)=1$ and 
$$Z(z;W_{\sG_0(\Q_p)}^{0}(\nu)\otimes \bar W_p)=\prod_{j=1}^{n-1} L\left(z+\nu_j, \bar \tau_p\right) \quad \text{ for all $z\in \C$ and $\nu=(\nu_j)_{j=1}^{n-1}\in \ft_{0,\C}^*$}.$$ 
\end{lem}
\begin{proof} The last equation is a part of the defining property of the essential vector. 
\end{proof}

\begin{cor} \label{ZWW=1}
Let $W_p \in \WW^{\psi_p}(\tau_p)$ be the essential vector, and set 
\begin{align*}
\phi_p(g_p)=\Delta_{\sG,p}(1)^{-1} \langle R(g_p)W_p|W_p\rangle_{\cP(\Q_p)}, \quad g_p\in \sG(\Q_p).
\end{align*}
Let $f_p$ be any function from $C_{\rm c}^\infty(\bK_1(M\Z_p)\bsl \sG(\Q_p)/\bK_1(M\Z_p))$ such that $\tilde f_p=\phi_p$. Then $f_p$ satisfies the conditions (iii) in \S\ref{SPEXP} and 
$$
Z\left(\tfrac{1}{2};W_{\sG_0(\Q_p)}^0(\nu)\otimes \Wcal_p^{\psi_p^{-1}}(\widetilde{\check f_p})\right)=1. 
$$
\end{cor}
\begin{proof} From \eqref{LapidMaoF}, $\Wcal_p^{\psi_p^{-1}}(\widetilde{\check f_p}; g_p)=\bar W_p(g_p)\,W_p(1_n)$. Thus $Z\left(\tfrac{1}{2};W_{\sG_0(\Q_p)}^0(\nu)\otimes \Wcal_p^{\psi_p^{-1}}(\widetilde{\check f_p})\right)=\prod_{j=1}^{n-1} L\left(\tfrac{1}{2}+\nu_j, \bar \tau_p\right)$. Since $\bar \tau_p$ is supercuspidal and $n>1$, $L(z,\bar \tau_p |\,|_v^{\nu_j})$ is identically $1$ (\cite[Proposition (8.1)]{JPSS1983}). 
\end{proof}

\subsection{Local non-vanishing over $S_0$} \label{LNVS0}
Let $p\in S_0$ and $\Omega_p^{(s)}$ be the bi-$\bK_p$-invariant matrix coefficient of $I_{p}^{\sG}(s)$ such that $\Omega_p^{(s)}(1_n)=1$. Then $\widehat f_p(s)=\int_{\sG(\Q_p)}f_p(g)\,\Omega_p^{(s)}(g)\,\d g$ for $f_p \in C_{\rm c}^{\infty}(\bK_p\bsl \sG(\Q_p)/\bK_p)$. Let $W_{\sG(\Q_p)}^{0}(s)$ be the normalized unramified $\psi_p$-Whittaker function defined in \S\ref{Jacquetint}. 

\begin{lem} \label{StableIntMatCo}
For any $g_p\in \sG(\Q_p)$ and $s\in \ii\ft^*$,  
\begin{align}
\int_{\sU(\Q_p)}^{\rm st} \Omega_p^{(s)}(g_p^{-1} u)\psi_{\sU,p}(u)^{-1}\d u
=\Delta_{\sG,p}(1)\,L(1,I_{p}^{\sG}(s)\times I_{p}^{\sG}(-s))^{-1}\,\overline{W_{\sG(\Q_p)}^{0}(s;g_p)}.
\label{StableIntMatCo-0}
\end{align}
\end{lem}
\begin{proof} From \eqref{LapidMaoF}, 
\begin{align}
\int_{\sU(\Q_p)}^{\rm st} \langle R(g_p^{-1}u)W_{\sG(\Q_p)}^{0}(s)|W_{\sG(\Q_p)}^{0}(s)\rangle_{\cP(\Q_p)}\,\psi_{\sU,p}(u)^{-1}\,\d u=\Delta_{\sG,p}(1)\,W_{\sG(\Q_p)}^{0}(s;1_n)\,\overline{W_{\sG(\Q_p)}^{0}(s;g)}. 
\label{StableIntMatCo-1}
\end{align} 
We have 
\begin{align}
\langle R(g_p) W_{\sG(\Q_p)}^{0}(s)|W_{\sG(\Q_p)}^{0}(s)\rangle_{\cP(\Q_p)}=
L\left(1,I_p^{\sG}(s)\times I_p^{\sG}(-s)\right)\,\Omega_p^{(s)}(g_p), \quad g_p\in \sG(\Q_p).
 \label{StableIntMatCo-3}
\end{align}
Indeed, since the left-hand side, viewed as a function in $g_p$, is the spherical matrix coefficient of $I_p^{\sG}(s)$, it coincides with $\Omega_p^{(s)}(g_p)$ up to a constant; the constant is determined by comparing the values at $1_{n}$ which is known by \eqref{normLocalUrWhitt}. By \eqref{StableIntMatCo-1} and \eqref{StableIntMatCo-3}, we are done. 
\end{proof}

Define a Haar measure on $\fX^0_p(1):=\{s\in(\R/2\pi(\log p)^{-1}\Z)^{n}|\,\sum_{j=1}^{n}s_j=0\}$ as 
$$\d_0s =(2\pi)^{-(n-1)}(\log p)^{n-1}\,\prod_{j=1}^{n-1}\d s_j, \quad s=(s_j)_{j=1}^{n}\in \fX_v^0(1), 
$$   
where $\d s_j$ is the Lebesgue measure on $\R$. Note that the total volume of $\fX_p^{0}(1)$ is $1$. 

\begin{prop} \label{LNVpadicL1}
 Let $p\in S_0$ and $f_p \in C_{\rm c}^{\infty}(\bK_p\bsl \sG(\Q_p)/\bK_p)$. If $\Re \nu_j>-1/2$ $(j \in [1,n-1]_{\Z})$, 
{\allowdisplaybreaks\begin{align*}
Z\left(\tfrac{1}{2},W_{\sG_0(\Q_p)}^0(\nu)\otimes \Wcal_p^{\psi_p^{-1}}(\widetilde{\check f_p})\right)
=\int_{\fX_p^{0}(1)} 
\widehat{f_p}(-\ii s)\,\frac{\prod_{j=1}^{n-1} L\left(\tfrac{1}{2}+\nu_j, I_p^{\sG}(-\ii s)\right)}{\zeta_p(1)L(1,I_p^{\sG}(\ii s);\Ad)}\,\Delta_{\sG,p}(1)\,\d\mu^{\rm Pl}(s),
\end{align*}}where 
$$\d\mu^{\rm Pl}(s)=\frac{\zeta_{p}(1)^{n}\Delta_{\sG,p}(1)^{-1}}{n!}|c(\ii s)|^{-2}\d_{0}s
\quad\text{with} \quad c(s)=\prod_{1\leq i<j\leq n}\frac{\zeta_{p}(s_i-s_j)}{\zeta_{p}(s_i-s_j+1)}.
$$
\end{prop}
\begin{proof}
By \cite[Theorem 4.7]{Tadic} applied to ${\PGL}_n(\Q_p)$, $\tilde f_p(g_p)=\int_{\fX_v^{0}(1)} \widehat {f_p }(\ii s) \Omega_{p}^{(-\ii s)}(g_p)\,\d\mu^{\rm Pl}(s)$. Note $Q_p=\sum_{w\in \sS_n}p^{-l(w)}=\prod_{j=1}^{n}\frac{p^{-j}-1}{p^{-1}-1}=\zeta_p(1)^n\Delta_{\sG,p}(1)^{-1}$. For $g_p\in \sG(\Q_p)$, there exists an open compact subgroup $\Ucal_0\subset \sU(\Q_p)$ such that the stable integral on the left-hand side of \eqref{StableIntMatCo-0} reduces to the usual integral of $\Omega_p^{(-\ii s)}(g_p^{-1}u)\psi_{\sU,p}(u)^{-1}$ over $\Ucal$ for all open compact subgroups $\Ucal\subset \sU(\Q_p)$ containing $\Ucal_0$. Hence for a sufficiently large open compact subgroup $\Ucal\subset \sU(\Q_p)$ containing $\Ucal_0$ and the support of $u\mapsto \tilde f_p(g_p^{-1}u)$, by changing the order of integrals, we have
{\allowdisplaybreaks\begin{align*}
\Wcal_p^{\psi_p^{-1}}(\widetilde{\check f_p};g_p)&= \int_{\Ucal}\tilde f_p(g_p^{-1}u)\,\psi_{\sU,p}(u)^{-1}\,\d u 
\\
&=\int_{\Ucal} \biggl\{\int_{\fX_p^0(1)} {\widehat {f_p}} (\ii s) \Omega_{p}^{(-\ii s)}(g_p^{-1}u)\,\d\mu^{\rm Pl}(s) \biggr\}\,\psi_{\sU,p}(u)^{-1}\,\d u
\\
&=\int_{\fX_v^0(1)} \widehat{f_p} (\ii s) \biggl\{\int_{\Ucal} \Omega_{p}^{(-\ii s)}(g_p^{-1}u)\psi_{\sU,p}(u)^{-1}\,\d u \biggl\}\,\d\mu^{\rm Pl}(s)
\\
&=\int_{\fX_p^0(1)} \widehat{f_p} (\ii s) \biggl\{\int_{\sU(\Q_p)}^{\rm st} \Omega_{p}^{(-\ii s)}(g_p^{-1}u)\psi_{\sU,p}(u)^{-1}\,\d u \biggl\}\,\d\mu^{\rm Pl}(s).
\end{align*}}Substituting \eqref{StableIntMatCo-0} to this, we obtain
{\allowdisplaybreaks\begin{align}
\Wcal_p^{\psi_p^{-1}}(\widetilde{\check f_p};g_p)
=\int_{\fX_v^{0}(1)} \widehat{f_p} (\ii s) 
\frac{\overline{W^{0}_{\sG(\Q_p)}(- \ii s;g_p)}}{L(1,I_p^{\sG}(\ii s)\times I_p^{\sG}(-\ii s))}\,\Delta_{\sG,p}(1)\,\d\mu^{\rm Pl}(s) \label{LNVpadicL1-f1}
\end{align}}for all $g_p\in \sG(\Q_p)$. Suppose $\nu \in \ii \ft_{0}^*$ for a while. Since $\fX_p^0(1)$ is compact, due to the absolute convergence of the zeta integral at $z=1/2$ (Lemma~\ref{Ishii-Stade}), we exchange the order of integrals to have 
{\allowdisplaybreaks\begin{align*}
&Z\left(\tfrac{1}{2}, W_{\sG_0(\Q_p)}^{0}(\nu)\otimes \Wcal_p^{\psi_p^{-1}}(\widetilde{\check f_p})\right)
\\
&=\int_{\sU_0(\Q_p)\bsl \sG_0(\Q_p)} W_{\sG_0(\Q_p)}^{0}(\nu;h_p)\,\Wcal_p^{\psi_p^{-1}}(\widetilde{\check f_p};\iota(h_p))\,\d h_p
\\
&=\int_{\sU_0(\Q_p)\bsl \sG_0(\Q_p)} W_{\sG_0(\Q_p)}^{0}(\nu;h_p)
\biggl\{\int_{\fX_p^0(1)} \widehat{f_p} (\ii s) 
\frac{\overline{W^{0}_{\sG(\Q_p)}(-\ii s;\iota(h_p))}}{L(1,I_p^{\sG}(\ii s)\times I_p^{\sG}(-\ii s))}\,\Delta_{\sG,p}(1)\,\d\mu^{\rm Pl}(s) \biggr\}\,\d h_p
\\
&=\int_{\fX_p^{0}(1)} \widehat{f_p} (\ii s) \frac{Z\left(\tfrac{1}{2}; W_{\sG_0(\Q_p)}^{0}(\nu)\otimes \overline{W_{\sG(\Q_p)}^{0}(-\ii s)}\right)}{L(1,I_p^{\sG}(\ii s)\times I_p^{\sG}(-\ii s))}\,\Delta_{\sG,p}(1)\,\d\mu^{\rm Pl}(s)
\end{align*}}with the local zeta-integral in the last formula being given as \eqref{URZetaInt}. Note that $I_p^{\sG}(\ii s)\cong I_p^{\sG}(-\ii s)$. Thus by the variable change $s\rightarrow -s$
$$
Z\left(\tfrac{1}{2}, W_{\sG_0(\Q_p)}^{0}(\nu)\otimes \Wcal_p^{\psi_p^{-1}}(\widetilde{\check f_p})\right)=\int_{\fX_p^0(1)} \widehat{f_p}(-\ii s)
\frac{\prod_{j=1}^{n-1}L\left(\tfrac{1}{2}+\nu_j, I_p^{\sG}(-\ii s)\right)}
{L(1,I_p^{\sG}(\ii s)\times I_p^{\sG}(-\ii s))}\,\Delta_{\sG,p}(1)\,\d\mu^{\rm Pl}(s)
$$
is established for $\nu \in \ii\ft_{0}^*$. Since $\fX_p^{}(1)$ is compact, from the explicit form of the local $L$-factors in the integrand, the right-hand side is holomorphic on the region $\Re\,\nu_j>-1/2\,(j \in [1,n-1]_{\Z})$; hence the identity is true for $\nu$ on that region by analytic continuation. 
\end{proof}

\begin{lem}\label{LNVpadicL2}
Let $\nu\in \ft_{0,\C}^{*}$ with $\Re \nu_j>-1/2$ $(j \in [1,n-1]_{\Z})$. For any non empty $\sS_n$-invariant open set $\Ncal_p \subset \fX_p^0(1)$, there exists a compactly supported and $\sS_n$-invariant $C^\infty$-function $\a_p(s)$ on $\fX_p^{0}:=\R^{n}/2\pi (\log p)^{-1}\Z^n$ such that ${\rm supp}(\a_p)\cap \fX_p^{0}(1) \subset \Ncal_p$ and  
$$
\int_{\fX_p^0(1)} \a_p(s)\,
\frac{\prod_{j=1}^{n-1}L\left(\tfrac{1}{2}+\nu_j,I_p^{\sG}(-\ii s)\right)}
{L(1,I_p^{\sG}(\ii s);\Ad)}\,\d\mu^{\rm Pl}(s)\not=0.
$$
\end{lem}
\begin{proof} The support of the density function
\begin{align*}
&\frac{\prod_{j=1}^{n-1}L\left(\tfrac{1}{2}+\nu_j,I_p^{\sG}(-\ii s)\right)}
{L(1,I_p^{\sG}(\ii s);\Ad)}\,\frac{\d\mu^{\rm Pl}(s)}{\d_{0}s}=\frac{Q_p}{n!}\zeta_p(1)^{-n+1}\frac{
\bigl|\prod_{1\leq i<j\leq n}(1-p^{\ii(-s_i+s_j)})\bigr|^{2}}{
\prod_{i=1}^{n-1}\prod_{j=1}^{n}(1-p^{-1/2-\nu_i+\ii s_j})}
\end{align*}
is evidently $\fX_v^0(1)$. 
\end{proof}

\subsection{Local non-vanishing at $\infty$} \label{LNVAR}
Recall $\psi_\infty(x)=\exp(2\pi \ii x)$ for $x\in \R$. In this subsection, we set $C_G=\Delta_{\sG,\infty}(1)^{-1}$ and $C_{G_0}=\Delta_{\sG_0,\infty}(1)^{-1}$. Define a Haar measure on the space $\fX_\infty^0(1)=\{s\in \R^n|\,\sum_{j=1}^{n}s_j=0\}$ 
%is identified with the dual group of $\sT(\R)/\sZ(\R)(\sT(\R)\cap \bK_\infty)$ as in \S\ref{GWP}. Note that $\sT(\R)\cong (\R^\times)^n$, $\sT(\R)\cap \bK_\infty\cong \{\pm 1\}^{n}$ and $\sZ(\R)\cong \R^\times$ diagonally embedded to $(\R^\times)^{n}$. If we use $(s_1,\dots,s_{n-1})$ as a coordinate system on $\fX_\infty^0(1)$, the Haar measure $\d_0s$ on $\fX_\infty^0(1)$ defined in \S\ref{GWP} is given by
as
$$
\d_0 s=(4\pi)^{-(n-1)}\prod_{j=1}^{n-1}\d s_j, \quad s=(s_j)_{j=1}^{n}\in \fX_\infty^0(1).
$$ 
Set $\ft^*[\zeta]=\{s\in \ft^*\mid \sum_{j=1}^{n}s_j=\zeta\}=\zeta(\frac{1}{n},\dots,\frac{1}{n})+\fX_\infty^{0}(1)$ ($\zeta \in \R$), and  $(\ft^*[0])^{+}=\{s\in \ft^{*}[0]|s_j\geq s_{j+1}\,(j \in [1,n-1]_{\Z})\}$. We endow $\ft^*[\zeta]$ with the Haar measure $\d_\zeta s$ obtained from $\d_0 s$ on $\fX_\infty^{0}(1)$ as the image of the isomorphism $s_0\mapsto \zeta(\frac{1}{n},\dots,\frac{1}{n})+s_0$. Then $\int_{\ft^*}\phi(s)\d s=(4\pi)^{n-1}\int_{\zeta \in \R}(\int_{\ft^*[\zeta]}\phi(s)\d_\zeta s)\,\d \zeta$ for $\phi\in L^1(\ft^*)$. Recall the Jacquet integral $J_{\sG(\R)}^{\psi_{\infty}}(s;g)$ defined in \S\ref{Jacquetint}. First we quote the spherical Whittaker-Plancherel theorem from \cite[\S12]{Wallach2} ({\it cf}. \cite{GoldfeldKontorovich} for ${\bf GL}(n)$ case).   

\begin{lem} \label{LNVArchL1}
For any $f_\infty\in C_{\rm{c}}^{\infty}(\bK_\infty\bsl \sG(\R)/\bK_\infty)$, we have
\begin{align*}
\Wcal_\infty^{\psi_\infty^{-1}}(\widetilde{\check f_\infty};g)=\frac{C_G}{n!} \int_{\ft^{*}[0]} \widehat{f_\infty}(-\ii s) J_{\sG(\R)}^{\psi_\infty}(\ii s;1_n)\,{J_{\sG(\R)}^{\psi_\infty^{-1}}(-\ii s;g)}\,\frac{\d_{0}s}{|c(\ii s)|^2},
\end{align*}
where $c(s)=\prod_{1\leq i<j\leq n} ({\Gamma_{\R}(s_i-s_j)}/{\Gamma_\R(s_i-s_j+1)})$.
\end{lem}
\begin{proof} As explained in the introduction of \cite{Wallach2}, errors contained in \cite[\S15]{Wallach} are corrected by \cite{Wallach2}. In this proof, we write $G$, $U$, $T$, $\psi$ and $\psi_{U}$ in place of $\sG(\R)$, $\sU(\R)$, $\sT(\R)$, $\psi_{\infty}$ and $\psi_{\sU,\infty}$, respectively, to simplify notation. Set $K=\bK_\infty$. Since the function $g\mapsto W(g):=\Wcal_\infty^{\psi_\infty^{-1}}(\check f_{\infty};g)$ is compactly supported modulo $U$ on $G$, it belongs to the $K$-invariant part of the Schwartz space ${\mathcal C}(U \bsl G,\psi_{U}^{-1})$ defined in \cite[\S4]{Wallach2}. By applying the final formula of \cite[\S12]{Wallach2} to this, 
 we have
\begin{align}W(g)=
\frac{C_G}{(4\pi)^{n}\,n!} \int_{\ft^*} \widehat{f_\infty}(-\ii s)\,J_{G}^{\psi}(\ii s;1_n)\,J_{G}^{\psi^{-1}}(-\ii s;g)\,\frac{\d s}{|c(\ii s)|^2}.
 \label{LNVArchL1-ff6}
\end{align}
For convenience, we recall the definition used in \cite{Wallach} and \cite{Wallach2}. The measures on $K$ and $U$ normalized in \cite[\S10.1.7]{Wallach} coincide with ours. The objects $c(s)$, $\gamma_{A}$ and $c_{A}$ appearing in Wallach's formula is defined and made explicit in our case as follows, where the group $A$ of \cite[\S12]{Wallach2} is our $T^\circ$. The function $c(s)=\int_{U}e^{\langle s+\rho_{\sB},H(w_\ell u)\rangle} \d u$ is the $C$-function of the principal series $I^{G}(s)$, $\gamma_{T^\circ}=\int_{U}e^{\langle 2\rho_{\sB},H(w_\ell u)\rangle}\,\d u$ (\cite[Proposition 12.5.3]{Wallach}) and $c_{T^\circ}$ is the constant defined in \cite[\S13.3.2]{Wallach}. Note that $[T:T^\circ]^{-1}\d g(=2^{-n}\d g)$ coincides with the Haar measures on $G$ defined in \cite[\S10.1.7]{Wallach} and used in \cite{Wallach2}. By the Gindikin-Karpelevic formula, $c(s)=\prod_{1\leq i<j\leq n} \frac{\Gamma_{\R}(s_i-s_j)}{\Gamma_\R(s_i-s_j+1)}$. From this, $\gamma_{T^\circ}=c(\rho_{\sB})=\Gamma_{\R}(1)^{n-1}\prod_{j=2}^{n}\Gamma_\R(j)^{-1}=C_{G}$ is confirmed. Since the Haar measure on $T^\circ\cong (\R_+)^{n}$ dual to the Euclidean measure on $\ft^*$ is $\prod_{j=1}^{n}\frac{\d t_j}{t_j}$, the constant $c_{T^\circ}$ defined by the relation
$c_{T^\circ}\int_{\ft^*}\int_{T^{\circ}} u(t)e^{\langle -\ii s, H(t)\rangle}\,\d t\,\d s=u(1_n)$ for $u\in C_{\rm c}^{\infty}(T^\circ)$ is equal to $(2\pi)^{n}$. Since $J_G^{\psi^{-1}}(\ii s;zg)=|z|^{\ii(s_1+\cdots+s_n)}J^{\psi^{-1}}_G(\ii s;g)$ for all $z\in \R^\times$, we have\begin{align*}
W(zg)&=\frac{C_G}{(4\pi)^{n}\,n!}(4\pi)^{n-1}\int_{\zeta \in \R}|z|^{\ii \zeta} \d \zeta \int_{\ft^*[\zeta]} \widehat{f_\infty}(-\ii s)\,J_{G}^{\psi}(\ii s;1_n)\,J_{G}^{\psi^{-1}}(-\ii s;g)\,\frac{\d_\zeta s}{|c(\ii s)|^2}, \quad z\in \R^\times.
\end{align*}
By the Fourier inversion for the duality between $z\in \R_+$ and $\zeta\in \R$ defined by $z^{\ii\zeta}$, we obtain 
\begin{align*}
\Wcal_v^{\psi_v^{-1}}(\widetilde{\check f_\infty};g)&=2\int_{\R_+}W(zg)\,\d^* z
&=4\pi\times \frac{C_G}{4\pi\,n!}\int_{\ft^*[0]} \widehat{f_\infty}(-\ii s)\,J_{G}^{\psi}(\ii s;1_n)\,J_{G}^{\psi^{-1}}(-\ii s;g)\,\frac{\d_0 s}{|c(\ii s)|^2}. \end{align*}
\end{proof}

From \S\ref{Jacquetint}, recall the notation $\tilde s$ for $s\in \ft_{\C}^*$ together with the formula \eqref{W-centralTwist} for $\sG(\R)$. We need a uniform estimate of $J_{\sG}^{\psi_\infty}(s;t)$ due to Blomer-Harcos-Maga \cite{BlomerHarcosMaga} in the following form. 
\begin{lem}\label{BHMestimate}
There exists a constant $d=d(n)>0$ such that for any $\e \in (0,1)$ 
\begin{align*}
|J_{\sG(\R)}^{\psi_{\infty}}(s;t)|\ll_{\e} (1+\|s\|)^{d} \,\delta_{\sB}(t)^{1/2-\e}\,\exp\bigl(-{\cT(s)}^{-1}\sum_{j=1}^{n-1}({t_j}/{t_{j+1}})\bigr)  
\end{align*}
for $s=(s_j)_{j=1}^{n}\in \ii \ft^{*}$ and for $t=\diag(t_j|j\in [1,n]_\Z) \in T_{\infty}$, where $\cT(s)=\max(2,|s_1|,\dots,|s_n|)$. 
\end{lem}
\begin{proof} From \eqref{W-centralTwist} and $z\in \ii\R$, we have $|J_{\sG(\R)}^{\psi_\infty}(s;g)|=|J_{\sG(\R)}^{\psi_\infty}(\tilde s;g)|$. We apply \cite[Theorem 1]{BlomerHarcosMaga} to ${\mathcal W}_{\tilde s}(g)=J_{\sG(\R)}^{\psi_\infty}(\tilde s;g)$.  
\end{proof}

\begin{lem}\label{JIest}
For any compact set $\Ncal\subset \ft^*$ and a compact set $\Ucal \subset \sG(\R)$, there exists a constant $d>0$ such that
$$
|J_{\sG(\R)}^{\psi_\infty}(s;g)|\ll_{\Ncal, \Ucal}(1+\|s\|)^{d}, \quad s\in \fJ(\Ncal),\,g\in \Ucal.
$$
\end{lem}
\begin{proof} From \eqref{W-centralTwist}, we have $|J_{\sG(\R)}^{\psi_\infty}(s;g)|=|\det\,g|_\infty^{\Re\,c}|J_{\sG(\R)}^{\psi_\infty}(\tilde s;g)|$ for $g\in \sG(\R)$. Choose $\kappa>\delta>0$ such that $|\Re\,s|<\kappa-\delta$ for all $s\in \fJ(\Ncal)$. Let $\Ucal_T$ be the image of $\Ucal$ by the projection $\sG(\R)=\sU(\R)T_\infty\bK_\infty\rightarrow T_{\infty}$. Then we apply \cite[Theorem 2]{BlomerHarcosMaga} to ${\mathcal W}_{\tilde s}(t)=J_{\sG(\R)}^{\psi_{\infty}}(\tilde s;t)$ for $t\in \Ucal_T$ and $s\in \fJ(\Ncal)$ with $\kappa$ and $\delta$ as above, noting $|t_j|^{\Re\,s_j}_\infty\ll 1$ ($t\in \Ucal_T$, $s\in \fJ(\Ncal)$). \end{proof}

\begin{lem} \label{ConvergenceArchZeta-L}
There exists a constant $d>0$ such that for any small $\e>0$
\begin{align}
\int_{\sU_0(\R)\bsl \sG_0(\R)}|W_{\sG_0(\R)}^{0}(\nu;h)|\,|J_{\sG(\R)}^{\psi_\infty^{-1}}(-\ii s;\iota(h))|\,\d h\ll_{\e} (1+\|s\|)^{d}, \quad \nu \in \ii \ft_{0}^*,\, s\in \ft^{*}.
 \label{ConvergenceArchZeta-L1}
\end{align}
\end{lem}
\begin{proof}
From \cite[Corollary (p.358)]{Stade} with the modification for central character ({\it cf}. \eqref{W-centralTwist}), we have $|W_{\sG_0(\R)}^{0}(\nu;h)|\leq W_{\sG_0(\R)}^{0}(0;h)$ for $\nu \in \ii \ft_{0}^*$ and $h\in \sG(\R)$. From Lemma~\ref{BHMestimate}, we have a constant $d>0$ such that
{\allowdisplaybreaks\begin{align*}
&|W_{\sG_0(\R)}^{0}(\nu;a_{\sG_0}(y))| \leq W_{\sG_0(\R)}^{0}(0;a_{\sG_0}(y))\ll_{\e} \delta_{\sB_0}(a_{\sG_0}(y))^{1/2-\e}\,\exp(-2^{-1}\sum_{j=1}^{n-2}y_j), \\
&|J_{\sG(\R)}^{\psi_\infty^{-1}}(-\ii s;\iota(a_{\sG_0}(y))|\ll_{\e}(1+\|s\|)^{d}\delta_{\sB}(\iota(a_{\sG_0}(y)))^{1/2-\epsilon}\exp(-\cT(s)^{-1}\sum_{j=1}^{n-1}y_j)
\end{align*}}uniformly in $\nu\in \ii\ft_{0}^*$, $s\in \ft^*$ and $y\in \R_+^{n-1}$. We use the first estimate omitting the exponential factor. Then by the Iwasawa decomposition of $\sG_0(\R)$ and \eqref{y-coordinate}, the integral \eqref{ConvergenceArchZeta-L1} is bounded from above by the product of $(1+\|s\|)^{d}$ and the integral 
\begin{align*}
&\int_{y\in \R_+^{n-1}} \{\prod_{j=1}^{n}y_{j}^{j(n-j)}\}^{1/2-\e}\times \{\prod_{j=1}^{n-1}y_j^{j(n-1-j)}\}^{1/2-\e}\times \{\prod_{j=1}^{n-1}y_j^{j(n-1-j)}\}^{-1} \prod_{j=1}^{n-1} \exp(-\cT(s)^{-1}y_j)\,\prod_{j=1}^{n-1}\d^{*}y_j.
\end{align*}
If $\epsilon>0$ is small enough so that $A_{\e,j}:=\frac{j}{2}-\e j(2n-2j-1)$ $(j\in[1,n-1]_\Z)$ are all positive, the last integral becomes
$$
\prod_{j=1}^{n-1}\int_{0}^{\infty} y_j^{A_{\e,j}}\exp(-\cT(s)^{-1}y_j)\,\d^* y_j
=\prod_{j=1}^{n-1}\cT(s)^{A_{\e,j}}\Gamma(A_{\e,j}). 
$$
Since $\cT(s):=\max(2,|s_1|,\dots,|s_n|) \leq 2(1+\|s\|)$ and $A_{\e,j}\leq j/2$, $\cT(s)^{A_{\e,j}}\ll_{j}(1+\|s\|)^{j/2}$ for $s\in \ft_{\C}^{*}$. This completes the proof. 
\end{proof}

\begin{lem}\label{LNVAr-L10-L} Let $\cN \subset \ft_{\C,0}^{*}$ be a compact set such that $\Re\,\nu_j>-1/2\,(j \in [1,n-1]_{\Z})$ for all $\nu\in \cN$. Then there exists a constant $r=r(\cN)>0$ such that 
$$
|M_{\sG,\infty}(-\ii s)^{-1}\,\prod_{j=1}^{n-1}L\left(\tfrac{1}{2}+\nu_j, I_\infty^{\sG}(-\ii s)\right)|\ll (1+\|s\|)^{r}, \quad s\in (\ft^*[0])^{+}, \,\nu \in \cN. 
$$
\end{lem}
\begin{proof} Recall $M_{\sG,\infty}(\ii s)=\prod_{1\leq i<j\leq n}\Gamma_{\R}(1+\ii(s_i-s_j))$ ({\it cf}. \eqref{f-MsG0v}).  By Stirling's formula, 
\begin{align*}
|M_{\sG,\infty}(-\ii s)^{-1}\prod_{j=1}^{n-1}L\left(\tfrac{1}{2}+\nu_j,I_\infty^{\sG}(-\ii s)\right)|&\ll_{\cN}  
\prod_{1\leq i<j\leq n}|\Gamma\left(\tfrac{1}{2}-\ii\tfrac{s_i-s_j}{2}\right)|^{-1} \times \prod_{i=1}^{n-1}\prod_{j=1}^{n}|\Gamma\left(\tfrac{1}{4}+\tfrac{\nu_i}{2}-\ii\tfrac{s_j}{2}\right)|\\
&\ll \exp\left(\tfrac{-\pi}{4}A(s,\nu)\right)\,P(\nu,s),
\end{align*}
where $P(\nu,s)$ is a positive quantity majorized by $(1+\|s\|)^{r}$ with some $r>0$ uniformly in $(\nu,s)\in \cN\times (\ft^*[0])^{+}$, and 
$$
A(\nu,s)=\sum_{j=1}^{n}\sum_{i=1}^{n-1}|s_j-\Im\,\nu_i|-\sum_{1\leq i<j\leq n}(s_i-s_j). 
$$
It suffices to show that there exists a constant $C=C(\cN)\in \R$ such that $A(\nu,s)\geq C$ for all $(\nu,s)\in \cN\times (\ft^*[0])^{+}$. Since $\cN$ is compact, we have
\begin{align*}
A(\nu,s)+O_{\cN}(1)&=\sum_{j=1}^{n}\{(n-1)|s_j|+(j-1)s_j-(n-j)s_{j}\}.
\end{align*}
It suffices to confirm the right-hand side is non-negative. For this, we set $I_{+}=\{i\in [1,n]_\Z\mid s_{i}\geq 0\}$ and $I_{-}=\{i\in [1,n]_\Z\mid s_j<0\}$, and write the sum as
{\allowdisplaybreaks\begin{align*}
&\sum_{j\in I_{+}}\{(n-1)s_j+(j-1)s_j-(n-j)s_{j}\}+\sum_{j\in I_{-}}\{(n-1)(-s_j)+(j-1)s_j-(n-j)s_{j}\}
\\
&=\sum_{j\in I_{+}}2(j-1)s_j+\sum_{j\in I_{-}}2(n-j)(-s_{j}),
\end{align*}}which is obviously non-negative. This completes the proof. \end{proof}

\begin{prop}\label{LNVAr-L10}
 Let $f_\infty \in C_{\rm c}^{\infty}(\bK_\infty \bsl \sG(\R)/\bK_\infty)$. If $\nu \in \ft_{0,\C}^*$ is such that $\Re \nu_j>-1/2$ $(j \in [1,n-1]_{\Z})$, 
\begin{align*}
Z(1/2;W_{\sG_0(\R)}^0(\nu)\otimes \Wcal_\infty^{\psi_\infty^{-1}}(\widetilde{\check f_\infty}) )
&=\int_{\fX_\infty^{0}(1)} 
\widehat{f_\infty}(-\ii s)W^{0}_{\sG(\R)}(s;1_n)\frac{\prod_{j=1}^{n-1}L\left(\tfrac{1}{2}+\nu_j, I_\infty^{\sG}(-\ii s)\right)}{L(1,I^{\sG}_\infty(\ii s);\Ad)}\,\d\mu^{\rm Pl}(s),
\end{align*}
where 
$$\d\mu^{\rm Pl}(s)=\frac{\Delta_{\sG,\infty}(1)^{-1}}{n!}|c(\ii s)|^{-2}\d_{0}s \quad \text{with}\quad c(s)=\prod_{1\leq i<j\leq n} \frac{\Gamma_{\R}(s_i-s_j)}{\Gamma_{\R}(s_i-s_j+1)}.$$ 
\end{prop}
\begin{proof} From the Stirling formula, we have 
\begin{align}
|\Gamma\left
(\tfrac{\ii y}{2}\right)| \gg \exp\left(-\tfrac{\pi}{2}y\right)(1+|y|)^{-1/2}, \quad 
|\Gamma\left(\tfrac{\ii y+1}{2}\right)| \ll \exp\left(-\tfrac{\pi}{2}y\right)
, \quad (y>0),
 \notag
\end{align}
which yields the bound 
\begin{align}
|c(\ii s)|\gg \prod_{1\leq i<j\leq n}(1+s_i-s_j)^{-1/2}, \quad s\in (\ft^*[0])^{+}.
\label{LNVAr-L10-f1}
\end{align}
From Lemmas~\ref{JIest} and \ref{SpectExpPerL-4-1}, for any $l>0$, 
\begin{align}
|\widehat{f_\infty}(-\ii s)\,J_{\sG(\R)}^{\psi_\infty}(\ii s;1_n)|\ll_{l} (1+\|s\|)^{-l}, \quad s\in (\ft^{*}[0])^{+}. 
\label{LNVAr-L10-f2}
\end{align}
Suppose $\nu \in \ii \ft_{0}^*$ for a while. Combining \eqref{LNVAr-L10-f1} and \eqref{LNVAr-L10-f2} with the bound \eqref{ConvergenceArchZeta-L1}, we obtain the majorization\begin{align*}
&\int_{(\ft[0])^{+}}|\widehat{f_\infty}(-\ii s)\,J_{\sG(\R)}^{\psi_\infty}(\ii s;1_n)|\times \biggl\{
\int_{\sU_0(\R)\bsl \sG_0(\R)}|W_{\sG_0(\R)}^{0}(\nu;h)|\,|J_{\sG(\R)}^{\psi_\infty^{-1}}(-\ii s;\iota(h))|\,\d h \biggr\}\times |c(\ii s)|^{-2}\,\d_0 s \\
&\ll_{l} \int_{(\ft^*[0])^{+}} (1+\|s\|)^{d-l} \prod_{1\leq i<j\leq n}(1+s_i-s_j) \,\d_0 s,
\end{align*}
whose majorant is finite for sufficiently large $l>0$. Since the integrand of the formula in Lemma~\ref{LNVArchL1} is $\sS_{n}$-invariant, we may replace the integral domain with $(\ft^*[0])^{+}$ multiplying the integrand by $n!$; by plugging the formula to the defining formula \eqref{LocalRSint} of $Z(1/2;W_{\sG_0(\R)}^0(\nu)\otimes \Wcal_\infty^{\psi_\infty^{-1}}(\widetilde{\check f_\infty}))$ and then by Fubini's theorem to change the order of integrals, we obtain the following expression of $Z(1/2;W_{\sG_0(\R)}^0(\nu)\otimes \Wcal_\infty^{\psi_\infty^{-1}}(\widetilde{\check f_\infty}))$: 
\begin{align*}
&C_{G} \int_{(\ft^*[0])^{+}} \,\widehat{f_\infty}(-\ii s)\,J_{\sG(\R)}^{\psi_\infty}(\ii s;1_n)\biggl\{\int_{\sU_0(\R)\bsl \sG_0(\R)} W_{\sG_0(\R)}^{0}(\nu;h)\,J_{\sG(\R)}^{\psi_\infty^{-1}}(-\ii s;\iota(h))\,\d h \biggr\}\, \frac{\d_0 s}{|c(\ii s)|^{2}}.
\end{align*}
By Lemma~\ref{Ishii-Stade} to compute the inner integral, we get the equality
\begin{align}
Z\left(\tfrac{1}{2};W_{\sG_0(\R)}^0(\nu)\otimes \Wcal_\infty^{\psi_\infty^{-1}}(\widetilde{\check f_\infty})\right)=&C_{G} \int_{(\ft^*[0])^{+}} \,\widehat{f_\infty}(-\ii s)\,J_{\sG(\R)}^{\psi_\infty}(\ii s;1_n)\,\frac{\prod_{j=1}^{n-1}L\left(\tfrac{1}{2}+\nu_j,I_\infty^{\sG}(-\ii s)\right)}{M_{\sG,\infty}(-\ii s)}\,\frac{\d_0 s}{|c(\ii s)|^{2}}
 \label{LNVAr-L10-f3}
\end{align} 
for $\nu \in \ii \ft_{0}^*$. From \eqref{LNVAr-L10-f1}, \eqref{LNVAr-L10-f2} and Lemma~\ref{LNVAr-L10-L}, when $\nu \in \ft_{0,\C}$ varies in a compact set $\cN$ of $\Re\,\nu_j>-1/2\,(j \in [1,n-1]_{\Z})$ the integrand  of the right-hand side of \eqref{LNVAr-L10-f3} is uniformly majorized by $(1+\|s\|)^{-l}$ for any $l>0$. Thus the right-hand side converges absolutely and uniformly in $\nu \in \cN$ defining a holomorphic function of $\nu$. Hence the equality holds true for all $\nu$ such that $\Re\,\nu_j>-1/2\,(j \in [1,n-1]_{\Z})$ by analytic continuation. By the relations $L(1,I_\infty^{\sG}(\ii s)\times I_\infty^{\sG}(-\ii s))=|M_{\sG,\infty}(\ii s)|^2$ and $W_{\sG(\R)}^{0}(\ii s;1_n)=M_{\sG,\infty}(\ii s)J_{\sG(\R)}^{\psi_\infty}(\ii s; 1_n)$, the formula \eqref{LNVAr-L10-f3} becomes the desired form; to resume the integral domain $\ft^*[0]$, we note the $\sS_n$-invariance of the integrand.  
\end{proof}

\begin{lem} \label{LNVArchL3}
Suppose $\Re \nu_j>-1/2$ $(j \in [1,n-1]_{\Z})$. For any non empty $\sS_n$-invariant open set $\Ncal_\infty \subset \ft^*[0]$, there exists a compactly supported and $\sS_n$-invariant $C^\infty$-function $\a_\infty(s)$ on $\ft^*$ such that ${\rm supp}(\a_\infty)\cap \ft^*[0] \subset \Ncal_\infty$ and  
$$
\int_{\fX_\infty^0(1)} \a_\infty(s)\,W_{\sG(\R)}^{0}(\ii s;1_n)
\frac{\prod_{j=1}^{n-1}L\left(\tfrac{1}{2}+\nu_j, I_\infty^{\sG}(-\ii s)\right)}{L(1,I_\infty^{\sG}(\ii s);\Ad)}\,\d\mu^{\rm Pl}(s)\not=0.
$$
\end{lem}
\begin{proof}
From the equation $\int_{G(\R)}W_{\sG(\R)}^{0}(s;g)f(g)\d g={\widehat f}(s)\,W_{\sG(\R)}^{0}(s;1_n)$ for $f\in C_{\rm c}^\infty(\sG(\R))$, we see that $W_{\sG(\R)}^{0}(s;1_n)$ is non-vanishing for all $s\in \ii\ft^*[0]$ because $W^0_{\sG(\R)}(s)$ is a non-zero smooth function on $\sG(\R)$. Thus the support of the density function
\begin{align*}
&W_{\sG(\R)}^{0}(\ii s;1_n)\,\frac{\prod_{j=1}^{n-1}L\left(\tfrac{1}{2}+\nu_j,
I_\infty^{\sG}(-\ii s)\right)}
{L(1,I_\infty^{\sG}(\ii s);\Ad)}\,\frac{\d\mu^{\rm Pl}(s)}{\d_{0}s}=\frac{C_{G}W_{\sG(\R)}^{0}(\ii s;1_n)}{n!}\frac{\prod_{i=1}^{n-1}\prod_{j=1}^{n}\Gamma_\R(1/2+\nu_i-\ii s_j)}{\bigl|\prod_{1\leq i<j\leq n}\Gamma_\R(\ii s_i-\ii s_j)\bigr|^{2}}
\end{align*}
is obviously $\ft^*[0]$.    
\end{proof}

\section{Control of periods of old forms} \label{sec:OldF}
Let $p$ be a prime number and $\pi_p\in \widehat{\sG(\Q_p)}_{\rm gen}^{\rm ur}$ with trivial central character. In this section (see Corollary \ref{OldF-L6}), we obtain an explicit evaluation of the period $\PP_{\psi_p}^{(\nu)}(\pi_p,\bK_1(p\Z_p);\cchi_{\bK_1(p\Z_p)})$ defined by \eqref{LocalWHittPer}. Combining our explicit formula with the Luo-Rudnick-Sarnak bound (\cite{LRS}) of the Satake parameter of $\pi_p$, we prove the following:
\begin{prop} \label{OldF-L7}
 There exists a constant $C=C(n)>0$ with the following property. For any prime number $p$ and for any $\pi \cong \otimes_{v}\pi_v \in \Pi_{\rm cusp}(\sG)_\sZ$ such that $\pi_{p}$ is spherical, we have
$$
|\PP_{\psi_p}^{(\nu)}(\pi_p,\bK_1(p\Z_p);\cchi_{\bK_1(p\Z_p)})|\leq C\,p^{-1-\frac{n-1}{n^2+1}} \quad \text{for all $\nu \in \C^{n-1}$ with $\Re \nu_j \geq 0\,(1\leq j \leq n-1)$}.
$$ 
\end{prop}
The proof of this proposition is given in \S\ref{sec:PrfOldF-L7}. 
\subsection{Symmetric polynomials} \label{sec:SymF}
Let $T=(T_1,\dots,T_r)$ be a set of indeterminates. For $l\in \N_{0}$, the elementary symmetric polynomial $e_l(T)=e_l(T_1,\dots,T_r)$ of degree $l$ and the complete symmetric polynomial $h_l(T)=h_{l}(T_1,\dots,T_r)$ of degree $l$ are defined by the explicit formulas 
\begin{align*}
e_l(T_1,\dots,T_r)=\sum_{1\leq i_1<\dots<i_{l}\leq r }T_{i_1}\cdots T_{i_l}, \qquad 
h_{l}(T_1,\dots T_r)=\sum_{1\leq i_1\leq \dots \leq i_{l}\leq r}T_{i_1}\cdots T_{i_{l}}.
\end{align*}  
These objects are also defined collectively by the generating series : 
\begin{align}
\prod_{j=1}^{r}(1-T_jZ)=\sum_{m=0}^{r}(-1)^{m}e_{m}(T)\,Z^{m}, \qquad 
\prod_{j=1}^{r}(1-T_jZ)^{-1}=\sum_{l=0}^{\infty}h_{l}(T)\,Z^{l},
 \label{SymF-f0}
\end{align}
which allow us to deduce the following formulas easily: 
\begin{align}
&\sum_{j=1}^{r} T^{j-1}_m\,\frac{(-1)^{r-j}e_{r-j}(\hat T_l)}{\prod_{\substack{1\leq \alpha  \leq r \\ \alpha\not=l}}(T_l-T_\alpha)}=\delta_{m,l}\quad (1\leq m,l\leq r), \label{SymF-f1}
\\
&T_m^{l-1}=\sum_{\alpha=1}^{l}(-1)^{\alpha-1}e_{\alpha-1}(\hat T_{m})\, h_{l-\alpha}(T)\quad (1\leq m,l\leq r),
\label{SymF-f2}
 \end{align}
where $\hat T_l$ denotes the set of variables $(T_1,\dots, T_{l-1},T_{l+1},\dots,T_r)$. 
%Indeed, \eqref{SymF-f1} is equivalent to the obvious relation $P_l(T_m)=\delta_{lm}\,\prod_{\substack{1\leq \alpha \leq r \\ \alpha\not=l}}(T_m-T_\alpha)$ ($1\leq m,l \leq r)$, where $P_{l}(Z)$ denotes the polynomial $\prod_{\substack{1\leq \alpha \leq  r \\ \alpha\not=l}}(Z-T_\alpha)=\sum_{j=1}^{r}(-1)^{r-j}e_{r-j}(\hat T_l)\,Z^{j-1}$. The second formula \eqref{SymF-f2} is obtained from 
%$$
%(1-ZT_m)^{-1}=\prod_{\substack{1\leq \alpha \leq r \\ \alpha\not=m}}(1-ZT_\alph%a) \times \prod_{j=1}^{r}(1-ZT_j)^{-1}
%$$
%by substituting the expansions \eqref{SymF-f0} and $(1-ZT_m)^{-1}=\sum_{l=1}^{\infty}T_m^{l-1}Z^{l-1}$ and by comparing the coefficient of $Z^{l-1}$. 
Define $r\times r$-matrices  
$${\mathbb V}(T):=(T_{i}^{j-1})_{ij}, \quad {\mathbb E}(T):=(e_{j-1}(\hat T_{i}))_{ij}, \quad {\mathbb H}(T):=((-1)^{i-1}h_{j-i}(T))_{ij} 
$$
with coefficients in $\C[T_1,\dots,T_r]$, where we set $h_{l}(T)=0$ for $l<0$. Then \eqref{SymF-f1} and \eqref{SymF-f2} yield matrix identities:  
\begin{align}
{\mathbb V}(T)^{-1}=\left(\frac{(-1)^{r-i}e_{r-i}(\hat T_j)}{\prod_{\substack{1\leq \alpha  \leq r \\ \alpha\not=j}}(T_j-T_\alpha)}\right)_{ij}, \quad {\mathbb V}(T)={\mathbb E}(T)\,{\mathbb H}(T). 
 \label{SymF-f3}
\end{align}
We set $D(T):=\det({\mathbb V}(T))=\prod_{1\leq i <j \leq r}(T_j-T_i)$, the Vandermonde determinant. Then $D(T)\times {\mathbb V}(T)$ belongs to $\Mat_{n}(\C[T_1,\dots,T_r])$.

\subsection{Computation of integrals}\label{sec:CoI}
Let $\pi \in \widehat{\sG(\Q_p)}^{\rm ur}_{\rm gen}$ and suppose that $\pi$ has the trivial central character. Let $W_{\pi}^{0} \in \WW^{\psi_p}(\pi)^{\bK_p}$ be as in \S\ref{sec:URWFTN}. Let ${s}=(s_j)_{j=1}^{n}\in \ii\fX_p^{0+}(1)/\sS_n$ be the spectral parameter of $\pi$ (see \S\ref{sec:URps}) so that $W_{\pi}^0=W_{\sG(\Q_p)}^0(s)$, and set $p^{-s}:=(p^{-s_j})_{j=1}^{n}$. Let $\cH:=C_{\rm c}^{\infty}(\bK_{\sG_0,p}\bsl \sG_0(\Q_p)/\bK_{\sG_0,p})$ be the spherical Hecke algebra of $\sG_0(\Q_p)$. Set $\fX_{\sG_0}^{0}=(\R/2\pi(\log p)^{-1}\Z)^{n-1}$, which parametrizes the tempered $\bK_{\sG_0,p}$-spherical unitary dual of $\sG_0(\Q_p)$ (see \S\ref{sec:Lfunction}). As in the case of $\sG(\Q_p)$, the Fourier transform of $\phi \in \cH$ at $\nu=(\nu_j)_{j=1}^{n-1}\in (\C/2\pi \ii (\log p)^{-1}\Z)^{n-1}$ is denoted by $\widehat \phi(\nu)$. We need the spherical Fourier inversion formula for $\sG_0(\Q_p)$:
\begin{align}
\phi(h)=\int_{\fX_{\sG_0,p}^0} \widehat \phi(\ii \nu)\,\Omega_{\sG_0(\Q_p)}^{(-\ii \nu)}(h)\,\d\mu_{\sG_0(\Q_p)}^{\rm Pl}(\nu), \quad h\in \sG_0(\Q_p),\,\phi \in \cH,
 \label{CoI-f0}
\end{align}
where $\Omega_{\sG_0(\Q_p)}^{(\nu)}$ is the $\bK_{\sG_0,p}$-spherical matrix coefficient of $I_p^{\sG_0}(\nu)$ and $\d\mu^{\rm Pl}_{\sG_0(\Q_p)}$ the Plancherel measure for $\sG_0(\Q_p)$ ({\it cf}. $\Omega_p^{(s)}$ and $\d\mu^{\rm Pl}$ for $\sZ(\Q_p)\bsl \sG(\Q_p)={\bf PGL}_n(\Q_p)$ are given in \S\ref{LNVS0}). For an integer $0\leq j \leq n-1$, define a function $\phi_j$ on $\sG_0(\Q_p)$ by 
$\phi_{j}=p^{-j(n-1-j)/2}\,\cchi_{X_j}$, where $X_j:=\bK_{\sG_0,p}\diag(p 1_{j},1_{n-1-j})\,\bK_{\sG_0,p}$. Note that $\phi_0$ is the characteristic function of $\bK_{\sG_0,p}$. It is known that $\{\phi_{j}\mid j\in[0,n-1]_\Z \}$ is a set of generators of the $\C$-algebra $\cH$ (\cite[Theorem 6]{Satake}) and $\widehat {\phi_j}(\nu)=e_{j}(p^{-\nu_1}, \dots,p^{-\nu_{n-1}})$ for all $\nu=(\nu_j)_{j=1}^{n-1}\in \C^{n-1}$ (\cite[(8.14)]{Satake}). Define 
\begin{align}
W^{(j)}(g)=\int_{\sG_0(\Q_p)}W^0_{\pi}(g \,\iota(h^{-1}))\,\phi_j(h)\,|\det h|_{p}^{1/2}\,\d h, \qquad g\in \sG(\Q_p)
 \label{CoI-f1}
\end{align}
for $j \in[0,n-1]_{\Z}$. Note that $W^{(0)}=W^0_{\pi}$. Define a matrix $F^{(\nu)}_\pi\in \Mat_{n}(\C)$ (resp. $G_\pi$) by saying that its $(i,j)$-entry is $W^{(n-i)}(1_n)\times Z^{0}(\nu;\bar W^{(n-j)})$ (resp. $\langle W^{(n-i)}|W^{(n-j)}\rangle_{\Pcal(\Q_p)}$) for $i,j\in[1,n]_\Z$. From \cite[Theorem 1]{Reeder}, the functions $W^{(j)}\,(j \in [0,n-1]_{\Z})$ form a $\C$-basis of the space $\WW^{\psi_p}({\pi})^{\bK_1(p\Z_p)}$, which implies that $G_{\pi}$ is a non-singular matrix. By Gram-Schimidt method applied to $W^{(j)}\,(j \in [0,n-1]_{\Z})$, we obtain an orthonormal basis $V_j\,(j \in [1,n]_{\Z}))$ together with a transformation matrix $A=(a_{i-1\,j-1})_{ij}\in \Mat_{n}(\C)$, i.e., $V_i=\sum_{j=0}^{n-1}a_{ji}\,W^{(j)}\,(i \in [0,n-1]_{\Z})$. Then $1_{n}={}^t A\, (w_{\ell}G_{\pi}w_{\ell})\,\bar A$, where $w_\ell$ is the longest element of $\sS_n$.

\begin{lem} \label{OldF-L2}
\begin{align}
\PP_{\psi_p}^{(\nu)}(\pi,\bK_1(p\Z_p);\cchi_{\bK_1(p\Z_p)})=\tfrac{1}{p^n-1}\,{L(1,\pi \times \bar \pi)}\,\tr(G_{\pi}^{-1}\,F^{(\nu)}_{\pi}). 
 \label{OldF-L2-1}
\end{align}
\end{lem}
\begin{proof} Since $V_j*{\cchi_{\bK_1(p\Z_p)}}=\vol(\bK_1(p\Z_p))\,V_j$ and $\vol(\bK_1(p\Z_p))=[\bK_p:\bK_1(p\Z_p)]^{-1}=({p^n-1})^{-1}$, from \eqref{LocalWHittPer}, $\PP_{\psi_p}^{(\nu)}(\pi,\bK_1(p\Z_p);\cchi_{\bK_1(p\Z_p)})$ equals
{\allowdisplaybreaks 
\begin{align*}
&L(1,\pi\times \bar \pi)\,\tfrac{1}{p^n-1}\,\sum_{j=0}^{n-1} \sum_{\alpha=0}^{n-1}\sum_{\beta=0}^{n-1}a_{\beta\, j}\,W^{(\beta)}(1_n)\,Z^{0}(\nu;\bar W^{(\alpha)})\,\bar a_{\alpha\, j}
\\
&=L(1,\pi\times \bar \pi)\,\tfrac{1}{p^n-1}\,\tr({}^t A\,w_\ell F_{\pi}^{(\nu)}w_\ell\,\bar A)=L(1,\pi\times \bar \pi)\,\tfrac{p-1}{p^n-1}\,\tr(w_\ell  \bar A {}^t A\,w_\ell F_{\pi}^{(\nu)}).
\end{align*}} By the relation $1_{n}={}^t A\, (w_{\ell}G_{\pi}w_{\ell})\,\bar A$, we get $G_{\pi}^{-1}=w_\ell \bar A\,{}^tA\,w_\ell$.  
\end{proof}

The next proposition gives us a simple expression of the matrix $F_{\pi}^{(\nu)}$ in terms of the elementary symmetric polynomials. 
\begin{prop}\label{OldF-L3}
Let $\nu=(\nu_j)_{j=1}^{n-1} \in \C^{n-1}$. 
{\allowdisplaybreaks\begin{align}
Z(z,W_{\sG_0(\Q_p)}^{0}(\nu)\otimes \overline{W^{(j)}})&=p^{-jz}\,\widehat{\phi_j}(\nu)\, \prod_{j=1}^{n-1}L(z+\nu_j, \bar \pi) \quad (\Re z\gg 0), 
 \label{OldF-L3-f1}
\\
W^{(j)}(1_n)&=\delta_{0,j}\quad (j \in [1,n-1]_{\Z}).
 \label{OldF-L3-f2}
\end{align}}The $i$-th row of the matrix $F_{\pi}^{(\nu)}\in \Mat_n(\C)$ is the zero vector if $1\leq i\leq n-1$ and is equal to the vector $(p^{-(n-j)/2}\,e_{n-j}(p^{-\nu}))_{j=1}^{n}$ if $i=n$, where $p^{-\nu}:=(p^{-\nu_j})_{j=1}^{n-1}$.
\end{prop}
\begin{proof} By a simple variable change and by noting $|\det x|_p=p^{-j}$ if $\phi_j(x)\not=0$, we have 
$$
Z(z,W_{\sG_0(\Q_p)}^{0}(\nu)\otimes \overline{W^{(j)}})=p^{-jz}\,Z(z,(W_{\sG_0(\Q_p)}^{0}(\nu)*\check \phi_j)\otimes \overline{W^{0}_\pi})\quad \Re z\gg 0,$$
where $*$ denotes the convolution product on $\sG_0(\Q_p)$. The first formula \eqref{OldF-L3-f1} follows from this by $W_{\sG_0(\Q_p)}^{0}*\check \phi_j=\widehat \phi_j(\nu)\,W_{\sG_0(\Q_p)}^{0}(\nu)$ and by \eqref{URZetaInt}. Although the simple formula \eqref{OldF-L3-f2} may be known to experts, a proof is included for completeness. From \eqref{CoI-f1},  
{\allowdisplaybreaks\begin{align}
W^{(j)}(1_n)&=\int_{\sG_0(\Q_p)} W^0_\pi(\iota(h))\,\check \phi_j(h)|\det h|_p^{-1/2}\,\d h \label{CoI-f2}
\\
&=\int_{\sU_0(\Q_p)\bsl \sG_0(\Q_p)}W^{0}_\pi(\iota(h))\,\Wcal_{\sG_0(\Q_p)}^{\psi_p^{-1}}(\check \phi_j; h)\,|\det h|_p^{-1/2}\,\d h=\lim_{w\rightarrow 0}Z(w, \Wcal_{\sG_0(\Q_p)}^{\psi_p^{-1}}(\check \phi_j)\otimes W^{0}_\pi),
 \notag 
\end{align}}where $\Wcal_{\sG_0(\Q_p)}^{\psi_p^{-1}}(\phi;h):=\int_{\sU_0(\Q_p)}\psi_p(u)\phi(uh)\d u$ for $\phi\in \cH$, and $Z(w,\Wcal_{\sG_0(\Q_p)}^{\psi_p^{-1}}(\check \phi_j)\otimes W_\pi^{0})$ is defined by \eqref{LocalRSint}. In the same way as in the proof of \eqref{LNVpadicL1-f1}, by using the Fourier inversion \eqref{CoI-f0}, we obtain the formula 
\begin{align}
\Wcal_{\sG_0(\Q_p)}^{\psi_p^{-1}}(\phi;h)=\frac{(\log p)^{n-1}}{(2\pi)^{n-1}\,(n-1)!} \int_{\fX_{\sG_0,p}^{0}}\widehat \phi(\ii \nu)\,\overline{W^{0}_{\sG_0(\Q_p)}(\ii \nu;h)}\,|D(p^{-\ii \nu})|^2\,\prod_{j=1}^{n-1}\d \nu_j
 \label{CoI-f3}
\end{align}
for $\phi \in \cH$, where $D(X)=\prod_{1\leq i<j\leq n-1}(X_j-X_i)$. Let $\Re w\gg 0$ and study the integral $Z(w,\Wcal_{\sG_0(\Q_p)}^{\psi_p^{-1}}(\check \phi_j)\otimes W_\pi^{0})$, substituting the expression \eqref{CoI-f3}. Since the zeta-integral $Z(w,\overline{W^{0}_{\sG_0(\Q_p)}(\ii \nu)}\otimes W_\pi^{0})$ is absolutely convergent for $\Re w\gg 0$ and since $\fX_{\sG_0,p}^{0}$ is compact, Fubini's theorem can be applied. Thus, by \eqref{URZetaInt},  
{\allowdisplaybreaks\begin{align*}
Z(w, \Wcal_{\sG_0(\Q_p)}^{\psi_p^{-1}}(\check \phi_j)\otimes W_\pi^{0})
&=
\frac{(\log p)^{n-1}}{(2\pi)^{n-1}\,(n-1)!}\, \int_{\fX_{\sG_0,p}^0} \widehat{\check \phi_j}(\ii\nu)\,\{\prod_{j=1}^{n-1}L(w-\ii \nu_j, I_p^{\sG}(s))\}\,|D(p^{-\ii \nu})|^2\,\prod_{j=1}^{n-1}\d \nu_j.
\end{align*}}Set $z_j=p^{-\ii \nu_j}\,(1\leq j \leq n-1)$ and $t_\alpha=p^{-s_\alpha}\,(1\leq \alpha \leq n)$. Then we have $\widehat{\check \phi_j}(\nu)=\widehat \phi_j(-\nu)=e_{j}(z_1^{-1},\dots,z_{n-1}^{-1})=(\prod_{i=1}^{n-1}z_i^{-1})\,e_{n-j-1}(z_1,\dots,z_{n-1})$ and the expression
\begin{align*}
Z(w,\Wcal_{\sG_0(\Q_p)}^{\psi_p^{-1}}(\check \phi_j)\otimes W^{0}_\pi)=\frac{1}{(2\pi \ii)^{n-1}\,(n-1)!}\, \int_{\cD} e_{n-j-1}(z)\, \frac{(-1)^{(n-1)(n-2)/2}\,D(z)^2}{\prod_{j=1}^{n-1}\prod_{\alpha=1}^{n}(z_j-p^{-w}t_\alpha)}\prod_{j=1}^{n-1}\d z_j,
\end{align*}
where $\cD=\{z\in \C^{n-1}\mid |z_j|=1\,(j \in [1,n-1]_\Z)\}$. Now we need a lemma, which is proved by a successive application of the residue theorem: 
\begin{lem} \label{OldF-L3-LL} Let $F(z_1,\dots,z_{n-1})$ be a holomorphic function defined in a neighborhood of the polydisc $\{z\in \C^{n-1}\mid |z_j|<1\}$ such that $F(\sigma z)=F(z)$ for all $\sigma\in \sS_{n-1}$ and such that $F(z)=0$ if $D(z)=0$. Let $x=(x_\alpha)_{\alpha=1}^{n}$ be an $n$-tuple of complex numbers which is regular $($ i.e., $x_\alpha\not=x_\beta$ for $\alpha\not=\beta$ $)$ and $|x_\alpha|<1\,(\alpha \in [1,n]_\Z)$. Then 
{\allowdisplaybreaks \begin{align*}
\frac{1}{(n-1)!}\int_{\cD}\frac{F(z_1,\dots,z_{n-1})}{\prod_{j=1}^{n-1}\prod_{\alpha=1}^{n}(z_j-x_\alpha)}\prod_{j=1}^{n-1}\d z_j=
(2\pi \ii)^{n-1}\,\sum_{\alpha=1}^{n}\frac{(-1)^{\frac{(n-1)(n-2)}{2}}F(\hat x_{\alpha})}{\,D(\hat x_\alpha)^2\,\prod_{i\not=\alpha}(x_i-x_\alpha)},
\end{align*}}where $\hat x_\alpha$ denotes the $(n-1)$-tuple $(x_1,\dots, x_{\alpha-1},x_{\alpha+1},\dots,x_{n})$ obtained from $x$ by deleting its $\alpha$-th coordinate.  \end{lem}
Suppose the Satake parameter $p^{-s}=(t_{\alpha})_{\alpha=1}^{n}$ of $\pi$ is regular for a while. Then by applying Lemma~\ref{OldF-L3-LL} to the function $F(z)=e_{n-j-1}(z)\,D(z)^2$ and the numbers $x_\alpha=p^{-w}t_\alpha\,(1\leq \alpha \leq n-1)$ with sufficiently large $\Re w$, we obtain 
$$
Z(w,\Wcal_{\sG_0(\Q_p)}^{\psi_p^{-1}}(\check \phi_j)\otimes W^{0}_\pi)=p^{wj}\,\sum_{\alpha=1}^{n}\frac{e_{n-j-1}(\hat t_{\alpha})}{\prod_{i\not=\alpha}(t_i-t_\alpha)}.
$$
By comparing the first column vector of the both sides of ${\mathbb V}(T)^{-1}\,{\mathbb V}(T)=1_{n}$ (see \eqref{SymF-f3}), the last sum is $\delta_{j0}$. Hence, $Z(w,\Wcal_{\sG_0(\Q_p)}^{\psi_p^{-1}}(\check \phi_j)\otimes W_\pi^{0})=p^{wj}\,\delta_{j0}$ for $\Re w \gg 0$, and then for all $w\in \C$ by analytic continuation. Note that the integral $Z(w,\Wcal_{\sG_0(\Q_p)}^{\psi_p^{-1}}(\check \phi_j)\otimes W_\pi^{0})$ converges absolutely for all $w\in \C$ defining an entire function, because $h\mapsto \Wcal_{\sG_0(\Q_p)}^{\psi_p^{-1}}(\check \phi_j;h)$ is compactly supported on $\sG_0(\Q_p)$ modulo $\sU_0(\Q_p)$. Thus from \eqref{CoI-f2}, we have $W^{(j)}(1_n)=\delta_{j0}$ as desired. Having formulas \eqref{OldF-L3-f1} and \eqref{OldF-L3-f2}, we can compute the entries of $F_{\pi}^{(\nu)}$ by \eqref{20201116}. \end{proof}

Next, we shall compute the matrix $G_\pi=(\langle W^{(n-i)}|W^{(n-j)}\rangle_{\Pcal(\Q_p)})_{ij}$ in terms of symmetric polynomials explicitly. For that, we need additional notation. Let $T=(T_1,\dots, T_n),\,Z$ be indeterminates. Set 
$$
 Q(Z;T):=\prod_{j=1}^{n}(1-ZT_j^{-1}) \in \C[T_1^{-1},\dots,T_n^{-1},Z], 
$$
and define matrices $\DD(Z,T),\,{\mathbb P}(Z,T) \in \Mat_{n}(\C(T_1,\dots, T_n,Z))$ by
\begin{align}
\DD(Z,T):={\mathbb V}(T)^{-1}\times \diag(Q(Z T_1;T)^{-1},\dots Q(Z T_n;T)^{-1})\times {\mathbb V}(T),
 \label{DDZT-def}
\\
{\mathbb P}(Z,T):={\mathbb V}(T)^{-1}\times \diag((1-Z T_1)^{-1}, \dots,(1-Z T_n)^{-1})\times {\mathbb V}(T). 
 \notag
\end{align}
\begin{lem}\label{OldF-L4}
There exist matrices $\DD^{*}(Z,T)$ and ${\mathbb P}^{*}(Z,T)$ in $\Mat_n(\C[T_1,\dots,T_n]^{\sS_n}[Z])$ such that 
\begin{align*}
\DD(Z,T)&=\frac{\,\DD^{*}(Z,T)}{e_n(T)^{n-1}\,\prod_{i=1}^{n}\prod_{j=1}^{n}(1-ZT_iT_j^{-1})}, \quad{\mathbb P}(Z,T)=\frac{{\mathbb P}^{*}(Z,T)}{\prod_{i=1}^{n}(1-ZT_i)}. 
\end{align*}
\end{lem}
\begin{proof} Set $D(T)=\prod_{1\leq i<j \leq n}(T_j-T_i)$. Since $D(T)\times {\mathbb V}(T)^{-1}$ has polynomial entries, the entries of the matrix $$A(T):=(T_1\dots T_n)^{n-1}\,D(T)\, \prod_{i,j=1}^{n}(1-ZT_iT_j^{-1}) \times \DD(Z,T)$$ belong to the ring $\C[T_1, \dots, T_n][Z]$. For $w\in \sS_n$, set $wT=(T_{w^{-1}(1)}, \dots, T_{w^{-1}(n)})$. Recall that $w\in \sS_n$ is identified with a permutation matrix in ${\bf GL}_n(\Q)$. By \eqref{wAwinv}, we have $D(wT)=\det(w)\,D(T)$, ${\mathbb V}(wT)=w\,{\mathbb V}(T)$, $Q(Z;wT)=Q(Z;T)$; from these, we easily have $A(wT)=\det(w)\,A(T)$, which means that all the entries of $A(T)$ are alternating polynomials in $T_1,\dots,T_n$ with coefficients in $\C[Z]$. Now apply the fact that for any alternating polynomial $f(T) \in \C[Z][T_1,\dots,T_n]$, there exists a symmetric polynomial $g(T)\in \C[Z][T_1,\dots,T_n]$ such that $f(T)=D(T)\,g(T)$. The existence of ${\mathbb P}^{*}(Z,T)$ is proved in the same way. \end{proof}
Recall the matrix ${\mathbb H}(T)$ defined in \S\ref{sec:SymF}. Set $R(Z):=\diag(Z^{n-i}|\,1\leq i \leq n)$. 

\begin{prop}\label{OldF-L5} Let $p^{-s}=(p^{-s_j})_{j=1}^{n}$ be the Satake parameter of $\pi$. Then, 
\begin{align}
{}^t G_\pi^{-1}=(1-p^{-n})\,(-1)^{n-1}\, \HH(x(s))\times   
\DD^*(p^{-1},p^{-s})\times R(-p). 
 \label{OldF-L5-f0}
\end{align}
\end{prop}
\begin{proof} Let $i,j\in[1,n]_\Z$. Recall the definition \eqref{WhitProd}. Since $|\det h|_p=p^{-i}$ if $\phi_i(h)\not=0$, 
{\allowdisplaybreaks\begin{align*}
&(1-p^{-n})\,\langle W^{(n-i)}|W^{(n-j)}\rangle_{\cP(\Q_p)}
\\
&=\int_{\sU_0(\Q_p)\bsl \sG_0(\Q_p)} W^{(n-i)}(\iota(h))\,\bar W^{(n-j)}(\iota(h))\,\d h
\\
&=p^{-(2n-i-j)/2} \int_{\sU_0(\Q_p)\bsl \sG_0(\Q_p)}\biggl\{ \int_{\sG_0(\Q_p)}W^{0}_\pi(\iota(hx^{-1}))\phi_{n-i}(x)\,\d x \biggr\}\,\biggl\{ \int_{\sG_0(\Q_p)}\bar W^{0}_\pi(\iota(hy^{-1}))\phi_{n-j}(y)\,\d y \biggr\}\,\d h.
\end{align*}}Let $w\in \C$ and consider the integral: 
$$
I_{ij}(w):=\int W^{0}_\pi(\iota(hx^{-1}))\,\bar W^{0}_\pi(\iota(hy^{-1}))\,\phi_{n-i}(x)\,\phi_{n-j}(y)\,|\det x|_p^{-w}\,|\det h|_p^{2w}\,\d h\,\d x\,\d y
$$
over $(h,x,y)\in ({\sU_0(\Q_p)\bsl \sG_0(\Q_p)) \times \sG_0(\Q_p)\times \sG_0(\Q_p)}$. Since $\phi_{n-i}$ and $\phi_{n-j}$ are compactly supported on $\sG_0(\Q_p)$ the absolute convergence of the integral $I_{ij}(w)$ for $\Re w \geq 0$ follows from the convergence of the zeta-integral $\Psi(z,W^{(n-i)},\bar W^{(n-j)},\Phi)$ for $\Re z\geq 1$. By the variable change $(h,x,y)\mapsto (hy,x^{-1}hy,y)$, {\allowdisplaybreaks
\begin{align*}
I_{ij}(w)&=\int W_\pi^{0}(\iota(x))\bar W_\pi^0(\iota(h)) \phi_{n-i}(x^{-1}h y)\,\phi_{n-j}(y)\,|\det h|_p^{w}\,|\det x|_p^{w}\,|\det y|_p^{w}\,d h \,\d x\,\d y \\
&=p^{-(n-j)w}\,\int_{\sU_0(\Q_p)\bsl \sG_0(\Q_p)}\int_{\sG_0(\Q_p)} W_\pi^0(\iota(x))\bar W_\pi^0(\iota(h))\,(\phi_{n-i}*\check \phi_{n-j})(x^{-1}h)\,|\det h|_p^{w}|\det x|_p^{w}\,\d h\,\d x
\\
&=p^{-(n-j)w}\,\int_{\sU_0(\Q_p)\bsl \sG_0(\Q_p)}\int_{\sU_0(\Q_p)\bsl \sG_0(\Q_p)} W_\pi^0(\iota(x))\bar W_\pi^0(\iota(h))\\
&\qquad\quad  \times \biggl\{\int_{\sU_0(\Q_p)} \psi_{\sU_0,p}(u)^{-1}\,\phi_{ij}(x^{-1} u h)\,\d u \biggr\}\,|\det h|_p^{w}|\det x|_p^{w}\,\d h\,\d x,
\end{align*}}where $\phi_{ij}=\phi_{n-i}*\check \phi_{n-j}$. From the same computation in the proof of Proposition~\ref{LNVpadicL1} transcribed for the group $\sG_0(\Q_p)$, {\allowdisplaybreaks\begin{align*}
&\int_{\sU_0(\Q_p)}\psi_{\sU_0,p}(u)^{-1}\,\phi(x^{-1}u h)\,\d u
\\
&=
\int_{\fX_{\sG_0,p}^{0}}
 \widehat \phi(\ii \nu)\,\biggl\{\int_{\sU_0(\Q_p)}^{\rm st} \Omega_{\sG_0(\Q_p)}^{(-\ii \nu)}(x^{-1}u h)\psi_{\sU_0,p}(u)^{-1} \d u\biggr\}\,\d\mu_{\sG_0(\Q_p)}^{\rm Pl}(\nu)
\\
&=\int_{\fX_{\sG_0,p}^{0}}
 \widehat \phi(\ii \nu)\,\Delta_{\sG_0,p}(1)L(1,I_p^{\sG_0}(\ii \nu)\times I^{\sG_0}_p(-\ii \nu))^{-1}W_{\sG_0(\Q_p)}^{0}(-\ii\nu;h)\,\overline{W_{\sG_0(\Q_p)}^0(-\ii \nu;x)}
\,\d\mu_{\sG_0(\Q_p)}^{\rm Pl}(\nu)
\\
&=\frac{(2\pi)^{1-n}(\log p)^{n-1}}{(n-1)!}\,\int_{\fX_{\sG_0,p}^{0}}
 \widehat \phi(\ii \nu)\,
W_{\sG_0(\Q_p)}^{0}(-\ii\nu;h)\,\overline{W_{\sG_0(\Q_p)}^0(-\ii \nu;x)} 
|D(p^{-\ii \nu})|^2\,\prod_{j=1}^{n-1}\d\nu_j. 
\end{align*}}Substitute this expression to the last formula of $I_{ij}(w)$. Then by the absolute convergence of the zeta-integrals $Z(w,W_{\sG_0(\Q_p)}^{0}(\ii\nu)\otimes \bar W_\pi^{0}))$ for $\Re w \gg 0$, we can apply Fubini's theorem to change the order of $\nu$-integral and the $(x,h)$-integral. Thus, by \eqref{URZetaInt}, $(\log p)^{1-n}\times I_{ij}(w)$ equals 
{\allowdisplaybreaks \begin{align*}
&\frac{p^{(j-n)w}(2\pi)^{1-n}}{(n-1)!} \int_{\fX_{\sG_0,p}^{0}}
 \widehat \phi_{ij}(-\ii \nu)\,\{\prod_{j=1}^{n-1}\overline{L(\bar w+\ii\nu_j+1/2,\bar \pi)}\,{L(w+\ii \nu_j+1/2, \bar \pi)}\}\,|D(p^{-\ii \nu})|^2\,\prod_{j=1}^{n-1}\d\nu_j.
\end{align*}}We set $z_l=p^{-\ii \nu_l}\,(1\leq l \leq n-1)$ and $x_\alpha=p^{-s_\alpha}\,(1\leq \alpha \leq n)$. Then $\widehat \phi_{ij}(-\ii \nu)=(\prod_{l=1}^{n-1}z_l)^{-1}e_{i-1}(z)\, e_{n-j}(z)$, $|D(p^{-\ii \nu})|^2=(-1)^{\frac{(n-1)(n-2)}{2}}\prod_{l=1}^{n-1}z_l^{-(n-2)}\,D(z)^2$. Thus $I_{ij}(w)$ equals
{\allowdisplaybreaks\begin{align*}
&\frac{p^{(j-n)w}(2\pi)^{1-n}}{(n-1)!} \int_{\cD} \frac{(-1)^{\frac{(n-1)(n-2)}{2}}{\mathcal F}_{ij}(x;z)}
{\prod_{\substack{1\leq l \leq n-1 \\ 1\leq \alpha\leq n}}(z_l-p^{-\frac{1}{2}-w}x_\alpha)}\,\prod_{l=1}^{n-1}\d z_l \quad \text{with} \quad 
{\mathcal F}_{ij}(x;z):=\frac{e_{i-1}(z)e_{n-j}(z)\,D(z)^2}{\prod_{\substack{1\leq l \leq n-1 \\ 1\leq \alpha\leq n}}(1-p^{-\frac{1}{2}-w}x_\alpha^{-1} z_l)}. 
\end{align*}}Suppose the Satake parameter $p^{-s}=(x_\alpha)_{\alpha}$ of $\pi$ is regular. Then by applying Lemma~\ref{OldF-L3-LL} to $F(z)={\mathcal F}_{ij}(x;z)$ and the points $p^{-1/2-w}x_\alpha$ with sufficiently large $\Re w$, we obtain the following expression of $I_{ij}(w)$:  
{\allowdisplaybreaks\begin{align*}
&p^{(j-n)w}\,(p^{1/2+w})^{j-i}\,\sum_{\alpha=1}^{n} \frac{e_{i-1}(\hat x_\alpha)\,e_{n-j}(\hat x_\alpha)}{\prod_{\substack {1\leq l \leq n \\ l \not=\alpha}}(x_l-x_\alpha)\,\prod_{\substack{1\leq l\leq n \\ l\not=\alpha}}\prod_{m=1}^{n}(1-p^{-1-2w}x_m^{-1} x_l)}\\
&=(-1)^{n-j}p^{(2j-i-n)w}\, p^{\frac{j-i}{2}}\,\prod_{l=1}^n\prod_{m=1}^{n}(1-p^{-1-2w}x_m^{-1}x_l)^{-1}\,\sum_{\alpha=1}^{n} \frac{(-1)^{n-j}e_{n-j}(\hat x_\alpha)}{\prod_{\substack {1\leq l \leq n \\ l \not=\alpha}}(x_l-x_\alpha)}\,e_{i-1}(\hat x_\alpha)\,Q(p^{-1-2w}x_\alpha; x)
. 
\end{align*}}Set $G_{\pi}(w):=((1-p^{-n})^{-1}\,p^{w(n-2j+i)-(2n-i-j)/2}\,I_{ij}(w))_{ij}$ so that $G_{\pi}(0)=G_\pi$. Then by \eqref{SymF-f3} and the last formula of $I_{ij}(w)$, we have the matrix identities holding for $\Re w\gg 0$: 
{\allowdisplaybreaks\begin{align*}&\,(1-p^{-n})\, {}^tG_{\pi}(w)
\\
&=\prod_{l=1}^n\prod_{m=1}^{n}(1-p^{-1-2w}x_m^{-1}x_l)^{-1} \times R(-p^{-1})\,
{\mathbb V}(x)^{-1}\, \diag(Q(q^{-1-2w}x_j;x)\mid j\in [1,n]_\Z)
\,{\mathbb E} (x)
\\
&=\prod_{l=1}^n\prod_{m=1}^{n}(1-p^{-1-2w}x_m^{-1}x_l)^{-1}\times R(-p^{-1})\, {\mathbb D}(p^{-1-2w},x)^{-1}\, {\mathbb H}(x)^{-1}
\\
&=R(-p^{-1})\,{\mathbb D}^*(p^{-1-2w},x)^{-1}\, {\mathbb H}(x)^{-1}.
 \end{align*}}Here, in the last step, we use Lemma~\ref{OldF-L4} noting $e_{n}(x)=\prod_{l=1}^{n}x_l=1$ from the triviality of the central character of $\pi$. Note that, since ${\mathbb H}(x)$ is an upper-triangular matrix with diagonal entries being $\pm 1$, it is invertible. From Lemma~\ref{OldF-L4}, ${\mathbb D}^{*}(p^{-1-2w},x)$  is a polynomial in $p^{-2w}$. Since all $I_{ij}(w)$ are holomorphic on the interior of the absolute convergence region $\Re w\geq 0$, the equality ${}^tG_{\pi}(w)^{-1}=(1-p^{-n}){\mathbb H}(x)\,{\mathbb D}^*(p^{-1-2w},x)\,R(-p)$ holds for $\Re w>0$. By taking the limit $w \rightarrow +0$, we obtain the identity \eqref{OldF-L5-f0} for $\pi$ with regular Satake parameter. By Lemma~\ref{OldF-L4}, the general case is obtained from the regular case by continuity. 
\end{proof} 

\begin{cor} \label{OldF-L6} Let $p^{-s}=(p^{-s_j})_{j=1}^{n}$ be the Satake parameter of $\pi$, and $\nu=(\nu_i)_{i=1}^{n-1} \in \C^{n-1}$. 
\begin{align*}
\PP_{\psi_p}^{(\nu)}(\pi,\bK_1(p\Z_p);\cchi_{\bK_1(p\Z_p)})=\tfrac{1}{p^{n}}\,{L(1,\pi \times \bar \pi)}\,\sum_{i=1}^{n}(-1)^{i-1}(\DD^{*}(p^{-1},p^{-s}))_{ni}\,p^{(n-i)/2}\,e_{n-i}(p^{-\nu}),
\end{align*}
where $(A)_{ni}$ denotes the $(n,i)$-entry of a matrix $A\in \Mat_n(\C)$. 
\end{cor}
\begin{proof} This follows from Lemma \ref{OldF-L2} and Propositions \ref{OldF-L3} and \ref{OldF-L5}. Note that ${}^t{\mathbb H}(p^{-s})\, F_{\pi}^{(\nu)}=(-1)^{n-1}F_{\pi}^{(\nu)}$ because the first $n-1$ rows of $F_{\pi}^{(\nu)}$ are zero, and that ${}^t{\mathbb H}(x)$ is a lower-triangular matrix whose $(n,n)$-th entry is $(-1)^{n-1}$.\end{proof}

\subsection{The proof of Proposition ~\ref{OldF-L7}} \label{sec:PrfOldF-L7} 
For $m\in \N_0$, set  
\begin{align}
{\mathbb P}_m(T): ={\mathbb V}(T)^{-1}\times \diag(T_j^{m}\mid j\in [1,n]_\Z)\times {\mathbb V}(T) \quad \in \Mat_{n}(\C(T_1,\dots,T_n)).
 \label{PrfOldL-L7-f2}
\end{align}The same argument as in the proof of Lemma~\ref{OldF-L4} shows that all the entries of the matrix ${\mathbb P}_m(T)$ are symmetric polynomials in $T_1,\dots,T_n$. From the formula of $\mathbb V(T)^{-1}$ (see \eqref{SymF-f3}), the $(i,j)$-th entry ${\mathbb P}_m(T)_{ij}$ has the rational expression
$\sum_{\alpha=1}^{n}\frac{(-1)^{n-i}e_{n-i}(\hat T_\alpha)}{\prod_{l\not=\alpha}(T_\alpha-T_l)}\,T_\alpha^{m+j-1}$, by which the homogeneity ${\mathbb P}_{m}(tT)_{ij}=t^{m+j-i}\,{\mathbb P}_{m}(T)_{ij}\,(\forall t\in \C^\times)$ is easily confirmed. Recall the polynomial $h_m(T)$, which is homogeneous of degree $m$.
\begin{lem} \label{PrfOldL-L7-L1}
 Let ${\mathcal T}:=\{t=(t_j)_{j=1}^{n}\in \C^{n}\mid |t_j|\leq 1\,(j\in [1,n]_\Z)\}$. 
Then for any $\theta>0$ there exists a constant $M_{\theta}>0$ such that 
\begin{align*}
|h_{m}(\bar t)|\leq M_{\theta}\,(2^{\theta/2})^{m}, \quad \max_{i,j}|{\mathbb P}_m(t)_{ij}|\leq M_\theta\,(2^{\theta/2})^{m} \quad (t\in \mathcal T, \, m\in \N_0).
\end{align*} 
\end{lem}
\begin{proof} The trivial inequality $|h_{m}(\bar t)|\leq \binom{m+n-1}{n-1}$ for $t\in \mathcal T$ yields the desired bound for $h_m$. From Lemma~\ref{OldF-L4} and \eqref{PrfOldL-L7-f2}, we have the series expansion ${{\mathbb P}^*(z,t)}{\prod_{i=1}^{n}(1-z t_i)}^{-1}=\sum_{m=0}^{\infty}{\mathbb P}_m(t)\,z^{m}$ convergent on $|z|<2^{-\theta/3}$ where the left-hand side is holomorphic. Cauchy's formula yields the inequality $\max_{i,j}|{\mathbb P}_m(t)_{ij}|\leq M_{\theta}'' (2^{\theta/2})^{m}$ for all $t\in \mathcal T$ and $m\in \N_0$, where $M_\theta''$ is the maximum of $\max_{i,j} \|{\mathbb P}^{*}(z,t)_{ij}|\prod_{i=1}^{n}|1-z t_i|^{-1}$ for $t\in {\mathcal T}$ and $|z|=2^{-\theta/2}$. 
\end{proof}
Let ${\mathcal Z}$ denote the subset of all those points $(z,x)\in \C \times \C^{n}$ such that $|z|<\min_{(i,j)}\mid x_i\bar x_j|^{-1}$, $e_n(x)\not=0$ and that $x^{-1}:=(x_j^{-1})_{j=1}^{n}$ and $\bar x=(\bar x_j)_{j=1}^{n}$ belong to the same $\sS_n$-orbit. Set ${\mathcal Z}^{*}:=\{(z,x)\in {\mathcal Z}\mid D(x)\not=0\,\}$. Then for $(z,x)\in {\mathcal Z}^*$, from \eqref{DDZT-def} and \eqref{SymF-f0},  
\begin{align}
{\mathbb D}(z,x)&= {\mathbb V}(x)^{-1}\diag\bigl( \sum_{m=0}^{\infty} h_{m}(x^{-1})\,x_i^{m}z^{m}\,\bigm|\,i\in [1,n]_{\Z}\bigr){\mathbb V}(x)=\sum_{m=0}^{\infty}h_{m}(\bar x)\,{\mathbb P}_{m}(x)\,z^{m},
 \label{PrfOldL-L7-f1}
\end{align}
which is absolutely convergent. Note that $h_{m}(x^{-1})=h_{m}(\bar x)$ because $x^{-1}$ and $\bar x$ are $\sS_n$-equivalent. Lemma~\ref{OldF-L4} shows that for a fixed $x$ with $e_n(x)\not=0$ the function $z\mapsto \DD(z,x)$ is holomorphic on $|z|<\min_{(i,j)}\mid x_i\bar x_j|^{-1}$ and that ${\mathbb D}(z,x)$ is continuous on ${\mathcal Z}$. Since ${\mathcal Z}^{*}$ is dense in ${\mathcal Z}$, by the Cauchy estimate, the expansion \eqref{PrfOldL-L7-f1} holds true for $(z,x)\in {\mathcal Z}$. From \cite{LRS}, for any $\pi\cong \otimes_{v}\pi_v\in \Pi_{\rm cusp}(\sG)_\sZ$ and any prime number $p$ where $\pi_p$ is unramified, the spectral parameter $(s_j)_{j=1}^{n}$ of $\pi_p$ satisfies $|\Re s_j|\leq \frac{1}{2}-\frac{1}{n^2+1}\,(j\in [1,n]_\Z)$. Set $\theta:=\frac{1}{n^2+1}$. Hence the point $(p^{-1},p^{-s})$ with $p^{-s}:=(p^{-s_j})_{j=1}^{n}$ belongs to the set ${\mathcal Z}$. Set $t_{p}(\pi):=p^{\theta-1/2}\,p^{-s}$ so that the points $t_p(\pi)$ and $\bar t_p(\pi)$ belongs to the set ${\mathcal T}$ in Lemma~\ref{PrfOldL-L7-L1}. Then by \eqref{PrfOldL-L7-f1} and the homogeneity of $h_{m}(T)$ and ${\mathbb P}_m(T)_{ij}$ remarked above and by $e_n(x_p(\pi))=1$, 
{\allowdisplaybreaks\begin{align*}
&|L(1,\pi_p \times \bar\pi_p)\,\DD^{*}(p^{-1},x_p(\pi))_{ij}|
=|\DD(p^{-1},x_p(\pi))_{ij}|
\\
&=|\sum_{m=0}^{\infty} (p^{1/2-\theta})^{m}\,h_m(\bar t_{p}(\pi))\times (p^{1/2-\theta})^{m+j-i}\,{\mathbb P}_m(t_p(\pi))_{ij}\,p^{-m}|\\
&\leq (p^{1/2-\theta})^{j-i} \sum_{m=0}^{\infty} |h_{m}(\bar t_p(\pi))\,{\mathbb P}_m(t_p(\pi))_{ij}|\,p^{-2\theta m }
\leq (p^{1/2-\theta})^{n-1} \sum_{m=0}^{\infty} |h_{m}(\bar t_p(\pi))\,{\mathbb P}_m(t_p(\pi))_{ij}|\,2^{-2\theta m}.  
\end{align*}}By Lemma ~\ref{PrfOldL-L7-L1}, the last series is dominated by $M_{\theta}^2 \sum_{m=0}^{\infty}(2^{\theta/2})^{m}\times (2^{\theta/2})^{m}\times 2^{-2\theta m}=M_{\theta}^2 \sum_{m=0}^{\infty} 2^{-\theta m}<\infty$ uniformly. Hence there exists a constant $C>0$ independent of $p$ and $\pi$ such that $|L(1,\pi_p \times \bar \pi_p)\,\DD^*(p^{-1},x_p(\pi))_{in}|\leq C\,p^{(n-1)/2}\times p^{-(n-1)\theta}$. We apply this estimate to the formula in Corollary~\ref{OldF-L6}, noting $|e_{l}(p^{-\nu})|=O_{l}(1)$ for $\Re \nu_j \geq 0\,(1\leq j \leq n-1)$, to complete the proof of Proposition~\ref{OldF-L7}.

\section{Proof of main results} \label{sec:PfMTHM}
In this section, we give proofs of the theorems stated in \S\ref{Intro}. For any $M'\in \N$, set $K_{M'}=\prod_{p\in S(M')}\bK_1(M'\Z_p)$. Note that $\bK_{1}(M')=K_{M'}\,\prod_{p\not\in S(M')}\bK_p$. Let $f_{M}=\otimes_{p\in S(M)}f_p$ and $f_{N}^0=\otimes_{p \not\in S_0\cup S(M)}f_p$ be the functions specified in \S\ref{GWP}. 

\subsection{Proof of Theorem~\ref{MAINTHM1}}\label{PfMTHM1} 
Recall the quantity ${\bf I}_{f}(\nu)$ defined in \S\ref{sect:spectralside} and ${\bf I}_{\psi}(\nu;\bK_1(MN),f)$ defined by \eqref{GlobalWhittPer}; we have ${\bf I}_f(\nu)=(\Delta_{\sG_0}(1)^{*})^{-1}{\bf I}_\psi(\nu;\bK_1(MN),f)$ from Lemmas~\ref{AverageWHittPerL1} and \ref{unramifiedPer}. From definitions, 
\begin{align}
\sum_{\varphi \in \Bcal(\pi^{\bK_1(MN)\bK_\infty})} Z\left(\tfrac{1}{2},\hat E(\nu),\bar\varphi *\check f_{M,N}\right)\,\Wcal^\psi(\varphi)=\frac{\Delta_{\sG}(1)^*}{\Delta_{\sG_0}(1)^*}\, A_{\psi}^{(\nu)}(\pi;K_{MN},f_{M,N}),
 \label{PfMTHM1-f1}
\end{align}
where $f_{M,N}=\otimes_{p\in S(MN)}f_p$. Recall the quantities ${\bf J}_f^\circ(\nu)$ defined by \eqref{Jcircnu} and ${\bf J}_{\psi}^0(\nu;f_{S_0,\infty})$ defined in Theorem~\ref{MAINTHM1}. Applying Corollary~\ref{ZWW=1}, Lemmas~\ref{LNVpadicL1} and \ref{LNVArchL1} to \eqref{Jcircnu}, we have the relation ${\bf J}_f^{\circ}(\nu)=(\Delta_{\sG_0}(1)^{*})^{-1}{\bf J}_\psi^{0}(\nu;f_{S_0,\infty})$ if $\Re\,\nu_j>-1/2\,(j\in [1,n-1]_\Z)$. Now Theorem~\ref{MAINTHM1} results from the identity \eqref{RTF-f1} in Theorem~\ref{RTF}. \qed

\subsection{Proof of Theorem~\ref{MAINTHM2.5}} 
Set $S=S_0\cup\{\infty\}$, $\fX_S^{0}=\prod_{v\in S}\fX_v^0$ and $W_S=\prod_{v\in S}\sS_n$, where $\fX_v^0$ is the space defined in \S\ref{sec:URps}. The group $W_S$ acts on $\fX_S^{0}$ in such a way that $\fX_S^0/W_S\cong \prod_{p\in S}(\fX_v^0/\sS_n)$. For a prime number $p$ and $s=(s_j)_{j=1}^{n} \in \C^n$, set $p^{-s}=(p^{-s_j})_{j=1}^{n}\in(\C^\times)^{n}$. Let $X=(X_1,\dots,X_n)$ be a set of indeterminates. To attain the non self-duality condition (iii) in Theorem \ref{MAINTHM2.5}, we introduce polynomials $P_{v}(X)\in \C[X_1,\dots,X_n]$ for $v\in S_0\cup\{\infty\}$. For $p\in S_{0}$, define $P_p(X):= \prod_{w\in \sS_n}\prod_{j=1}^{n}(1-X_jX_{w(j)})$. Then $P_p(X)$ is an $\sS_n$-invariant polynomial with the property: $P_p(p^{-\ii s})=0$ for all $s\in \fX_p^0$ such that $s=-ws$ for some $w\in \sS_n$. The subset $\fX_p^0(1)$ is defined as $\sum_{j=1}^{n} s_j \equiv 0 \pmod{2\pi (\log p)^{-1}\Z}$, which is the locus of the equation $X_1\cdots X_n-1=0$ with $X_j=p^{-s_j}\,(j\in[1,n]_\Z)$. Note that $X_1\cdots X_n-1$ is not a divisor of $P_p(X)$ for $n\geq 3$. As is well-known, the map $f\mapsto {\widehat f}(s)$ with $X=p^{-s}$ yields the Satake isomorphism $C_{\rm c}^{\infty}(\bK_p\bsl \sG(\Q_p)/\bK_p)\rightarrow \C[X_1, X_1^{-1},\dots,X_n,X_n^{-1}]^{\sS_n}$; thus there exists $f_p\in C_{\rm c}^{\infty}(\bK_v\bsl \sG(\Q_v)/\bK_v)$ such that $\widehat {f_p}(s)=P_{p}(p^{-s})\, (s\in (\C/2\pi \ii (\log p)^{-1}\Z)^{n})$. Set $P_{\infty}(X)=\prod_{w\in \sS_n}\prod_{j=1}^{n}(X_j+X_{w(j)})$, which is an $\sS_n$-invariant polynomial with the property: $P_\infty(\ii s)=0$ for all $s\in \fX_\infty^0$ such that $s=-ws$ for some $w\in \sS_n$. From \cite[Propositions 8.20 and 8.22]{Knapp}, corresponding to $P_\infty(X)$, there exists an element $D$ in the center of $U(\fg_\infty)$ such that $\widehat{h_\infty*D}(s)=P_\infty(s)\,\widehat{h_\infty}(s)\, (s\in \C^n)$ for all $h_\infty\in C_{\rm c}^{\infty}(\bK_\infty\bsl \sG(\R)/\bK_\infty)$. Set $\fX_\infty(1):=(\fX_\infty^0(1))_{\C}$. Note that the map $h_{\infty}\mapsto \widehat{h_\infty}|\fX_\infty(1)$ is the composite of the map $h_\infty\mapsto \widetilde{h_\infty}$ (defined by \eqref{CentralProjLoc}) and the spherical Fourier transform of ${\rm PGL}_n(\R)$ which is injective (\cite[Proposition 3.8(Ch.IV \S3]{Helgason}). Hence the holomorphic function $\widehat{h_\infty}(s)$ on $\fX_\infty(1)$ cannot be identically zero on any non-empty open subset of $\fX_\infty(1)$ unless $\widetilde{h_{\infty}}=0$ identically. Note that $\fX_\infty(1)$ is the zero set of the linear polynomial $\sum_{j=1}^{n}X_j$, which is not a divisor of $P_{\infty}(X)$ for $n\geq 3$. Therefore, for any $h_{\infty}$ such that $\widetilde{h_\infty}\not=0$, $P_{\infty}(s)\widehat{h_\infty}(s)$ is not zero on any non-empty open subset of $\fX_\infty(1)$. Fix such an $h_\infty$ and set $f_{\infty}:=h_\infty*D$. Having $f_v\,(v\in S)$, we set $f_S:=\otimes_{v\in S}f_v$. Put 
$$
\fL_S^{(\nu)}(\bfs)=
\prod_{v\in S}\biggl\{W_{\sG(\Q_v)}^{0}(\ii s_v;1_n)\,\Delta_{\sG,v}^{(\infty)}(1)
\frac{\prod_{j=1}^{n-1}L\left(\tfrac{1}{2}+\nu_j,I_v^{\sG}(-\ii s_v)\right)}
{\zeta_v(1)L(1,I_v^{\sG}(\ii s_v);\Ad)}\biggr\}, \quad \bfs=(s_v)_{v\in S} \in \fX_S^{0},
$$
where $\Delta_{\sG,v}^{(\infty)}(1)$ denotes $\Delta_{\sG,v}(1)$ or $1$ according to $v<\infty$ or $v=\infty$. Since the function $\widehat{f_{S}}(\ii \bfs)$ is non-zero on any non empty $W_S$-invariant open subset of $\fX_S^0(1)$, by Lemmas~\ref{LNVpadicL2} and \ref{LNVArchL3}, we can find $\alpha_{S}\in C_{\rm c}^{\infty}({\fX_S^{0}}/W_S)$ such that 
\begin{align}
L:=\int_{\fX_S^{0}(1)} \alpha_S(\bfs)\,\widehat{f_{S}}(-\ii \bfs)
\,\fL_S^{(\nu)}(\bfs)\,\d\mu_S^{\rm Pl}(\bfs) \not=0
,
\label{ProofMAINTHM1-f2}
\end{align}
where $\d\mu_S^{\rm Pl}=\otimes_{v\in S}\d\mu_v^{\rm Pl}$. From \eqref{LNVAr-L10-f3}, we see that the $\infty$-factor of $\widehat{f_S}(-\ii \bfs)\,\fL_S^{(\nu)}(\bfs)$ is
$$
\,\widehat{f_\infty}(-\ii s)\,J_{\sG(\R)}^{\psi_\infty}(\ii s;1_n)\,{M_{\sG,\infty}(-\ii s)}^{-1}\,{\prod_{j=1}^{n-1} L\left(\tfrac{1}{2}+\nu_j, I_\infty^{\sG}(-\ii s)\right)}.
$$
Hence by Lemmas \ref{SpectExpPerL-4-1}, \ref{JIest} and \ref{LNVAr-L10-L}, it is easy to confirm that the function $\widehat{f_{S}}(-\ii \bfs)\fL_S^{(\nu)}(\bfs)$ is bounded by a constant ${\rm K}>0$ on $\fX_S^0(1)$. As recalled above, the space $\{\widehat {\xi_p}\mid \xi_p \in C_{\rm c}^{\infty}(\bK_p\bsl \sG(\Q_p)/\bK_p)\}$ with $p\in S_0$ coincides with $\C[X_1,X_1^{-1},\dots,X_n,X_n^{-1}]^{\sS_n}$, which is dense in $C_{\rm c}^{\infty}(\fX_p^0/\sS_n)$ by the Weierstrass approximation theorem. From \cite[Theorem 7.1 (Chap. IV \S7)]{Helgason}, $\{\widehat {\xi_\infty}\mid \xi_\infty \in C_{\rm c}^{\infty}(\bK_\infty\bsl \sG(\R)/\bK_\infty)\}$ coincides with the $\sS_n$-invariant part of the Paley-Wiener space ${\rm PW}(\C^{n})$ over $\C^{n}$; since $C_{\rm c}^{\infty}(\R^{n})$ is dense in the Schwartz space $\Scal(\fX_\infty^0)$ of $\fX_\infty^0=\R^{n}$, the space ${\rm PW}(\C^{n})$ as the Fourier image of $C_{\rm c}^{\infty}(\R^{n})$ is also dense in $\Scal(\fX_\infty^0)$. Hence we can find functions $\xi_v\in C_{\rm c}^{\infty}(\bK_v\bsl \sG(\Q_v)/\bK_v)$ for $v\in S$ such that \begin{align}
\sup_{\bfs\in \fX_{S}^{0}}|\widehat{\xi_S}(-\ii \bfs)-\alpha_S(\bfs)|\leq ({2{\rm K}})^{-1}|L|,
\label{ProofMAINTHM1-f0}
\end{align}where $\xi_{S}:=\otimes_{v\in S}\xi_v$ and $\widehat{\xi_S}:=\otimes_{v\in S}\widehat{\xi_v}$. From Theorem~\ref{MAINTHM1} applied with the test function whose $S$-component is $f_S*\xi_S=\otimes_{v\in S}(f_{v}*\xi_v)$, we have a constant $N_0=N_0(f_{M}, f_{S}*\xi_{S})$ such that 
\begin{align}
&(N-1)\,{\bf I}_{\psi}(\nu;\bK_1(MN), f_{N}^0\otimes f_{M}\otimes (f_S*\xi_S)) ={\bf J}_{\psi}^{0}(\nu; f_S*\xi_S),
 \label{ProofMAINTHM1-f1}
\end{align} 
for all prime numbers $N>N_0$ with $N\not\in S(M)\cup S_0$. From now on $N$ denotes one of such prime numbers. Let ${\bf I}^{*}_{\psi}(\nu;\bK_1(MN), f_{N}^0\otimes f_{M}\otimes (f_S*\xi_S))$ be the quantity obtained from ${\bf I}_{\psi}(\nu;\bK_1(MN), f_{N}^0\otimes f_{M}\otimes (f_S*\xi_S))$ by reducing the summation range in the defining formula \eqref{GlobalWhittPer} to only those $\pi \in \Pi_{\rm cusp}(\sG)_\sZ^{\bK_1(MN)\bK_\infty}$ with $c(\pi_N)=N$, where $c(\pi_N)$ is the conductor of $\pi_N$ (see \S\ref{sec:Lfunction}). Note that $\pi \cong \otimes_{v}\pi_v \in \Pi_{\rm cups}(\sG)_\sZ^{\bK_1(MN)\bK_\infty}$ and $c(\pi_N)\not=N$ implies $c(\pi_N)=1$, or equivalently $\pi_N$ is spherical. Then by Proposition~\ref{OldF-L7}, we obtain 
{\allowdisplaybreaks\begin{align*}
&|{\bf I}_{\psi}(\nu;\bK_1(MN), f_{N}^0\otimes f_{M}\otimes (f_S*\xi_S))
-{\bf I}_{\psi}^{*}(\nu;\bK_1(MN), f_{N}^0\otimes f_{M}\otimes (f_S*\xi_S))|
\\
&\leq \sum_{\pi \in \Pi_{\rm cusp}(\sG)_\sZ^{\bK_1(M)\bK_\infty}} |\widehat {f_S*\xi_S}(\bar \pi_{S})|\,|A_{\psi}^{(\nu)}(\pi;\bK_{M}, f_M)| \times |\PP_{\psi_N}^{(\nu)}(\pi_N,\bK_1(N\Z_N);\cchi_{\bK_1(N\Z_N)})|
\\
&\leq C\,N^{-1-(n-1)/(n^2+1)}\times  \sum_{\pi \in \Pi_{\rm cusp}(\sG)_\sZ^{\bK_1(M)\bK_\infty}} |\widehat {f_S*\xi_S}(\bar \pi_{S})|\,|A_{\psi}^{(\nu)}(\pi;K_M ,f_M)|,
\end{align*}}where the last sum over $\pi\in \Pi_{\rm cusp}(\sG)_\sZ^{\bK_0(M)\bK_\infty}$ is convergent by Proposition~\ref{SpectExpPerL-3} and may be viewed as a constant independent of $N$. From this and \eqref{ProofMAINTHM1-f1}, we have a constant $C'>0$ independent of $N$ such that 
\begin{align}
|{\bf J}_{\psi}^{0}(\nu; f_S*\xi_S)-(N-1)\,{\bf I}_{\psi}^{*}(\nu;\bK_1(MN), f_{N}^0\otimes f_{M}\otimes (f_S*\xi_S))|\leq C'\,N^{-(n-1)/(n^2+1)}
 \label{ProofMAINTHM1-f10}
\end{align}
for all prime numbers $N>N_0$ such that $N\not\in S(M)\cup S_0$. From \eqref{ProofMAINTHM1-f2} and \eqref{ProofMAINTHM1-f0}, 
{\allowdisplaybreaks\begin{align*}
\left|{\bf J}_{\psi}^0(\nu;f_S*\xi_S)-L\right|
&\leq \int_{\fX_S^0(1)}|\widehat{\xi_S}(-\ii\bfs)-\alpha_S(\bfs)|\times |\widehat {f_S}(-\ii\bfs)|\,|\fL_S^{(\nu)}(\bfs)|\,\d\mu_S^{\rm Pl}(\bfs)
\leq 2^{-1}|L|,
\end{align*}}which combined with \eqref{ProofMAINTHM1-f10} yields 
$$|(N-1){\bf I}_{\psi}^{*}(\nu;\bK_1(MN), f_{N}^0\otimes f_{M}\otimes (f_S*\xi_S))-L|\leq 2^{-1}|L|+C'N^{-(n-1)/(n^2+1)}
$$
for all $N>N_0$ as above. From this and $|L|>0$, there exists a constant $N_1$ such that ${\bf I}^{*}_{\psi}(\nu;\bK_1(MN), f_{N}^0\otimes f_{M}\otimes (f_S*\xi_S))\not=0$ for all primes $N>N_1$, $N\not\in S(M)\cup S_0$. For such an $N$, there exists $\pi\cong \otimes_{v}\pi_v \in \Pi_{\rm cusp}(\sG)_\sZ^{\bK_1(MN)\bK_\infty}$ such that $c(\pi_N)=N$, $\widehat{f_S}(-\ii \nu_{S}(\pi))\not=0$ and $A_\psi^{(\nu)}(\pi;K_{MN},f_{M,N})\not=0$, where $\ii\nu_v(\pi)\in \ii\fX_v^{0+}(1)$ is the spectral parameter of $\pi$ at $v\in S$ (\S\ref{sec:URps}) and $\nu_{S}(\pi):=(\nu(\pi_v))_{v\in S}$. From $\widehat {f_S}(-\ii\nu_{S}(\pi))\not=0$, we have $P_p(p^{-\ii\nu_p(\pi)})\not=0$ for all $p\in S_0$ and $P_\infty(\ii\nu_\infty(\pi))\not=0$, which imply $\nu_v(\pi)\not=-w\nu_v(\pi)\,(\forall w\in \sS_n)$ for all $v\in S$. Since $I^{\sG}_v(\ii s)^\vee \cong I^{\sG}_v(\ii s')$ with $s,s'\in \fX_v^{0+}(1)$ if and only if $s=-ws'$ for some $w\in \sS_n$, we conclude $\pi_v$ is not self-dual for all $v\in S$. From $A^{(\nu)}_{\psi}(\pi;K_{MN},f_{M,N})\not=0$ and \eqref{PfMTHM1-f1}, there exists $\varphi\in \Bcal(\pi^{\bK_1(MN)\bK_\infty})$ such that $\bar \varphi*\check{f}_{M,N} \not=0$. Thus the operator $\pi_{p}({f}_p)$ is non-zero for all $p\in S(M)$. Since ${f_p}$ is a matrix coefficient of a supercuspidal representation ${\tau_p}$, from \cite[Proposition (A 3.g, p.58)]{DKV}, we conclude $\pi_p\cong \tau_p$ for all $p\in S(M)$. From \eqref{GlobalWhittPer-f1}, we also have $\prod_{j=1}^{n-1} L\left(\tfrac{1}{2}+\nu_j,\bar \pi\right)\not=0$. \qed

%\smallskip
%\noindent
%{\bf Remark}: For $v\in S_0 \cup \{\infty\}$, let $\Fcal_v$ be an $\sS_{n}$-invariant subset of the zero-locus in $\fX_v^{0}(1)$ of an $\sS_n$-invariant non-zero polynomial $Q_v(X)\in \C[X_1,X_1^{-1},\dots,X_{n},X_n^{-1}]^{\sS_n}$ relatively prime to $X_1\cdots X_n-1$. Then by the same argument as above replacing $P_v(X)$ with $P_v(X)Q_v(X)$, we can find $\pi$ in Theorem~\ref{MAINTHM2.5} such that $\pi_v\cong I_v^{\sG}(\ii s_v)$ with $s_v\in \fX_v^{0+}(1)-\Fcal_v$ for all $v\in S_0\cup \{\infty\}$. 

%%%%%%%%%%%%%%%%%%%%%%%%%%%%
\section*{Acknowledgments}
%%%%%%%%%%%%%%%%%%%%%%%%%%%%

The author thanks the anonymous referee for his/her careful reading of the manuscript and valuable comments. This research was supported by Kakenhi (C) 15K04795.

%%%%%%%%%%%%%%%%%%%%%%%%%%%%%%%%%%%%%%%%%%%%%%%%%%%%%%%%%%%%%%%%%%%%%%%%%%%%%%%%

\end{document}